%% file: SpV1-memoirs.tex
\documentclass{memo-l}

\usepackage{}

\input{macros.tex}

\newtheorem{theorem}{Theorem}[chapter]
\newtheorem{lemma}[theorem]{Lemma}
\newtheorem{proposition}[theorem]{Proposition}

\theoremstyle{definition}
\newtheorem{definition}[theorem]{Definition}

\theoremstyle{remark}
\newtheorem{remark}[theorem]{Remark}

\numberwithin{section}{chapter}
\numberwithin{equation}{chapter}
\numberwithin{figure}{chapter}

\makeindex

\begin{document}

\frontmatter

\title{The Algebra of Semi-flows: A Tale of Two Topologies}


\author{Kelly Spendlove}
\address{Mathematical Institute }
\curraddr{University of Oxford, Oxford UK}
\email{spendlove@maths.ox.ac.uk}
\thanks{This work was supported in part by NWO GROW grant 040.15.044/3192 and by EPSRC grant EP/R018472/1.}

\author{Robert Vandervorst}
\address{ Department of Mathematics}
\curraddr{VU University Amsterdam\\
               The Netherlands}
\email{r.c.a.m.vander.vorst@vu.nl}
\thanks{}

\date{August 14, 2023}




\begin{abstract}
To capture the global structure of a dynamical system  we reformulate dynamics in terms of appropriately constructed
topologies, which we call \emph{flow topologies}; we call this process \emph{topologization}.
This yields a description of a semi-flow in terms of a \emph{bi-topological space}, with the first topology corresponding to the (phase) space and the second to the flow topology.
A study of topology is facilitated  through \emph{discretization}, i.e. defining and examining appropriate finite sub-structures. Topologizing the dynamics provides an elegant solution to their discretization by discretizing the associated flow topologies.
We introduce \emph{Morse pre-orders}, an instance of a more general bi-topological discretization, which synthesize the space and flow topologies, and encode the directionality of  dynamics.
We describe how 
Morse pre-orders can be augmented with appropriate (co)homological information in order to describe invariance of the dynamics; this ensemble provides an \emph{algebraization} of the semi-flow.
An illustration of the main ingredients is provided 
by an application to 
the theory of discrete parabolic flows.  
Algebraization  yields a new invariant for positive braids in terms of a bi-graded differential module which  contains Morse theoretic information of  parabolic flows. 
\end{abstract}

\maketitle

\tableofcontents


\mainmatter
\include{prelude.tex}

\include{topanddisc}

\include{flowastop.tex}

\include{connectionMatrix-parallel.tex}

\include{parabolic.tex}

\include{postlude.tex}

\appendix
\include{binrelprops.tex}
\include{order2.tex}

\include{appendixCMT.tex}

\backmatter
\bibliographystyle{amsplain}
\bibliography{references.bib}
\printindex

\end{document}

%% file: macros.tex
 \usepackage{hyperref}
 \hypersetup{
     colorlinks=true,
     linkcolor=blue, 
     filecolor=magenta,      
     urlcolor=cyan,
}
\usepackage{amsfonts, amssymb, amsbsy, amsmath}
\usepackage[skip=2pt, font=small]{caption} 
\usepackage{mathabx}
\usepackage{extpfeil}
\usepackage{stmaryrd}
 \usepackage{blkarray}
 \usepackage{graphicx}
 \usepackage{float}
 \usepackage{epstopdf}
 \usepackage{algorithmic}
 \ifpdf
   \DeclareGraphicsExtensions{.eps,.pdf,.png,.jpg}
 \else
   \DeclareGraphicsExtensions{.eps}
 \fi
 \usepackage{mathrsfs}
 \usepackage{amscd,empheq} 
 \usepackage{yfonts,bm}

 \usepackage[pdf]{graphviz}

 \usepackage[all,2cell]{xy}
 \usepackage{pb-diagram}
 \usepackage{pb-xy}

 \usepackage{tikz}
 \usepackage{tikz-cd}
 \tikzcdset{
  cells={font=\everymath\expandafter{\the\everymath\displaystyle}},
}
 \usetikzlibrary{patterns}
 \usetikzlibrary{positioning}
 \usetikzlibrary{shapes,decorations.markings,arrows.meta}
 \usetikzlibrary{matrix,arrows}
 \tikzset{hfit/.style={rounded rectangle, inner xsep=0pt},
           vfit/.style={rounded corners}}

 \usepackage{morewrites}
 \usepackage{scalerel}
 \usepackage{mathabx,ascii}
 \usepackage{palatino}
 \usepackage{adjustbox}
\usepackage{dsfont}

\usepackage{aurical}
\usepackage[T1]{fontenc}
 
\usepackage{enumitem}
\setlist[enumerate]{leftmargin=.5in}
\setlist[itemize]{leftmargin=.5in}

\usepackage{mathtools,amssymb,stackengine}

\usepackage{booktabs}

\usepackage{amsopn}

\DeclareMathOperator{\Int}{\mathrm{int}}
\DeclareMathOperator{\cInt}{\mathbf{int}}

\newcommand{\rmGamma}{\mathrm{cl}}
\newcommand{\Gammatau}{\mathrm{\Gamma}_\tau^+}
\newcommand{\rmPhi}{\mathrm{\Phi}}

\newcommand{\bmPhi}{{\bm{\Phi}}}
\newcommand{\bmphi}{{\bm{\phi}}}
\newcommand{\bmTheta}{{\bm{\Theta}}}

\newcommand{\bflt}{\textrm{block-flow topology~}}
\newcommand{\bflta}{\textrm{block-flow topology}}
\newcommand{\discresol}{\textrm{condensed Morse pre-order~}}
\newcommand{\discresols}{\textrm{condensed Morse pre-orders~}}
\newcommand{\Discresol}{\textrm{Condensed Morse pre-order~}}

\newcommand{\dff}{{\mathrm d}}
\newcommand{\dffbf}{{\mathrm d}}

\newcommand{\Conf}{\mathscr{D}}
\newcommand{\Sing}{\mathrm{\Sigma}}
\newcommand{\sing}{{\mathrm{sing}}}

\newcommand{\rel}{~{\rm rel}~}
\newcommand{\scrC}{{\mathscr{C}}}
\newcommand{\scrO}{{\mathscr{O}}}

\newcommand{\scrE}{{\mathscr{E}}}

\newcommand{\aclop}{{\mathscr{CO}}}

\newcommand{\Tor}{{\mathrm{Tor}}}

\newcommand{\cross}{{\bm{\lambda}}}
\newcommand{\Lap}{{\bm{\mathrm{\Lambda}}}}

\newcommand{\sSC}{\mathsf{SC}}

\newcommand{\sFwdset}{\mathsf{Invset}^+}

\newcommand{\Invsetpl}{\mathsf{Invset^+}}
\newcommand{\Invsetneg}{\mathsf{Invset^-}}

\newcommand{\cl}{\mathrm{cl}}
\newcommand{\ccl}{\mathbf{cl}}
\newcommand{\bccl}{\mathbf{\bar cl}}
\newcommand{\bcl}{\mathrm{\bar cl}}
\newcommand{\Gr}{{\mathrm{Gr\,}}}

\newcommand{\uclbf}{\cl^-_\sqbullet}
\newcommand{\uclbff}{\ccl^-_\sqbullet}

\newcommand{\der}{\mathrm{\Gamma}}
\newcommand{\derr}{\bm{\Gamma}}
\newcommand{\uder}{\Gamma^+_\sqbullet}
\newcommand{\uderm}{\Gamma^-_\sqbullet}
\newcommand{\udermm}{\bm\Gamma^-_\sqbullet}
\newcommand{\uderup}{{\bmPhi}}
\newcommand{\bvtheta}{{\mathbf{c}}}

\newcommand{\deromega}{{\mathrm{\Gamma}}_\omega}
\newcommand{\lebf}{\le^-_\sqbullet}
\newcommand{\gebf}{\ge^-_\sqbullet}

\newcommand{\scrTbf}{\scrT^-_\sqbullet}
\newcommand{\scrTbfop}{\scrT_-^\sqbullet}

\newcommand{\lbr}{\Vert}
\newcommand{\rbr}{\Vert}

\newcommand{\sAtt}{{\mathsf{ Att}}}
\newcommand{\sABlock}{{\mathsf{ABlock}}}
\newcommand{\sABlockR}{{\mathsf{ABlock}}_{\mathscr{R}}}
\newcommand{\sABlockC}{{\mathsf{ABlock}}_{\mathscr{C}}}

\newcommand{\sRBlockO}{{\mathsf{RBlock}}_{\mathscr{O}}}

\newcommand{\sCo}{{\mathsf{Co}}}

\newcommand{\BB}{\mathscr{B}}

\newcommand{\scrR}{\mathscr{R}}
\newcommand{\scrA}{\mathscr{A}}
\newcommand{\scrT}{\mathscr{T}}

\newcommand{\IIi}{{{\sO(\sP)}^{\bm{2}}}}

\newcommand{\spi}{\mathsf{p}}

\newcommand{\smin}{\smallsetminus}
\newcommand{\rmin}{\!-\!}
\newcommand{\pred}{\blacktriangleleft}

\newcommand{\setof}[1]{\left\{ {#1}\right\}}

\newcommand{\dyn}{\mathrm{dyn}}
\newcommand{\grd}{\mathrm{disc}}
\newcommand{\disc}{\mathrm{disc}}
\newcommand{\cm}{{\dff}}

\newcommand{\possim}{\stackrel{\mbox{\textnormal{\raisebox{0ex}[0ex][0ex]{\scaleto{+}{2.5pt}}}}}{\sim}}
\newcommand{\sdoteq}{\stackrel{\mbox{\textnormal{\raisebox{0ex}[0ex][0ex]{\scaleto{+}{2.5pt}}}}}{=}}
\newcommand{\wbinom}[2]{\genfrac{}{}{0pt}{}{#1}{#2}}

\newcommand{\grade}{\mathrm{grd}}

\newcommand{\ppart}{\mathrm{part}}
\newcommand{\para}{\mathrm{para}}
\newcommand{\flt}{\mathrm{flt}}
\newcommand{\tess}{\mathrm{tess}}
\newcommand{\tile}{\mathrm{tile}}

\newcommand{\cell}{\mathrm{cell}}
\newcommand{\skel}{\mathrm{skel}}
\newcommand{\lap}{\mathrm{lap}}
\newcommand{\pb}{\mathrm{pb}}

\newcommand{\ind}{\mathrm{ind}}
\newcommand{\tessph}{\amalg}
\newcommand{\otessph}{\overline{\amalg}}

\newcommand{\rk}{\mathrm{rank~}}

\newcommand{\topc}{{\top}\!}
\newcommand{\BM}{\mathrm {BM}}

\definecolor{Blue}{RGB}{0,162,255}
\definecolor{Orange}{RGB}{243,144,25}
\definecolor{Red}{RGB}{236,93,87}

\newcommand{\scrS}{\mathscr{S}}

\newcommand{\bh}{{\bf h}}

\newcommand{\bv}{{\bf v}}

\newcommand{\K}{{\mathbb{K}}}
\newcommand{\N}{{\mathbb{N}}}
\newcommand{\R}{{\mathbb{R}}}
\newcommand{\T}{{\mathbb{T}}}
\newcommand{\Z}{{\mathbb{Z}}}
\newcommand{\E}{{\mathbb{E}}}

\newcommand{\rA}{{{C}}}

\newcommand{\RR}{\mathbf{R}}

\newcommand{\ccC}{\text{\small\Fontauri{C}}}
\newcommand{\cccC}{\text{\tiny\Fontauri{C}}}

\newcommand{\cE}{\text{\huge\Fontauri{e}}}
\newcommand{\cF}{{{\bm{\psi}}^\top}}

\newcommand{\cR}{{\bm{\psi}}}
\newcommand{\cS}{{\mathcal S}}

\newcommand{\ccT}{\text{\small\Fontauri{T}}}
\newcommand{\cccT}{\text{\tiny\Fontauri{T}}}
\newcommand{\ccfT}{\text{\footnotesize\Fontauri{T}}}
\newcommand{\cU}{\text{\Fontauri\slshape U}}

\newcommand{\cX}{\text{\small\Fontauri{X}}}
\newcommand{\ccX}{\text{\footnotesize\Fontauri{X}}}
\newcommand{\cccX}{\text{\tiny\Fontauri{X}}}

\newcommand{\cXr}{{\mathcal X}_{\rm reg}}
\newcommand{\cXs}{{\mathcal X}_{\rm sing}}

\newcommand{\sA}{{\mathsf A}}
\newcommand{\sB}{{\mathsf B}}

\newcommand{\sE}{{\mathsf E}}

\newcommand{\sJ}{{\mathsf J}}
\newcommand{\sO}{{\mathsf O}}
\newcommand{\sK}{{\mathsf K}}
\newcommand{\sL}{{\mathsf L}}
\newcommand{\sM}{{\mathsf M}}
\newcommand{\sN}{{\mathsf N}}
\newcommand{\sP}{{\mathsf P}}
\newcommand{\sQ}{{\mathsf Q}}

\newcommand{\sT}{{\mathsf T}}
\newcommand{\sU}{{\mathsf U}}

\newcommand{\sAnt}{{\sO(\sP)}}

\newcommand{\tH}{{\vec{H}}}

\newcommand{\sSet}{{\mathsf{Set}}}

\def\setof#1{\left\{{#1}\right\}}

\def\w#1{\mbox{#1}}

\newcommand{\id}{\mathrm{id}}

\DeclareMathOperator{\rank}{rank}

\DeclareMathOperator{\st}{\mathbf{star}}
\DeclareMathOperator{\nst}{\mathrm{star}}
\DeclareMathOperator{\bd}{bd}

\DeclareMathOperator{\Inv}{Inv}

\newcommand{\image}{\mathrm{im~}}

\newcommand{\Sub}{\mathsf{Sub}}

\newcommand{\bFPoset}{\text{{\bf FPoset}}}

\newcommand{\bFDLat}{\text{{\bf FDLat}}}

\newcommand{\bFPreOrd}{\text{{\bf FPreOrd}}}
\newcommand{\sRmod}{R\text{-{\bf Mod}}}

\newcommand{\bfE}{\text{\bf E}}

{\begin{snugshade}\begin{quote}}
{\hfill\end{quote}\end{snugshade}}

\definecolor{shadecolor}{rgb}{0.8,0.8,0.8}

%% file: prelude.tex
\chapter{Prelude}
\label{prelude}

We introduce the point of view that dynamics can be studied as a topology.  Thus an analysis of a dynamical system results in an analysis of two topologies (i.e. a bi-toplogical space): the first topology corresponds to the (phase) space, and the second to the dynamics. 
Topology and associated algebraic invariants have long played a prolific role in the theory of dynamical systems \cite{conley:cbms,morse,poincare,smale}.  Loosely stated, a dynamical system engenders topological data, both local (e.g. fixed points) and global (e.g. attractors) and the directionality of the dynamics organizes the data. 
The topological data have associated algebraic invariants which may further codify the relationship between local and global and often recover the invariance of the dynamics, i.e. provide information about the existence and structure of the invariant sets.
%
%

\section{Topologization and discretization}

The novelty of our approach -- and the first theme of this text -- is to formalize the dynamical system itself as a topology, and capture both topology and dynamics in the formalism of bi-topological spaces, i.e. a \emph{topologization}\index{Topologization} of the dynamical system.
Recall that a
 \emph{semi-flow}\index{Semi-flow} on a topological space $(X,\scrT)$\footnote{A topology on $X$ is denoted by $\scrT$.} 
is a continuous map $\varphi\colon \R^+\times X \to X$ such that 
\begin{enumerate}
    \item[(i)] $\varphi(0,x)=x$ for all $x\in X$, and
    \item[(ii)] $\varphi(t,\varphi(s,x)) =\varphi(t+s,x)$ for all $s,t\in \R^+$ and $x\in X$.
\end{enumerate}
For a semi-flow $\varphi$  the
backward image  is denoted by $\varphi(-t,x)$, $t>0$ and is defined as $\varphi(-t,x) := \{x'\in X\mid \varphi(t,x) = x'\}$.
One way to regard a topological space is via a closure operator on the algebra of subsets of $X$, cf.\ Sect.\ \ref{CA-disc}. Using this point of view, a map $\varphi$ as defined above yields a natural closure operator on $X$ through backward or forward images. If we disregard the continuity of $\varphi$ in $(X,\scrT)$  this defines  Alexandrov topologies\index{Alexandrov topology} on $X$ which are denoted by $\scrT^-$ and $\scrT^+$ respectively, cf.\ Sect.\ \ref{twotopos}.
The topologies $\scrT^-$ and $\scrT^+$ record directionality of the flow, but discard the sense of time and invariance. They are also independent of the continuity properties of $\varphi$.
We therefore define a topology which allows one to incorporate the continuity of $\varphi$ in $(X,\scrT)$ and  which is more suited for capturing the important characteristics of dynamics such as invariance. We will refer to this topology as the (derived) \emph{\bflta}\index{Block-flow topology} denoted by $\scrTbf$, cf.\ Sect.\  \ref{derivesflowtop}.\footnote{As we have cast dynamics as topology, it is worthwhile to ask the question: what can dynamical systems theory say about topology? Section~\ref{tesshom} may be regarded as steps in this direction.}
The topologies $\scrT$ and $\scrTbf$ comprise the bi-topological space $(X,\scrT,\scrTbf)$ which becomes our model for a semi-flow $\varphi$.\footnote{The topologies $\scrT$ and $\scrTbf$ are in general related while the topologies $\scrT$ and $\scrT^-,\scrT^+$ are independently defined. The \bflt is not Alexandrov in general. An explanation of the notation of the various flow topologies we use is given  in Sect.'s \ref{derivesflowtop} and \ref{exttopology}.}
Sets that are closed in $(X,\scrT)$ and open  $(X,\scrTbf)$, so called \emph{pairwise clopen sets},\index{Pairwise clopen set}    are closed attracting blocks for $\varphi$, cf.\ Thm.\ \ref{charclattbl}. 
Of course the \bflt $\scrTbf$ discards some  information about $\varphi$.
However, we can define suitable (co)homology theories on $(X,\scrT,\scrTbf)$ which allow one to describe fundamental invariant structures of the dynamics, cf.\ Sect.\ \ref{sec:cm}. 

The second theme of this paper concerns discretization of both topology and dynamics. 
The last few decades have seen ever more sophisticated uses of discrete approximation 
in order to explore global dynamical features \cite{cmdb,cmdbchaos,MCG,GKV,kkv,kmv,mischaikow:mrozek:95}; these techniques are largely based on Conley theory, a topological generalization of Morse theory \cite{conley:cbms}.  
As before, we describe topology in terms of a closure algebra, which 
provides a powerful formalism for discretization and extends these techniques.
 Moreover, as we encode a semi-flow $\varphi$ as a topology, topological discretization also provides a means of discretizing $\varphi$.
If we describe a topological space via a closure algebra $\bigl(\sSet(X),\cl\bigr)$, then
discretization may be regarded as determining a finite sub-algebra in the category of closure algebras, cf.\ Sect.\ \ref{CA-disc}. 
A finite closure algebra may be represented by $\bigl(\sSet(\cX),\cl\bigr)$, where $\cX$ is a finite set, and is
 equivalently described by a finite pre-order $(\cX,\le)$, called the \emph{specialization pre-order}\index{Specialization pre-order}
 of the associated Alexandrov topology on $\cX$.
Duality of the latter pre-order defines
a continuous map
\[
\disc\colon X \twoheadrightarrow \cX,
\]
which is called the \emph{discretization map}\index{Discretization map} and provides a discretization of $(X,\scrT)$ by a finite topological space $(\cX,\le)$, cf.\ Eqn.\ \eqref{induceddisc}. 
\begin{figure}[hbt]
\begin{minipage}{.24\textwidth}
\centering
\begin{tikzpicture}[scale=0.8]
\def\h{.1}
\foreach \x in {0,1,2}
    \foreach \y in {0,1,2}
        \filldraw[black,fill opacity=.15](\x+\h,\y+\h) rectangle (\x+1-\h,\y+1-\h);
\foreach \x in {0,1,2,3}
    \foreach \y in {0,1,2}
        \filldraw[black, fill opacity=1.0, thick] (\x,\y+\h) -- (\x,\y+1-\h);
\foreach \x in {0,1,2}
    \foreach \y in {0,1,2,3}
        \filldraw[black, fill opacity=1.0 ,thick] (\x+\h,\y) -- (\x+1-\h,\y);
\foreach \x in {0,1,2,3}
    \foreach \y in {0,1,2,3}
        \fill[fill = black, fill opacity=1.0] (\x,\y) circle (1.5pt);
\fill[white]  (2+.5*\h, -\h) rectangle ( 3+\h, 1-.5*\h);
\end{tikzpicture}
\end{minipage}
\begin{minipage}{.24\textwidth}
\centering
\begin{tikzpicture}[2cell/.style={rectangle,fill,inner sep=1.7pt},
edge/.style = {circle, draw,inner sep=1pt},
vert/.style = {circle, fill, inner sep=1pt},
line width=.1pt,scale=.4]
\def\h{.15}
\foreach \x in {0,2,4,6}
    \foreach \y in {0,2,4,6}
        \node(\x\y) at (\x,\y) [vert] {};
\foreach \x in {0,2,4,6}
    \foreach \y in {1,3,5}
        \node(\x\y) at (\x,\y) [edge] {};
\foreach \y in {0,2,4,6}
    \foreach \x in {1,3,5}
        \node(\x\y) at (\x,\y) [edge] {};
\foreach \x in {1,3,5}
    \foreach \y in {1,3,5}
        \node(\x\y) at (\x,\y) [2cell] {};
\foreach \x[count=\xi]  in {1,3,5}
    \foreach \y[count=\yi] in {1,3,5} {
        \draw[-latex] (\x,\y+\h) to (\x,\y+1-\h);
        \draw[-latex] (\x+\h,\y) to (\x+1-\h,\y);
        \draw[-latex] (\x-\h,\y) to (\x-1+\h,\y);
        \draw[-latex] (\x,\y-\h) to (\x,\y-1+\h);
    }
\foreach \x[count=\xi]  in {1,3,5}
    \foreach \y[count=\yi] in {0,2,4,6} {
        \draw[-latex] (\x+\h,\y) to (\x+1-\h,\y);
        \draw[-latex] (\x-\h,\y) to (\x-1+\h,\y);
    }
\foreach \x[count=\xi]  in {0,2,4,6}
    \foreach \y[count=\yi] in {1,3,5} {
        \draw[-latex] (\x,\y+\h) to (\x,\y+1-\h);
        \draw[-latex] (\x,\y-\h) to (\x,\y-1+\h);
    }
\fill[white]  (4+\h, -\h) rectangle ( 6+\h, 2-\h);
\end{tikzpicture}
\end{minipage}
\begin{minipage}{.24\textwidth}
\centering
 \begin{tikzpicture}[dot/.style={draw,circle,fill,inner sep=.75pt},line width=.3pt,scale=0.82, decoration={markings, 
    mark= at position 0.55 with {\arrow{latex}}}]
    \draw (0,0) -- (0,3) -- (3,3) -- (3,1) -- (2,1); 
    \draw (1,0)--(1,3);
    \draw (2,2) -- (3,2);
    \draw (2,0)--(2,3);
    \draw (0,1) -- (2,1);
    \draw (0,2)--(2,2);
    \draw (0,0)--(2,0);
    \node[Orange] (00) at (.5,.45)[dot] {};
    \node[Orange] (01) at (.75,1.35)[dot] {};
    \node[Orange] (02) at (0.4,2.5)[dot] {};
    \node[Orange] (11) at (1.25,1.65)[dot] {};
    \node[Orange] (23) at (2,2.5)[dot] {};
    \node[Orange] (12) at (1.025,1.5)[dot] {};
    

    
    \draw[Orange, postaction={decorate}] (11) to[out=80,in=-10] (02);
    \draw[Orange, postaction={decorate}] (23) to[out=180,in=50] (02);
    
    \draw[Orange, postaction={decorate}] (01) to[out=80,in=160] (11);
    \draw[Orange, postaction={decorate}] (11) to[out=270,in=340] (01);
    \draw[Orange]  (.85,1.425) to[out=90,in=120] (1.16,1.575);
    \draw[Orange] (1.16,1.575) to[out=290,in=280] (.85,1.425);
    
    \draw[Orange, postaction={decorate}] (01) to[out=280,in=80] (00);
    
    \draw[Orange, postaction={decorate}] (2.5,3.1) to[out=290,in=0] (23);
    
    \draw[Orange, postaction={decorate}] (-.1,2.35) to[out=0,in=250] (02);
    \draw[Orange, postaction={decorate}] (-.1,3.1) to[out=330,in=150] (02);

    \draw[Orange, postaction={decorate}] (3.1,2.5) to[out=240,in=25] (02);
    \draw[Orange, postaction={decorate}] (2.1,3.1) to[out=240,in=60] (02);
    \draw[Orange, postaction={decorate}] (-0.1,1.65) to[out=0,in=350] (02);
    \draw[Orange, postaction={decorate}] (-0.1,1.95) to[out=20,in=300] (02);
    
    \draw[Orange, postaction={decorate}] (1.125,3.1) to[out=230,in=100] (02);
    
    \draw[Orange, postaction={decorate}] (2.1,1.65) to[out=170,in=10] (02);
    
    \draw[Orange, postaction={decorate}]
    (3.1, 1.45) .. controls (2.75, 1.85)  and (2.55, 2.1) .. (2.2, 2.25);
    \draw[Orange, postaction={decorate}] (2.4,1.75) .. controls (1.9,2.4) and (1.75,1.5)  .. (1.295, 2.18);
    
    
    \draw[Orange, postaction={decorate}] (3.075,3.1) .. controls (2.5,2.25) and (2,2.25)  .. (1.75, 2.41);
    
    \draw[Orange, postaction={decorate}] (2.025,-.1) .. controls (1.95,0.25) and (1.5,0.3)  .. (.95, .39);



     \draw[Orange, postaction={decorate}] (-.1,1.5) to[out=0,in=150] (01);
     \draw[Orange, postaction={decorate}] (-.1,0.5) to[out=0,in=170] (00);
     \draw[Orange, postaction={decorate}] (-.15,-.05) to[out=0,in=190] (00);
    \draw[Orange, postaction={decorate}] (.5,-0.1) to[out=80,in=290] (00);
    \draw[Orange, postaction={decorate}] (2.1,0.4) to[out=180,in=-10] (00);
    
    \draw[Orange, postaction={decorate}] (2.4,.75) .. controls (2,1.4) and (1.5, .6) .. (00);
    
    \draw[Orange, postaction={decorate}] (2.1,1.35) to[out=170,in=51] (00);
    \draw[Orange, postaction={decorate}] (-.1,1.35) to[out=10,in=90] (00);
    \draw[Orange, postaction={decorate}] (-.1,1.03) to[out=-10,in=120] (00);
    \draw[Orange, postaction={decorate}] (1.55,-.1) to[out=100,in=-30] (00);
    \draw[Orange, postaction={decorate}] (1.05,-.1) to[out=110,in=-40] (00);
    
    \draw[Orange, postaction={decorate}] (3.1,1.1) to[out=140,in=-10] (11);
    
    \draw[Orange, postaction={decorate}] (1.795,3.1) .. controls (1.65,2.95) and (1.55,2.85)  .. (1.3, 2.825);
    
    \draw[Orange, postaction={decorate}] (3.1,1.95) .. controls (2.65,2.15) and (2.4,2.175)  .. (2.25, 2.235);
    
    \draw[Orange, postaction={decorate}] (2.8,.8) .. controls (2,1.2)   .. (1.35, .74);


\end{tikzpicture}
\end{minipage}
\begin{minipage}{.24\textwidth}
\centering
\begin{tikzpicture}[2cell/.style={rectangle,fill,inner sep=1.7pt},
edge/.style = {circle, draw,inner sep=1pt},
vert/.style = {circle, fill, inner sep=1pt},
line width=.1pt, scale=.4]
\def\h{.15}
\foreach \x in {0,2,4,6}
    \foreach \y in {0,2,4,6}
        \node(\x\y) at (\x,\y) [vert] {};
\foreach \x in {0,2,4,6}
    \foreach \y in {1,3,5}
        \node(\x\y) at (\x,\y) [edge] {};
\foreach \y in {0,2,4,6}
    \foreach \x in {1,3,5}
        \node(\x\y) at (\x,\y) [edge] {};
\foreach \x in {1,3,5}
    \foreach \y in {1,3,5}
        \node(\x\y) at (\x,\y) [2cell] {};
        
\foreach \x[count=\xi]  in {1}
    \foreach \y[count=\yi] in {1,5} {
        \draw[-latex] (\x,\y+\h) to (\x,\y+1-\h);
        \draw[-latex] (\x+\h,\y) to (\x+1-\h,\y);
        \draw[-latex] (\x-\h,\y) to (\x-1+\h,\y);
        \draw[-latex] (\x,\y-\h) to (\x,\y-1+\h);
        \draw[-latex] (\x+\h,\y+\h) to (\x+1-\h,\y+1-\h);
        \draw[-latex] (\x-\h,\y+\h) to (\x-1+\h,\y+1-\h);
        \draw[-latex] (\x-\h,\y-\h) to (\x-1+\h,\y-1+\h);
        \draw[-latex] (\x+\h,\y-\h) to (\x+1-\h,\y-1+\h);

    }
\foreach \x in {2,3,4,6}
    \draw[latex-] (\x-\h,5) to (\x-1+\h,5);
\draw[latex-latex] (5-\h,5) to (5-1+\h,5);
\foreach \x[evaluate = \x as \xi using int(\x+1)] in {3,5} {
    \draw[latex-] (\xi-\h,6-\h) to (\xi-1+\h,6-1+\h);
    \draw[latex-] (\x,6-\h) to (\x,6-1+\h);
}
\foreach \x[evaluate = \x as \xi using int(\x+1)] in {5} {
    \draw[latex-] (\xi-\h,4+\h) to (\xi-1+\h,4+1-\h);
    \draw[latex-] (\x,4+\h) to (\x,4+1-\h);
}
\foreach \x[evaluate = \x as \xi using int(\x+1)] in {3} {
    \draw[latex-] (\xi-\h,0+\h) to (\xi-1+\h,0+1-\h);
    \draw[latex-] (\x,0+\h) to (\x,0+1-\h);
}
\foreach \y in {2,3,4} {
    \draw[latex-latex] (3,\y+\h) to (3,\y+1-\h);
}
\draw[-latex] (3,1+\h) to (3,1+1-\h);
\foreach \x in {3,4}
    \draw[latex-] (\x-\h, 1) to (\x-1+\h, 1);
\foreach \x in {1,2}
    \draw[latex-latex] (\x+\h, 3) to (\x+1-\h, 3);
\foreach \y in {2,3}
    \draw[latex-] (4-\h,\y) to (3+\h,\y);
\draw[latex-] (4-\h,4-\h) to (3+\h,3+\h);

\draw[latex-] (3-\h,3+\h) to (2+\h,4-\h);
\draw[latex-] (3-\h,3-\h) to (2+\h,2+\h);

\draw[-latex] (1,3+\h) to (1,4-\h);

\draw[latex-] (1,3-\h) to (1,2+\h);
\draw[latex-] (0+\h,3) to (1-\h,3);

\draw[-latex] (4+\h, 2) to (5-\h, 2);
\draw[-latex] (4+\h, 3) to (5-\h, 3);
\draw[-latex] (5+\h, 3) to (6-\h, 3);
\draw[-latex] (5, 4-\h) to (5, 3+\h);

\draw[-latex] (4+\h, 4-\h) to (5-\h, 3+\h);
\draw[-latex] (5+\h, 3-\h) to (6-\h, 2+\h);

\draw[-latex] (5,3-\h) to (5,2+\h);

\fill[white]  (4+\h, -\h) rectangle ( 6+\h, 2-\h);
\end{tikzpicture}
\end{minipage}
\vspace{2ex}
\caption{
A discretization of $(X,\scrT)$ with the associated face partial order $\le$ [left 1 and 2]. Example of a semi-flow $\varphi$ on $X$ and pre-order $\lebf$ which is a discretization of the \bflt $\scrTbf$ [right 3 and 4].
}\label{morsetes23}
\end{figure}
%
The elements of $\cX$ are denoted by $\xi$ and are called \emph{cells} in $\cX$.~The closure algebra for a bi-topological space such as $(X,\scrT,\scrTbf)$ is given by $\bigl(\sSet(X),\cl,\uclbf\bigr)$ and by considering finite sub-structures we obtain discretization maps $\disc\colon X\to \cX$ that are continuous with respect to two finite topologies $(\cX,\le)$ and $(\cX,\lebf)$. 
That is, for the bi-topological space  $(X,\scrT,\scrTbf)$ a discretization is a finite bi-topological space $(\cX,\le,\lebf)$ with continuous discretization map $\disc\colon (X,\scrT,\scrTbf)
\twoheadrightarrow (\cX,\le,\lebf)$; this procedure is described for general bi-topological spaces in Sect.\ \ref{bi-top}. 
Fig. \ref{morsetes23}  illustrates discretization of $(X,\scrT,\scrTbf)$ via two compatible pre-orders on a finite topological  space $\cX$.  
Attracting blocks, cf.\ Eqn.\ \eqref{clattbl12}, play a central role in the study of the gradient-like and recurrent dynamics of a semi-flow $\varphi$.
As earlier noted, 
the pairwise clopen sets in $(X,\scrT,\scrTbf)$ comprise
the closed attracting blocks for $\varphi$.
We can describe such sets in terms of a discretization which synthesizes both topologies. 
A \emph{Morse pre-order}, cf.\ Defn.\ \ref{morsepreorder1aa}, is a pre-order $(\cX,\le^\dagger)$ such that both discretization maps
$\disc\colon (X,\scrT)\twoheadrightarrow (\cX,\le^\dagger)$ and $\disc\colon (X,\scrTbf)\twoheadrightarrow (\cX,\ge^\dagger)$
are continuous; Morse pre-orders  are particular instances of \emph{antagonistic pre-orders},\index{Antagonistic pre-order} which are defined purely in terms of bi-topological spaces, cf.\ Defn.\ \ref{antacoarse}.\index{Bi-topological space}
Closed sets in $(\cX,\le^\dagger)$ correspond to closed attracting blocks for $\varphi$, cf.\ Fig.\ \ref{MPOandTile}[right]. 

The lattice $\sO(\cX,\le^\dagger)$ of closed sets in $(\cX,\le^\dagger)$  can be represented by the down-sets of a finite poset $(\sSC,\le)$, cf.\ Eqn.\ \eqref{defnofSC}, Fig.\ \ref{MPOandTile}[left]. The map~$\dyn\colon {(\cX,\le^\dagger)} \twoheadrightarrow\! (\sSC,\le\!)$ is defined as the dual of $\sO(\cX,\le^\dagger) \cong \sO(\sSC,\le) \rightarrowtail \sSet(\cX)$, cf.\ Eqn.'s \eqref{dubbleconst2}-\eqref{idred}.\footnote{This map is christened $\dyn$ as it may be regarded as a grading by the dynamics.  Combining \eqref{defndyn} and \eqref{defnofSC} 
provides a formula for $\dyn$, q.v. Thm.\ \ref{thethmdyn} and Eqn.\ \eqref{thethmdynform}.} 
 Depending on the topology on $\cX$ the map $\dyn$ is order-preserving, or order-reversing\footnote{The dashed arrows in Figure \ref{dia:intro} indicated order-reversing maps. We use this notation throughout the paper.} as is displayed in the following diagram:
%

\begin{equation}
   \label{dia:intro}
    \begin{tikzcd}[column sep=huge, row sep=large]
& (\cX,\le) \arrow[d, "\id"'] \arrow[rd, "\dyn"]                 &   \\
(X,\scrT,\scrTbf) \arrow[r, "\disc"', bend left=29, shift right] \arrow[r, "\disc", dashed, bend right=29, shift left] \arrow[ru, "\disc", shift left=2] \arrow[rd, "\disc"',  shift right=2] & (\cX,\le^\dagger) \arrow[r, "\dyn"]                                  & (\sSC,\le) \\
& (\cX,\lebf) \arrow[u, "\id", dashed] \arrow[ru, "\dyn"', dashed] &  
\end{tikzcd}
\end{equation}

%
The composition $X\xrightarrow{\disc} \cX\xrightarrow{\dyn} \sSC$ yields the $T_0$-discretization
$\tile\colon X\to \sSC$, cf.\ App.\ \ref{gradfilt};
in particular, 
$\tile\colon (X,\scrT) \to (\sSC,\le)$ and
$\tile\colon (X,\scrTbf)\to (\sSC,\ge)$ are continuous. 
%
%
%
The discretization $\tile$ 
defines a \emph{Morse tessellation}\index{Morse tessellation} with locally closed tiles $T=\tile^{-1}\cS$, $\cS\in \sSC$,
i.e. the sets $T$ form a tessellation of $X$ such that $\big\downarrow T$ is closed and $\varphi(t,x)\in \Int \big\downarrow T$ for every $x\in T$ and for all tiles $T$, 
 cf.\ Defn.\ \ref{morsetess45} and \cite[Defn.\ 8]{lsa3}.
Fig.\ \ref{MPOandTile}[right] illustrates how a Morse tessellation is obtained from a Morse pre-order. 
Conversely, if $(\sT,\le)$ is a Morse tessellation we obtain a Morse pre-order by defining $(\sT,\le)$ to be a Morse pre-order, and thus a discretization of $(X,\scrT,\scrTbf)$, cf.\ Sect. \ref{Morsetessll}.
%
\begin{figure}[hbt]
\begin{minipage}{.5\textwidth}
\centering
\begin{tikzpicture}[2cell/.style={rectangle,fill,inner sep=1.7pt}, edge/.style = {circle, draw,inner sep=1pt},vert/.style = {circle, fill, inner sep=1pt},line width=.35pt,scale=.45]
\def\h{.15}
\foreach \x in {0,2,4,6}
    \foreach \y in {0,2,4,6}
        \node(\x\y) at (\x,\y) [vert] {};
\foreach \x in {0,2,4,6}
    \foreach \y in {1,3,5}
        \node(\x\y) at (\x,\y) [edge] {};
\foreach \y in {0,2,4,6}
    \foreach \x in {1,3,5}
        \node(\x\y) at (\x,\y) [edge] {};
\foreach \x in {1,3,5}
    \foreach \y in {1,3,5}
        \node(\x\y) at (\x,\y) [2cell] {};
\foreach \x[count=\xi]  in {1}
    \foreach \y[count=\yi] in {1,5} {
        \draw[latex-latex] (\x,\y+\h) to (\x,\y+1-\h);
        \draw[latex-latex] (\x+\h,\y) to (\x+1-\h,\y);
        \draw[latex-latex] (\x-\h,\y) to (\x-1+\h,\y);
        \draw[latex-latex] (\x,\y-\h) to (\x,\y-1+\h);
        \draw[latex-latex] (\x+\h,\y+\h) to (\x+1-\h,\y+1-\h);
        \draw[latex-latex] (\x-\h,\y+\h) to (\x-1+\h,\y+1-\h);
        \draw[latex-latex] (\x-\h,\y-\h) to (\x-1+\h,\y-1+\h);
        \draw[latex-latex] (\x+\h,\y-\h) to (\x+1-\h,\y-1+\h);

    }
    
\foreach \x in {3}
    \draw[-latex] (\x-\h,5) to (\x-1+\h,5);
    
\foreach \x in {4,5,6}
    \draw[latex-latex] (\x-\h,5) to (\x-1+\h,5);
    
\foreach \x[evaluate = \x as \xi using int(\x+1)] in {3,5} {
    \draw[latex-latex] (\xi-\h,6-\h) to (\xi-1+\h,6-1+\h);
    \draw[latex-latex] (\x,6-\h) to (\x,6-1+\h);
}
\foreach \x[evaluate = \x as \xi using int(\x+1)] in {5} {
    \draw[latex-latex] (\xi-\h,4+\h) to (\xi-1+\h,4+1-\h);
    \draw[latex-latex] (\x,4+\h) to (\x,4+1-\h);
}
\foreach \x[evaluate = \x as \xi using int(\x+1)] in {3} {
    \draw[latex-latex] (\xi-\h,0+\h) to (\xi-1+\h,0+1-\h);
    \draw[latex-latex] (\x,0+\h) to (\x,0+1-\h);
}
\foreach \y in {1,2,3,4} {
    \draw[latex-latex] (3,\y+\h) to (3,\y+1-\h);
}
\foreach \x in {3}
    \draw[-latex] (\x-\h, 1) to (\x-1+\h, 1);
\foreach \x in {4}
    \draw[latex-latex] (\x-\h, 1) to (\x-1+\h, 1);
\foreach \x in {1,2}
    \draw[latex-latex] (\x+\h, 3) to (\x+1-\h, 3);
\foreach \y in {2,3}
    \draw[latex-latex] (4-\h,\y) to (3+\h,\y);
\draw[latex-latex] (4-\h,4-\h) to (3+\h,3+\h);

\draw[-latex] (3-\h,3+\h) to (2+\h,4-\h);
\draw[-latex] (3-\h,3-\h) to (2+\h,2+\h);

\draw[-latex] (1,3+\h) to (1,4-\h);
\draw[-latex] (1,3-\h) to (1,2+\h);

\draw[latex-latex] (0+\h,3) to (1-\h,3);
\draw[latex-latex] (5-\h,4) to (4+\h,4);

\draw[-latex] (3-\h,5+\h) to (2+\h,6-\h);
\draw[-latex] (3-\h,5-\h) to (2+\h,4+\h);
\draw[-latex] (3+\h,5-\h) to (4-\h,4+\h);
\draw[-latex] (5-\h,5+\h) to (4+\h,6-\h);
\draw[-latex] (5-\h,5-\h) to (4+\h,4+\h);

\draw[-latex] (1-\h,3+\h) to (0+\h,4-\h);
\draw[-latex] (1-\h,3-\h) to (0+\h,2+\h);

\draw[-latex] (1+\h,3+\h) to (2-\h,4-\h);
\draw[-latex] (1+\h,3-\h) to (2-\h,2+\h);

\draw[-latex] (3+\h,3-\h) to (4-\h,2+\h);

\draw[latex-latex] (5+\h,3-\h) to (6-\h,2+\h);
\draw[latex-latex] (5+\h,3) to (6-\h,3);
\draw[latex-latex] (5,3-\h) to (5,2+\h);

\draw[-latex] (5+\h,3+\h) to (6-\h,4-\h);
\draw[-latex] (5,3+\h) to (5,4-\h);
\draw[-latex] (5+\h,3-\h) to (4+\h,4-\h);
\draw[-latex] (5-\h,3) to (4+\h,3);
\draw[-latex] (5-\h,3-\h) to (4+\h,2+\h);
\draw[-latex] (5-\h,2) to (4+\h,2);

\draw[-latex] (3-\h,1+\h) to (2+\h,2-\h);
\draw[-latex] (3-\h,1-\h) to (2+\h,\h);
\draw[-latex] (3+\h,1+\h) to (4-\h,2-\h);

\fill[white]  (4+\h, -\h) rectangle ( 6+\h, 2-\h);
\def\d{.275}
\draw (-\d,-\d) -- (2+\d,-\d) -- (2+\d,2+\d) -- (-\d,2+\d) -- cycle ; 
\draw (-\d,4-\d) -- (2+\d,4-\d) -- (2+\d,6+\d) -- (-\d,6+\d) -- cycle ;

\draw (6+\d,6+\d) -- (3-\d,6+\d) -- (3-\d,3+\d) -- (0-\d,3+\d) -- (0-\d,3-\d) -- (3-\d,3-\d) -- (3-\d,-\d) -- (3-\d,-\d) -- (4+\d,-\d) -- (4+\d, 4-\d) -- (6+\d,4-\d) -- cycle;

\draw (5-\d,2-\d) -- (5-\d,3+\d) -- (6+\d,3+\d) -- (6+\d,2-\d) -- cycle ; 

\end{tikzpicture}
\end{minipage}
\begin{minipage}{.40\textwidth}
\centering
 \begin{tikzpicture}[dot/.style={draw,circle,fill,inner sep=.75pt},line width=.3pt,scale=0.9, decoration={markings, 
    mark= at position 0.55 with {\arrow{latex}}}]
    \draw (0,0) -- (0,3) -- (3,3) -- (3,1) -- (2,1); 
    \draw (1,0)--(1,1);
    \draw (1,2)--(1,3);
    \draw (2,2) -- (3,2);
    \draw (2,0)--(2,2);
    \draw (0,1) -- (1,1);
    \draw (0,2)--(1,2);
    \draw (0,0)--(2,0);
    \node[Orange] (00) at (.5,.45)[dot] {};
    \node[Orange] (01) at (.75,1.35)[dot] {};
    \node[Orange] (02) at (0.4,2.5)[dot] {};
    \node[Orange] (11) at (1.25,1.65)[dot] {};
    \node[Orange] (23) at (2,2.5)[dot] {};
    \node[Orange] (12) at (1.025,1.5)[dot] {};
    
    \node at (-.25,.5) {$T_0$};
    \node at (-.25,1.5) {$T_2$};
    \node at (-.25,2.5) {$T_1$};
    \node at (3.35,1.525) {$T_3$};


    
    \draw[Orange, postaction={decorate}] (11) to[out=80,in=-10] (02);
    \draw[Orange, postaction={decorate}] (23) to[out=180,in=50] (02);
    
    \draw[Orange, postaction={decorate}] (01) to[out=80,in=160] (11);
    \draw[Orange, postaction={decorate}] (11) to[out=270,in=340] (01);
    \draw[Orange]  (.85,1.425) to[out=90,in=120] (1.16,1.575);
    \draw[Orange] (1.16,1.575) to[out=290,in=280] (.85,1.425);
    
    \draw[Orange, postaction={decorate}] (01) to[out=280,in=80] (00);
    
    \draw[Orange, postaction={decorate}] (2.5,3.1) to[out=290,in=0] (23);
    
    \draw[Orange, postaction={decorate}] (-.1,2.35) to[out=0,in=250] (02);
    \draw[Orange, postaction={decorate}] (-.1,3.1) to[out=330,in=150] (02);

    \draw[Orange, postaction={decorate}] (3.1,2.5) to[out=240,in=25] (02);
    \draw[Orange, postaction={decorate}] (2.1,3.1) to[out=240,in=60] (02);
    \draw[Orange, postaction={decorate}] (-0.1,1.65) to[out=0,in=350] (02);
    \draw[Orange, postaction={decorate}] (-0.1,1.95) to[out=20,in=300] (02);
    
    \draw[Orange, postaction={decorate}] (1.125,3.1) to[out=230,in=100] (02);
    
    \draw[Orange, postaction={decorate}] (2.1,1.65) to[out=170,in=10] (02);
    
    \draw[Orange, postaction={decorate}]
    (3.1, 1.45) .. controls (2.75, 1.85)  and (2.55, 2.1) .. (2.2, 2.25);
    \draw[Orange, postaction={decorate}] (2.4,1.75) .. controls (1.9,2.4) and (1.75,1.5)  .. (1.295, 2.18);
    
    \draw[Orange, postaction={decorate}] (3.075,3.1) .. controls (2.5,2.25) and (2,2.25)  .. (1.75, 2.41);
    
    \draw[Orange, postaction={decorate}] (2.025,-.1) .. controls (1.95,0.25) and (1.5,0.3)  .. (.95, .39);



     \draw[Orange, postaction={decorate}] (-.1,1.5) to[out=0,in=150] (01);
     \draw[Orange, postaction={decorate}] (-.1,0.5) to[out=0,in=170] (00);
     \draw[Orange, postaction={decorate}] (-.15,-.05) to[out=0,in=190] (00);
    \draw[Orange, postaction={decorate}] (.5,-0.1) to[out=80,in=290] (00);
    \draw[Orange, postaction={decorate}] (2.1,0.4) to[out=180,in=-10] (00);
    
    \draw[Orange, postaction={decorate}] (2.4,.75) .. controls (2,1.4) and (1.5, .6) .. (00);
    
    \draw[Orange, postaction={decorate}] (2.1,1.35) to[out=170,in=51] (00);
    \draw[Orange, postaction={decorate}] (-.1,1.35) to[out=10,in=90] (00);
    \draw[Orange, postaction={decorate}] (-.1,1.03) to[out=-10,in=120] (00);
    \draw[Orange, postaction={decorate}] (1.55,-.1) to[out=100,in=-30] (00);
    \draw[Orange, postaction={decorate}] (1.05,-.1) to[out=110,in=-40] (00);
    
    \draw[Orange, postaction={decorate}] (3.1,1.1) to[out=140,in=-10] (11);
    
    \draw[Orange, postaction={decorate}] (1.795,3.1) .. controls (1.65,2.95) and (1.55,2.85)  .. (1.3, 2.825);
    
    \draw[Orange, postaction={decorate}] (3.1,1.95) .. controls (2.65,2.15) and (2.4,2.175)  .. (2.25, 2.235);
    
    \draw[Orange, postaction={decorate}] (2.8,.8) .. controls (2,1.2)   .. (1.35, .74);

\end{tikzpicture}
\vspace{2ex}
\end{minipage}
    \caption{The antagonistic coarsening of $\le$ and $\lebf$ in Fig.\ \ref{morsetes23} makes a Morse pre-order $\le^\dagger$ [left], with $\sSC$ depicted as the outlined sets. The realization yields a Morse tessellation $(\sT,\le)$ [right].}
    \label{MPOandTile}
\end{figure}
%
For Boolean discretizations, i.e. discretizations for which the closure operators for $(X,\scrT)$ and $(\cX,\le)$ commute,\footnote{For example, CW-decompositions.} cf.\ Defn.\ \ref{regularCA}, we can exploit the fact that sets in $\sO(\cX,\le^\dagger)$ correspond to regular closed sets, cf.\ Thm.\ \ref{clattbl}. In this case we can define a pre-order $\le^\topc$ on the top cells\footnote{Top cells are elements in $(\cX,\le)$ that are maximal with respect to $\le$.} $\xi^\topc\in \cX^\topc$ such that $\sO(\cX^\topc,\le^\topc) \cong \sO(\cX,\le^\dagger)$.
Such a pre-order $(\cX^\topc,\le^\topc)$ is called a \emph{\discresol}  and drastically reduces the amount of data to analyze, cf.\ Sect.'s \ref{dyngrad12} and \ref{algimpl}.\footnote{In the application to parabolic systems in Section \ref{gldecomppf} one may  achieve a data reduction of orders of magnitude using discrete resolutions, e.g. for a cubical CW-decomposition of a $d$-cube we have $4^{-d} < |\cX^\topc|/|\cX| < 2^{-d}$.
}
\Discresol\index{\Discresol} do not lose information about the initial Morse pre-order and associated Morse tessellation. The map $\cU^\topc \mapsto \cl~\cU^\topc$, with $\cU^\topc\in \sO(\cX^\topc,\le^\topc)$, defines an injective lattice homomorphism $\cl\colon\sO(\cX^\topc,\le^\topc) \rightarrowtail \sO(\cX,\le)$, whose image
$\sO(\cX,\le^\dagger)$ is the lattice of down-sets of the Morse pre-order, and yields the factorization:
\[
\begin{tikzcd}[column sep=large, row sep=huge]
\sO(\cX^\topc,\le^\topc) \arrow[r] \arrow[rd] & \sO(\cX,\le^\dagger) \arrow[l, "\cong"'] \arrow[r, tail] \arrow[d]                                & \sO(\cX,\le) \arrow[r, tail] & \sSet(\cX) \\
                       & \sO(\sSC,\le) \arrow[u, "\cong"'] \arrow[lu, "\cong"] \arrow[ru, tail] \arrow[rru, "\dyn^{-1}"', tail] &                   &  
\end{tikzcd}
\]
cf.\ Thm.\ \ref{dyngdres}.
The discretization $\dyn\colon (\cX,\le) \to(\sSC,\le)$ is the map that allows us to alternate between Morse pre-orders and \discresols, cf.\ Sect.\ \ref{dyngrad12}, and is dual to the injection $\cl\colon\sO(\cX^\topc,\le^\topc) \rightarrowtail \sO(\cX,\le)$.  A formula for $\dyn$ is given in Theorem \ref{thethmdyn}.
In summary, $(\cX,\le^\dagger)$ is the relevant topology that contains the information about closed attracting blocks. A discrete resolution $(\cX^\topc,\le^\topc)$ is a coarser data structure  than  $(\cX,\le^\dagger)$ and which yields the same down-sets, cf.\ Fig.\ \ref{morsetess1ex}.
The Morse pre-order can be retrieved from the \discresol, cf.\ Thm.\ \ref{thethmdynform}. In Sections \ref{subsec:discbr} and \ref{sec:parabolic:model} we exploit this principle in the application to parabolic systems.

\begin{figure}[h!]
\begin{minipage}{.4\textwidth}
\centering
\begin{tikzpicture}[scale=1]
\def\a{4}
\def\b{6}
\def\c{5}
\def\h{.1}

\node[rectangle, draw=black,fill opacity=.2, text opacity=1] (00) at (0, 0){\scriptsize $\xi_0 $};
\node[rectangle, draw=black,fill opacity=.2, text opacity=1] (01) at (0, 1){\scriptsize $ \xi_2 $};
\node[rectangle, draw=black,fill opacity=.2, text opacity=1] (02) at (0, 2){ \scriptsize $\xi_1 $};

\node[rectangle, draw=black,fill opacity=.2, text opacity=1] (10) at (1, 0){\scriptsize $\xi_3 $};
\node[rectangle, draw=black,fill opacity=.2, text opacity=1] (11) at (1, 1){\scriptsize $\xi_4 $};
\node[rectangle, draw=black,fill opacity=.2, text opacity=1] (12) at (1, 2){\scriptsize $\xi_5 $};

\node[rectangle, draw=black,fill opacity=.2, text opacity=1] (22) at (2, 2){ \scriptsize $\xi_6 $};
\node[rectangle, draw=black,fill opacity=.2, text opacity=1] (21) at (2, 1){ \scriptsize $\xi_7 $};

\draw[-latex,thick] (12) to (02);
\draw[-latex ,thick] (10) to (00);
\draw[latex-latex ,thick] (01) to (11);
\draw[latex-latex ,thick] (10) to (11);
\draw[latex-latex ,thick] (11) to (12);
\draw[latex-latex ,thick] (12) to (22);
\draw[-latex,thick] (01) to (02);
\draw[-latex,thick] (01) to (00);

\draw[-latex,thick] (21) to (11);
\draw[-latex,thick] (21) to (22);

\end{tikzpicture}
\end{minipage}
\begin{minipage}{.5\textwidth}
\centering
\begin{tikzpicture}[node distance=.25cm and .25cm]
\def\a{-.25}
\def\b{2.25}
\def\c{1}

\node[draw, ellipse] (0) at (\a, 0) {\scriptsize $\cS_1: \{\xi_1\}$};
\node[draw, ellipse] (1) at (\b, 0) {\scriptsize $\cS_0: \{\xi_0\}$};
\node[draw, ellipse] (2) at (\c,1) {\scriptsize $\cS_2: \{\xi_i\}_{2\leq i \leq 6}$};
\node[draw, ellipse] (3) at (\c,2) {\scriptsize $\cS_3: \{\xi_7\}$};
 
\draw[->,>=stealth,thick] (2) to (0);
\draw[->,>=stealth,thick] (2) to (1);
\draw[->,>=stealth,thick] (3) to (2);

\end{tikzpicture}
\end{minipage}
\vspace{2ex}
\caption{An associated discrete resolution $\le^\topc$ on the top-cells $\cX^\topc$ of CW-decomposition in Fig.\ \ref{morsetes23} [left] and the poset $\sSC$ of its partial equivalence classes [right], which is order-isomorphic to the poset  $(\sT,\le)$.}\label{morsetess1ex}
\end{figure}


\section{Algebraization}\label{algebraization}\index{Algebraization}

The above constructions have reduced dynamics to a (continuous) discretization $\tile\colon X\to \sSC$ of the space $X$. A discretization as such captures robust directionality properties of a flow. However, information about invariant sets is lost. 
The third theme in this paper is the algebraization of semi-flows; that is, the order-theoretic structures that encode the directionality are to be augmented with (co)homological data which carry information about (robust) invariant dynamics via Wazewski's principle.\index{Wazewski's principle}
A discretization $\tile\colon X\to \sSC$, or equivalently a Morse tessellation, gives rise to a filtering of $X$ consisting of regular closed 
attracting blocks, i.e. a lattice homomorphism $\alpha \mapsto F_\alpha X$, where $\alpha\in \sO(\sSC)$ is a down-set in $\sSC$ and $F_\alpha X = \tile^{-1} \alpha \in \sABlockR(\varphi)$ is an attracting block.  In the case of homology with field coefficients the representation theory of Cartan-Eilenberg systems,\index{Cartan-Eilenberg system} cf.\ Sect.\ \ref{CEsystems}, in particular Franzosa's connection matrix theory\index{Connection matrix}\index{Franzosa's connection matrix} \cite{fran2,fran3,fran,hms,robbin:salamon2}, describes a 
strict
$\sSC$-graded chain complex $\bigl(C^\tile(X),\cm^\tile \bigr)$ whose grading is given by
$C^\tile(X)=\bigoplus_{\cS\in \sSC} H(F_{\downarrow\cS} X,F_{\downarrow\cS^\pred} X)$,  where $H$ is the singular homology functor, cf.\ Fig.\ \ref{morsetess2ex}[left], cf.\ App.\ \ref{gradedcellchain}.
From the
graded chain complex 
$\bigl(C^\tile(X),\cm^\tile \bigr)$ all homologies $H(F_\beta X,F_\alpha X)$, with $\alpha,\beta\in\sO(\sSC)$, can be computed as homology of the sub-quotient chain complex $G_{\beta\smin\alpha} C^\tile(X)$ and which is denoted by $H^\tile(G_{\beta\smin\alpha}X)$.  This data can be visualized as a poset isomorphic to $\sSC$, 
whose elements are pairs $\bigl(\cS,P^\tile_\mu(G_\cS X)\bigr) \in \sSC\times \Z_+[\mu]$, where $P^\tile_\mu(G_\cS X)$
is the Poincar\'e polynomial of $H^\tile(G_\cS X)$ which uses the natural dimension grading of singular homology,
cf.\ Fig.\ \ref{morsetess2ex}[right].
Such a poset  will be referred to as the\index{Tessellar phase diagram} \emph{tessellar phase diagram}\footnote{The tessellar phase diagram is expressed with respect to Borel-Moore homology.}
$(\tessph,\le^\dagger)$ for  $(\cX,\le^\dagger)$, cf.\ Sect.\ \ref{doublegr}.

\begin{figure}[h!]

\begin{minipage}{.46\textwidth}
\centering
\begin{tikzcd}[column sep = 1.5em]
0 \arrow{r} &
\Z_2\langle \cS_2\rangle 
\arrow{r}{\tiny \dff_1^\tile=\begin{pmatrix}
1\\
1
\end{pmatrix}}&[5em] 
\bigoplus_{i=\{0,1\}}\Z_2\langle \cS_i\rangle \arrow{r} &
0
\end{tikzcd}
\end{minipage}
\begin{minipage}{.52\textwidth}
\centering
\begin{tikzpicture}[node distance=.2cm and .25cm]
\def\a{0}
\def\b{2}
\def\c{1}

\node[draw, ellipse] (0) at (\a, 0) {\scriptsize $\cS_1 : \mu^0$};
\node[draw, ellipse] (1) at (\b, 0) {\scriptsize $\cS_0 : \mu^0$};
\node[draw, ellipse] (2) at (\c,1) {\scriptsize $\cS_2 : \mu^1$};
\node[draw, ellipse] (3) at (\c,2) {\scriptsize $\cS_3 : 0$};
 
\draw[->,>=stealth,thick] (2) to (0);
\draw[->,>=stealth,thick] (2) to (1);
\draw[->,>=stealth,thick] (3) to (2);

\end{tikzpicture}
\end{minipage}
\vspace{2ex}
\caption{The $\sSC$-graded  chain complex  $(C^\tile(X),\cm^\tile)$  [left] of Fig.\ \ref{morsetes23} (computed using $\Z_2$ coefficients) and associated the tessellar phase diagram $(\tessph,\le^\dagger)$ [right]. We display the pair $\bigl(\cS,P^\tile_\mu(G_\cS X)\bigr)$ as $\cS$ and $P^\tile_\mu(G_\cS X)$ as a matter of style.}\label{morsetess2ex}
\end{figure}



A discretization $\tile\colon X\to \sSC$ may be considered purely from a topological perspective, independent of dynamics.  Given a sufficiently `nice' discretization, the connection matrix theory can be applied to $\tile$.   We regard this as part of the synthesis of dynamics and topology, and using dynamical tools to analyze topology.  In Section \ref{tesshom}, we show how this leads to a homology theory which we call \emph{tessellar homology} and which, in contradistinction to cellular homology, uses general tiles instead of CW-cells. In Section \ref{cellhom12} we show how cellular homology specifically is recovered.


One of the key advantages of using Morse pre-orders is that $\tile$ factors through the discretization of $(X,\scrT)$, which enables a computational approach to connection matrix theory using the algorithms of \cite{hms,hms2} and associated software \cite{cmcode}.
For the sake of the simplicity, we explain this in the case that $(X,\scrT)$ admits a finite CW-decomposition.
The idea is encapsulated in the following diagram, cf.\ Sect.\ \ref{cellhom12}:

\begin{equation}\label{factor3ind}
\begin{tikzcd}[column sep=large, row sep=huge]
                                 & (\cX,\leq) \arrow[d, "\dyn"] &   \\
X \arrow[ru, "\cell"] \arrow[r, "\tile"] & \sSC \arrow[r, "\ind"] & \Z
\end{tikzcd}
%
%
\end{equation}
%
%
In the case of a CW-decomposition map $\cell$, the associated pre-order is a partial order (face partial order) and the above diagram
provides the factorization via\! $(\cX,\le)$. 
The fact that we factor $\tile$ through $(\cX,\le)$ allows us to compute $(C^\tile(X),\cm^\tile)$ by instead computing connection matrices in two  simpler settings. First for the discretization $\cell\colon X\to (\cX,\le)$ to obtain the cellular chain complex $C^\cell(X)$ graded over the poset $\cX$ of cells which represents the singular homology as cellular homology; then a second time by re-grading $C^\cell(X)$ by $\sSC$ via $\dyn\colon \cX\to \sSC$ and using the algorithm of \cite{hms} to produce $(C^\tile(X), \cm^\tile)$.
%
As indicated above the discretization $\tile\colon X\to\sSC$ is  equivalent to a Morse tessellation $(\sT,\le)\xrightarrow[]{\cong} (\sSC,\le)$, where the tiles $T\in \sT$ are given by $T=\tile^{-1}\cS$.
Since $\{\cS\} = \beta\smin\alpha$, the homology of $T$ is given by
\[
H^\dyn(\beta\smin\alpha)\cong H^\tile(G_{\beta\smin\alpha}X) \cong H^\cell(G_{\beta\smin\alpha} X) \cong H^\BM(T),
\]
cf.\ Thm.\ \ref{tilehomtoBM}, where the latter is the \emph{Borel-Moore homology} of $T$, cf.\ \cite{BM,Bredon,Geog,goresky,ginzburg,iversen}. Since $T$ is the set-difference of two attracting blocks for $\varphi$ it is an isolating block (neighborhood) 
for $\varphi$, cf.\ \cite{lsa3}, and $H^{\tile}(\beta\smin\alpha)$ represents  the Conley index of $T$.
%

An additional scalar discretization $\ind\colon\sSC\to \Z$, cf.\ \eqref{factor3ind}, allows a second grading of $H^\tile$ which make $H^\tile$ a bi-graded homology theory (in the case of field coefficients) denoted by $H^\tile_{p,q}(X)$, cf. Sect.\ \ref{dimensiongrading}. 
The discretization $\ind$ induces a spectral sequence which allows an additional $\Z$-grading of the tessellar homology in accordance with the $\sSC$-grading of $C^\tile$. As a matter of fact $H^\tile_{p,q}(G_{\cU\smin\cU'}X)$ is well-defined for any convex set $\cU\smin\cU'$ in $\sSC$.
In this context $(C^\tile,\cm^\tile)$ may be regarded as chain complex as well as $\Z$-graded differential module/vector space. Define the Poincar\'e polynomials
$
    P_{\lambda,\mu}(C^\tile) = \sum_{\cS\in \sSC}P^\tile_{\lambda,\mu}(G_\cS X)
$ with
$P^\tile_{\lambda,\mu}(G_\cS X) = \sum_{p,q\in \Z} \bigl(\rk H^\tile_{p,q}(G_\cS X)\bigr) \lambda^p \mu^q$.
Then, the following
variation on the standard Morse relations are satisfied, cf.\ Thm.\ \ref{morserel1},
\begin{equation*}
    \sum_{\cS\in \sSC}P^\tile_{\lambda,\mu}(G_\cS X)
    = P_{\lambda,\mu}^\tile(X) + \sum_{r=1}^\infty (1+\lambda^r \mu) Q^r_{\lambda,\mu},
\end{equation*}
where $Q^r_{\lambda,\mu} \ge 0$ and the sum over $r$ is finite.
The $p$-index is a manifestation of the \bflt  and the $q$-index of the phase space topology making it a true tale of two topologies. The impact is most apparent in the application to parabolic flows where we use a canonical discretization $\lap \colon \sSC \to \Z$.

%


\section{Parabolic flows and braid invariants}


Chapter \ref{parabolic} encompasses  the final theme  wherein we demonstrate the ideas and methods in this paper for a large class of flows,
called \emph{discrete parabolic flows}.
Such flows occur in a wide variety of settings, e.g. studying the infinite dimensional dynamics of scalar parabolic equations which can be realized via discrete parabolic equations, cf.\ \cite{gv}, \cite{day}.
Discrete parabolic equations and parabolic flows
also play a prominent role in the theory of monotone twist maps whose dynamics can be 
studied using parabolic flows, cf.\ \cite{angenent1}, \cite{CzechV}. 
The introduction of Morse theory on braids in \cite{im} was sparked by questions for fourth order Lagrangian dynamics which use parabolic flows
to describe periodic solutions. Finally,  parabolic flows play a pivotal role
in computing braid Floer homology, cf.\ \cite{bgvw}.

A parabolic flow $\varphi$ is defined via  differential equations
$
\dot x_i = R_i(x_{i-1},x_i,x_{i+1}),
$
with $R_{i+d}=R_i$, 
where the (smooth) functions $R_i(x_{i-1},x_i,x_{i+1})$ are monotone with respect to their first and third argument and their stationary equations, given by $R_i(x_{i-1},x_i,x_{i+1})=0$, are referred to as \emph{parabolic recurrence relations}.
Parabolic recurrence relations form a perfect symbiosis
with discretized braids.
A \emph{discretized braid (diagram)} $x$ on $n$ strands and $d$ discretization points is a unordered collection of sequences $\cdots, x_0^\alpha,x_1^\alpha,\cdots$ such that
$x_{i+d}^{\theta(\alpha)} = x_i^\alpha$, for a permutation $\theta\in S_n$, and $\alpha=1,\cdots,n$. 

\begin{figure}[h!]
\centering
\begin{minipage}{.45\textwidth}
\begin{tikzpicture}[dot/.style={draw,circle,fill,inner sep=.75pt},line width=.7pt,scale=.65]
\def\a{0}
\def\b{1}
\def\c{2}
\def\d{3}
\def\e{4}
\foreach \x in {0,1,2,3,4,5,6}
    \foreach \y in {\a,\b,\c,\d,\e}
        \node (\x\y) at (\x, \y)[dot] {};
\foreach \y in {\a,\e}
    \foreach \x  in {0,1,2,3,4,5}
        \draw (\x,\y) to (\x+1,\y);
\foreach \x in {0,2,3,5} {
     \draw (\x, \b) to (\x+1, \c);
     \draw (\x,\c) to (\x+1,\d);
     \draw (\x,\d) to (\x+1,\b);
     }
\foreach \x in {1, 4} {
    \draw (\x,\d) to (\x+1,\c);
    \draw (\x,\c) to (\x+1,\b);
    \draw (\x,\b) to (\x+1,\d);
}
\end{tikzpicture}
\end{minipage}
\begin{minipage}{.4\textwidth}
\centering
\begin{tikzpicture}[scale=.65,node distance=.3cm and .25cm]
\def\a{0}
\def\b{2}
\def\c{4}
\def\h{1.25}
\node[draw, ellipse] (0) at (\a, 0) {\scriptsize $\cS_0 : \lambda^0 \mu^0$};
\node[draw, ellipse] (1) at (\c, 0) {\scriptsize $\cS_1 : \lambda^0 \mu^0$};

\node[draw, ellipse] (4) at (\a, 3.25*\h) {\scriptsize $\cS_4 : \lambda^5 \mu^3$};
\node[draw, ellipse] (5) at (\c, 3.25*\h) {\scriptsize $\cS_5 : \lambda^5 \mu^3$};

\node[draw, ellipse] (7) at (\a, 5.2*\h) {\scriptsize $\cS_7 : \lambda^7 \mu^5$};
\node[draw, ellipse] (8) at (\c, 5.2*\h) {\scriptsize $\cS_8 : \lambda^7 \mu^5$};
\node[draw, ellipse] (2) at (\b, \h) {\scriptsize $\cS_2 : \lambda^3 \mu^1+ \lambda^3 \mu^2$};
\node[draw, ellipse] (3) at (\b, 2.3*\h) {\scriptsize $\cS_3 : \lambda^4 \mu^2$};
\node[draw, ellipse] (6) at (\b, 4.25*\h) {\scriptsize $\cS_6b : \lambda^6 \mu^4$};
\node[draw, ellipse] (9) at (\b, 6.2*\h) {\scriptsize $\cS_9 : \lambda^8 \mu^6$};



\draw[->,>=stealth,thick] (2) to (0);
\draw[->,>=stealth,thick] (2) to (1);
\draw[->,>=stealth,thick] (3) to (2);
\draw[->,>=stealth,thick] (4) to (3);
\draw[->,>=stealth,thick] (5) to (3);
\draw[->,>=stealth,thick] (6) to (4);
\draw[->,>=stealth,thick] (6) to (5);
\draw[->,>=stealth,thick] (7) to (6);
\draw[->,>=stealth,thick] (8) to (6);
\draw[->,>=stealth,thick] (9) to (7);
\draw[->,>=stealth,thick] (9) to (8);

\end{tikzpicture}
\end{minipage}
\vspace{2ex}
\caption{A discretized (pseudo-anosov) braid $y$ [left], and associated reduced tessellar phase diagram $\otessph(y)$ [right]. The vertices in   $\otessph(y)$ contain the Poincar\'e polynomials of the 
bi-graded 
parabolic homology, cf.\ Sect.\ \ref{lapgrading}.}\label{fig:braid:cmg12}
\end{figure}

By viewing such sequences as piecewise linear interpolations between the anchor points the various strands `intersect', cf.\ Fig.\ \ref{fig:braid:cmg12}[left]. There is a non-degeneracy condition on the intersections (no tangencies). 
A collection of  stationary solutions of $\varphi$ of integer period  forms a discretized braid diagrams denoted by $y$ and referred to as a \emph{skeleton}, cf.\ Fig.\ \ref{fig:braid:cmg12}[left].\index{Skeleton} 
If we add an additional
periodic sequence $x$  the ensemble of $x$ and $y$  generically forms a braid $x\rel y$, called a \emph{relative} discretized braid.\index{Discretized braid}\index{Discretized braid!relative} Since $y$ is stationary for $\varphi$  we may consider $\varphi(t,x)\rel y$. Whenever $\varphi(t,x)\rel y$ becomes singular, i.e. strands develop tangencies,  then the total number of intersections of $x$ with the strands in $y$ decreases strictly.
This principle is crucial and emphasizes the intimate relation between parabolic flows and discretized braids.
As such relative braids can be partially ordered using parabolic flows which form the backbone of the canonical discretization of the \bflt for parabolic flows.
We  construct special CW-decompositions, cf.\ Sect.\ \ref{parabolic}, and discretizations of the \bflt via
appropriately defined discrete Lyapunov functions,
which allows us to establish Morse tessellations and Morse representations, cf.\ Sect.\ \ref{sec:parabolic:model}.
In the language of \discresols the partial order on the relative braid classes defines the poset $(\sSC,\le)$ which comes from the canonical CW-decomposition of $X$ given by the skeleton $y$ and the \bflt given by $\varphi$.

\begin{figure}[h!]
\centering
\tikzcdset{every label/.append style = {font = \tiny}}
\begin{tikzcd}[column sep=1.6em, row sep =3.5em]   
0 \arrow{r} & 
\Z_2\langle \cS_9\rangle 
\arrow{r}{\cm^\tile_{6}=\begin{pmatrix}
1\\
1
\end{pmatrix}} & [1.5em]
\bigoplus_{i= 7,8}\Z_2\langle \cS_i\rangle  
\arrow{r}{\cm^\tile_{5} = \begin{pmatrix}
1 ~~ 1
\end{pmatrix}} 
\arrow[phantom, near start,""{coordinate, name=Z}]{d} & [1.5em]
\Z_2\langle \cS_6\rangle
  \arrow[
    rounded corners,
    to path={
      -- ([xshift=2ex]\tikztostart.east)
      |- (Z) [near start]\tikztonodes
      -| ([xshift=-2ex]\tikztotarget.west)
      -- (\tikztotarget)
    }
  ]{dlll}{\cm^\tile_{4}=\begin{pmatrix}
0\\
0
\end{pmatrix}} \\
\bigoplus_{i=4,5}\Z_2\langle \cS_i\rangle \arrow{r}{\cm^\tile_{3} = \begin{pmatrix}
0~~~1\\
1~~~1
\end{pmatrix}} & 
\bigoplus_{i=2,3} \Z_2\langle \cS_i\rangle \arrow{r}{\cm^\tile_{2} = \begin{pmatrix}
0 ~~ 0
\end{pmatrix}} & 
\Z_2\langle \cS_2\rangle  
\arrow{r}{\cm^\tile_{1}=\begin{pmatrix}
1\\
1
\end{pmatrix}} & 
\bigoplus_{i=0,1} \Z_2\langle \cS_i\rangle \arrow{r}{\cm^\tile_{0}=0} & 0
\end{tikzcd}
 \caption{Associated  complex $\scrA(\beta)$, computed over $\K=\Z_2$ coefficients, viewed as chain complex using dimension grading. The index $q$
 for $\cm^\tile_q$ reflects the natural grading induced by singular homology.
 }\label{fig:my_label2}
\end{figure}

In   Section \ref{algimpl} we give an overview of the algorithmic steps in computing the order structures and connection matrices and in Section \ref{lapgrading} we discuss the bi-graded parabolic homology.
In Section \ref{tesspara} 
we show that the parabolic differential module (and the associated  tessellar phase diagrams) we obtain for parabolic flows extend the results in \cite{im} and provide an  invariant for positive braids which also defines a new invariant for scalar parabolic equations, cf.\ Thm.\ \ref{stableh2}. 
%
The importance of Theorem \ref{stableh2} is that we obtain complete insight in the homology of loop spaces for scalar parabolic equations as well as the boundary homomorphisms which contain information about connecting orbits for parabolic equations. 
To best explain the discretization of topology and dynamics and the main statement of Theorem \ref{stableh2} we consider an example. Let $y$ be a discretized braid $y$ given in Figure \ref{fig:braid:cmg12}[left]
and $\varphi$ is any parabolic flow for which $y$ is stationary. This can be rephrased in terms of a bi-topological space for which $y$ defines a natural discretization. The  sub-poset of  homologically non-trivial braid classes in $(\sSC,\le)$ is given by the reduced tessellar phase diagram $(\otessph,\le^\dagger)$ 
displayed in Fig.\ \ref{fig:braid:cmg12}[right].
The more detailed information is given by the \emph{parabolic differential module} $\scrA(\beta)$ generated by $\tile\colon X\to \sSC$ and the scalar discretization $\lap\colon \sSC\to \Z$, which counts the intersections of $x$ with $y$ divided by two, cf.\ Sect.\ \ref{tesspara} for a detailed definition.
As described in Section \ref{lapgrading} 
$\lap$ provides a scalar grading on the tessellar homology $H^\tile(G_{\cU\smin\cU'}X)$, with $\cU\smin\cU'$ convex in $\sSC$, which yields the \emph{parabolic homology} $\tH_{p,q}(G_{\cU\smin\cU'}X)$ where
\[
H^\tile(G_{\cU\smin\cU'}X) = \bigoplus_{p,q\in\Z} \tH_{p,q}(G_{\cU\smin\cU'}X).
\]
Figure \ref{fig:braid:cmg12}[right] lists the Poincar\'e polynomials of the non-trivial tiles with respect to parabolic homology.
The differential module $\scrA(\beta)$
 may be  regarded as a chain complex in Figure \ref{fig:my_label2} with homology
 $H^\BM_q(G_\cS X) =\oplus_{p\in \Z}\tH_{p,q}(G_\cS X)$,
or as $\Z$-graded differential module in Figure \ref{fig:my_label2b} with homology
 $\tH_{p,*}(G_\cS X) =\oplus_{q\in \Z}\tH_{p,q}(G_\cS X)$. 
By making further specification of the entries in $\cm^\tile$ the differential module can also be represented as to show the $\otessph$-order, cf.\ Sect.\ \ref{tesspara}.
Theorem \ref{stableh2} shows that the parabolic differential module, and thus the reduced tessellar phase diagram is a topological invariant for the topological braid $\beta(y)$, i.e. all braid diagrams isotopic to $y$, cf.\ Sect.\ \ref{tesspara}. Due to Theorem \ref{stableh2} the parabolic differential module and reduced tessellar phase diagram can be denoted by  $\scrA(\beta)$ and $\otessph(\beta)$  respectively.
In summary, the application of the methods put forth in this paper to parabolic flows gives a novel approach towards computing algebraic topological information about infinite dimensional problems.
\begin{figure}
\centering
\tikzcdset{every label/.append style = {font = \tiny}}
\begin{tikzcd}[column sep=2.2em, row sep =3.5em]   
0 \arrow{r} & 
\Z_2\langle \cS_9\rangle 
\arrow{r}{\cm^\tile_{8}=\begin{pmatrix}
1\\
1
\end{pmatrix}} & [1.5em]
\bigoplus_{i= 7,8}\Z_2\langle \cS_i\rangle  
\arrow{r}{\cm^\tile_{7} = \begin{pmatrix}
1 ~~ 1
\end{pmatrix}} 
\arrow[phantom, near start,""{coordinate, name=Z}]{d} & [1.5em]
\Z_2\langle \cS_6\rangle
  \arrow[
    rounded corners,
    to path={
      -- ([xshift=2ex]\tikztostart.east)
      |- (Z) [near start]\tikztonodes
      -| ([xshift=-2ex]\tikztotarget.west)
      -- (\tikztotarget)
    }
  ]{dlll}{\cm^\tile_{6}=\begin{pmatrix}
0\\
0
\end{pmatrix}} 
\\
\bigoplus_{i=4,5}\Z_2\langle \cS_i\rangle \arrow{r}{\cm^\tile_{5} = \begin{pmatrix}
1~~~1
\end{pmatrix}} 
%
%
%
\arrow[rr, "\tiny{ \dff^\tile_5 =\begin{pmatrix} 0~~~0\\ 0~~~1\end{pmatrix} }", bend right=49]
& 
 \Z_2\langle \cS_3\rangle \arrow{r}{\cm^\tile_{4} = \begin{pmatrix}
0\\
0
\end{pmatrix}} 
& 
\Z_2\oplus\Z_2\langle \cS_2\rangle  
\arrow{r}{\cm^\tile_{3}=\begin{pmatrix}
1~~~0\\
1~~~1
\end{pmatrix}} & 
\bigoplus_{i=0,1} \Z_2\langle \cS_i\rangle \arrow{r}{} & 0\\
 & \phantom{.}  
\end{tikzcd}
 \caption{The differential module $\scrA(\beta)$, computed over $\K=\Z_2$ coefficients,
 with lap number grading. 
For the entry $\cm_p^\tile$ the index $p$ reflects the lap number.
}\label{fig:my_label2b}
\end{figure}

The themes throughout this paper touch upon many topics. Section \ref{postlude} concludes with a discussion of related remarks and open problems.

%% file: topanddisc.tex
\chapter{Topology, discretization and bi-topological spaces}
\label{dynastop}

%
Our philosophy is that discretization is the study of appropriate \emph{finite} substructures.
In this chapter we start with an exposition of topological spaces, closure algebras and modal algebras and discuss discretization in terms of Boolean algebras with operators. 
Closure algebras provide an equivalent way to describe a topological space. In general modal algebras are also related to topological spaces which provides the essential link to dynamical systems. The latter can be used to regard various aspects of dynamics in terms of topology.

\section{Closure algebras}\label{CA-disc}\index{Closure algebra}
Let $(X,\scrT)$ be a topological space and let $\sSet(X)$ denote the (complete and atomic) Boolean algebra of subsets of $X$.   For $(X,\scrT)$ we define an associated  \emph{closure algebra}\index{Closure algebra}  as the pair $\bigl(\sSet(X),\cl \bigr)$, where $\cl\colon \sSet(X)\to \sSet(X)$ is the operator defined as the closure of a subset, 
$\cl\, U = \bigcap\bigl\{U'\supset U\mid U'\text{~closed} \bigr\}$, cf.\ \cite{mcktar}, which is our first source of closure algebras.   
In general, an operator $\cl\colon\sSet(X) \to \sSet(X)$ is a \emph{closure operator}\index{Closure operator} if all four Kuratowski axioms{\index{Kuratowski axioms} for a closure operator are satisfied: for all $U,U'\subset X$,
\begin{enumerate}
    \item [(K1)] (normal) $\cl~\varnothing = \varnothing$;
    \item [(K2)] (additive) $\cl(U\cup U') = \cl~U \cup \cl~U'$;
    \item [(K3)] (sub-idempotent) $\cl\bigl(\cl~ U\bigr) \subset \cl~ U$;\footnote{Axiom (K3) in combination with Axiom (K4) this implies that
    $\cl$ is idempotent, i.e. $\cl\bigl(\cl~ U\bigr) = \cl~U$.}
    \item [(K4)] (expansive) $U \subset \cl~ U$.\footnote{The single condition (K): $U\cup \cl~U\cup \cl\bigl(\cl~ U'\bigr) = \cl(U\cup U')\smin \cl~\varnothing$ is equivalent to (K1)-(K4).} 
\end{enumerate}
Continuous maps between topological spaces can also be described in terms of closure algebras.
Let $g\colon X\to Y$ be a continuous map\index{Continuous map} between topological spaces.
Then, $g^{-1}$ defines a map from $\sSet(Y)$ to $\sSet(X)$.
As a matter of fact $g^{-1}\colon \sSet(Y) \to \sSet(X)$ is a completely additive\footnote{Closed with respect to arbitrary intersections and unions.} Boolean homomorphism\index{Boolean homomorphism}\index{Boolean homomorphism!completely additive}  of complete and atomic Boolean algebras\index{Boolean algebra}\index{Boolean algebra!complete and atomic}, cf.\ App. \ref{complBoolAlg}. Consider the (not necessarily commutative) diagram: 
\begin{equation}
\label{closalgsemidiag12}
\begin{tikzcd}[column sep=large, row sep=large]
\sSet(X) \arrow[r, "\cl_X"]                & \sSet(X)                 \\
\sSet(Y)  \arrow[r, "\cl_Y"] \arrow[u, "g^{-1}"] & \sSet(Y)  \arrow[u, "g^{-1}"']
\end{tikzcd}
\end{equation}
The continuity of $g$ is equivalent to the condition
$\cl_X g^{-1}(V) \subset g^{-1}(\cl_YV)$ for all $V\subset Y$
which makes $g^{-1}$ a \emph{semi-homomorphism}\index{Semi-homomorphism of closure algebras} of closure algebras.
In particular, if $V\subset Y$ is closed in $Y$, then $g^{-1}(V)$ is closed in $X$.
In case $\cl_X g^{-1}(V) = g^{-1}(\cl_YV)$ for all $V\subset Y$ the operator $g^{-1}$ is called a homomorphism of closure algebras,\footnote{ For a \emph{homomorphism} of closure algebras the diagram in \eqref{closalgsemidiag12} commutes.}\index{Homomorphism of closure algebras} in which case $g$ is an open continuous map.\footnote{If there is no ambiguity about the topological space the sub-index of $\cl$ will be omitted.}\index{Open continuous map}\index{Continuous map!open}
A second source for closure algebras are via pre-ordered sets.\index{Pre-order} Let $(X,\le)$ be a pre-order. Define
\[
\cl_\le U := \bigl\downarrow U =\bigl\{x\in X\mid x\le y \text{~~for some ~}y\in U \bigr\}.
\]
Then, $(\sSet(X),\cl_\le)$ is a closure algebra and the closure operator $\cl_\le$ is completely additive,\index{Closure operator!completely additive} i.e. Kuratowski axiom (K2) is satisfied with respect to arbitrary unions. The associated topological space is denoted by $(X,\scrT_\le)$ which is an Alexandrov topological space\index{Alexandrov topological space} and the topology $\scrT_\le$ is called an \emph{Alexandrov topology}, i.e.\index{Alexandrov topology} $\scrT_\le$ is closed under arbitrary intersections and unions.

On the other hand every topological space induces a natural pre-order as follows:
\[
x\le_\scrT x'\quad\text{if and only if}\quad x\in \cl\{x'\},
\]
which is called the \emph{specialization pre-order}\index{Specialization pre-order} associated to $(X,\scrT)$. In general, the topology $\scrT_{\le_\scrT}$ induced by $\le_\scrT$ is finer than $\scrT$.
In particular $\bigl\downarrow U$ is not equal to $\cl~U$ in that case.
If we start from an Alexandrov topology $\scrT$, then $\scrT_{\le_\scrT} = \scrT$. The above described duality between topological spaces and pre-orders will be used to treat discretization of topological spaces, cf.\ \cite{BMM} for further details on closure algebras.\index{Closure algebra}
A topological space yields a closure algebra where the Boolean algebra is complete and atomic. This concept can be defined for any Boolean algebra, cf.\ Sect.\ \ref{clandbiclos} and App.\ \ref{binrel}.

\section{Modal operators and modal algebras}
\label{modalstuff}\index{Modal operator}

A third source for closure algebras is given by modal operators and  binary relations  on a set $X$, cf.\ App.\ \ref{binrel}.
In general, an 
operator $\rmPhi\colon \sSet(X) \to \sSet(X)$ is called a \emph{modal operator}\index{Modal operator} if following axioms are satisfied: for all $U,U' \subset X$, 
\begin{enumerate}
    \item [(M1)] (normal) $\rmPhi \varnothing= \varnothing$;
    \item [(M2)] (additive) $\rmPhi (U\cup U') = \rmPhi U \cup \rmPhi U'$;
\end{enumerate} 
 The pair $(\sSet(X),\rmPhi)$ is called\index{Modal algebra} a \emph{modal algebra} and is an example of a Boolean algebra with operators, cf.\ \cite{JonssonTarski}, 
 \cite{mcktar}.\index{Boolean algebra with operators}
 A modal algebra  defines a topology  on $X$. Consider the set $\sFwdset(\rmPhi)$ consisting  of subsets $U\subset X$ such that $\rmPhi U\subset U$.\index{$\sFwdset(\Phi)$} 
\begin{proposition}
\label{modaltop}
    The set $\sFwdset(\rmPhi)$ defines a bounded, distributive lattice with binary operations $\cap$ and $\cup$. Moreover, $\sFwdset(\Phi)$ is closed under arbitrary intersections, i.e.  arbitrary intersections of sets in $\sFwdset(\rmPhi)$ are again in $\sFwdset(\rmPhi)$. 
\end{proposition}
\begin{proof}
The subsets $\varnothing$ and $X$ are in $\sFwdset(\rmPhi)$ since $\rmPhi$ is a normal operator.
Finite unions of sets in $\sFwdset(\rmPhi)$ are obviously  again in $\sFwdset(\rmPhi)$ since $\rmPhi$ is additive.
Let $\{U_i\}$ be an arbitrary collection is subsets in $\sFwdset(\rmPhi)$. Then, $\rmPhi\bigl( \bigcap_i U_i\bigr) \subset \rmPhi U_i \subset U_i$ and therefore $\rmPhi\bigl( \bigcap_i U_i\bigr) \subset \bigcap_i U_i$.
\end{proof}
The lattice $\sFwdset(\rmPhi)$ defines a topology $\scrT_\rmPhi$ on $X$ by declaring the subsets in $\sFwdset(\rmPhi)$ to be the closed sets.
%
The topology $\scrT_\rmPhi$ given by the lattice $\sFwdset(\rmPhi)$ can also be characterized by an associated closure operator.
\begin{proposition}
    \label{chartopviaclop}
Consider the operator $\cl_\Phi\colon\sSet(X)\to \sSet(X)$ defined by
\begin{equation}
\label{Phiclasop}
\cl_\rmPhi U = \bigcap\bigl\{ U'\supset U\mid U'\in \sFwdset(\rmPhi)\bigr\} \in \sFwdset(\Phi).
\end{equation}
Then, $\cl_\Phi$ is a closure operator and  $\sFwdset(\Phi) = \sFwdset(\cl_\Phi)$, i.e.
the sets in $\sFwdset(\Phi)$ are exactly the closed sets  defined by $\cl_\Phi$.
\end{proposition}
\begin{proof}
    The definition of $\cl_\Phi$ is a standard construction of a closure operator satisfying the Kuratowski axioms (K1)-(K4). 
    A set $U\in \sFwdset(\cl_\Phi)$ satisfies $\cl_\Phi U\subset U$ and thus $\cl_\Phi U = U$, i.e. $U\in \sFwdset(\Phi)$ and thus $\sFwdset(\cl_\Phi)\subset \sFwdset(\Phi)$.
    If $U\in \sFwdset(\Phi)$, then $\cl_\Phi U = U$ which yields 
    $\sFwdset(\Phi)\subset \sFwdset(\cl_\Phi)$. Combining both inclusions gives the desired statement.
\end{proof}

The lattice $\sFwdset(\Phi)$ is a complete \emph{co-Heyting algebra}\index{Co-Heyting algebra}
with binary subtraction\index{Binary subtraction} $U - U' := \cl_\Phi (U\smin U')$. 
%
%
 The specialization pre-order $\le_\rmPhi$ defined by the topology $\scrT_\rmPhi$ (or equivalently the closure operator $\cl_\rmPhi$)
is the transitive reflexive closure\index{Transitive reflexive closure} of the binary relation $\phi\subset X \times X$, referred to as the \emph{specialization relation}\index{Specialization relation}, which is defined by
\begin{equation}
    \label{dualmodbinrel}
    (x,x')\in \phi  \quad\text{if and only if}\quad x\in \rmPhi\{x'\},
\end{equation}
where $\Phi=\phi^{-1}$ is regarded as an operator on $\sSet(X)$ and is an example of a completely additive modal operator, cf.\ \ref{complBoolAlg}.
As before $\bigl\downarrow U$ does not  coincide with $\cl_\rmPhi U$ in general. This is due to the fact that a modal operator is not completely additive\footnote{An operator is said to be completely additive if it is closed under arbitrary unions.} in general. For the time being the above described  duality between $(X,\scrT_\rmPhi)$ and  the associated closure algebra $(\sSet(X),\cl_\rmPhi)$ suffices.

In the case $\Phi$ is a \emph{completely additive} modal operator on $\sSet(X)$ the specialization relation recovers 
$\Phi$ and vice versa.
To be more precise,\index{Modal operator!completely additive}
\[
(x,x')\in \phi\quad\text{~~if and only if~~}\quad(x',x)\in\phi^{-1}\quad \text{~~if and only if~~}\quad x\in \phi^{-1}\{x'\},
\]
where for the latter we regard the opposite relation $\phi^{-1}$ as modal operator: 
\begin{equation}
    \label{modalcorres12}
\Phi U:= \bigcup_{x\in U} \Phi\{x\},\quad \text{where}\quad \Phi\{x\}=\bigl\{ y\in X\mid (x,y) \in \phi^{-1}\text{~~for some~~} x\in U\bigr\}.
\end{equation}
The transitive reflexive closure $\phi^{\bm{+=}} = \bigcup_{k\ge 0}\phi^k$\index{Transitive reflexive closure}\index{Transitive reflexive closure!of a binary relation|see{Binary relation}}
defines a pre-order $(X,\phi^{\bm{+=}})$ and therefore $\Phi^{\bm{+=}} = \bigcup_{k\ge 0}\Phi^k$\index{Transitive reflexive closure! of a modal operator}
is a completely additive closure operator. In particular:
\begin{lemma}
\label{equivsofclos}
    $\Phi^{\bm{+=}} U = \bigl\downarrow U = \cl_\Phi U$.
\end{lemma}
\begin{proof}
    The first equility is a direct consequence of the duality between $\phi$ and $\Phi$.
    As for the second equality we argue as follows.
    If $\cl_\Phi U =U$, then $U\in \sFwdset(\Phi)$, i.e. $\Phi U \subset U$ and thus $\Phi^{\bm{+=}} U=U$. Conversely,
    if $\Phi^{\bm{+=}} U=U$, then $\Phi U\subset U$ and thus $\cl_\Phi U =U$. Therefore, $\Phi^{\bm{+=}} U=U$ if and only if $\cl_\Phi U=U$ and $\Phi^{\bm{+=}} U=U= \cl_\Phi U$.
\end{proof}
In this case a binary relation $\phi\subset X\times X$ is the source of a closure algebra: $\bigl(\sSet(X),\Phi^{\bm{+=}}\bigr)$. The associated topology is 
an Alexandrov topology  and is equivalent to the specialization pre-order, cf.\ \ref{binrel122}.

Let  $(\sSet(X),\Phi)$ and $(\sSet(Y),\Psi)$ be   modal algebras.
As for closure algebras a map $g\colon X\to Y$ yields a Boolean homomorphism
$g^{-1}\colon \sSet(Y) \to \sSet(X)$. The latter is a \emph{semi-homomorphism of modal algebras}\index{Semi-homomorphism!of a modal algebra} if 
$\Phi g^{-1}(V) \subset g^{-1}(\Psi V)$
for all $V\subset Y$ and is expressed in the (not necessarily commutative) diagram:
\begin{equation}\label{semiforMA}
\begin{tikzcd}[column sep=large, row sep=large]
\sSet(X) \arrow[r, "\rmPhi"]                & \sSet(X)                 \\
\sSet(Y)  \arrow[r, "\Psi"] \arrow[u, tail, "g^{-1}"] & \sSet(Y)  \arrow[u, tail, "g^{-1}"']
\end{tikzcd}
\end{equation}
The semi-homomorphism property for Boolean homomorphisms $g^{-1}\colon \sSet(Y) \to \sSet(X)$ is related to continuity of $g$:
\begin{proposition}
    \label{semiMAcont}
    Let  $g^{-1}\colon (\sSet(Y),\Phi) \to (\sSet(X),\Psi)$ be a semi-homomorphism of modal algebras for a map $g\colon X\to Y$, i.e. $\Phi g^{-1}(V) \subset g^{-1}(\Psi V)$
for all $V\subset Y$. Then, $g\colon (X,\scrT_\Phi) \to (Y,\scrT_\Psi)$ is a continuous map. 
\end{proposition}
\begin{proof}
The closure operators $\cl_\Psi$ and $\cl_\Psi$ are given by
\eqref{Phiclasop}. 
Since $g^{-1}$ is a completely additive Boolean homomorphism we have
\[
g^{-1}\bigl(\cl_\Psi V\bigr) =  \bigcap \bigl\{ g^{-1}(V') \mid V'\supset V,~~ \Psi V'\subset V'\bigr\}.
\]
The fact that $g^{-1}$ is a semi-homomorphism of modal algebras implies,
for $V'$ closed in $Y$, that $\Phi g^{-1}(V') \subset g^{-1}(\Psi V')\subset g^{-1}(V')$
and thus $g^{-1}(V')$ is closed in $X$.
This implies that $\cl_\Phi g^{-1}(V) \subset g^{-1}(\cl_\Psi V)$, which proves that 
$g\colon X\to Y$ is a continuous map.
\end{proof}
\begin{remark}
    \label{charsemihomMA}
    If $\Xi\colon \sSet(Y) \to \sSet(X)$ is completely additive Boolean homomorphism\index{Boolean homomorphism!completely additive} then $\Xi=g^{-1}$ for some map $g\colon X\to Y$, cf.\ Prop.\ \ref{charcompladd} and \cite{JonssonTarski}. The latter is given by $g(x) =y$ for the unique $y\in Y$ such that $x\in \Xi\{y\}$.
\end{remark}

\begin{remark}
If $\Phi$ and $\Psi$ are completely additive operators then
    Proposition \ref{semiMAcont} can be proved  using  transitive reflexive closure. Observe that $\Phi^k g^{-1}(V) \subset g^{-1}(\Psi^k V)$ for all $k\ge 0$. 
    Then,
    \[
    \cl_\Phi g^{-1}(V) = \bigcup_{k\ge 0} \Phi^k g^{-1}(V) \subset 
    \bigcup_{k\ge 0}  g^{-1}(\Psi^k V)  = g^{-1}\bigl(\bigcup_{k\ge 0}\Psi^k V\bigr)
    = g^{-1}(\cl_\Psi V),
    \]
    which establishes continuity.
\end{remark}

\begin{remark}
    \label{dualrelclosalg}
    For a pre-order $(X,\le)$ the associated dual given by $\cl_\le U=\bigl\downarrow U$ defines a \emph{complete and atomic and completely additive} closure algebra $(\sSet(X),\cl_\le)$, i.e.\index{Closure algebra!complete and atomic and completely additive} $\sSet(X)$ is a complete and atomic Boolean algebra
    and the closure operator $\cl_\le$ is completely additive.
    The completely additive closure operator $\cl_\le$ retrieves the pre-order. Similarly, a binary relation $\phi\subset X\times X$ yields
    a {complete and atomic, and completely additive} closure algebra $(\sSet(X),\cl_\Phi)$, where $\cl_\Phi = \Phi^{\bm{+=}}$ and
    $\Phi = \phi^{-1}$. The closure operator retrieves the transitive reflexive relation $\phi^{\bm{+=}}$, but not $\phi$ in general. For a complete and atomic, and completely additive modal
    algebra $(\sSet(X),\Phi)$ the operator $\Phi$ retrieves $\phi$.
    These principles play a role in the duality theory of closure algebras and modal algebras, cf.\ App.\ \ref{complBoolAlg}, Sect.\ \ref{clandbiclos} and \cite{mcktar}, \cite{JonssonTarski}. In this text we are mainly interested in closure algebras and their duality.
\end{remark}
 
 \begin{remark}
 \label{derivativermk}
 Closely related to a closure operator is the notion of a derivative operator. A modal operator $\der\colon \sSet(X) \to \sSet(X)$ is called a \emph{derivative operator}\index{Derivative operator} if following axioms are satisfied: for all $U,U' \subset X$, 
\begin{enumerate}
    \item [(D1)] (normal) $\der \varnothing= \varnothing$;
    \item [(D2)] (additive) $\der (U\cup U') = \der U \cup \der U'$;
    \item [(D3)] (quasi-idempotent) $\der(\der U)\subset U\cup \der U$.\footnote{Axiom (D3) is equivalent to from Axiom (K4) via $\cl=\id\cup\der$.}
\end{enumerate} 
 The pair $(\sSet(X),\der)$ is called a \emph{derivative algebra},\index{Derivative algebra} cf.\ \cite{Esakia}, \cite{mcktar}.
A derivative operator defines a closure operator via 
\begin{equation}
    \label{clfromder}
    \cl_\der U:= U\cup \der U.
\end{equation}
For every closure operator there exists a derivative operator such that \eqref{clfromder} holds, e.g.\ take $\der = \cl$.\footnote{The choice of a derivative operator is clearly not unique. An important non-trivial choice is given by the \emph{derived set}, the set of limit points of a set $U$: 
\[
\der U := \bigl\{ x\in X\mid N\cap U\smin\{x\}\neq \varnothing\text{~for all neighborhoods~} N\ni x\bigr\},
\]
which is \emph{not} equal to $\cl~U$ in general.
}
\end{remark}

\section{Discretization of topology}\index{Discretization!of topology}
\label{disc-top}

We start with a general description of discretization of topology in terms of closure algebras. This procedure can then be used for the same purpose in the setting of modal algebras. These techniques play a role in the discussion of treating dynamics in terms of topology.

\subsection{Closure algebra discretization}\index{Closure algebra discretization}
\label{closalgdisc12}
Discretization of a topological space $(X,\scrT)$ in the spirit of closure algebras is
an \emph{injective}  Boolean homomorphism\footnote{\label{footinj} In the context of discretizaztion we consider injective homomorphisms with finite range, i.e. finite subalgebras. The theory can be phrased in more general terms via homomorphisms.}   $|\cdot|\colon \sSet(\cX)\rightarrowtail \sSet(X)$, where $\sSet(\cX)$ is the
powerset of a finite set $\cX$, in combination with an appropriately chosen \emph{discrete} closure operator $\ccl\colon \sSet(\cX)\to \sSet(\cX)$ such that $\cl|\cU|\subset |\ccl \cU|$ for all $\cU\in \sSet(\cX)$, in which case $|\cdot|\colon \sSet(\cX)\to \sSet(X)$ is a {semi-homomorphism} of closure algebras. 
We refer to the elements $\xi\in \cX$ as \emph{cells}.\index{Cell}
%
%
%
This discretization is captured by the following  diagram in the category of closure algebras and semi-homomorphisms:\footnote{In the above mentioned  category of closure algebras we employ the morphisms are semi-homomorphisms of closure algebras.}
\begin{equation}\label{dia:disctop}
\begin{tikzcd}[column sep=large, row sep=large]
\sSet(X) \arrow[r, "\cl"]                & \sSet(X)                 \\
\sSet(\cX)  \arrow[r, "\ccl"] \arrow[u, tail, "|\cdot|"] & \sSet(\cX)  \arrow[u, tail, "|\cdot|"']
\end{tikzcd}
\end{equation}
\noindent
The closure algebra $(\sSet(\cX),\ccl)$ defines a finite topology on $\cX$ by declaring $\cU$ closed if and only if $\ccl\cU =\cU$.  As $\cX$ is a finite set any topology on $\cX$ is necessarily an {Alexandrov topology}\index{Alexandrov topology},
 and is equivalent to the specialization pre-order $(\cX,\le)$\footnote{The Alexandrov topology is $T_0$ if and only if the specialization pre-order is
a partial order.}  defined by
\begin{equation}
    \label{spec1}
\xi\le \xi' \quad\text{if and only if}\quad  \xi\in \ccl\{\xi'\},
\end{equation}
cf.\ App.\ \ref{sec:birkhoff}.
The discretization described above allows us to regard $\cX$ as an algebra $\bigl(\sSet(\cX),\ccl\bigr)$, as a pre-order $(\cX,\le)$, and as topological space $\cX$. 
In general we do not differentiate between the specialization pre-order and the Alexandrov topology, and we refer to the triple $(\cX,\ccl,|\cdot|)$ as a closure algebra discretization, or \emph{CA-discretization} of $(X,\scrT)$.\index{CA-discretization}
\begin{definition}\index{Boolean CA-discretization}\index{CA-discretization!Boolean}
\label{regularCA}
A CA-discretization is called \emph{Boolean} if $\cl|\cU| = |\ccl\cU|$ for all $U\in \sSet(\cX)$.
\end{definition}
This definition in particular implies that \eqref{dia:disctop} is commutative in which case the  map $|\cdot|\colon \sSet(\cX)\to \sSet(X)$ is  a {homomorphism} of closure algebras, cf.\ \cite{BMM}.
For a pre-order
$(\cX,\le)$ we  define\index{Down-set} 
a \emph{down-set} $\cU\subset \cX$ by the property: $\xi'\in \cU$, $\xi\le \xi'$, then $\xi\in \cU$. The set of down-sets is denoted by
 $\sO(\cX,\le)$\footnote{If there is no ambiguity about the pre-order we write $\sO(\cccX)$ for short. Another common notation is $\Invsetpl(\le)$.}
which by construction  is  a finite distributive lattice with binary operations $\cap $ and $\cup$, cf.\ App.\ \ref{posets}, \ref{sec:birkhoff}} and \cite{davey:priestley}. 
Note that $\sO(\cX,\le) = \sFwdset(\ccl)$, where $\ccl$ is the associated closure operator on $\sSet(\cX)$.
 In a similar fashion we can define the lattice of \emph{up-sets} $\sU(\cX,\le)=\sFwdset(\bccl)$,\index{Up-set}
where $\bccl\,\cU = \st \cU$
 is the\index{Conjugate closure operator}\index{Star}\index{Closure operator!conjugate} \emph{conjugate closure operator}, cf.\ App.\ \ref{operators} and \cite{JonssonTarski}.
The join-irreducible elements of both lattices are characterized as
\[
\ccl~\xi = \big\downarrow\xi := \setof{\xi'\mid \xi' \leq \xi} \quad\text{and}\quad \st \xi = \big\uparrow\xi := \setof{\xi'\mid \xi \leq \xi'},\quad \xi\in \cX,
\]
which are called the \emph{principal} down-sets and up-sets respectively.\index{Convex set}
Intersections of up-sets and down-sets are the \emph{convex sets} in $(\cX,\le)$\footnote{The convex sets in the pre-order $(\ccX,\le)$  are  the \emph{locally closed} subsets in $\ccX$ as topological space.}
and are denoted by $\sCo(\cX,\le)$ which is a meet-semilattice with respect to $\cap$.

\begin{remark}
\label{partialequiv}
For a pre-order $(\cX,\le)$ we can define the \emph{partial equivalence classes}\index{Partial equivalence class} by $\xi\sim \xi'$ if and only if
$\xi\le \xi'$ and $\xi'\le \xi$. The set of partial equivalence classes is denoted by $\cX/_\sim$. The latter is a poset via $[\xi]\le [\eta]$ if only only if $\xi\le \eta$. This yields the natural order-preserving projection
$\cX \xrightarrow{\pi} \cX/_\sim$ defined by $\xi \mapsto [\xi]$.
By construction $\sO(\cX,\le) \cong \sO(\cX/_\sim)$.
The map $X\twoheadrightarrow\cX\twoheadrightarrow\cX/_\sim$ is also a discretization.
The associated CA-discretization $\bigl(\cX/_\sim, \ccl,|\cdot| \bigr)$ is defined by
$\cl [\xi] = \big\downarrow [\xi]$ and $\bigl|[\xi]\bigr| = \bigcup_{\xi'\in [\xi]}|\xi'|$ and yields a $T_0$ Alexandrov topology.\index{Alexandrov topology}
\end{remark}


\subsection{Modal algebra discretization}\index{Modal algebra discretization}
Let $(X,\scrT_\rmPhi)$ be a topological space defined by a modal operator 
$\rmPhi\colon \sSet(X)\to \sSet(X)$. The associated closure algebra is $(\sSet(X),\cl_\rmPhi)$ with closure operator defined in \eqref{Phiclasop}.
Discretization in terms of closure algebras  can always be formulated in terms of modal algebras.
Consider a discrete modal operator\index{Modal operator!discrete} $\bmPhi\colon\sSet(\cX) \to \sSet(\cX)$.
As explained in Section \ref{modalstuff} we obtain a topological space $(\cX,\scrT_\bmPhi)$, whose discrete topology can be described by either the  specialization pre-order $\le_\bmPhi$ or the associated closure operator $\ccl_\bmPhi$.
 A triple $(\cX,\bmPhi,|\cdot|)$
 is a modal algebra discretization, or \emph{MA-discretization}\index{MA-discretization}\index{Modal algebra discretization} of $(X,\scrT_\rmPhi)$ if 
\begin{enumerate}
\item [(MA)] $\rmPhi |\cU| \subset |\bmPhi \cU|$ for all $\cU\in \sSet(\cX)$,
\end{enumerate}
 which is equivalent to the condition that $|\cdot|$ is a semi-homomorphism of modal algebras 
 and is expressed in the diagram:
\begin{equation}\label{dia:disctop2}
\begin{tikzcd}[column sep=large, row sep=large]
\sSet(X) \arrow[r, "\rmPhi"]                & \sSet(X)                 \\
\sSet(\cX)  \arrow[r, "\bmPhi"] \arrow[u, tail, "|\cdot|"] & \sSet(\cX)  \arrow[u, tail, "|\cdot|"']
\end{tikzcd}
\end{equation}
\begin{proposition}
\label{MAtoCA}
Let  $(\cX,\bmPhi,|\cdot|)$ be a MA-discretization of $(X,\scrT_\rmPhi)$. Then, the induced closure operator $\ccl_\bmPhi\colon \sSet(\cX) \to \sSet(\cX)$ given in \eqref{Phiclasop} defines a CA-discretization $(\cX,\ccl_\bmPhi,|\cdot|)$ of $X$.
In particular, the closure operator $\ccl_\bmPhi$ is given by $\ccl_\bmPhi =\bigcup_{k\ge 0} \bmPhi^k$.\footnote{The expression $\bmPhi^{\bm{+=}}:=\bigcup_{k\ge 0} \bmPhi^k$ is the transitive reflexive closure of $\bmPhi$, cf.\ App.\ \ref{binrel122}.\index{Transitive reflexive closure}}
\end{proposition}
\begin{proof}
Since $|\cdot|$ is a completely additive Boolean homomorphism it is the inverse image of a map $X\to \cX$, cf.\ Prop.\ \ref{charcompladd} and Rem.\ \ref{charsemihomMA}. By Axiom (MA) and Proposition \ref{semiMAcont}
the latter is continuous and thus
$\cl_\Phi |\cU| \subset |\ccl_\bmPhi \cU|$ for all $\cU\subset \cX$,
which proves that $(\cX,\ccl_\bmPhi,|\cdot|)$ is a CA-discretization of $X$.
Since a finite modal operator is completely additive the formula for $\bmTheta$ follows from Lemma \ref{equivsofclos}.
%
%
\end{proof}

The advantage of using a discrete modal operator on $\sSet(\cX)$ is a refinement of the specialization pre-order on $\cX$ in terms of specialization relation given by:\index{Specialization relation}
\begin{equation}
    \label{spec111}
(\xi,\xi')\in \bmphi \quad\text{if and only if}\quad  \xi\in \bmPhi\{\xi'\}.
\end{equation}
Observe that $\bmphi$ need not be transitive and is not reflexive in general. Moreover, the  transitive reflexive closure $\bmphi^{\bm{+=}}$\index{Transitive reflexive closure}  is the specialization pre-order of $\ccl_\bmPhi = \bmPhi^{\bm{+=}}$ in \eqref{spec1}, cf.\ Sect.\ \ref{modalstuff}. 
\begin{remark}
As for pre-orders the notion of down-set for a specialization relation is formulated as: $\cU\subset \cX$ is a down-set\index{Down-set} for $\bmPhi$ if  $\xi'\in \cU$ and $(\xi,\xi')\in \bmphi$, then $\xi'\in\cU$.
Note that $\sO(\cX,\le_\bmPhi) = \sO(\cX,\bmphi) = \sFwdset(\bmPhi)$.
\end{remark}
\begin{remark}
Modal algebra discretization using a derivative operator will be referred to as \emph{DA-discretizations}.\index{DA-discretization}
Binary relations coming from derivative operators will be called \emph{weakly transitive},\index{Weakly transitive relation}\index{wK4-relation} or {wK4}, cf.\ \cite{Esakia}. If $\der$ satisfies the stronger sub-idempotency axiom in  (K3), i.e.  $\der(\der U)\subset \der U$, 
we say that $\der$ is a \emph{strong
derivative operator}.\index{Derivative operator!strong} This is often the case in dynamics in which instance the associated  specialization relation\index{Specialization relation} is transitive (a K4-order).
\end{remark}

\section{Discretization maps}\label{sec:gradings}\index{Discretization map}

Let $(X,\scrT)$ be a topological space. 
A \emph{discretization map}\index{Discretization map} on $X$ is a \emph{surjective} map\footnote{\label{footinj2} As pointed in Footnote \ref{footinj}, in the context of discretization the maps are chosen to be surjective and are dual to injective closure algebra homomorphisms.}
\begin{equation}
    \label{spec0}
\grd\colon X \xtwoheadrightarrow{~} \cX,
\end{equation}
where $\cX$ is a finite set. Since unions and intersections are preserved under preimage, $\grd^{-1}\colon \sSet(\cX)\to \sSet(X)$, is an injective Boolean homomorphism, cf.\ Footn.'s \ref{footinj} and \ref{footinj2}.  In this context, we say that $\grd^{-1}$ is an \emph{evaluation map}\index{Evaluation map} and we use the notation: $|\cU|:=\grd^{-1} \cU$ for $\cU\in \sSet(\cX)$.
\subsection{Topology consistent pre-orders}\index{Topology consistent pre-order}
\label{topconspre}
As pointed out above any topology on $\cX$ is equivalent to its specialization pre-order $(\cX,\leq)$.  We say that $\leq$ is a \emph{$\scrT$-consistent pre-order}\index{$\scrT$-consistent pre-order} on $X$ with respect to $\grd$ if $\grd\colon (X,\scrT) \twoheadrightarrow (\cX,\leq)$ is continuous, which is equivalent to the condition that
$\cl~ \grd^{-1}\cU \subset \grd^{-1}\bigl\downarrow \cU$, for all $\cU\subset \cX$, i.e. $\cl|\cU|\subset |\ccl~\cU|$, where
$\ccl~\cU = \bigl\downarrow \cU$.  Consequently, when $\leq$ is $\scrT$-consistent, the triple $(\cX,\ccl,|\cdot|)$ is a CA-discretization, and $\disc\colon (X,\scrT) \twoheadrightarrow (\cX,\le)$ is a \emph{continuous} discretization map.   
If $\grd$ is a continuous \emph{open} map then  $|\cdot|=\grd^{-1}$ is a homomorphism of closure algebras.\footnote{By \eqref{spec1} $\xi\le \xi'$ if and only if $\xi\in \ccl~\xi'$ which is equivalent to $\{\xi\}\subset \ccl~ \xi'$. In the case that $|\cdot|$ is an injective homomorphism of closure algebras we obtain the equivalent statement $|\xi|\subset |\ccl~\xi'| = \cl|\xi'|$. Here we use the convention $\ccl\{\xi\} = \ccl~\xi$.}
A  specialization relation\index{Specialization relation} $\bmphi\subset \cX\times\cX$ is $\scrT$-consistent if 
the reflexive transitive closure is $\scrT$-consistent. 
%


Conversely, given a CA-discretization $(\cX,\ccl,|\cdot|)$, we define $\grd\colon X\twoheadrightarrow \cX$ via
\begin{equation}
\label{induceddisc}
\grd(x) = \xi \text{ where } x\in |\xi|,
\end{equation}
which is a well-defined map since $\bigcup_\xi |\xi|=X$ and $|\xi|$ and $|\xi'|$ are mutually disjoint for all $\xi\neq\xi'$, cf.\ Rem.\ \ref{charsemihomMA}. By Birkhoff duality the injectivity of $|\cdot|$ implies the surjectivity of $\disc$, cf.\ Thm.\ \ref{thm:birkhoff} and \cite[Thm.\ 5.19]{davey:priestley}.  Moreover, since $\cl|\cU| \subset |\ccl~ \cU|$, $|\cdot|=\grd^{-1}$, the map
$\grd$ is a continuous map, thus $\leq$ is $\scrT$-consistent.
 For a discretization of a subset $X$ of the plane with the associated face partial order $\le$ in  Fig.\ \ref{morsetes23}[left 1 and 2] the pre-order $(\cX,\le)$ is a discretization of the topology $\scrT$ of $X$. The map $\grd\colon X\twoheadrightarrow \cX$ assigns a vertex, edge or square to any point in $X$.  For closure algebras, $\disc$ is induced by sending a point $x$ to the cell in which it is contained, as in Equation \eqref{induceddisc}.
We can summarize these considerations as follows:
\begin{proposition}
 A surjective, continuous discretization map $\disc\colon (X,\scrT)\xtwoheadrightarrow{} (\cX,\le)$ is equivalent to a CA-discretization $(\cX,\ccl,|\cdot|)$ with $(\sSet(\cX),\ccl)$ dual\footnote{In terms of Boolean algebras with operators, cf.\ App.\ \ref{sec:birkhoff} the pre-order is the dual to the closure algebra and vice versa. This duality can also be understood in terms of (co)-Heyting algebras.} to $(\cX,\le)$ and $|\cdot|=\disc^{-1}$.
\end{proposition}

A pre-order $\le$ is a \emph{$\scrT$-co-consistent}\index{$\scrT$-co-consistent pre-order} with respect to $\grd$ if $\grd\colon X \to (\cX,\ge)$ is continuous, where $\ge $ is the opposite pre-order. This is equivalent to the condition 
$\cl~ \grd^{-1} \cU\subset \grd^{-1}\bccl\, \cU$, for all $\cU\subset \cX$,
where $\bccl\,\cU:=\st \cU$ is the \emph{conjugate closure operator}, cf.\ \cite{mcktar}. In terms of realization this reads $\cl|\cU|\subset |\st\cU|$. 
If $\cU$ is $\bccl$-closed, i.e. $\bccl\,\cU = \st \cU = \cU$ which implies that $\cU$ is open in $(\cX,\le)$, then the $\scrT$-co-consistency implies that $|\cU|$ is a closed set. 
Indeed, $\cl |\cU| \subset |\st\cU|=|\cU| \subset \cl |\cU|$ and thus $\cl|\cU|=|\cU|$. Moreover, the closed sets $\cU\in \sO(\cX,\le)$ for a $\scrT$-co-consistent pre-order are open under realization. If $\cU\in\sO(\cX,\le)$ 
then $\ccl~\cU=\cU$ and therefore $\cU^c$ is open which implies that $\st\cU^c=\cU^c$. By the previous $|\cU^c| = |\cU|^c$ is closed and thus
$|\cU|$ is open.

\begin{remark}
\label{arbdiscmap}
If we allow the evaluation map $|\cdot|\colon \sSet(\cX) \to \sSet(X)$ to be an arbitrary homomorphism then the map $\disc$  defined 
in \eqref{induceddisc} is still valid
which allows us to treat the theory of discretization  with arbitrary, not necessarily surjective, continuous maps $\disc\colon X\to \cX$.
In this paper we restrict to  surjective discretization maps unless stated otherwise.
\end{remark}

\begin{remark}\label{indiscrete}
For a discretization map $\grd\colon X\twoheadrightarrow \cX$, there is always a $\scrT$-consistent pre-order.  Namely, the trivial, or indiscrete topology on $\cX$: $\cl\,\varnothing=\varnothing$ and $\cl\,\cU = \cX$ for all $\cU\neq\varnothing$, i.e. $\le$ is an equivalence: $\xi\le \xi'$ and $\xi'\le \xi$ for all $\xi,\xi'\in \cX$.
\end{remark}

\subsection{Filtering and grading}
\label{filtandgrad}

If $\le$ is a $\scrT$-consistent pre-order for a discretization map $\grd$, then by \eqref{spec1} the down-sets $\cU$ for $\le$ correspond to the closed sets in $\cX$ and therefore by the continuity of $\grd$ we have that $\grd^{-1}\cU \in \scrC(X,\scrT)$,\index{$\scrC(X)$}
where $\scrC(X)=\scrC(X,\scrT)$\footnote{If there is no ambiguity about the topology we write $\scrC(X)$ for short. The same applies to the open sets $\scrO(X)$.} denotes the closed sets in $X$.\index{$\scrO(X)$} 
This yields the lattice homomorphism
\begin{equation}
    \label{spec2}
    \grd^{-1}\colon \sO(\cX,\le) \longrightarrow \scrC(X) 
    \xrightarrow[]{\subset} \sSet(X).
\end{equation}
In this setting we refer to $\grd^{-1}$ as an $\sO(\cX/_\sim)$-\emph{filtering} on $X$ and  use the filtering notation:\index{Filtering!of a space}
$
F_\cU X := \grd^{-1}\cU = |\cU|.
$
Similarly, a subset $\cU\subset \cX$ is {\em open} if it is an up-set of $\cX$ and the image under $\grd^{-1}$ are open sets in $X$.
Dual to the $\sO(\cX/_\sim)$-filtering is the \emph{grading} $X = \bigcup_{\xi\in [\xi]} G_{[\xi]} X$ with the property that $G_{[\xi]}X\neq \varnothing$ for all $\xi\in \cX$, cf.\ App.\ \ref{gradfilt}. The latter is an $\cX/_\sim$-grading given by $G_{[\xi]}X\xmapsto{\grade} [\xi]$, where $\bigl\{G_{[\xi]}X \mid[\xi]\in (\cX/_\sim,\le) \bigr\}$ is an ordered tessellation.\index{Grading!of a space} 
The following scheme shows the duality between CA-discretizations and continuous discretization maps, and between filterings and gradings, cf.\ App.\ \ref{gradfilt}:



\[
\begin{tikzcd}[column sep=large, row sep=large]
{\displaystyle{\wbinom{(\cX,|\cdot|,\cl)}{ \text{CA-discretization}}}} \arrow[r] \arrow[d] & {\displaystyle{\wbinom{|\cdot|\colon \sO(\cX,\le) \to \sSet(X)}{ \text{filtering} }}} \arrow[d] \\
{\displaystyle{\wbinom{\disc\colon (X,\scrT)\twoheadrightarrow (\cX,\le)}{\text{ continuous discretization}}}} \arrow[r] \arrow[u] & {\displaystyle{\wbinom{X = \bigcup_{[\xi]}G_{[\xi]} X}{\text{grading}}}} \arrow[u]
\end{tikzcd}
\]
where $G_{[\xi]}X=\disc^{-1}[\xi]$.
In general a grading on a topological space yields a discretization which is not necessarily continuous.\footnote{Let $X=\bigcup_{p\in \sP} G_pX$ be a $\sP$-graded decomposition of $X$.
Then, the map $\grade\colon X\to \sP$, defined by $\grade(x) =
\grade(G_pX)=p$ for all $x\in G_pX$, is a discretization map in the sense of Rmk.\ \ref{arbdiscmap}. By restricting to the range one obtains a surjective discretization map, cf.\ App.\ \ref{gradfilt}.} 
This implies that the above diagram do not necessarily point in the opposite direction.
Finally we define a class of discretization maps which are favorable for using homology theories.

\begin{definition}
\label{regudisc}
A discretization map $\disc\colon X\twoheadrightarrow \cX$ is \emph{natural}\index{Discretization map!natural} 
if it is continuous and  the associated filtering $\disc^{-1}\colon \sO(\cX,\le)\to  \sSet(X)$ consists of mutually\index{Good pair} good pairs.\footnote{Recall that a pair $(X,A)$, $A\subset X$ closed, is a \emph{good pair}  if 
$A$ is a deformation retract of a neighborhood in $X$, q.v.\ \cite[Thm.\ 2.13]{Hatcher}.}
\end{definition}


\section{Bi-topological spaces and discretization}\index{Bi-topological space}\index{Bi-topological space!discretization of}
\label{bi-top}
A triple $(X,\scrT,\scrT')$ is called a \emph{bi-topological space}\index{Bi-topological space} if the factors $(X,\scrT)$ and $(X,\scrT')$ are well-defined topological spaces. 
The associated closure algebra for $(X,\scrT,\scrT')$ is the Boolean algebra with operators\index{Boolean algebra with operators} $(\sSet(X),\cl,\cl')$, where $\cl$ and $\cl'$ are the closure operators for $\scrT$ and $\scrT'$ respectively and is referred to as \emph{bi-closure algebra}\index{Bi-closure algebra} for  $(X,\scrT,\scrT')$.
%
A subset $U\subset X$ is a \emph{$(\scrT,\scrT')$-pairwise clopen set}\index{$(\scrT,\scrT')$-pairwise clopen set}\index{Pairwise clopen set} for $X$ if
$U$ is closed in $\scrT$ and open in $\scrT'$.  We denote the set of $(\scrT,\scrT')$-pairwise clopen sets in a bi-topological space $(X,\scrT,\scrT')$ by $\aclop(X)$.\index{$\aclop(X)$}\index{$\aclop^*(X)$}
Similarly, we can define $(\scrT',\scrT)$-pairwise clopen sets which are
open in $\scrT$ and closed in $\scrT'$ and are denoted by
$\aclop^*(X)$.

The next step is to consider discretization for bi-topological spaces.
A discretization map $\disc\colon X\twoheadrightarrow \cX$ is a p-continuous map\footnote{J.C. Kelly refers to such map that are continuous with respect to both topologies as \emph{p-continuous}, or \emph{pairwise continuous} maps, cf.\ \cite{Kelly}.\index{$p$-continuous map}\index{Pairwise continuous map}}  between bi-topological spaces if there exists pre-orders $(\cX,\le)$ and $(\cX,\le')$ that are $\scrT$-consistent and $\scrT'$-consistent respectively. We write
$\disc\colon (X,\scrT,\scrT') \twoheadrightarrow (\cX,\le,\le')$. The associated \emph{bi-topological CA-discretization}\index{Bi-topological CA-discretization}\index{CA-discretization!bi-topological} is denoted by $(\cX,\ccl,\ccl',|\cdot|)$ where $\ccl$ and $\ccl'$ are the associated closure operators.
Let $(\cX,\ccl,\ccl',|\cdot|)$ be a bi-topological CA-discretization for $(X,\scrT,\scrT')$. This is equivalent
to choosing a discretization map $\disc\colon X\twoheadrightarrow \cX$ and pre-orders $(\cX,\le)$ and $(\cX,\le')$ such that $\disc$ is continuous with respect to both $\scrT$ and $\scrT'$.
Since the pre-orders $\le$ and $\le'$ represent Alexandrov topologies we can coarsen the finite topologies using both up-sets and down-sets. 
\begin{definition}
\label{antacoarse}
Let $\disc\colon X\twoheadrightarrow \cX$ be a discretization map.
An \emph{antagonistic pre-order}\index{Antagonistic pre-order} for $(X,\scrT,\scrT')$ is a pre-order $(\cX,\le^\dagger)$ such that
\begin{enumerate}
    \item [(i)] $\le^\dagger$ is $\scrT$-consistent with respect to $\grd$; 
    \item [(ii)] $\le^\dagger$  is $\scrT'$-co-consistent with respect to $\grd$.
\end{enumerate}
These conditions   translate as 
\begin{equation}
    \label{dubbleconst}
    \cl|\cU|\subset | \ccl^\dagger\cU|,\quad \cl'|\cU| \subset |\st^\dagger\cU|,\quad \forall \cU\subset \cX,
\end{equation}
where $\st^\dagger = \bccl^\dagger$, the conjugate closure operator,\index{Conjugate closure operator} cf.\ Sect.\ \ref{topconspre} and \cite{JonssonTarski}.
The triple $(\cX,\ccl^\dagger,|\cdot|)$ is called an \emph{antagonistic CA-discretization}\index{Antagonistic CA-discretization}\index{CA-discretization!antagonistic} for $(X,\scrT,\scrT')$.

\end{definition}


\begin{remark}
\label{doubleconti}
An equivalent way to say that $(\cX,\le^\dagger)$ is an antagonistic pre-order is that both
\[
\disc\colon (X,\scrT) \xtwoheadrightarrow{} (\cX,\le^\dagger),\quad\text{and}\quad \disc\colon (X,\scrT') \xtwoheadrightarrow{} (\cX,\ge^\dagger),
\]
are continuous.
\end{remark}

\begin{remark}
\label{antatop}
We can use the pairwise clopen sets in $(X,\scrT,\scrT')$ as a base (of closed sets) for the topology $\scrT^\dagger$.
Closed sets in $(\cX,\le^\dagger)$ are pairwise clopen sets in $(X,\scrT,\scrT')$,
 which yields the continuous discretization map $\disc\colon (X,\scrT^\dagger)\twoheadrightarrow (\cX,\le^\dagger)$. We say that $\scrT^\dagger$ is the \emph{antagonistic topology}\index{Antagonistic topology} with respect to the pair $(\scrT,\scrT')$. Since closed sets in $(X,\scrT^\dagger)$ are not necessarily pairwise clopen it is preferable to use the concept of pairwise clopen sets in the bi-topological space $(X,\scrT,\scrT')$.
\end{remark}

\begin{remark}
Antagonistic pre-orders can also be defined by reversing the role of $\scrT$ and $\scrT'$.
\end{remark}

Antagonistic pre-orders yield discrete topologies on $\cX$ since 
such  topologies  are necessarily Alexandrov. One cannot play the same game with $\scrT$ and $\scrT'$ on $X$ since topologies are not Alexandrov in general. For this reason one chooses the formalism of bi-topological spaces and anatagonistic pre-orders.

Let $(\cX,\le^\dagger)$ be an antagonistic pre-order as in Definition
\ref{antacoarse}.
Closed sets $\cU$ in $(\cX,\le^\dagger)$ are pairwise clopen sets $|\cU|\in \aclop(X)$, cf.\ Sect.\ \ref{topconspre}.
%
Conversely, if $(\cX,\ccl,\ccl',|\cdot|)$ is a bi-topological CA discretization for $(X,\scrT,\scrT')$ then the lattice embedding:
\begin{equation}
      \label{dubbleconst2}
\iota\colon\sAnt \rightarrowtail \sO(\cX,\le) \cap \sU(\cX,\le') =: \aclop(\cX),
\end{equation}
yields an antagonistic pre-order for $(X,\scrT,\scrT')$. 
Indeed, 
if we use Birkhoff duality\index{Birkhoff duality} to dualize the homomorphism
$\iota\colon\sAnt \rightarrowtail\sSet(\cX)$, q.v.\ Thm.\ \ref{thm:birkhoff} and Rem.\ \ref{extraformdual}, we obtain the order-preserving surjection
\begin{equation}
    \label{defndyn}
    \begin{aligned}
    \pi\colon \cX \twoheadrightarrow \sP,\quad \xi \mapsto \pi(\xi) &:= \min\bigl\{ p\in \sP\mid \xi\in \iota\bigl(\big\downarrow p\bigr)\bigr\}\\
    &~=\max\Bigl\{\min\bigl\{ \cU\in \sJ(\sAnt)\mid \xi\in \cU\bigr\}\Bigr\},
     \end{aligned}
\end{equation}
where $\sJ\bigl( \sAnt\bigr)$ is the   poset of join-irreducible elements in $\sAnt$.
%
By construction  a pre-order $(\cX,\le^\dagger)$ is defined by
\[
\xi\le^\dagger \xi' \quad\hbox{if and only if}\quad \pi(\xi) \le \pi(\xi'),
\]
is an antagonistic pre-order with $\sO(\cX,\le^\dagger)\cong \sAnt$. 
In this case we say that $\le^\dagger$ is
an \emph{antagonistic coarsening} of $(\cX,\ccl,\ccl',|\cdot|)$.
The lattice embedding $\sO(\cX,\le^\dagger) \rightarrowtail\sSet(\cX)$ is dual to the identity map 
\begin{equation}
\label{idred}
    \id\colon \cX \to (\cX,\le^\dagger).
\end{equation}
%

\begin{theorem}
    \label{dubbleconst3}
    A pre-order $\le^\dagger$ is an
     antagonistic pre-order for $\disc\colon (X,\scrT,\scrT') \to (\cX,\le^\dagger)$ 
     if and only if there exists a bi-topological CA-discretization $(\cX,\ccl,\ccl',|\cdot|)$ such
     that $\sO(\cX,\le^\dagger)$ is given by
     \eqref{dubbleconst2}, i.e. an antagonistic pre-order is equivalent to an antagonistic coarsening.
\end{theorem}

\begin{proof}
One direction is given be the construction in \eqref{dubbleconst2}.
It remains to show that an antagonistic pre-order satisfying Definition \ref{antacoarse}(i)-(ii) is an antagonistic coarsening. If we define the discrete closure operators  $\ccl = \ccl^\dagger$ and $\ccl'=\st^\dagger$, then
$\sO(\cX,\le) = \sO(\cX,\le^\dagger)$ and $\sU(\cX,\le') = \sO(\cX,\le^\dagger)$, which proves the theorem.\end{proof}

In  practical situations, given a bi-topological CA-discretization $(\cX,\le,\le')$, we can choose $\sO(\cX,\le^\dagger) = \aclop(\cX)$. 
Let $\sJ(\sO(\cX,\le^\dagger))$ be the poset of join-irreducible elements in $\aclop(\cX)$. From the results in \cite{lsa3,kkv}  consider a representation $(\sSC,\le)$ of
$\sJ(\sO(\cX,\le^\dagger))$ defined by
\begin{equation}
\label{defnofSC}
\sSC := \bigl\{ \cS = \cU\smin\cU^\pred\mid \cU\in \sJ(\sO(\cX,\le^\dagger))\bigr\},
\end{equation}
 with
$\cS\le \cS'$ if and only if $\cU\subset \cU'$, with $\cU,\cU'\in \sJ(\sO(\cX,\le^\dagger))$ uniquely
determined by $\cS= \cU\smin\cU^\pred$ and $\cS' = \cU'\smin\cU'^\pred$.\footnote{For every $\cU\in \sJ(\sO(\cccX,\le^\dagger))$ there exists a unique immediate predecessor $\cU^\pred\in  \sO(\cccX,\le^\dagger)$.} 
If we regard the pre-order $(\cX,\le^\dagger)$ as a directed graph then
the sets $\cS\in \sSC$ correspond to the \emph{strongly connected components}\index{Strongly connected components!of a directed graph}\index{$\sSC$} of the directed graph which is the motivation for the
abbreviation $\sSC$, cf.\ Rem.\ \ref{perspectives}.
The posets
$(\sSC,\le)$ and $\bigl(\sJ(\sO(\cX,\le^\dagger)),\subset \bigr)$ are isomorphic by construction and
the elements in $\sSC$ correspond to the partial equivalence classes in $(\cX,\le^\dagger)$, cf.\ Rem.\ \ref{partialequiv}, which yields the (order-preserving) projection $\dyn\colon(\cX,\le^\dagger) \twoheadrightarrow (\sSC,\le) = \cX/_\sim$.\footnote{We use the terminology $\dyn$ as it is used later on in the setting of the \bflt.} The latter may be regarded as  a \emph{finite} discretization of $\cX$.
%
%
The embedding
$\sO(\cX,\le^\dagger) \xrightarrow{\subset}\sO(\cX,\le)$ implies that 
 $\dyn$ is also an order-preserving map $\dyn\colon (\cX,\le)\twoheadrightarrow (\sSC,\le)$ which factors as
$(\cX,\le) \xtwoheadrightarrow{\id} (\cX,\le^\dagger) 
 \xtwoheadrightarrow{\dyn} (\sSC,\le)$. Indeed, a downset in $(\cX,\le^\dagger)$ is a downset in $(\cX,\le)$ and therefore $\le\, \subset \,\le^\dagger$ as pre-orders.
%
By the same token we have that $\dyn$ is order-reversing with respect to $\le'$ which follows from the embedding $\sO(\cX,\le^\dagger) \xrightarrow{\subset}\sU(\cX,\le')$ and therefore $\ge'\, \subset\, \le^\dagger$ as pre-orders.
Summarizing, the maps $\dyn$ are order-preserving and order-reversing respectively:
\begin{equation}
\label{ordpreforgrd}
\begin{tikzcd}[column sep=huge, row sep=large]
& (\cX,\le) \arrow[d, "\id"'] \arrow[rd, "\dyn"]    \arrow[rrd, "\ppart", bend left]              &  & \\
(X,\scrT,\scrT') \arrow[r, "\disc"', bend left=29, shift right] \arrow[r, "\disc", dashed, bend right=29, shift left] \arrow[ru, "\disc", shift left=2] \arrow[rd, "\disc"',  shift right=2] & (\cX,\le^\dagger) \arrow[r, "\dyn"]                                  &  \sSC \arrow[r,"\pi", two heads] & \sP\\
& (\cX,\le') \arrow[u, "\id", dashed] \arrow[ru, "\dyn"', dashed] \arrow[rru, "\ppart"', dashed, bend right]  &  &
\end{tikzcd}
\end{equation} 
%
%
The map $\dyn$ is again a continuous discretization and the composition 
\begin{equation}
\label{spacegrading}
\begin{tikzcd}[column sep=huge]
X \arrow[r] \arrow[r, "\disc", two heads] \arrow[rr, "\tile", two heads, bend right, shift right] & \cX \arrow[r, "\dyn", two heads] & \sSC
\end{tikzcd}
\end{equation}
denoted by $\tile$,
may be regarded as a continuous map $\tile\colon (X,\scrT) \to (\sSC,\le)$
and as a continuous map  $(X,\scrT') \to (\sSC,\ge)$ 
by factoring through $(\cX,\le^\dagger)$ and $(\cX,\ge^\dagger)$ respectively, and which is equivalently obtained
 by factoring through $(\cX,\le)$
by factoring through $(\cX,\le')$ respectively.
Every antagonistic pre-order for $(X,\scrT,\scrT')$ defines a grading $G_\cS X\xmapsto{\tile}\cS$ of $X$ given by 
\[
X = \bigcup_{\cS\in \sSC} G_\cS X, \quad G_\cS X =
\tile^{-1} \cS,
\]
which is called an \emph{antagonistic tessellation}\index{Tessellation!antagonistic}\index{Antagonistic tessellation} of $X$. We apply these bi-topological discretization techniques  in the next two chapters in the context of discretizing semi-flows.

\begin{remark}
An antagonistic pre-order $(\cX,\le^\dagger)$ satisfies $\le\, \subset \,\le^\dagger$ and $\ge'\, \subset \,\le^\dagger$ as pre-orders 
and thus $\le\vee\ge'\,\subset \,\le^\dagger$ as pre-orders.
The `vee' on the pre-order is the transitive closure of the union.
The case $\sO(\cX,\le^\dagger) = \aclop(\cX)$ corresponds to $\, \le\vee\ge'  \,= \,\le^\dagger$.
\end{remark}

\begin{remark}
\label{perspectives}
There are various (equivalent) perspectives for presenting binary relations on $\cX$:
\begin{enumerate}
    \item[(a)] as a binary relation, i.e.\  $\bmphi \subset \cX\times \cX$;
    \item[(b)] as a directed graph (digraph) $\bmphi$ with vertices $\cX$ and edge set 
    \[
    \bigl\{\xi \to \xi'\mid (\xi,\xi')\in \bmphi^{-1}\bigr\};
    \]
    \item[(c)] as modal operator  $\bmPhi=\bmphi^{-1}\colon\sSet(\cX) \to \sSet(\cX)$. 
\end{enumerate}
We alternate between these perspectives, using whichever is  conceptually most convenient. As digraph the strongly connected components are found via \eqref{SCbin}.
\end{remark}



\section{Examples and further extensions}
\label{exdisc}
In this section we discuss examples of discretization of topology starting with regular closed sets. The latter will be applied in the setting of CW-decompositions.\index{CW-decomposition} We start with outlining regular closed sets.

\subsection{Regular closed sets}
\label{regclsets}\index{Regular closed set}
Of particular importance in this text are the  regular closed sets which play a central role in the construction of Morse tessellations. A subset $U\subset X$ is \emph{regular closed}\index{Regular closed set} if $\cl\Int U = U$.
The set of regular closed sets in any topological space $X$ is denoted by $\scrR(X,\scrT)$\index{$\scrR(X)$} and the latter forms a complete Boolean algebra with unary operation $U^\#=\cl\,U^c$    and binary operations $U\vee U' = U\cup U'$ and $U\wedge U' = \cl\Int(U\cap U')$, cf.\ \cite[Sect.\ 2.3]{walker}.
By the same token we can define the regular closed sets in a CA-discretization $(\cX,\cl,|\cdot|)$\index{$\scrR(\cX)$}
and we denote regular closed sets in $\cX$ by $\scrR(\cX)$.\footnote{In this text we mainly consider regular closed sets with repect to one topology $\scrT$. Therefore the notation $\scrR(X)$ and $\scrR(\cccX)$ does not cause any ambiguities.}
%
Regular closed subsets can be obtained from closed subsets $\sO(\cX,\le)$.
\begin{proposition}[cf.\ \cite{lsa3}, Lem.\ 22]
\label{rgcomplex}
The map $\ccl\cInt \colon \sO(\cX,\le) \to \scrR(\cX)$, defined by $\cU \mapsto \ccl\cInt \cU$, is a surjective lattice homomorphism.\footnote{This result holds for any topological space $X$.}
\end{proposition}
%
The atoms in $\sSet(\cX)$ are given by the set $\cX$. Since  $\scrR(\cX)$ is a finite Boolean algebra it is a power set on a set of atoms. Define the maximal elements in $\cX$ with respect to $\le$ by $\cX^{\topc}$, i.e. $\xi\in \cX^{\topc}$ if and only if $\st\xi = \xi$. Such elements are called \emph{top cells} and form an anti-chain in $(\cX,\le)$.\index{Top cell}\index{Cell!top}
\begin{proposition}
  \label{atomreg}
  The atoms of $\scrR(\cX)$ are given by the set $\bigl\{\ccl\, \xi\mid \xi\in \cX^{\topc} \bigr\}$.
Moreover, $\ccl\colon \sSet(\cX^{\topc}) \to \scrR(\cX)$ is an isomorphism with inverse $\cU\mapsto \cU\cap \cX^{\topc}$.
\end{proposition}
\begin{proof}
We start with the observation that $\xi\in \cX^{\topc}$ represents an open subset.
Indeed, $\st\xi=\xi$ and thus $\xi=\cInt\xi$. Then, $\ccl\,\xi = \ccl\cInt\xi$ is a regular closed set since $\ccl\cInt\ccl\cInt \xi = \ccl\cInt \xi$.
Suppose $\xi\not\in \cX^{\topc}$, i.e. $\xi$ is not maximal in $(\cX,\le)$. 
Since $\xi=\{\xi\}$ is a singleton set and since $\cInt\{\xi\}\subset\{\xi\}$ we conclude that $\cInt\xi=\xi$, or $\cInt\xi=\varnothing$. Suppose the former holds. Then, since $\xi$ is not maximal, $\{\xi\}\subsetneq \st \xi = \{\xi\}$, a contradiction.
%
%
Therefore, $\cInt\xi=\varnothing$. 
Let $\cU = \bigcup \{\xi\}$ be a regular closed set. Then,  by invoking Proposition \ref{rgcomplex} we have
\[
\cU = \ccl\cInt\cU = \ccl\cInt\Bigl(\bigcup\{\xi\}\Bigr)
= \bigcup \ccl\cInt \xi =\bigcup\Bigl\{\ccl\,\xi\mid \xi\in \cX^{\topc}\Bigr\},
\]
which proves that every regular closed set is union of elements $\ccl\,\xi$, for $\xi\in \cX^{\topc}$. The latter form an anti-chain in $\scrR(\cX)$, i.e.,
again invoking Proposition \ref{rgcomplex},
\begin{equation}
    \label{emptyforwed}
\ccl\,\xi \wedge \ccl\,\xi' = \ccl\cInt\xi \wedge \ccl\cInt\xi'  =
\ccl\cInt \bigl(  \xi \cap  \xi'\bigr) =\varnothing,
\end{equation}
which proves that $\ccl\,\xi$, for $\xi\in \cX^{\topc}$, are atoms.
The map $\ccl\colon \sSet(\cX^{\topc}) \to \scrR(\cX)$ preserves union. Consider $\ccl~ \cU\wedge \ccl~\cU'$. By \eqref{emptyforwed} we have
\[
\ccl~ \cU\wedge \ccl~\cU' = \bigcup\ccl~\xi \wedge \bigcup\ccl~\xi'
= \bigcup \ccl~\xi'' = \ccl \Bigl(\bigcup\{\xi''\} \Bigr),
\]
where $\bigcup\{\xi''\} = \cU\cap \cU'$, which proves that 
$\ccl\colon \sSet(\cX^{\topc}) \to \scrR(\cX)$ is a homomorphism.
Any $\cU\in \scrR(\cX)$ is uniquely represented as $\cU=\bigcup \ccl~\xi=
\ccl\bigl(\bigcup\{\xi\} \bigr)
$, $\xi\in \cX^\top$ which shows that $\cU\cap \cX^\top$ yields the unique set of generating cells $\xi\in \cX^\top$. 
\end{proof}



Regular closed sets in $\cX$ do not necessarily yield regular closed sets in $X$ under evaluation.
However, if $(\cX,\ccl,|\cdot|)$ is \emph{Boolean}, cf.\ Defn.\ \ref{regularCA},\index{Boolean CA-discretization}\index{CA-discretaization!Boolean} then we have the following correspondence:
$|\cU| = |\ccl\cInt \cU| = \cl|\cInt\cU| = \cl\Int|\cU|$\footnote{If $(\cccX,\ccl,|\cdot|)$ is Boolean then $\cInt|\cU| = |\cInt\cU|$ follows from the relation for closure.} and thus a subset $\cU\subset\cX$ is regular closed in $\cX$ if and only if $|\cU|\subset X$ is regular closed in $X$. 
This way the image of $|\cdot|\colon \scrR(\cX) \to \scrR(X)$ 
yields
the finite  Boolean subalgebra $\scrR_0(X)$  contained in $\scrR(X)$. The subalgebra $\scrR_0(X)$ is  generated by the 
set of atoms $\sJ\bigl( \scrR_0(X)\bigr) := \bigl\{\cl|\xi|\mid \xi\in \cX^{\topc}\bigr\}$.\footnote{The join-irreducible element in a finite Boolean algebra and the atoms that generate the Boolean algebra.}

\begin{proposition}
\label{realization}
Let $(\cX,\ccl,|\cdot|)$ be a Boolean CA-discretization for $X$.\index{Boolean CA-discretization}
Then, the map 
\[
\lbr\cdot\rbr\colon \sSet(\cX^{\topc}) \to \scrR_0(X),\quad \xi \mapsto \lbr\xi\rbr := \ccl|\xi|,\footnote{The notation $\lbr\cdot\rbr$ is called \emph{closed realization}.}
\]
is a lattice isomorphism and thus a Boolean isomorphism.
\end{proposition}

Conversely, a finite sub-algebra $\scrR_0(X)\subset \scrR(X)$ yields a finite
closure algebra $\bigl(\sSet(\cX),\ccl\bigr)$ as follows. 
 Define $[\scrR_0(X)]$ as the smallest sub-algebra  in $\sSet(X)$ containing $\scrR_0(X)$. Denote the set of atoms
by $\sJ\bigl( [\scrR_0(X)]\bigr) := \bigl\{|\xi| \mid \xi\in \cX\bigr\}$ for some finite set $\cX$.
This defines the pre-order $(\cX,\le)$ via the relation: $\xi\le \xi'$ if and only if $|\xi|\subset \ccl|\xi'|$. Via the pre-order we obtain a closure operator on $\sSet(\cX)$ via \eqref{spec1}.

\begin{proposition}
\label{realization22}
Let $\scrR_0(X)\subset \scrR(X)$ be a finite sub-algebra.
Then, $[\scrR_0(X)]\subset \sSet(X)$ defines a unique finite closure algebra
 $\bigl(\sSet(\cX),\ccl\bigr)$, where the $\ccl$ is defined by \eqref{spec1}.
\end{proposition}

The above statement can be rephrased as:
a finite sub-algebra $\scrR_0(X)\subset \scrR(X)$  induces a unique Boolean CA-discretization $(\cX,\ccl,|\cdot|)$  for $X$, where $|\cdot|$ an injective homomorphism of closure algebras.

\subsection{CW-decomposition maps}\label{CW}\index{CW-decomposition map}\index{CW decomposition}

Consider a  discretization
$\disc\colon X \twoheadrightarrow \cX$ without indicating a specific pre-order on $\cX$ for now. 
%
Let  $B^q$ and $\bar B^q$
denote the open and closed unit balls in $\R^q$ respectively (where $B^0$ and $\bar{B}^0$ denote the one point space).  We say that $\xi$ is an \emph{$q$-cell} if $|\xi|$ is homeomorphic to an open ball $B^q$.  The integer $q$ is called the \emph{dimension}\index{Dimension} of $\xi$ and is denoted $\dim \xi$.  
Suppose $\disc$ has the property that
%
%
every $\xi$ is  an $q$-cell for some $q$.
Given such a discretization map assigning dimension to a cell is an order-preserving map 
\[
\dim\colon \cX \to (\N,\le),
\]
when $\cX$ is regarded as anti-chain.
Note therefore that the anti-chain $\cX$ is a naturally graded set with respect to $\dim$, i.e., $\cX=\bigcup_{q\in \N} G_q\cX$, where $G_q\cX=\dim^{-1}\!q$ is the set of $q$-cells. The latter also
yields the filtering 
\[
\big\downarrow q \mapsto F_{\downarrow q}\cX,\quad \big\downarrow q \in \sO(\N),
\]
where, by Birkhoff duality, $F_{\downarrow q}\cX = \dim^{-1}\!\big\downarrow q$, $\big\downarrow q =\{0,1,\cdots,q\}$ and $q\in \N$. 
The composition $X\xrightarrow{\disc}\cX\xrightarrow{\dim}\N$, denoted by $\skel$,  also defines a discretization on $X$ and yields the filtering
$\big\downarrow q \mapsto F_{\downarrow q} X = \skel^{-1} \big\downarrow q$, $\big\downarrow q \in \sO(\N)$.

\begin{definition}
\label{CWdefn}
A \emph{CW-decomposition map}\index{CW-decomposition map}\index{CW decomposition} on $X$, denoted by $\cell\colon X\twoheadrightarrow\cX$, is a discretization map where each $\xi\in \cX$ is an $q$-cell for some $q$.  Moreover, for every $\xi\in \cX$ there is a continuous map $f_\xi\colon \bar{B}^q\to X$, { where $q=\dim \xi$}, such that 
\begin{enumerate}
    \item [(i)] $f_\xi$ restricts to a homeomorphism $f_\xi|_{B^q}\colon B^q\to |\xi|$;
    \item [(ii)] $f_\xi ( \bar{B}^q\smin B^q)\subset F_{\downarrow (q-1)}X$. 
\end{enumerate}
A subset $U\subset X$ is open (closed) if and only if    $f_\xi^{-1}(U)$ is open (closed) in $\bar B^q$
for all $\xi\in \cX$.  A CW-decomposition map is \emph{regular}\index{CW-decomposition map!regular} if the maps $f_\xi$ are embeddings.
\end{definition}

Note that if $X$ admits a (finite) CW-decomposition map then $X$ is a compact Hausdorff space, cf.\ \cite{Hatcher}. 
Since $X$ is Hausdorff, it follows that $\lbr \xi\rbr = f_\xi(\bar{B}^q)$.\footnote{By continuity $f_\xi(\bar{B}^q) \subset \cl f_\xi({B}^q)= \lbr\xi\rbr$. On the other hand $\lbr \xi\rbr = \cl f_\xi({B}^q) \subset \cl f_\xi(\bar{B}^q) = f_\xi(\bar{B}^q)$ since the continuous image of a compact set is compact and in a Hausdorff space compact sets are closed.} CW-decompositions are general enough to include simplicial and cubical complexes.  
From a CW-decomposition map we can define the following finite topology in terms of a pre-order on $\cX$:
\[
\xi\le \xi' \quad\text{if and only if}\quad |\xi|\subset \cl|\xi'|,
\]
which is called the \emph{face partial order}\index{Face partial order} on $\cX$. 
%
\begin{lemma}\label{lem:faceposet}
The pre-order $(\cX,\le)$ is a partial order and the associated Alexandrov topology is
a $T_0$ topology. 
\end{lemma}
\begin{proof}
Suppose $\xi\neq\xi'$ and $\xi\sim\xi'$, i.e. $\xi\le \xi'$ and $\xi'\le \xi$, which implies that $\cl|\xi|=\cl|\xi'|$ and thus $f_\xi(\bar B^q) = f_{\xi'}(\bar B^q)$.
Furthermore,
$f_\xi(\bar B^q)  = f_\xi( \bar{B}^q\smin B^q \cup B^q) = f_\xi( \bar{B}^q\smin B^q) \cup |\xi|$,
and similarly $f_{\xi'}(\bar B^q) = f_{\xi'}( \bar{B}^q\smin B^q) \cup |\xi'|$.
By assumption $|\xi|\cap |\xi'|=\varnothing$, which yields that 
\[
|\xi|\subset f_{\xi'}( \bar{B}^q\smin B^q) \subset F_{\downarrow (q-1)}X,\quad
|\xi'|\subset f_{\xi}( \bar{B}^q\smin B^q) \subset F_{\downarrow (q-1)}X.
\]
This contradicts the definition of $F_{\downarrow (q-1)}X$ and the fact that all cells are realized as disjoint sets in $X$. Therefore $\xi\sim \xi'$ if and only if $\xi=\xi'$ and $\le$ is a partial order and the Alexandrov topology is $T_0$.
\end{proof}

The next step is to show that $\cell\colon X\to (\cX,\leq)$ is a continuous discretization map, i.e., that $\le$ is $\scrT$-consistent.  Moreover, we also show that $\dim$ is order-preserving.

\begin{lemma}
\label{facepart1}
Let $\cl\colon \sSet(\cX) \to \sSet(\cX)$ be the closure operator defined by face partial order $\le$. Then,
\begin{enumerate}
    \item [{\rm (i)}] 
    $\leq$ is a $\scrT$-consistent partial order and $\cl|\xi| = |\ccl~\xi|$ for all $\xi\in \cX$;
    \item [{\rm (ii)}] $\dim\colon (\cX,\le) \to (\N,\le)$ is order-preserving.
\end{enumerate}
\end{lemma}

\begin{proof}
We have that $\cl|\xi| = \lbr\xi\rbr = f_\xi(\bar B^q)$ and thus, as before using Definition \ref{CWdefn}(i)-(ii),
$f_\xi(\bar B^q)  = f_\xi( \bar{B}^q\smin B^q \cup B^q) = f_\xi( \bar{B}^q\smin B^q) \cup |\xi|$.
This implies that $\cl|\xi|$ is a union of sets $|\xi'|$ and more precisely
\[
\cl|\xi| = \bigcup \bigl\{ |\xi'|\mid |\xi'|\subset \cl|\xi|\bigr\}
=\bigcup \bigl\{ |\xi'|\mid \xi'\le \xi\bigr\} = \bigl|\cl~\xi\bigr|,
\]
which proves (i). 
As for (ii) we have that $\cl|\xi| \smin |\xi|=f_\xi( \bar{B}^q\smin B^q) = \bigcup \bigl\{ |\xi'|\mid \xi'\le \xi\bigr\}\subset F_{\downarrow (q-1)}X$
and thus $\dim \xi'\le q-1$. Consequently, $\xi'\le \xi$ implies $\dim\xi'\le \dim\xi$.
\end{proof}

Lemma \ref{facepart1} shows in particular that $\cell$ is continuous open map and thus a discretization with respect to the face partial order. Moreover, $|\cdot|$ is an injective homomorphism of closure algebras. Indeed by additivity of $|\cdot|$ and $\cl$ we have that 
$\cl|\cU| = |\ccl~\cU|$ for all $\cU\in \sSet(\cX)$.
The associated CA-discretization $(\cX,\ccl,|\cdot|)$ is Boolean and is called a \emph{CW-decomposition}\index{CW decomposition} for $X$.

\begin{lemma}
\label{regCW}
A CW-decomposition map\index{CW-decomposition map}\index{CW decomposition}\index{Discretization map!natural} $\cell\colon X\to \cX$ is a natural discretization map.\footnote{cf.\ Defn.\ \ref{regudisc}.}
\end{lemma}

\begin{proof}
For every  closed subset $\cU\subset \cX$ the realization $|\cU| \subset X$ is a sub CW-decomposition and therefore a deformation retract of a neighborhood in $X$, cf.\ \cite[Prop.\ A.5.]{Hatcher}.
\end{proof}

\subsection{General closure and bi-closure algebras}
\label{clandbiclos}

A more general notion of closure algebra is given by a Boolean algebra $\sB=(\sB,\vee,\wedge,^\neg)$ and an operator $\cl\colon \sB \to \sB$
(an abstract closure operator)\index{Closure operator!abstract} satisfying Axiom (K1)-(K4).\footnote{Replace $\varnothing$ and $X$ by the neutral elements $0$ and $1$ respectively, as well as the binary operations and complement.} The Boolean algebra $\sB$ is not necessarily complete, nor atomic. Such algebras are referred to as \emph{closure algebras} and are denoted as $(\sB,\cl)$.\index{Closure algebra} In \cite{mcktar} various representation results for general closure algebras are given. 
For a closure algebra  the lattice of closed elements is given by $\sFwdset(\cl)=\{b\in \sB\mid\cl~b\le b\}$.
The latter can be considered as \emph{generalized topological space},\index{Generalized topological space} q.v.\ \cite{naturman}.
In terms of closed sets this entails that $\sFwdset(\cl)$ is a bounded lattice, cf.\ Prop.\ \ref{modaltop} and if $\bigwedge b$, $b\in \sFwdset(\cl)$, exists in $\sB$, then $\bigwedge b\in \sFwdset(\cl)$.
Moreover, the expression $\max\bigl\{a\mid b\le a,~\cl~ a=a\bigr\}$ exists in $\sB$. By definition the latter satisfies $b\le \max\bigl\{a\mid b\le a,~\cl~ a=a\bigr\} \le \cl~b$ and thus it exists and $\cl~b = \bigwedge\bigl\{a\mid b\le a,~\cl~ a=a\bigr\}$. The generalized topological space is denoted $\bigl(\sB,\sFwdset(\cl)\bigr)$. Closure algebras are equivalent to generalized topological spaces.
In a similar fashion a \emph{bi-closure algebra}\index{Bi-closure algebra} is given by a Boolean algebra and two abstract closure operators $\cl,\cl'\colon \sB\to \sB$ and is denoted by $(\sB,\cl,\cl')$.  For  bi-closure algebras we have an associated generalized bi-topological space.
The same consideration hold if we used derivative operators.\index{Derivative operator}

For bi-closure algebras discretization can be formulated as bofore. An embedding
\[
|\cdot|\colon\bigl( \sSet(\cX),\ccl,\ccl'\bigr) \rightarrowtail \bigl(\sB,\cl,\cl'\bigr),
\]
is a bi-topological CA-discretization if $\cl|\cU| \subset |\ccl\cU|$
and $\cl'|\cU| \subset |\ccl'\cU|$ for all $\cU\subset \cX$.
In order to describe discretization in terms of continuous discretization maps we need to use a representation of the bi-closure algebra such as the approach by Mckinsey-Tarski, cf.\ \cite{mcktar}, or
Jonsson-Tarski, cf.\ \cite{JonssonTarski}.

In a slightly more general setting one may define a \emph{modal algebra} by specifying a Boolean algebra $\sB$ and an  operator $\Phi\colon \sB\to \sB$ satisfying (M1)-(M2).\footnote{As for closure algebras use the neutral elements $0$ and $1$.} The latter is   called a(n) (abstract) modal operator\index{Modal operator!abstract} and the associated modal algebra is denoted by $(\sB,\Phi)$.
For a modal algebra the lattice of closed elements is given by  $\sFwdset(\Phi) = \{b\in \sB\mid \Phi b\le b\}$.
Similarly a \emph{bi-modal algebra}\index{Bi-modal algebra} is given by two modal operators $\Phi,\Phi'\colon\sB \to \sB$ and is denoted by $(\sB,\Phi,\Phi')$. 
Even though the closed element  almost yield  generalized topological spaces the correspondence is more involved in this case.

%% file: flowastop.tex
\chapter{Flow topologies and discretization of dynamics}
\label{sec:disc-dyn}
In the previous sections we discussed 
topological and bi-topological spaces in terms of closure algebras 
which is the appropriate language for formalizing discretization of topology.
The next step 
is to model the dynamics of
 semi-flows on topological spaces 
 via  appropriately constructed topologies on $X$.
Such topologies may be realized in many different ways and we refer to the dynamics induced topologies as \emph{flow topologies}.
The objective  is not to develop explicit methods for discretizing dynamics but to describe the contours of a theory that discretizes dynamics in terms of discretizing two topologies.

\section{Dynamics as topology}
\label{twotopos}
As discussed in the previous sections a topological space $(X,\scrT)$  can be equivalently described via the closure algebra $\cl\colon\sSet(X) \to \sSet(X)$. 
This description is convenient for introducing new topologies in relation to dynamical systems.

\subsection{Basic flow topologies}
\label{basicflowtop}\index{Flow topology}\index{Relational flow topology}
 
For a  semi-flow $\varphi$ define the (completely additive) modal operators\index{Modal operator!completely additive}
$\der^-, \der^+\colon \sSet(X)\to \sSet(X)$  given by
\begin{equation}
    \label{secondcl3}
   U \mapsto \der^- U := \bigcup_{t> 0} \varphi({-t},U),\quad U \mapsto \der^+ U := \bigcup_{t> 0} \varphi(t,U),
\end{equation}
The operators $\der^-$ and $\der^+$ which are called the \emph{strict backward image}\index{Strict backward image} and \emph{strict forward image}\index{Strict forward image} operators respectively.
The operators $\cl^-=\id\cup \der^-$ and $\cl^+ = \id \cup \der^+$ are obtained by taking $t\ge 0$ and  satisfy
%
all four Kuratowski axioms (K1)-(K4) for closure operators.
%
The derivative operators\index{Derivative operator} $\der^-$ and $\der^+$ define the topologies $\scrT^-$ and $\scrT^+$ on $X$ respectively, which are Alexandrov topologies on $X$. 
The associated specialization pre-order on $X$ defined by $\scrT^+$ will   be denoted by $\le^+$ and is
defined by $y\le^+ x$ if and only if $y\in \rmGamma^+\{x\}$. The latter is  
 characterized by
\[
y\le^+ x\quad\text{if and only if}\quad y = \varphi(t,x)\quad\text{for some}\quad t\ge 0.
\]
The pre-order $\le^+$ does record the directionality of the flow $\phi$ but discards the sense of time and invariance.
The closed sets in $\scrT^+$ are the forward invariant sets for $\phi$ and are denoted by $\Invsetpl(\varphi)$.
As a matter of fact, using the notation in Sect.\ \ref{modalstuff}, we have that $\Invsetpl(\varphi) = \sFwdset(\der^+) = \sFwdset(\cl^+)$.
The two topologies $\scrT$ and $\scrT^+$ combined comprise the bi-topological space $(X,\scrT,\scrT^+)$.
The associated specialization pre-order on $X$ defined by $\scrT^-$ will   be denoted by $\le^-$ and is
defined by $y\le^- x$ if and only if $y\in \rmGamma^-\{x\}$. This pre-order is characterized by
$y\le^- x$ if and only if $y \in\varphi({-t},x)$ for some $t\ge 0$, i.e. $x=\varphi(t,y)$ for some $t\ge 0$. This shows that $\le^-$ is the opposite pre-order to $\le^+$
and closed sets in $\scrT^+$ are open sets in $\scrT^-$ and vice versa.
The Alexandrov topologies $\scrT^-$ and $\scrT^+$ are each other's opposites.
The closed sets in $\scrT^-$ are the backward invariant sets for $\varphi$ and are denoted by $\Invsetneg(\varphi)$.
The two topologies $\scrT$ and $\scrT^-$ combined comprise the bi-topological space $(X,\scrT,\scrT^-)$.

\begin{remark}
\label{aboutstar}
For the topologies $\scrT^-$ and $\scrT^+$  the closure  and conjugate closure operators are related: $\bcl^+ = \nst^+= \cl^-$ and $\bcl^- = \nst^- = \cl^+$.
\end{remark}


Since $\cl^-$ and $\cl^+$ are closure operators $\der^-$ and $\der^+$ are
(canonical) \emph{derivative operators} satisfying  the Axioms (D1)-(D3).
Observe that 
\begin{equation}
\begin{aligned}
\label{strongdop144}
\der^+\bigl(\der^+ U\bigr)  &= \bigcup_{t>0} \varphi\bigl(t,\bigcup_{s>0} \varphi(s,U)\Bigr) = \bigcup_{t>0} \bigcup_{s>0}\varphi(t,\varphi(s,U))\\
&= \bigcup_{s+t >0} \varphi(s+t,U) = \der^+ U,
\end{aligned} 
\end{equation}
and therefore
 $\der^+$ satisfies the stronger idempotency axiom  (K3), i.e. $\der^+(\der^+ U) = \der^+ U$ and $\der^+$ is an idempotent derivative operator.
 In the same way one proves that $\der^-$ satisfies the idempotency axiom in (K3).

The  derivative operator $\der^+$ is associated with a  binary relation on $X$:
$y<^+ x$ if and only if $y = \varphi(t,x)$ for some $t> 0$.
Observe that the $<^+$ is a transitive relation and the reflexive closure yields the specialization pre-order $\le^+$. 
Points $x\in X$ for which  $\der^+ \{x\} =\{x\}$ correspond to fixed points of $\phi$ and are examples of reflexive points for $<^+$, i.e. $x<^+ x$.
Other reflexive points are given by periodic orbits for $\phi$, cf.\ \cite{Akin}.
The derivative operator $\der^+$ does not  detect invariant sets in general. Indeed, for 
$\varphi(t,x) = x+t$ 
the set $U=(0,\infty)$ satisfies $\der^+ U=U$ but is not invariant. To capture invariance one can use $\tau$-forward image operator $\Gammatau$ and topology $\scrT^+_\tau$ by considering forward images from $t\ge\tau$ and similarly for $\tau$ negative.
cf.\ Sect.\ \ref{topologization}.
%
In view of the considerations in Sect.\ \ref{modalstuff} one can also consider the topology $\scrT_\tau$ defined via the modal operator $\Phi$ defined by $\Phi U = \varphi(\tau,U)$ for some $\tau\neq 0$.
In particular we have that 
\[
\scrT^+ \subset \scrT_\tau^+ \subset \scrT_\tau,\quad \tau>0,
\]
and the same for $\tau<0$ and all topologies are Alexandrov.
The flow topologies capture directionality but do not require any continuity properties on $\phi$. An interesting feature of topological dynamics is to study convergence and decompositions. To do so we will now explore an alternative way to recast dynamics in terms of topology.

\begin{remark}
\label{contflow12}
If $\varphi$ is a continuous semi-flow then the continuity of $\varphi(t,\cdot)$
with respect to the topology $\scrT^+$ is immediate.
Continuity of $\varphi$ with respect to $\scrT$ yields interaction of the two topologies as a manifestation of the continuous semi-flow $\varphi$ on $X$. 
This implies the \emph{subcommutativity}\index{Subcommutativity of topologies} for $\scrT$ and $\scrT^+$, i.e.
$U\subset X$ $\scrT^+$-closed implies that $\cl~U$ is $\scrT^+$-closed. This 
 makes the space $(X,\scrT,\scrT^+)$ a $(\scrT,\scrT^+)$-subcommutative bi-topological space.\index{$(\scrT,\scrT^+)$-subcommutative bi-topological space}\index{Subcommutative bi-topological space}\index{Bi-topological space!subcommutative}
\end{remark}

\subsection{The \bflt $\scrTbf$}
\label{derivesflowtop}

As our aim is an algebraization of dynamics that recovers invariance
based on Wazewski's principle, cf.\  \cite{conley:cbms}, we study the 
%
attracting and repelling blocks of $\varphi$.
Recall that closed \emph{attracting blocks}\index{Attracting block}\index{Attracting block!closed} are defined by
\begin{equation}
\label{clattbl12}
\sABlockC(\varphi) := \bigl\{ U\subset X\mid \cl~U=U,~~\varphi(t,U) \subset \Int U,~~\forall t>0\bigr\}.
\end{equation}
 Let $U,U'\in \sABlockC(\varphi)$.\index{$\sABlockC(\varphi)$} 
 If the singular homology satisfies
$H(U,U')\neq 0$, then $\Inv(U\smin U') \neq \varnothing$, cf.\ Sect.\ \ref{MDs}.
In order to incorporate attracting blocks into the theory of (bi)-closure algebras we construct topologies derived 
from the basic flow topologies. In general, there are a number of options to define such  topologies. We highlight one particular choice that serves the purpose of constructing closed attracting blocks, cf.\ Sect.\ \ref{topologization}. 

\begin{lemma}
\label{super}
The operator $\uderm:=\der^-\cl$ is a modal operator.
\end{lemma}

\begin{proof}
Axioms (M1)-(M2) in Section \ref{modalstuff} are satisfied since $\cl$  is a closure operator and $\der^-$ is a derivative operator,
which establishes $\uderm$ as a normal, additive operator on $\sSet(X)$.
\end{proof}


The operator $\uderm$ does not define a derivative operator since (D3) is not satisfied in general. 
Using the theory in
 Section \ref{modalstuff}
 $\uderm$  can be used to  define a topology on $X$.
\begin{definition}
\label{udermdefn}
A subset $U\subset X$ is closed with respect to $\uderm$ if and only if $\uderm U \subset U$. Such set are denoted by $\sFwdset(\uderm)$.
\end{definition}
By Proposition \ref{modaltop} $\sFwdset(\uderm)$ defines a topology on $X$ which is denoted by $\scrTbf$\index{$\scrTbf$} and is called the \emph{\bflt}\index{Block-flow topology} on $X$.
Observe that the condition that $U\subset X$ is $\scrTbf$-closed is equivalent to the condition $\varphi(-t,\cl~U)\subset U$ for all $t>0$.
%
%
The associated closure operator is given by
\begin{equation}
    \label{udermtop3}
    \uclbf   U := \bigcap\bigl\{U'\supset U\mid U' \in \sFwdset(\uderm) \bigr\}.
\end{equation}
By definition   the \bflt $\scrTbf$ is not necessarily an Alexandrov topology.
A subset $U\subset X$ is \emph{$\scrTbf$-closed} if $\uclbf   U = U$.

\begin{proposition}
\label{altder222}
The \bflt $\scrTbf$ is a coarsening of the Alexandrov topology $\scrT^-$ which is generated by $\cl^-$, i.e. $\cl^- \subset \uclbf  $.
\end{proposition}
\begin{proof}
By definition $\der^- U \subset \uderm U\subset U$ which implies $\der^- U\subset U$. Therefore, 
$\cl^- \subset \uclbf  $.
\end{proof}

\begin{proposition}
\label{continderived}
The  maps $\varphi(t,\cdot)$ are continuous in the \bflt $\scrTbf$ for all $t\ge 0$.
\end{proposition}

\begin{proof}
For the backward image it holds that $\varphi(-t,\der^- U) = \der^-\varphi(-t,U)$, $t\ge 0$, and thus $\varphi(-t,\der^-\cl~ U) = \der^-\varphi(-t,\cl~ U) \supset \der^-\cl~\varphi(-t, U)$. 
Suppose $U$ is $\scrTbf$-closed, i.e. $\uderm U\subset U$.
Then, 
\[
\uderm \varphi(-t,U)\subset \varphi(-t,\uderm U) \subset \varphi(-t,U),
\]
which proves that $\varphi(-t,U)$ is $\scrTbf$-closed. 
The latter proves that for every $t\ge 0$ the map $\varphi(t,\cdot)$ is continuous in the $\scrTbf$-flow topology.
%
\end{proof}

Observe, by Proposition \ref{continderived}, that $\varphi(t,\cdot)$ is continuous in both topologies $\scrT$ and $\scrTbf$.
The \bflt $\scrTbf$ 
defines a new, derived subcommutative bi-topological space\index{Subcommutative bi-topological space}\index{Bi-topological space!subcommutative} $(X,\scrT,\scrTbf)$. 
Open and closed sets in the \bflt $\scrTbf$  can be characterized via the semi-flow $\varphi$.
\begin{lemma}
\label{charattbl12345}
A subset $U\subset X$ is closed in the \bflt $\scrTbf$ if and only if
\begin{equation}
    \label{attblock14bc}
  \varphi(-t,\cl~U)\subset U,\quad \forall t>0.
\end{equation}
Similarly, a subset $U\subset X$ is open in the \bflt $\scrTbf$ if and only if
\begin{equation}
    \label{attblock15bc}
  \varphi(t,U)\subset \Int U,\quad \forall t>0.
\end{equation}
\end{lemma}
\begin{proof}
By definition a subset $U\subset X$ is $\scrTbf$-closed if and only if $\uclbf   U=U$
which implies $\uderm U \subset U$ and thus $\der^- \cl~U\subset U$. The latter implies that $\varphi(-t,\cl~U)\subset U$ for all $t>0$.
Conversely, if \eqref{attblock14bc} holds then $\der^- \cl~U\subset U$ and thus
$\uclbf   U=U$.



By definition a subset $U\subset X$ is $\scrTbf$-open if and only if $U^c$ is $\scrTbf$-closed, i.e.
$\varphi(-t,\cl~U^c) \subset U^c$ for all $t>0$. The latter is equivalent to $\varphi(-t,\cl~U^c)^c \supset U$ for all $t>0$, which results in the equivalent statement that $\varphi(-t,\Int U) \supset U$ for all $t>0$.
If we compose the latter with $\varphi(t,\cdot)$ we obtain
\[
\varphi(t,U) \subset \varphi\bigr(t,\varphi(-t,\Int U) \bigr)\subset \Int U,\quad\forall t>0.
\]
On the hand,  composition of  \eqref{attblock15bc} with the inverse image $\varphi(-t,\cdot)$ gives
\[
U \subset \varphi\bigl(-t,\varphi(t,U) \bigr) \subset \varphi(-t,\Int U),\quad\forall t>0,
\]
 which prove that $U$ is $\scrTbf$-open.
\end{proof}


The $(\scrT,\scrTbf)$-pairwise  clopen sets in $(X,\scrT,\scrTbf)$ are
defined as sets $U\subset X$ that are closed in $\scrT$ and open in $\scrTbf$.
\begin{theorem}
\label{charclattbl}
A subset $U\subset X$ is a
 {closed attracting block},\index{Attracting block} cf.\ \eqref{clattbl12},
 if and only if $U$ is a $(\scrT,\scrTbf)$-pairwise  clopen set in $(X,\scrT,\scrTbf)$.
\end{theorem}

\begin{proof}
This follows from Lemma \ref{charattbl12345} and the definition of closed attracting blocks, cf.\ Eqn.\ \eqref{clattbl12}.
\end{proof}
The
set of   closed attracting blocks $\sABlockC(\varphi)$ is
a sublattice of $\sSet(X)$, cf.\ \cite{lsa,lsa2}, \cite{kkv}.
This lattice is not complete in general.
The $(\scrTbf,\scrT)$-pairwise clopen sets correspond to open repelling blocks. Indeed, $U$ is open and \eqref{attblock15bc} we have that $\varphi(-t,\cl~U) \subset U$ for all $t>0$. We denote the open \emph{repelling blocks}\index{Repelling block}\index{Repelling block!open} by $\sRBlockO(\varphi)$.\index{$\sRBlockO(\varphi)$}

\begin{remark}\label{rem:phi_continuity}
\label{usecont}
The continuity of $\varphi$ in the $\scrT$-topology 
can be relaxed to an $\R^+$-parameter family $\varphi(t,\cdot)$ of  continuous maps in the $\scrT$-topology for all $t\ge 0$. 
 This implies in particular that we can apply discretization and topologization to other families of maps such as $t\in \Z^+$ which is equivalent to iterating a map, i.e. discrete time dynamics. Continuity of $\varphi(t,\cdot)$
 implies a more equal role for both topologies.
 The continuity of $\varphi$ on $\R^+\times X$ comes in in two instances, (i) equivalent discretization via \discresols, cf.\ Sect.\ \ref{secondgrad}, and (ii) algebraization in order to invoke Wazewski's principle for finding invariant sets, cf.\ Sect.\ \ref{topologization} and \cite{conley:cbms}.
\end{remark}


\begin{remark}
\label{otherdertop}
If we consider variations of the operator $\uderm$ such as $\uder = \der^+\cl$  then 
the $(\scrTbf,\scrT)$-clopen sets
are open attracting blocks\index{Attracting block!open}
and the $(\scrT,\scrTbf)$-clopen sets are the closed repelling blocks.\index{Repelling block!closed}
Other variations entail $\cl~\der^+$ and $\cl~\der^-$, cf.\ Rem.\ \ref{contflow12} and Sect.\ \ref{topologization}. Our definition of \bflt is suitable for the theory in this text.
\end{remark}

\section{Discretization of the \bflt}
\label{discdyn123}
In Section \ref{disc-top} we discussed discretization of topology in terms of closure algebras and derivative algebras. The standard CW-decompositions of  spaces are examples of such 
discretizations. In this section we apply the closure algebra discretization to the \bflt which provides the appropriate discretization of dynamics.
%
As we have encoded $\varphi$ as a topological space $(X,\scrT,\scrTbf)$, we can apply the tools from Section~\ref{disc-top}, e.g., CA-discretizations, MA-discretizations, discretization maps and topology consistent pre-orders.  

Let  $\cX$ be a finite sets and let $\uclbff  \colon \sSet(\cX) \to \sSet(\cX)$ be a closure
operator
such that 
$\uclbf   |\cU| \subset |\uclbff   \cU|$ for all $\cU\in \sSet(\cX)$.
This induces a pre-order $\lebf$ by \eqref{spec1} and yields the continuity of $\grd\colon (X,\scrTbf) \to (\cX,\lebf  )$
defined in Eqn.\ \eqref{induceddisc}.
Conversely, if $\lebf  $ is any $\scrTbf$-consistent pre-order with  respect to $\grd$ then the associated
closure operator $\uclbff$ defines a CA-discretization for $(X,\scrTbf)$.
%
\begin{lemma}
\label{admis1a1}
A pre-order $\lebf  $ on $\cX$ with associated closure operator $\uclbff  \colon \sSet(\cX) \to \sSet(\cX)$ is $\scrTbf$-consistent with respect to $\disc\colon X \twoheadrightarrow\cX$ if and only if
\begin{equation}
    \label{secondtop6}
\varphi\bigl(-t,\cl |\xi|\bigr) \subset \bigl| \uclbff  \xi\bigr|,\quad 
 \forall t > 0,
\quad\text{and}\quad \forall \xi\in \cX.
\end{equation}
\end{lemma}

\begin{proof}
The discretization $\disc\colon X\twoheadrightarrow \cX$ is continuous if and only if $\cU$ closed in $(\cX,\lebf  )$ implies that $|\cU|$ is $\scrTbf$-closed.
Assume \eqref{secondtop6} is satisfied. Let $\uclbff   \cU = \cU$, then 
\[
\begin{aligned}
\varphi(-t,\cl|\cU|) &= \bigcup_{\xi\in \cU} \varphi(-t,\cl|\xi|)\subset \bigcup_{\xi\in \cU}|\uclbff  \xi|\\
&= \Bigl|\uclbff  \Bigl(\bigcup_{\xi\in\cU}\{\xi\}\Bigl)\Bigr| = |\uclbff  \cU| =|\cU|,\quad\forall t>0,
\end{aligned}
\]
which proves that $\uderm|\cU| \subset |\cU|$ and thus $|\cU|$ is $\scrTbf$-closed.
%
%
Conversely, if $\grd$ is continuous, 
then $\cU$ closed in $(\cX,\lebf  )$ implies $|\cU|$ is $\scrTbf$-closed and
therefore $\varphi(-t,\cl|\cU|)\subset |\cU|$. Choose $\cU = \uclbff  \xi$. This implies that
\[
\varphi(-t,\cl|\xi|) \subset \varphi(-t,\cl|\uclbff  \xi|)\subset |\uclbff  \xi|,\quad\forall t>0,\quad\text{and}\quad \forall \xi\in \cX.
\]
%
which establishes \eqref{secondtop6}.
\end{proof}

In the discrete setting we can define the discrete analogues of the \bflt $(X,\scrTbf)$ via discretizations of $\scrT$ and $\scrT^-$.
Let $\derr^-\colon \sSet(\cX) \to \sSet(\cX)$ be a discrete  derivative for $\scrT^-$ and define the additive operator $\udermm:= \derr^-\ccl\colon \sSet(\cX) \to \sSet(\cX)$. 
\begin{lemma}
\label{choiceforgam1}
The triple $(\sSet(\cX,\udermm,|\cdot)$ defines a MA-discretization of $(X,\scrTbf)$. The associated discrete closure operator is given by $\uclbff = \bigcup_{k\ge 0} \bigl( \udermm \bigr)^k$ and 
the triple $(\cX,\uclbff  ,|\cdot|)$
is a CA-discretization for $(X,\scrTbf)$.
\end{lemma}
\begin{proof}
In order to establish $(\sSet(\cX,\udermm,|\cdot)$ as a MA-discretization we use the fact that $\ccl$ defines a CA-discretization  $(\cX,\ccl,|\cdot|)$ for $(X,\scrT)$ and $\derr^-$ defines a MA-discretization $(\cX,\derr^-,|\cdot|)$ for $(X,\scrT^-)$. This implies, for $\cU\subset \cX$, that
\begin{equation}
\label{MAcheck}
\uderm |\cU| = \der^-\cl|\cU| \subset \der^-|\ccl\cU|\subset |\derr^-\ccl~\cU|
= |\udermm\cU|,
\end{equation}
which by Proposition \ref{MAtoCA} shows that $(\sSet(\cX,\udermm,|\cdot)$ is a MA-discretization of $(X,\scrTbf)$
and provides the expression for $\uclbff$.
It remains to show that $\disc\colon X\twoheadrightarrow\cX$ is continuous.
Let $\cU=\{\xi\}$, then \eqref{MAcheck} yields
$\varphi(-t,\cl|\xi|) \subset \uderm |\xi| \subset |\udermm \xi| \subset  |\uclbff  \xi|$ for all $t>0$ which, by Lemma \ref{admis1a1}, proves that $\disc$ continuous.
\end{proof}

The following lemma  formulates a criterion
 for discretizing the \bflt $\scrTbf$ without using a discretization for $\scrT^-$.
\begin{lemma}
\label{udermlem}
Let $\uderup\colon \sSet(\cX) \to \sSet(\cX)$ be a modal 
operator such that
\begin{equation}
    \label{secondtop5}
\varphi\bigl(-t,\cl|\xi|\bigr) \subset \bigl| \uderup\xi\bigr|,\quad \forall t> 0,\quad\text{and}\quad \forall \xi\in \cX.
\end{equation}
Then, the  operator $\uclbff = \bmPhi^{\bm{+=}}=\bigcup_{k\ge 0}  \bmPhi^k \colon \sSet(\cX)\to \sSet(\cX)$ is a  closure operator and yields a  CA-discretization $(\cX,\uclbff,|\cdot|)$ of the \bflt  $(X,\scrTbf)$. 
\end{lemma}

\begin{proof}
The fact that $\uderup^{\bm+=}$ is closure operator follows from  Lemma \ref{choiceforgam1}.
By assumption $\bigcup_{t>0} \varphi(-t,\cl |\xi|) = \uderm |\xi| \subset |\uderup \xi|$.
As in the proof of Lemma \ref{choiceforgam1} this implies that 
$\varphi(-t,\cl|\xi|) \subset \uderm |\xi| \subset |\uderup \xi| \subset |\uderup^{\bm{+=}} \xi| \subset |\uclbff  \xi|$
for all $t>0$ which proves that $\disc$ is continuous, completing the proof.
%
\end{proof}

The following result gives a local version of the above criterion and provides a practical method for constructing CA-discretizations for the \bflta. Fig.\ \ref{morsetes2} displays an example of a discretization of both $(X,\scrT)$ and $(X,\scrTbf)$.


\begin{figure}
\centering
\begin{minipage}{.175\textwidth}
\centering
\begin{tikzpicture}[dot/.style={draw,circle,fill,inner sep=.75pt},line width=.3pt,scale=1.2, decoration={markings, 
    mark= at position 0.55 with {\arrow{latex}}}]
\draw (0,0) -- (0,3) -- (1,3) -- (1,0) -- cycle; 
\draw (0,1) -- (1,1);
\draw (0,2) -- (1,2);

\node[Orange] (1) at (.5,1)[dot] {};
\node[Orange] (2) at (.5,1.6)[dot] {};

\draw[Orange, postaction={decorate}] (2) to[out=270,in=100] (1);
\draw[Orange, postaction={decorate}] (-.1,1.35) to[out=0,in=160] (1);
\draw[Orange, postaction={decorate}] (-.1,1.15) to[out=-10,in=210] (1);
\draw[Orange, postaction={decorate}] (1.1,1.45) to[out=180,in=20] (1);
\draw[Orange, postaction={decorate}] (1.1,1.25) to[out=190,in=-30] (1);
\draw[Orange, postaction={decorate}] (-.1,.5) to[out=0,in=250] (1);
\draw[Orange, postaction={decorate}] (.5,-.1) to[out=90,in=280] (1);
\draw[Orange, postaction={decorate}] (1.1,.5) to[out=180,in=310] (1);

\draw[Orange, postaction={decorate}] (-.1,1.5) to[out=0,in=180] (2);
\draw[Orange, postaction={decorate}] (1.1,1.65) to[out=180,in=10] (2);
\draw[Orange, postaction={decorate}] (-.1,2.85) to[out=-30,in=120] (2);
\draw[Orange, postaction={decorate}] (.5,3.1) to[out=280,in=100] (2);
\draw[Orange, postaction={decorate}] (1.1,2.65) to[out=210,in=60] (2);
\end{tikzpicture}

$\vphantom{X}$
\end{minipage}
\begin{minipage}{.225\textwidth}
\centering
\begin{tikzpicture}[dot/.style={circle,fill, inner sep = 0.75, outer sep=1.0}, scale=1.25]
\def\h{.1}
\foreach \x in {0}
    \foreach \y in {0,1,2}
        \fill[black,fill opacity=.15](\x+\h,\y+\h) rectangle (\x+1-\h,\y+1-\h);
\foreach \x in {0}
    \foreach \y in {1,2}
        \filldraw[thick] (\x+\h,\y) -- (\x+1-\h,\y);
\foreach \x in {0}
    \foreach \y in {1, 2}
        \node at (\x+\h,\y) [dot] {};
\foreach \x in {1}
    \foreach \y in {1, 2}
        \node at (\x-\h,\y) [dot] {};

\draw[black,thick] (\h,\h) -- (1-\h,\h);
\draw[black,thick] (\h,3-\h) -- (1-\h,3-\h);
\foreach \x in {0}
    \foreach \y in {0,1,2} {
        \draw[black,thick] (\x+\h,\y+\h) -- (\x+\h,\y+1-\h);
        \draw[black,thick] (\x+1-\h,\y+\h) -- (\x+1-\h,\y+1-\h);
    }
\def\a{0}
\node[anchor=west] at (1+\a,.5)     {\scriptsize $\xi_1$};
\node[anchor=west] at (1+\a,1)      {\scriptsize $\xi_2$};
\node[anchor=west] at (1+\a,1.5)    {\scriptsize $\xi_3$};
\node[anchor=west] at (1+\a,2)      {\scriptsize $\xi_4$};
\node[anchor=west] at (1+\a,2.5)    {\scriptsize $\xi_5$};
\end{tikzpicture}

\hspace{-5mm} {\scriptsize $\cX$}
\end{minipage}
\begin{minipage}{.08\textwidth}
\begin{tikzpicture}[2cell/.style={diamond,fill,inner sep=1.5pt},
2cella/.style={regular polygon,regular polygon sides=6,fill,inner sep=1.5pt},
2cellb/.style={regular polygon,regular polygon sides=5,draw, inner sep=1.2pt},
edge/.style={circle,draw,inner sep = 1.3pt}, line width=.45pt,scale=.75]
\def\h{.15}
\foreach \y in {0,4}
    \node(\y) at (0,\y) [2cell] {};
\foreach \y in {2}
    \node(\y) at (0,\y) [2cella] {};
\foreach \y in {1,3}
    \node(\y) at (0,\y) [2cellb] {};
\draw[-latex] (0) to (1);
\draw[-latex] (2) to (1);
\draw[-latex] (2) to (3);
\draw[-latex] (4) to (3);
\end{tikzpicture}

{\scriptsize $\leq$}
\end{minipage}
\begin{minipage}{.125\textwidth}
\begin{tikzpicture}[2cell/.style={diamond,fill,inner sep=1.5pt},
2cella/.style={regular polygon,regular polygon sides=6,fill,inner sep=1.5pt},
2cellb/.style={regular polygon,regular polygon sides=5,draw, inner sep=1.2pt},
edge/.style={circle,draw,inner sep = 1.3pt}, line width=.45pt,scale=.75]
\def\h{.15}
\foreach \y in {0,4}
    \node(\y) at (0,\y) [2cell] {};
\foreach \y in {2}
    \node(\y) at (0,\y) [2cella] {};
\foreach \y in {1,3}
    \node(\y) at (0,\y) [2cellb] {};
\draw[-latex] (0) to (1);
\draw[latex-latex] (1) to (2);
\draw[-latex] (2) to (3);
\draw[-latex] (3) to (4);
\draw[-latex] (0) to[out=25,in=-25] (2);
\draw[-latex] (2) to[out=25,in=-25] (4);
\draw[-latex] (0) to[loop left] (0);
\draw[-latex] (1) to[loop left] (1);
\draw[-latex] (2) to[loop left] (2);
\draw[-latex] (4) to[loop left] (4);

\end{tikzpicture}

{\scriptsize $\uderup^{-1}$}
\end{minipage}
\begin{minipage}{.1\textwidth}
\begin{tikzpicture}[2cell/.style={diamond,fill,inner sep=1.5pt},
2cella/.style={regular polygon,regular polygon sides=6,fill,inner sep=1.5pt},
2cellb/.style={regular polygon,regular polygon sides=5,draw, inner sep=1.2pt},
edge/.style={circle,draw,inner sep = 1.3pt}, line width=.45pt,scale=.75]
\def\h{.15}
%
\foreach \y in {0,4}
    \node(\y) at (0,\y) [2cell] {};
\foreach \y in {2}
    \node(\y) at (0,\y) [2cella] {};
\foreach \y in {1,3}
    \node(\y) at (0,\y) [2cellb] {};
\draw[-latex] (0) to (1);
\draw[latex-latex] (2) to (1);
\draw[-latex] (2) to (3);
\draw[latex-] (4) to (3);
\end{tikzpicture}

{\scriptsize $\lebf$}
\end{minipage}
\begin{minipage}{.1\textwidth}
\begin{tikzpicture}[2cell/.style={diamond,fill,inner sep=1.5pt},
2cella/.style={regular polygon,regular polygon sides=6,fill,inner sep=1.5pt},
2cellb/.style={regular polygon,regular polygon sides=5,draw, inner sep=1.2pt},
edge/.style={circle,draw,inner sep = 1.3pt}, line width=.45pt,scale=.75]
\def\h{.15}
\foreach \y in {0,4}
    \node(\y) at (0,\y) [2cell] {};
\foreach \y in {2}
    \node(\y) at (0,\y) [2cella] {};
\foreach \y in {1,3}
    \node(\y) at (0,\y) [2cellb] {};
\draw[latex-] (0) to (1);
\draw[latex-latex] (2) to (1);
\draw[latex-] (2) to (3);
\draw[-latex] (4) to (3);
\end{tikzpicture}

{\scriptsize $\gebf$}
\end{minipage}
\begin{minipage}{.1\textwidth}
\begin{tikzpicture}[2cell/.style={diamond,fill,inner sep=1.5pt},
2cella/.style={regular polygon,regular polygon sides=6,fill,inner sep=1.5pt},
2cellb/.style={regular polygon,regular polygon sides=5,draw, inner sep=1.2pt},
edge/.style={circle,draw,inner sep = 1.3pt}, line width=.45pt,scale=.75]
\def\h{.15}
\foreach \y in {0,4}
    \node(\y) at (0,\y) [2cell] {};
\foreach \y in {2}
    \node(\y) at (0,\y) [2cella] {};
\foreach \y in {1,3}
    \node(\y) at (0,\y) [2cellb] {};
\draw[latex-latex] (0) to (1);
\draw[latex-latex] (2) to (1);
\draw[latex-latex] (2) to (3);
\draw[-latex] (4) to (3);
\end{tikzpicture}

{\scriptsize $\le^\dagger$}
\end{minipage}
\caption{
The topological space $(X,\scrT)$ and semi-flow $\phi$ [left 1].  A discretization of $X$ with $\scrT$-consistent pre-order $(\ccX,\le)$  [middle 2, 3]. 
The relation $\uderup^{-1}$ [middle 4] which generates the 
 $\scrTbf$-consistent pre-order $(\ccX,\lebf  )$ [right 5].
The discretization map $\disc\colon X\to\ccX$ is continuous with respect to both topologies. Common coarsening of $(\ccX,\le)$ and $(\ccX,\gebf  )$ [right 6] resulting in a Morse pre-order $(\ccX,\le^\dagger)$
[right 7].   All pre-orders are represented by their Hasse diagrams.}
\label{morsetes2}
\end{figure}

\begin{theorem}
\label{localconstrGam}
Let $(X,\scrT)$ be compact and let $\uderup\colon \sSet(\cX) \to \sSet(\cX)$ be a modal operator.
Assume that for every $\xi\in \cX$ there exists $t_\xi>0$ such that 
\begin{equation}
    \label{secondtop555}
\varphi\bigl(-t,\cl|\xi|\bigr) \subset \bigl| \uderup\xi\bigr|,\quad \forall 0<t\le t_\xi,\quad\text{and}\quad \forall \xi\in \cX.
\end{equation}
Then, $(\cX,\uclbff  ,|\cdot|)$, with $\uclbff  = \uderup^{\bm{+=}}$, is a CA-discretization for  $(X,\scrTbf)$.
\end{theorem}

\begin{proof}
The proof of based on the following observation. By the compactness of $X$ we may assume, without loss of generality, that $t_\xi\ge t_*>0$ for all $\xi\in \cX$. 
Then, $\varphi(-t,\cl|\xi|) \subset |\uderup\xi|$ for all $0<t\le t_*$ and for all $\xi\in \cX$.
Observe that 
\[
\begin{aligned}
\varphi(-2t,\cl|\xi|) &= \varphi\bigl(-t, \varphi(-t,\cl|\xi|)\bigr)\subset \varphi\bigl(-t,|\uderup\xi|\bigr)\\
&\subset  \varphi\bigl(-t,\cl|\uderup \xi|\bigr)\subset | \uderup^2\xi|,
\end{aligned}
\]
which yields $\varphi(-kt,\cl|\xi|)\subset  | \uderup^k\xi|
\subset |\uderup^{\bm{+=}}\cU|$,  for all $k\ge 0$ and for all $0<t\le t_*$. As  in the proof of Lemma \ref{udermlem},
$\varphi(-t,\cl|\xi|)  \subset |\uderup^{\bm{+=}} \xi| \subset |\uclbff  \xi|$
for all $t>0$ which proves by Lemma \ref{admis1a1} that $\disc$ is continuous,
and thus $(\cX,\uderup^{\bm{+=}},|\cdot|)$ is a CA-discretization for  $(X,\scrTbf)$.
%
\end{proof}

\begin{remark}
The operator $\uderup$ defines a relation on $\cX$: $(\eta,\xi) \in \bmphi$ if and only if $\eta\in \uderup \xi$, cf.\ App.\ \ref{complBoolAlg}. The transitive, reflexive closure of $\bmphi$ is the pre-order $\lebf  $ associated to
$\uclbff   := \uderup^{\bm{+=}}$. 
\end{remark}

\begin{remark}
\label{discotherfltop}
For discretizing the flow topologies $\scrT^-$ and $\scrT^+$ we can use the criteria in Lemmas \ref{admis1a1} and \ref{secondtop5} by discarding the topology $\scrT$, i.e. take $\cl$ to be the identity map. For example a  discrete closure operator $\ccl^+\colon\sSet(\cX) \to \sSet(\cX)$ yields a CA-discretization for $\scrT^+$ if and only if $\varphi(t,|\xi|)\subset |\ccl^+\xi|$ for all $t\ge 0$ and for all $\xi\in \cX$, i.e. $\scrT^+$-consistency for the associated pre-order $\le^+$. 
\end{remark}

\begin{remark}
\label{modalextra}
The definition of the \bflt in Section \ref{derivesflowtop} uses the modal operator $\uderm$. This construction works for any modal operator $\Phi$ on $\sSet(X)$ as is explained in Section \ref{modalstuff}.
\end{remark}


\section{Morse pre-orders}
\label{secondgrad}\index{Morse pre-order}
In this section we explain the implications of discretization with respect to two topologies in the sense of $(\scrT,\scrTbf)$- pairwise clopen sets. 
We consider the  bi-topological space $(X,\scrT,\scrTbf)$ and we use the theory in Section \ref{bi-top} to discuss discretization in this setting.
Let $\disc\colon X\twoheadrightarrow \cX$ be a discretization map and let
$(\cX,\le^\dagger)$ be an anatagonistic pre-order\index{Antagonistic pre-order} 
for $(X,\scrT,\scrTbf)$, which motivates the following definition:
%
%
 \begin{definition}
 \label{morsepreorder1aa}
 Let $\disc\colon X\twoheadrightarrow \cX$ be a discretization map.
 A \emph{Morse pre-order}\index{Morse pre-order} on $\cX$ is an anatagonistic pre-order $\le^\dagger$ 
 for $(X,\scrT,\scrTbf)$, i.e. 
 \begin{enumerate}
     \item [(i)] $\le^\dagger$ is $\scrT$-consistent with respect to $\grd$; 
     \item [(ii)] $\le^\dagger$  is $\scrTbf$-co-consistent with respect to $\grd$.
 \end{enumerate}
 The associated closure operator 
 is denoted by $\cl^\dagger\colon \sSet(\cX)\to \sSet(\cX)$.
 \end{definition}
%
\begin{figure}
\begin{minipage}{0.5\textwidth}
\centering
\begin{tikzpicture}[scale=.75, dot/.style={circle,fill, inner sep = 1.0, outer sep=2.0}]
\def\h{1.1}
\def\v{.9}
\def\l{.1}

\node (e) at (0,0) [dot] {};
\node[anchor=west] (le) at (\l,0) {\scriptsize $\varnothing$};
\node (2) at (-\h,\v) [dot] {};
\node[anchor=east] (l2) at (-\h-\l,\v) {\scriptsize $\{\xi_2\}$};
\node (4) at (\h,\v) [dot] {};
\node[anchor=west] (l4) at (\h+\l,\v) {\scriptsize $\{\xi_4\}$};
\node(12) at (-2*\h,2*\v) [dot] {};
\node[anchor=east](l12) at (-2*\h-\l,2*\v) {\scriptsize $\{\xi_1,\xi_2\}$};
\node(24) at (0, 2*\v) [dot] {};
\node[anchor=west](l24) at (0+\l, 2*\v) {\scriptsize $\{\xi_2,\xi_4\}$};
\node(45) at (2*\h, 2*\v) [dot] {};
\node[anchor=west](l45) at (2*\h+\l, 2*\v) {\scriptsize $\{\xi_4,\xi_5\}$};
\node(124) at (-\h, 3*\v) [dot] {};
\node[anchor=east](l124) at (-\h-\l, 3*\v) {\scriptsize $\{\xi_1,\xi_2,\xi_4\}$};

\node(245) at (\h, 3*\v) [dot] {};
\node[anchor=west](l245) at (\h+\l, 3*\v) {\scriptsize $\{\xi_2,\xi_4,\xi_5\}$};
\node(234) at (0, 3.5*\v) [dot] {};
\node[anchor=west] (l234) at (0+\l, 3.5*\v) {\scriptsize $~~\,\{\xi_2,\xi_3,\xi_4\}$};
\node(1234) at (-\h, 4.5*\v) [dot] {};
\node[anchor=east](l1234) at (-\h-\l, 4.5*\v) {\scriptsize $\{\xi_1,\xi_2,\xi_3,\xi_4\}$};

\node(2345) at (\h, 4.5*\v) [dot] {};
\node[anchor=west](l2345) at (\h+\l, 4.5*\v) {\scriptsize $\{\xi_2,\xi_3,\xi_4,\xi_5\}$};
\node(12345) at (0, 5.5*\v) [dot] {};
\node[anchor=west](l12345) at (0+\l, 5.5*\v) {\scriptsize $\{\xi_1,\xi_2,\xi_3,\xi_4,\xi_5\}$};







\draw (2) -- (e);
\draw (4) -- (e);

\draw (12) -- (2);
\draw (24) -- (2);
\draw (24) -- (4);
\draw (45) -- (4);

\draw (124) -- (12);
\draw (124) -- (24);
\draw (245) -- (24);
\draw (245) -- (45);
\draw (234) -- (24);


\draw (1234) -- (234);
\draw (1234) -- (124);
\draw (2345) -- (234);
\draw (2345) -- (245);

\draw (12345) -- (1234);
\draw (12345) -- (2345);

\end{tikzpicture}
\end{minipage}
 $\bigcap\quad$ 
\begin{minipage}{0.15\textwidth}
\centering
\begin{tikzpicture}[node distance=.1cm, dot/.style={circle,fill, inner sep = 1.0, outer sep=2.0}]
\def\h{1.1}
\def\v{.9}
\def\l{.1}

\node (e) at (0,0) [dot] {};
\node[anchor=west] (le) at (\l,0) {\scriptsize $\varnothing$};

\node (1) at (0,\v) [dot] {};
\node[anchor=west] (l1) at (\l,\v) {\scriptsize $\{\xi_1\}$};

\node (2) at (0,\v*2) [dot] {};
\node[anchor=west] (l2) at (\l,\v*2) {\scriptsize $\{\xi_1,\xi_2,\xi_3\}$};

\node (3) at (0,\v*3) [dot] {};
\node[anchor=west] (l3) at (\l,\v*3) {\scriptsize $\{\xi_1,\xi_2,\xi_3,\xi_4\}$};

\node (4) at (0,\v*4) [dot] {};
\node[anchor=west] (l4) at (\l,\v*4) {\scriptsize $\{\xi_1,\xi_2,\xi_3,\xi_4,\xi_5\}$};

\draw (1) -- (e);
\draw (2) -- (1);
\draw (3) -- (2);
\draw (4) -- (3);
\end{tikzpicture}
\end{minipage}
$=\quad$
\begin{minipage}{0.15\textwidth}
\centering
\begin{tikzpicture}[node distance=.1cm, dot/.style={circle,fill, inner sep = 1.0, outer sep=2.0}]
\def\h{1.1}
\def\v{.9}
\def\l{.1}

\node (e) at (0,0) [dot] {};
\node[anchor=west] (le) at (\l,0) {\scriptsize $\varnothing$};

\node (1) at (0,\v) [dot] {};
\node[anchor=west] (l1) at (\l,\v) {\scriptsize $\{\xi_1,\xi_2,\xi_3,\xi_4\}$};

\node (2) at (0,\v*2) [dot] {};
\node[anchor=west] (l4) at (\l,\v*2) {\scriptsize $\{\xi_1,\xi_2,\xi_3,\xi_4,\xi_5\}$};

\draw (1) -- (e);
\draw (2) -- (1);
\end{tikzpicture}
\end{minipage}



\caption{Intersection of both the lattice of down-sets $\sO(\ccX,\leq)$ and up-sets $\sU(\ccX,\lebf  )$ for Fig. \ref{morsetes2} yields the coarsening $(\sSC,\leq)$ [right] of the Morse pre-order $(\ccX,\le^\dagger)$.}
\label{morsetes222}
\end{figure}
By Theorem \ref{dubbleconst3} $(\cX,\le^\dagger)$ is an antagonistic coarsening of discretizations for both 
$\scrT$ and $\scrTbf$, cf.\ Sect.\ \ref{bi-top}. 
%
The conditions for an antagonistic pre-order imply that for $\cU\subset \cX$ we have that (i) $\cl|\cU| \subset |\ccl^\dagger\cU|$ and (ii) $\uclbf  |\cU| \subset |\st^\dagger\cU|$. In particular, if $\cU\in \sO(\cX,\le^\dagger)$, then $\ccl^\dagger\cU=\cU$ and thus  $|\cU|$ is $\scrT$-closed. Moreover, 
if $\cU\in \sO(\cX,\le^\dagger)$, then 
$|\cU|$ is $\scrTbf$-open, cf.\ Sect.\ \ref{topconspre}, which implies that $|\cU|$ satisfies
 $\varphi(t,|\cU|) \subset \Int |\cU|$ for $t>0$, cf.\ Lem.\ \ref{charattbl12345}.
These facts combined show that $\cU\in \sO(\cX,\le^\dagger)$ implies that $|\cU|$
is a $(\scrT,\scrTbf)$-pairwise clopen set and therefore a closed attracting block, i.e.
$|\cU|\in \sABlockC(\varphi)$.
We have the following commutative diagram:
\begin{equation}
\label{AB12}
\begin{tikzcd}[column sep=large, row sep=large]
\sSet(X)                      & \sABlockC(\varphi) \arrow[l, "\supset"', tail]                       \\
\sSet(\cX) \arrow[u, "|\cdot|", tail] & \sO(\cX,\le^\dagger) \arrow[u, "|\cdot|"', tail] \arrow[l, "\supset"', tail]
\end{tikzcd}
\end{equation}

\begin{theorem}
    \label{Morsepreordequiv}
    Let $\disc\colon X\twoheadrightarrow \cX$ be a discretization map.
    A pre-order $(\cX,\le^\dagger)$ is a 
{Morse pre-order}\index{Morse pre-order} on $\cX$ for $(X,\scrT,\scrTbf)$ if and only if
$|\ccl^\dagger\xi|$ is $\scrT$-closed and 
\begin{equation}
    \label{critMPO}
    \varphi(t,|\xi|) \subset \Int|\ccl^\dagger\xi|,\quad\forall t>0,
\end{equation}
for all $\xi\in \cX$.
\end{theorem}
\begin{proof}
    If $(\cX,\le^\dagger)$ is a Morse pre-order then for every $\cU\in \sO(\cX,\le^\dagger)$,
    $|\cU|$ is $\scrT$-closed and  $\varphi(t,|\cU|) \subset \Int |\cU|$ for $t>0$. Take $\cU=\ccl^\dagger \xi$. Then, $|\ccl^\dagger\xi|$ is $\scrT$-closed and 
    $\varphi(t,|\xi|)\subset \varphi(t,|\ccl^\dagger\xi|) \subset \Int|\ccl^\dagger\xi|$ for all $t>0$.

    Conversely, suppose $|\ccl^\dagger \xi|$ is $\scrT$-closed and \eqref{critMPO} is satisfied. To prove that $(\cX,\le^\dagger)$ is a Morse pre-order we show that it is $\scrT$-consistent and $\scrTbf$-co-consistent with respect to $\disc$.
    Let $\cU\in \sO(\cX,\le^\dagger)$. Then, $\ccl^\dagger \cU=\cU$ and
    \[
    |\cU| = |\ccl^\dagger \cU| = \bigl|\ccl^\dagger\bigcup \{\xi\}\bigr| = \bigl|\bigcup \ccl^\dagger \xi\bigr| = \bigcup \bigl|\ccl^\dagger \xi\bigr|,
    \]
    is $\scrT$-closed. Therefore, $\cU$ closed in $(\cX,\le^\dagger)$ implies that $\disc^{-1}\cU = |\cU|$ is $\scrT$-closed and thus $\le^\dagger$ is $\scrT$-consistent with respect to $\disc$.
    Moreover, 
    \[
    \begin{aligned}
        \varphi(t,|\cU|) &= \varphi\bigl( t,\bigcup|\xi|\bigr) = \bigcup \varphi(t,|\xi|)\subset \bigcup \Int |\ccl^\dagger \xi|\\
        &\subset \Int \bigcup |\ccl^\dagger \xi| = \Int |\ccl^\dagger \cU| = \Int |\cU|,\quad\forall t>0,
    \end{aligned}
    \]
    which implies that $|\cU|$ is $\scrTbf$-open. Therefore, $\cU$ closed in $(\cX,\le^\dagger)$ implies that $\disc^{-1}\cU = |\cU|$ is $\scrTbf$-open by \eqref{attblock15bc}.
    Recall   that $\scrTbf$-co-consistency can be characterized as follows: $\cU^c\in \sO(\cX,\ge^\dagger)$, then $|\cU|^c = |\cU^c|$ is $\scrTbf$-closed,
    which is equivalent to $\cU\in \sO(\cX,\le^\dagger)$, then $|\cU|$ is $\scrTbf$-open, cf.\ Sect.\ \ref{topconspre}.
    Using the latter proves that $\le^\dagger$ is $\scrTbf$-co-consistent with respect to $\disc$.
\end{proof}
\begin{remark}
    The $\scrT$-consistency of $(\cX,\le^\dagger)$ with respect to $\disc$ implies that $\cl|\xi|\subset |\ccl^\dagger \xi|$ for all $\xi\in \cX$. In particular, this implies that $\varphi(t,\cl|\xi|) \subset \varphi(t,|\ccl^\dagger\xi|)\subset \Int |\ccl^\dagger\ccl^\dagger \xi| = \Int |\ccl^\dagger \xi|$, for all $t>0$ and for all $\xi\in \cX$.
\end{remark}

If $\uderup\colon \sSet(\cX) \to \sSet(\cX)$ is a modal operator such that 
\begin{equation}
    \label{secondtop512}
\varphi\bigl(t,|\xi|\bigr) \subset \Int\bigl| \uderup\xi\bigr|,\quad \forall t> 0,\quad\text{and}\quad \forall \xi\in \cX.
\end{equation}
then Theorem \ref{critMPO} implies that the  operator $\ccl_\uderup = \bmPhi^{\bm{+=}}=\bigcup_{k\ge 0}  \bmPhi^k \colon \sSet(\cX)\to \sSet(\cX)$  is an antagonistic  closure operator for $(X,\scrT,\scrTbf)$,
cf.\ Lem.\ \ref{udermlem}. If $(X,\scrT)$ is compact then \eqref{secondtop512} can be weakened to $0<t\le t_\xi$,  cf.\ Thm.\ \ref{localconstrGam}

\subsection{Morse tessellations}
\label{Morsetessll}

For a Morse pre-order $(\cX,\le^\dagger)$ the down-sets  yield a sublattice $\sO(\cX,\le^\dagger)$ of closed attracting blocks. 
Following the theory in Section \ref{bi-top}
 a Morse  pre-order $(\cX,\leq^\dagger)$ yields a finite discretization map  $\dyn\colon\cX \twoheadrightarrow
\sSC$ which is defined by combining the formulas in \eqref{defndyn} and \eqref{defnofSC}. 
The latter is dual to the embedding $\sO(\sSC) \cong\sO(\cX,\le^\dagger)\rightarrowtail \sSet(\cX)$.
%
 The composition 
\[
X\xtwoheadrightarrow{ \grd} \cX \xtwoheadrightarrow{\dyn} \sSC,
\]
defines a continuous $T_0$-discretization of $X$ which is denoted by $\tile\colon X\to \sSC$,
cf.\ \eqref{spacegrading} and which defines an $\sSC$-grading on $X$, cf.\ App.\ \ref{gradfilt}.\index{$\tile$} 
Compare the latter with the composition $X\xtwoheadrightarrow{\disc}\cX\xtwoheadrightarrow{\dim}\N$ which defines an $\N$-grading on $X$. 
The diagrams in \eqref{dia:intro} and \eqref{ordpreforgrd}  show how $\dyn$ is order-preserving and order-reversing with respect to $(\cX,\le)$ and $(\cX,\lebf  )$ respectively, cf.\ Sect.\ \ref{bi-top}.
Diagram \eqref{dia:intro} also shows the continuous $T_0$-discretization maps 
\[
\tile\colon (X,\scrT) \to (\sSC,\le) \quad \hbox{and}\quad \tile\colon (X,\scrTbf) \to (\sSC^*,\ge),
\]
 by factoring through $(\cX,\le)$
and  through $(\cX,\lebf  )$ respectively. 
A Morse pre-order can therefore be thought of as grading of $X$ which, as $T_0$-discretization, is a continuous map $\tile$ in the above sense.
%
%
A down-set $\cU$ in $(\sSC,\le)$ is a down-set in $(\cX,\le)$ and an up-set in $(\cX,\lebf  )$ and therefore closed  in $(X,\scrT)$ and and open in $(X,\scrTbf)$ respectively.
By Lemma \ref{charattbl12345} this implies that down-sets $\cU$ in $(\sSC,\le)$ realize to closed attracting neighborhoods $\tile^{-1} \cU\in \sABlockC(\varphi)$.
On the level of classes $\cS\in \sSC$ the inverse image of $\tile$ yields an $\sSC$-graded tessellation
$(\sT,\le)$ with 
\begin{equation}
\label{Morsetessll11}
\sT:= \bigl\{ T = \tile^{-1}\cS \mid \cS \in \sSC \bigr\},
\end{equation}
 such that $\big\downarrow T$ is $\scrT$-closed and
 $\scrTbf$-open, i.e. $\varphi(t,x)\in \Int \big\downarrow T$ for every $x\in T$ and for all tiles $T$. The latter follows since $\big\downarrow T = \big\downarrow\tile^{-1}\cS = \tile^{-1}\big\downarrow \cS$ and $\big\downarrow\cS$ is a down-set in $\sSC$,
 cf.\ \cite[Defn.\ 8]{lsa3}. 
 This motivates the  definition:
 \begin{definition}[cf.\ \cite{lsa3}, Cor.\ 4]
\label{morsetess45}
An ordered tessellation $(\sT,\le)$  of $X$, cf.\ Defn.\ \ref{ordtessdefn}, 
 is called a \emph{Morse tessellation} for $\phi$ if  for every $T\in \sT$
\begin{enumerate} 
\item[(i)] $\big\downarrow T$ is $\scrT$-closed;
\item[(ii)] $\big\downarrow T$ is $\scrTbf$-open, i.e. $\varphi\bigl(t,x\bigr) \in \Int \big\downarrow T$, for all $x\in T$ and for all $t>0$.
\end{enumerate}
The sets $T\in \sT$ are called \emph{Morse tiles}.\footnote{Equivalently, for every $I\in \sO(\sT,\le)$, $|I|:=\bigcup_{T\in I} T \in \sABlock(\varphi)$, i.e. $\sT=\sT(\sN)$ where $\sN$ is given by  $\sN=\{|I|\mid I\in\sO(\sT,\le)\}\subset \sABlock(\varphi)$.}
\end{definition}

 Conversely, Morse tessellations yield Morse pre-orders and associated space discretizations.
Indeed, for a Morse tessellation $(\sT,\le)$ we declare the tiles to be the cells in $\cX$ and the partial order is the Morse pre-order on $\cX$. By definition this defines a discretization for both $(X,\scrT)$ and $(X,\scrTbf)$. In the next subsection we discuss a more refined reconstruction based on regular closed sets.

\begin{remark}
\label{recon23}
One can obviously build larger sets $\cX$ by for example considering $\scrT$ or $\scrTbf$ closure of the tiles $T$. One can also define fine structure within the tiles. In the next section we explain a specific reconstruction in the case of regular closed tiles.
\end{remark}

Morse tessellations are a defining structure for Morse representations, cf.\ \cite{lsa3}. The considerations
in this subsection explain that Morse tessellations are equivalent to Morse pre-orders which will be the central structure for discussing connection matrix theory in Sect.\ \ref{sec:cm}. 

\subsection{Regular closed attracting blocks}
\label{regbldiscrresol}\index{Attracting block!regular closed}
This section 
 discusses a special property of closed attracting blocks, cf.\ \eqref{clattbl12}. 
\begin{theorem}
\label{clattbl}
Let $U\in \sABlockC(\varphi)$ be a closed attracting block, then $U$ is a regular closed attracting block, i.e. $\sABlockC(\varphi) = \sABlockR(\varphi)$, where the latter denotes the lattice of
regular closed attracting blocks. Moreover, 
\[
U\wedge U' = U\cap U',
\]
for all $U,U'\in \sABlockC(\varphi)$.
\end{theorem}
\begin{proof}
Let $U\subset X$ be a closed attracting block. By definition 
$\cl \Int U  \subset  U$. 
Suppose $U \smin \cl \Int U\neq\varnothing$.
Let $x\in  U \smin \cl \Int U$.
Choose $t_n>0$ with $t_n \to 0$ as $n \to \infty$. Since $U$ is an attracting block we have that $y_n := \varphi(t_n,x) \in \Int U$ for all $t_n>0$.
By the continuity $y_n \to x$ as $n\to \infty$ which implies  that $x\in \cl \Int U$, a contradiction. Therefore
$ U=\cl \Int U$, which proves that $U$ is a regular closed attracting block, cf.\ \cite{lsa2,lsa3}.

If $U,U'\in \sABlockC(\varphi)$, then $U\cap U'\in \sABlockC(\varphi)$ and thus $U\cap U'$ is a closed attracting block and therefore a regular closed attracting block. This implies that $U\cap U' = 
\cl \Int(U\cap U') = U\wedge U'$.
%
%
\end{proof}

By Theorem \ref{dubbleconst3} we may assume that a Morse pre-order $(\cX,\le^\dagger)$ 
is induced by a bi-topological CA-discretization $(\cX,\ccl,\uclbff  ,|\cdot|)$  for $(X,\scrT,\scrTbf)$.
Assume without loss of generality that $\sO(\cX,\le^\dagger) = \aclop(\cX)$ and 
$\ccl = \ccl^\dagger$ and $\uclbff   = \st^\dagger$.
Moreover, assume that 
$(\cX,\ccl,|\cdot|)$ is Boolean CA-discretization.
Let $\cU\in \aclop(\cX)$, then, since $|\cU| \in \sABlockR(\varphi)$, the set $\cU$ is regular closed and 
$|\cU\cap \cU'|=|\cU|\cap|\cU'| = |\cU|\wedge|\cU'| = |\cU\wedge\cU'|$,
which
yields the following regular closed analogue of  \eqref{AB12}:
%
%
\begin{equation}
\label{resolution1}
\begin{tikzcd}[column sep=large, row sep=large]
\scrR(X)                      & \sABlockR(\varphi) \arrow[l, "\supset"', tail]                       \\
\scrR(\cX) \arrow[u, "|\cdot|", tail] & \sO(\cX,\le^\dagger) \arrow[u, "|\cdot|"', tail] \arrow[l, "\supset"', tail]
\end{tikzcd}
\end{equation}
The down-sets for a Morse pre-order yield a sublattice of $\sABlockC(\varphi)$.
Conversely, for a finite sublattice  $\sN\subset \sABlockC(\varphi)$ we can construct a Boolean
CA-discretization, cf.\ Sect.\ \ref{regclMT12} and Rem.\ \ref{reverseconstr1}. 
Morse pre-orders 
provide an important reduction of discretization data as is explained in the next section.

\section{Condensed Morse pre-orders}
\label{dyngrad12}
In this section we assume that 
$(\cX,\ccl,|\cdot|)$ is Boolean CA-discretization of $(X,\scrT)$.
The fact that closed attracting blocks are regular closed sets, cf.\ Thm.\ \ref{clattbl}, is a crucial property for reducing the data structures in the theory of Morse pre-orders. Such a discretization will be referred to as a \emph{\discresol}.\index{Condensed Morse pre-order} 

\subsection{Pre-orders on top cells}\index{Top cell}\index{Cell!top}
The bottom embedding $\sO(\cX,\le^\dagger) \rightarrowtail \scrR(\cX)$ in  \eqref{resolution1} is inclusion
since $\cU\cap \cU' = \cU\wedge \cU'$, which follows from the fact that the evaluation map  $|\cdot|\colon\scrR(\cX)\rightarrowtail \scrR(X)$ is a homomorphism and $(\cX,\ccl,|\cdot|)$ is Boolean.
%
Since $\scrR(\cX) \cong \sSet(\cX^{\topc})$, cf.\  Prop.\ \ref{atomreg}, 
 we can dualize the   homomorphism 
\[
\sO(\sSC,\le) \xrightarrow{\cong}\sO(\cX,\le^\dagger) \rightarrowtail \scrR(\cX) \xrightarrow{\cong}\sSet(\cX^{\topc}),
\]
which yields  the surjection $\pi\colon\cX^{\topc} \twoheadrightarrow (\sSC,\le)$, where $\cX^\topc$ is unordered and where $\sSC \cong \sJ\bigl( \sO(\cX,\le^\dagger)\bigr)$. 
\begin{definition}
\label{discmodel22}
The 
 induced pre-order $\le^\topc$ on $\cX^{\topc}$, defined by 
\[
\xi \le^\topc\xi'  \quad \text{if and only if}
\quad \pi(\xi) \le \pi(\xi'),
\]
is called a \emph{\discresol}\index{Condensed Morse pre-order}\index{Morse pre-order!condensed} for $\le^\dagger$.\footnote{Since  $\sO(\cccX,\le^\dagger)\cong \sO(\sSC,\le)\cong \sO(\cccX^\topc,\le^\topc)$ it follows that a pre-order $\le^\topc$ is the restriction of $\le^\dagger$ to the top cells $\ccX^\topc$, cf.\ Proof of Thm.\ \ref{mainthmfordrs}. }
\end{definition}

By construction we have that $\sO(\cX,\le^\dagger) \cong \sO(\cX^\topc,\le^\topc)\cong \sO(\sSC,\le)$.
The associated closure operator on $\sSet(\cX^{\topc})$ is denoted by $\ccl^\topc$.
By Proposition \ref{atomreg}, $\cU^\topc\in \sO(\cX^\topc,\le^\topc)$ implies $\ccl~\cU^{\topc} \in \sO(\cX,\le^\dagger)$ and thus, since $(\cX,\ccl,|\cdot|)$ is Boolean,
$\cl~|\cU^\topc|=|\ccl~\cU^\topc| \in \sABlockR(\varphi)$. Consequently $\varphi\bigl(t,|\ccl~\cU^\topc|\bigr)\subset \Int |\ccl~\cU^\topc|$ for all $t>0$.
Let $\cU^{\topc} = \ccl^\topc\xi$, with $\xi\in \cX^{\topc}$. Then,
$\varphi\bigl(t,\lbr\xi\rbr  \bigr) \subset \varphi\bigl(t,|\ccl~\ccl^\topc\xi|\bigr)\subset \Int|\ccl~\ccl^\topc\xi| = \Int \cl|\ccl^\topc\xi|
=\Int\lbr\ccl^\topc\xi\rbr$ for all $t>0$.
The latter is a condition on only the top cells. We show below that any pre-order $(\cX^\topc,\le^\topc)$ satisfying the latter is
a \discresol induced by a Morse pre-order.
\begin{theorem}
\label{mainthmfordrs}
A pre-order $(\cX^\topc,\le^\topc)$ is a \discresol for $\le^\dagger$
 if and only if 
 \begin{equation}
    \label{firstchar1}
\varphi\bigl(t,\lbr\xi\rbr\bigr) \subset 
 \Int \bigl\lbr \ccl^\topc\xi\bigr\rbr,\quad \forall t>0.
\end{equation}
\end{theorem}

Condition \eqref{firstchar1} is a characterization of \discresols and can be used as alternative definition of \discresol.

\begin{remark}
\label{discmodel2a}
If $(X,\scrT)$ is a compact topological space  then
the criterion in \eqref{firstchar1} is equivalent  to the condition: for every $\xi\in \cX^{\topc}$ there exists a $t_\xi>0$, such that
$\varphi\bigl(t,\lbr\xi\rbr\bigr) \subset \Int \bigl\lbr\ccl^\topc\xi 
\bigr\rbr$ for all 
$0< t \le t_\xi$.
\end{remark}

\begin{remark}

For a binary relation $\bmphi\subset \cX^\topc\times\cX^\topc$ the transitive, reflexive closure\index{Transitive reflexive closure}  $\bmphi^{\bm{+=}}$   defines a pre-order on $\cX^\topc$.
If   \eqref{firstchar1} is satisfied with $\ccl^\topc_\bmphi = (\bmphi^{\bm{+=}})^{-1}$, then $\bmphi$
will also be referred to as a \discresol for $\le^\dagger$.
The definition of \discresol is reminiscent of the notion of weak outer approximation in relation to the commutative diagram in \eqref{resolution1}, cf.\ \cite{kkv}, \cite[Defn.\ 3.7]{lsa2}, \cite{lsa3}.
\end{remark}

The Boolean algebra $\sSet(\cX^\topc)$ is a sublattice in $\sSet(\cX)$. The embedding does not preserve the top element and the inclusion is not therefore Boolean. 
The closure operator 
$
\ccl\colon \sSet(\cX^\topc) \to \sSet(\cX)
$, given by $\cU \mapsto \ccl~ \cU \in \sO(\cX,\le)$, 
is additive but not a lattice homomorphism in general. 
Since $i\colon\cX^\topc\hookrightarrow (\cX,\le)$, with $\cX^\topc$ unordered, is an  order-embedding\footnote{The top cells $\ccX^\top$ form an anti-chain in $\bigl(\ccX,\le\bigr)$, cf.\ Sect.\ \ref{regclsets}.} Birkhoff duality yields the surjective lattice homomorphism $j\colon \sO\bigl(\cX,\le\bigr) \twoheadrightarrow \sSet(\cX^\topc)$ given by $\cU \mapsto \cU^\topc:= \cU\cap \cX^\topc$. 
Schematically we pose the following lifting diagram:
\begin{equation}
    \label{lift1}
\begin{tikzcd}[column sep=large, row sep=large]
                                               & \sO(\cX,\le) \arrow[d, "j", two heads] \\
\sO(\cX^\topc,\le^\topc) \arrow[ru, "\ccl", dotted] \arrow[r, "\id", tail] & \sSet(\cX^\topc)                          
\end{tikzcd}
\end{equation}
Theorem \ref{thm:lattice:hom} below shows that the identity map can  be lifted as  closure.

\begin{theorem}\label{thm:lattice:hom}
Suppose $(\cX^\topc,\le^\topc)$ is a pre-order that satisfies \eqref{firstchar1}.
\label{dyngdres}
Then, the restriction $\ccl\colon \sO(\cX^\topc,\le^\topc)\rightarrowtail  \sO(\cX,\le)$ is an injective lattice homomorphism with $j\circ \ccl = \id$.
\end{theorem}


\begin{proof}
By definition $\cU^{\topc} = \bigcup_{\xi\in \cU^{\topc}} \{\xi\}$ and since $\lbr\cdot\rbr\colon \sSet(\cX^{\topc}) \to \scrR(X)$
 is an injective Boolean homomorphism, cf.\ Prop.\ \ref{realization}, we have that 
 $\lbr\cU^{\topc}\rbr = \bigcup_{\xi\in \cU^{\topc}} \lbr\xi\rbr$. In combination with \eqref{firstchar1}, the additivity of $\ccl^\topc$ and the fact that $\cU^{\topc}\in\sO(\cX^\topc,\le^\topc)$ we conclude
 \[
 \begin{aligned}
 \varphi\bigl(t,\lbr\cU^{\topc}\rbr\bigr) &= \bigcup_{\xi\in \cU^{\topc}} \varphi\bigl(t,\lbr\xi\rbr\bigr)
 \subset \bigcup_{\xi\in \cU^{\topc}}\Bigl( \Int\bigl\lbr \ccl^\topc\xi \bigr\rbr\Bigr)
 \subset \Int\Bigl(\bigcup_{\xi\in \cU^{\topc}} \bigl\lbr \ccl^\topc\xi \bigr\rbr\Bigr)\\
 &=\Int \Bigl\lbr\bigcup_{\xi\in \cU^{\topc}} \ccl^\topc\xi \Bigr\rbr =\Int \bigl\lbr\ccl^\topc \cU^{\topc}\bigr\rbr\subset \Int \lbr \cU^{\topc}\rbr,\quad \forall t>0,
 \end{aligned}
 \]
 which proves that $\lbr\cU^{\topc}\rbr$ is an attracting block for $\phi$. Since $\sO(\cX^\topc,\le^\topc)$ and $\sABlockR(\varphi)$ are sublattices of $\sSet(\cX^{\topc})$ and $\scrR(X)$ respectively, and since 
 $\lbr\cdot\rbr\colon \sSet(\cX^{\topc}) \rightarrowtail \scrR(X)$ is an injective Boolean homomorphism  the evaluation map $\lbr\cdot\rbr\colon \sO(\cX^\topc,\le^\topc) \rightarrowtail \sABlockR(\varphi)$ is a injective lattice homomorphism. In particular we conclude that $|\ccl~\cU^\topc|
 \in \sABlockR(\varphi)$ and $\ccl~\cU^\topc \in \sO(\cX,\le)$.

To show that the restriction of $\ccl$ is a homomorphism it remains to check that the unit and intersection are preserved.
By definition $\ccl~ \cX^\topc = \cX$ which proves that the unit is preserved. 
By Proposition \ref{atomreg} we have that $\ccl~\cU^\topc\in  
\scrR(\cX)$ and by 
Theorem \ref{clattbl}, $|\ccl~\cU^\topc|\cap |\ccl~\cU'^\topc| = 
|\ccl~\cU^\topc|\wedge |\ccl~\cU'^\topc|$.
This implies,
\[
\begin{aligned}
|\ccl(\cU^\topc\cap \cU'^\topc)| &= \cl|\cU^\topc\cap \cU'^\topc| = \lbr \cU^\topc\cap \cU'^\topc\rbr = \lbr\cU^\topc\rbr \wedge \lbr \cU'^\topc\rbr = \lbr\cU^\topc\rbr \cap \lbr \cU'^\topc\rbr\\
&= |\ccl~ \cU^\topc|\cap |\ccl~ \cU'^\topc| = |\ccl~\cU^\topc \cap \ccl~ \cU'^\topc|,
\end{aligned}
\]
where we use Proposition \ref{realization} to conclude that 
$\lbr \cU^\topc\cap \cU'^\topc\rbr = \lbr\cU^\topc\rbr \wedge \lbr \cU'^\topc\rbr$.
The fact that $|\cdot|$ is injective yields   $\ccl(\cU^\topc\cap \cU'^\topc) = \ccl~\cU^\topc\cap \ccl~\cU'^\topc$,  which completes the proof.
\end{proof}

\begin{proof}[Proof of Thm.\ \ref{mainthmfordrs}]
The direction that a \discresol satisfies \eqref{firstchar1} is given above. It remains to show that \eqref{firstchar1} yields a Morse pre-order.
Suppose \eqref{firstchar1} is satisfied. Then, by Theorem \ref{thm:lattice:hom}, $\ccl\colon \sO(\cX^\topc,\le^\topc) \to\sO(\cX,\le)$ provided an embedding sublattice.
From \eqref{dubbleconst2}-\eqref{idred} we obtain a pre-order $(\cX,\le^\dagger)$ such that the range of the above closure is
the lattice $\sO(\cX,\le^\dagger)$:
\[
\begin{tikzcd}
\sO(\cX,\le^\dagger) \arrow[rd, "\subset"', tail] &   & \sO(\cX^{\topc},\le^\topc) \arrow[ll, "\ccl"',"\cong"] \arrow[ld, "\ccl", tail] \\
                         & \sO(\cX,\le) &                                         
\end{tikzcd}
\]
The pre-order $(\cX,\le^\dagger)$ is the desired Morse pre-order that induces $(\cX^\topc,\le^\topc)$.
If we choose $\lebf   \,= \,\ge^\dagger$ we obtain $\le^\dagger$ as antagonistic pre-order for $(\cX,\ccl,\uclbf  ,|\cdot|)$,
cf.\ Thm.\ \ref{dubbleconst3}.
\end{proof}
The novelty of the above construction is that 
the (injective) composition 
\[
\begin{tikzcd}
\sO(\sSC,\le) \arrow[r, "\cong"] & \sO(\cX^\topc,\le^\topc) \arrow[r, "\ccl", tail] & \sO(\cX,\le)  \arrow[r, "\subset", tail] & \sSet(\cX)
\end{tikzcd}
\]
 dualizes to 
the finite discretization 
\begin{equation}
    \label{dyndefn12}
\dyn\colon (\cX,\le) \xtwoheadrightarrow{} (\sSC,\le), 
\end{equation}
which is defined in Section \ref{Morsetessll} and is given by the formulas in \eqref{defndyn} and \eqref{defnofSC}.
The finite discretization $\dyn$ recovers the Morse pre-order $\le^\dagger$ via $\xi\le^\dagger\xi'$ 
if and only if $\dyn\,\xi\le\dyn\,\xi'$,
and $\le^\topc$ is the restriction of  $\le^\dagger$ to $\cX^{\topc}$.
The advantage of using regular closed  sets is that the Morse pre-order is completely determined by the restriction $(\cX^\topc,\le^\topc)$ which is a much smaller data structure in general and bypasses the  topologies given by
$\le$ and $\lebf  $.
The following result gives a formula for determining $\dyn$ in terms of $\le^\topc$:
\begin{theorem}
\label{thethmdyn}
Suppose $(\cX^\topc,\le^\topc)$ is a pre-order that satisfies \eqref{firstchar1}.
The finite discretization $\dyn\colon (\cX,\le) \xtwoheadrightarrow{} (\sSC,\le)$ is given by
\begin{equation}
\label{thethmdynform}
\xi \mapsto \dyn(\xi)  = \min_{\sSC} \Bigl\{ \bigl[\eta^\topc\bigr]\mid \eta^\topc\in \st \xi \cap \cX^\topc \Bigr\},
\end{equation}
where $\bigl[\eta^\topc\bigr] \in \sSC$ is the partial equivalence class in $(\cX^\topc,\le^\topc)$ containing $\eta^\topc$. 
\end{theorem}

\begin{proof}
Consider the commutative diagram
\begin{equation}\label{dyndiapr12}
\begin{tikzcd}[column sep=large, row sep=large]
\cX \arrow[r, "\dyn", two heads] \arrow[d, "\iota_\cccX"'] & \sSC \arrow[d, "\iota_\sSC"] \\
\sJ(\sO(\cX,\le)) \arrow[r, "\sJ(\ccl)", two heads]                 & \sJ(\sO(\sSC))               
\end{tikzcd}
\end{equation}
 the maps $\iota_\cccX$ and $\iota_\sSC$ are given by $\xi \xmapsto{\iota_\cccX} \big\downarrow \xi$ and
$[\xi^\topc] \xmapsto{\iota_\sSC} \big\downarrow [\xi^\topc]$, and $\ccl\colon \sO(\sSC) 
\rightarrowtail\sO(\cX,\le)$.
By the commutativity we have that $\dyn = \iota_\sSC^{-1}\circ \sJ(\ccl)\circ \iota_\cccX$, and $\sJ(\ccl)(\cU) = \min \ccl^{-1}\bigl(\big\uparrow \cU\bigr)\in \sJ(\sO(\sSC))$, $\cU\in \sJ(\sO(\cX,\le))$, cf.\ Thm. \ref{thm:birkhoff}.
Recall that $\ccl^{-1}(\cU) =\{\cU^\topc\in \sO(\sSC)\mid \ccl~\cU^\topc =\cU\}$.
Note that in $\sO(\cX,\le)$ the up-set $\big\uparrow \iota_\cccX(\xi)$  is the set of closed subsets in $\cX$ that contain $\iota_\cccX(\xi)=\big\downarrow \xi$. 
By definition $\ccl^{-1}\bigl(\big\uparrow\iota_\cccX(\xi) \bigr)$ are all $\cU^\topc\in \sO(\sSC)$ such that $\ccl~\cU^\topc =\cU$ for some $\cU\in \sO(\cX,\le)$ with $\big\downarrow \xi\subset \cU$. The latter is equivalent to $\xi\in \cU$. Since $\ccl~\cU^\topc\in \sO(\sSC)$ this implies that $\ccl^{-1}\bigl(\big\uparrow\iota_\cccX(\xi) \bigr) = \bigl\{ \cU^\topc\in \sO(\sSC)\mid \xi\in \ccl~\cU^\topc\bigr\}$. Since join-irreducible elements generate all elements in a finite distributive lattice we have that
$\ccl^{-1}\bigl(\big\uparrow\iota_\cccX(\xi) \bigr) = \bigl\{ \cU^\topc\in \sJ(\sO(\sSC))\mid \xi\in \ccl~\cU^\topc\bigr\}$ and thus
\[
\sJ(\ccl)(\iota_\cccX(\xi)) = \min\bigl\{ \cU^\topc\in \sJ(\sO(\sSC))\mid \xi\in \ccl~\cU^\topc\bigr\},
\]
which is attained by a unique $\widehat \cU^\topc\in \sJ(\sO(\sSC))$.
Recall that $\cU^\topc\in \sJ(\sO(\sSC))$ if and only if $\cU^\topc=\big\downarrow\cS$, $\cS = [\xi^\topc]$ for some $\xi^\topc\in \cX^\topc$.
 Therefore, $\xi\in \ccl~\cU^\topc$ if and only if $\xi \in \ccl~\eta^\topc$ for some $\eta^\topc\in \big\downarrow [\xi^\topc]$.
By duality $\xi \in \ccl~\eta^\topc$ if and only if $\eta^\topc \in \st\xi\cap \cX^\topc$.
Consider the set
\[
\bigl\{[\eta^\topc]\in \sSC\mid \eta^\topc \in \st\xi\cap \cX^\topc \bigr\}.
\]
Let $[\eta^\topc]$ and $[\tilde \eta^\topc]$ be minimal. Then, $\big\downarrow [\eta^\topc]\subset \widehat\cU^\topc$ and $\big\downarrow [\tilde\eta^\topc]\subset \widehat\cU^\topc$,
which implies that $[\eta^\topc]=[\tilde \eta^\topc]$, and $\big\downarrow[\eta^\topc] = \widehat\cU^\topc$.
\end{proof}

\begin{remark}
For \discresols the commutative diagram in \eqref{resolution1} is replaced by
\begin{equation}
\label{resolution123a}
\begin{tikzcd}[column sep=large, row sep=large]
\scrR(X)                      & \sABlockR(\varphi) \arrow[l, "\supset"', tail]                       \\
\sSet(\cX^\topc) \arrow[u, "\Vert\cdot\Vert", tail] & \sO(\cX^\topc,\le^\topc) \arrow[u, "\Vert\cdot\Vert"', tail] \arrow[l, "\supset"', tail]
\end{tikzcd}
\end{equation}
Condensed Morse pre-orders give rise to regular closed Morse tessellations. 
\end{remark}

\subsection{Regular closed tessellations}
\label{regclMT12}
A \discresol yields a Morse pre-order. 
The range  of the injective lattice homomorphism
$\ccl\colon\sO(\cX^\topc,\le^\topc)\to \sO(\cX,\le)$ 
can be expressed as $\sO(\cX,\le^\dagger)$ via a pre-order $(\cX,\le^\dagger)$ --- a Morse pre-order.
%
%
%
For $\cU^\topc \in \sO(\cX^\topc,\le^\topc)$  we can give a representation of $\sSC$ in terms of regular closed tiles.
By construction $\lbr\cU^\topc\rbr\in \sABlockR(\varphi)$ and we denote the associated sublattice of regular closed attracting blocks by $\sN\subset \sABlockR(\varphi)$. Then, $\sSC \cong \sJ(\sN) \cong \sT$, where $T =  \lbr\cU^\topc\rbr - \lbr{\cU^\topc}^\pred\rbr :=\lbr\cU^\topc\rbr \wedge \lbr{\cU^\topc}^\pred\rbr^\#\in \scrR(X)$, cf.\ \cite{lsa3}.
From \cite[Lem.\ 23]{lsa3} we have that 
\[
\begin{aligned}
\lbr\cU^\topc\rbr - \lbr{\cU^\topc}^\pred\rbr  &= \cl~\bigl( \lbr\cU^\topc\rbr \smin \lbr{\cU^\topc}^\pred\rbr\bigr)  = \cl  \Bigl(|\ccl~\cU^\topc|\smin|\ccl~{\cU^\topc}^\pred |\Bigr)\\ 
 &= \cl  \Bigl(\bigl|\ccl~\cU^\topc \smin \ccl~{\cU^\topc}^\pred \bigr|\Bigr) =  \Bigl| \ccl  \bigl(\ccl~\cU^\topc \smin \ccl~{\cU^\topc}^\pred \bigr) \Bigr|\\
 &= \bigl| \ccl~\cU^\topc -\ccl~{\cU^\topc}^\pred\bigr|, 
 \end{aligned}
\]
which shows that the regular closed tiles are closure of the Morse tiles obtained in \eqref{defnofSC} and \eqref{Morsetessll11}.
The poset $(\sT,\le)$, which is isomorphic to $(\sSC,\le)$, is an example of a \emph{regular closed Morse tessellation}.\index{Regular closed Morse tessellation}
\index{Morse tessellation!regular closed}
The definition of a regular closed Morse tessellation is similar to Definition \ref{morsetess45}: the tiles are regular closed and Condition (i) is redundant.
Given a regular closed Morse tessellation we can reconstruct a Morse pre-order.
If we start with a regular closed Morse tessellation $(\sT,\le)$, then the Morse tiles  $T\in \sT$ generate 
a subalgebra of regular closed sets $\scrR_0(X)$ which in turn generates 
a finite subalgebra of $\sSet(X)$ represented by $\sSet(\cX)$ for some finite set $\cX$, cf.\ Prop.\ \ref{realization22}.
The elements in $\cX$ are again denoted by $\xi$
and their \emph{realization} in $X$ is denoted by $|\xi|$.
This way we obtain  a discrete space $\cX$ with two pre-orders:
(i) the {face pre-order} $\le$ defined by $\xi\le \xi'$ if and only if $|\xi|\subseteq \cl|\xi'|$, for which
closed sets in $(\cX,\le)$ correspond to closed sets in the topological space $(X,\scrT)$ and (ii) 
the {Morse pre-order} $\le^\dagger$, derived from the Morse tessellation, defined by
$\xi\le^\dagger\xi'$ if and only if $\varphi\bigl(t,|\xi|\bigr) \in \Int \big\downarrow T$ for all $t>0$ for some $T\supset |\xi'|$. Closed sets in $(\cX,\le^\dagger)$
 correspond to regular closed attracting blocks for the semi-flow $\varphi$, and the partial equivalence classes of $\le^\dagger$  retrieve the Morse tessellation partial order.
 Summarizing, a regular closed Morse tessellation gives rise to a bi-topological CA-discretization $(\cX,\le,\lebf  ,|\cdot|)$ where
$\lebf  =\ge^\dagger$.

\begin{remark}
\label{reverseconstr1}
If we choose an arbitrary finite sublattice $\sN\subset \sABlockR(\varphi)$,
then $\sN$ is a sublattice of some finite subalgebra $\scrR_0(X)\subset\scrR(X)$.
This induces a Boolean CA-discretization by Proposition \ref{realization22}.
\end{remark}

\section{Beyond semi-flows}
\label{beyondsemifl}
In this chapter the focus of applying bi-topological techniques is restricted to semi-flows. However, most of the ideas and methods apply to a much wider class of dynamical systems. In this section we outline some of these extensions and how these fits into the theory of this chapter.

A \emph{relational semi-flow}\index{Relational semi-flow}\index{Semi-flow!relational} $\phi = \{\phi^t\}_{t\in \T^+}$ is a family of binary relations $\phi^t\subset X\times X$ on a point set $X$ parametrized by (time) $t\in \T^+$ such that
\begin{enumerate}
\item[(i)] $\phi^0= \id$ on  $X$;
\item[(ii)] $\phi^s\circ\phi^t = \phi^{s+t}$ for all\footnote{For composition of relations and other properties cf.\ App. \ref{binrel}.} $s,t\in \T^+$.
\end{enumerate}
The time space $\T^+$ is either $\Z^+$ or $\R^+$. 
For negative time we define $\phi^{-t}$ to be the \emph{opposite relation},\index{Opposite relation} i.e.
$\phi^{-t} = \bigl\{(x,y)\in X\times X\mid (y,x)\in \phi^t\bigr\}$, cf.\ App.\ \ref{binrel}.
Therefore Axiom (ii) is equivalent to
\begin{enumerate}
\item[(ii)'] $\phi^s\circ\phi^t = \phi^{s+t}$ for all $s\cdot t \ge 0$ with $s,t\in \T$,
\end{enumerate}
where $\T$ is either $\Z$ or $\R$.
In Appendix \ref{binrel} we discuss additional properties of binary relations.
If  $\varphi(t,x):= \phi^t(x)$ defines a continuous map $\T^+\times X\to X$ then, if $\T^+=\R^+$, $\varphi$ is called a \emph{continuous semi-flow}\index{Semi-flow}\index{Semi-flow!continuous}\index{Continuous semi-flow} on $X$ which is the main point of focus in this text.
%
If $\T^+=\Z^+$ , $\phi$ is a called an \emph{iterated continuous map}.\index{Iterated continuous map}
In this case it suffices to only consider the map $f:=\phi^1$ since higher iterates are found via composition.
Backward images  define $\phi^t$ for negative times. 
Most considerations in this chapter are valid for relational semi-flows with $\T$ either discrete or continuous time. In particular the techniques carry over to iterated maps. We will indicate in which situations continuity will be required.

On the complete and atomic Boolean algebra $\sSet(X)$ there is a natural duality between binary relations and completely additive modal operators, cf.\ App.\ \ref{binrel}.
Let $\Phi$ be a modal operator on $\sSet(X)$. Recall from Sect.\ \ref{modalstuff} that specialization relation\index{Specialization relation} is given by $(x,x')\in \phi$ if and only if $x\in \Phi\{x'\}$ and the operator $\Phi=\phi^{-1}$ is defined via \eqref{modalcorres12}.
We apply this principle to a relational semi-flow by setting
\[
(x,x')\in \phi^{-t}\quad
\text{~~if and only if~~} \quad x\in \phi^t\{x'\},\quad t\in \T.
\]
If we coarsen the relations by discarding time we obtain the relation
\[
(x,x')\in \phi^{-t},\text{~~for some}\quad t>0\quad\text{~~if and only if~~}\quad x\in \der^+ \{x'\}:= \bigcup_{t>0}\phi^t\{x'\},
\]
which is remeniscent of the operator $\der^+$ defined in Section \ref{basicflowtop}. A similar definition can be made for $\der^-$. Via $\der^+$ and $\der^-$ one can define associated Alexandrov topologies $\scrT^+$ and $\scrT^-$ respectively.
It makes sense to define  finer topologies via appropriately defined modal operators. For $\tau\in \T^+$ define the topology $\scrT_\tau$ by declaring the sets $U\subset X$ such that $\phi^\tau U\subset U$ to be closed.
In particular we have that 
\[
\scrT^+ \subset \scrT_\tau^+ \subset \scrT_\tau,\quad \tau>0,
\]
and the same for $\tau<0$. All these topologies are Alexandrov. The specialization relation for $\scrT_\tau$  is given by
\[
(x,x')\in \phi^{-\tau} \quad\text{if and only if}\quad y \in \phi^\tau\{x\}.
\]
The flow topologies discussed in Section \ref{derivesflowtop} are not Alexandrov in general and the associated duality is more involved.
For example the \bflt $\scrTbf$ for a relational semi-flow is defined by considering the modal operator $\der^-$ as defined above in the setting of relational semi-flows.
Define the modal operator $\Phi_\sqbullet^- := \der^-\cl$ for the topology $\scrTbf$. Another interesting modal operator to consider is defined as: $\Phi_\sqbullet^\tau = \phi^\tau\cl$.

Finally, even though Theorem \ref{clattbl} does not hold in general for relational semi-flows one can also study regular closed attracting blocks for relational semi-flows, cf.\ \cite{lsa3}.

%% file: connectionMatrix-parallel.tex
\chapter{Algebraization of dynamics}
\label{sec:cm}

 
 In this section, we elaborate on the third theme of this text: augmentation of Morse pre-orders with algebraic topological data  in order to characterize invariance of the dynamics, i.e. the algebraization of dynamics. 
 In particular, we use techniques from algebraic topology in a way in that enables a computational theory.
 The starting point is a 
 discretization. 
 A Morse pre-order $(\cX,\le^\dagger)$
 is the choice of a  pre-order such that the discretization map $\disc\colon (X,\scrT,\scrTbf) \twoheadrightarrow (\cX,\le^\dagger)$  is both $\scrT$-consistent and $\scrTbf$-co-consistent.
In particular,
the composed maps
\[
\adjustbox{scale=0.95}{%
\begin{tikzcd}
(X,\scrT) \arrow[r, two heads, "\disc"] \arrow[rr, two heads, "\tile", bend right] & (\cX,\le^\dagger) \arrow[r] \arrow[r] \arrow[r, two heads, "\dyn"] & (\sSC,\le),
\end{tikzcd}
\begin{tikzcd}
(X,\scrTbf) \arrow[r, two heads, "\disc"] \arrow[rr, two heads, "\tile", bend right] & (\cX,\ge^\dagger) \arrow[r] \arrow[r] \arrow[r, two heads, "\dyn"] & (\sSC^*,\ge)
\end{tikzcd}
}
\]
are continuous $T_0$-discretizations, denoted by $\tile$, which define $\sSC$-gradings on $X$, cf.\ App.\ \ref{gradfilt}.
We explain how factorized gradings can be used to discretize algebraic topological invariants of topological spaces. We apply these methods in the context of space and flow topologies.   Recall that the first theme was linking topology and dynamics by formulating dynamics as a topology. It is worthwhile to then ask of the reverse direction: what happens when topology is analyzed as dynamics? The beginning of this chapter explores this direction, leading to a construction we entitle \emph{tessellar homology}, which, in contradistinction to cellular homology, uses general tiles instead of CW-cells.

\section{Cartan-Eilenberg systems}
\label{CEsystems}
The notion of a Cartan-Eilenberg system over a (countable) total order was first introduced in \cite{CE}.This notion is generalized to arbitrary total orders in \cite{HelleRognes} and to arbtrary posets in \cite{Matschke} and \cite{SpV2}.
Here we use this concept for finite distributive lattices, cf.\  \cite{SpV2}.
To some extend Cartan-Eilenberg systems may be regarded as a type of generalized homology theory. These systems provide the right data structure for considering algebraic topological invariance in combination with filterings and discretizations of a space.

\subsection{Cartan-Eilenberg systems over finite distributive lattices}
\label{CEsystemsFDLat}
Birkhoff's representation theorem, cf.\ Thm.\ \ref{thm:birkhoff},
yields that every finite distributive lattice can be represented as the down-set lattice $\sO(\sP)$ for some finite poset $(\sP,\le)$.
Regard $\sO(\sP)$
 as small (thin) category\index{Category}\index{Poset!as category} where the objects are the elements in the lattice and the order relations $\alpha\le \beta$ (i.e. $\alpha\subset \beta$) account for the morphisms, or arrows, i.e. $\alpha\le \beta$ yields the arrow $\alpha \to \beta$. 
The \emph{arrow category}\index{Arrow category} of $\sO(\sP)$ consists of pairs $(\alpha,\beta)$, with $\alpha\le \beta$, 
and unique morphisms $(\alpha,\beta) \to (\gamma,\delta)$ for $\alpha\le \gamma$ and 
$\beta\le \delta$,
and is denoted by $\IIi$ and corresponds to commutative diagrams in $\sO(\sP)$. 
Following \cite{HelleRognes} we consider the covariant functors $\spi_0$, $\spi_1$ and $\spi_2$ 
given by $(\alpha,\beta,\gamma) \xmapsto{\spi_0} (\alpha,\beta)$, $(\alpha,\beta,\gamma) \xmapsto{\spi_1} (\alpha,\gamma)$ and
$(\alpha,\beta,\gamma) \xmapsto{\spi_2} (\beta,\gamma)$ respectively, and  natural transformations 
$\imath\colon \spi_0\Rightarrow\spi_1$ and
$\jmath\colon \spi_1\Rightarrow\spi_2$ 
whose components 
are given by
$(\alpha,\beta) \xmapsto{\imath} (\alpha,\gamma)$ and
$(\alpha,\gamma) \xmapsto{\jmath} (\beta,\gamma)$ respectively.
\begin{definition}[cf.\ \cite{SpV2}]
    \label{CEsys}
    Let $(\sP,\le)$ be a finite poset.
A \emph{Cartan-Eilenberg system}\index{Cartan-Eilenberg system} over  $\sO(\sP)$ consists of a covariant functor $\sE\colon \IIi\to \sRmod$\footnote{The category of $R$-modules is denoted by $\sRmod$. Cartain-Eilenberg systems can be formulated in any abelian category such as abelian groups, $R$-modules or $\K$-vector spaces.}   and a natural transformation $k\colon \sE \spi_2 \Rightarrow \sE\spi_0$ between the composite functors $\sE\spi_2$ and $\sE\spi_0$,
called the \emph{connecting homomorphism},\index{Connecting homomorphism}  
such that 
\[
\begin{tikzcd}[column sep=small]
\sE \spi_0 \arrow[rr, "\sE\imath", Rightarrow] &                               & \sE\spi_1\arrow[ld, "\sE\jmath", Rightarrow] \\
                              & \sE\spi_2 \arrow[lu, "k", Rightarrow] &                              
\end{tikzcd}
\]
is an exact triangle, where the natural transformations $\sE\imath$ and $\sE\jmath$ are the right whiskerings of $\sE$ and $\imath$, and $\sE$ and $\jmath$ respectively.
A Cartan-Eilenberg system over $\sO(\sP)$ is denoted by $\bfE = \bigl(\IIi,\sE,k\bigr)$.\footnote{Cartan-Eilenberg systems can be defined over any poset, eg.\ all subsets of a topological space, closed subsets, etc., cf.\ \cite{SpV2}.}
\end{definition}

Unpacking the above definition yields
\[
\sE\colon \IIi \longrightarrow \sRmod, \quad (\alpha,\beta) \mapsto \sE(\alpha,\beta)=E^\beta_\alpha \in \sRmod.
\]
The functor $\sE$ yields the homomorphisms    $\ell\colon E^\beta_\alpha \to E^\delta_\gamma$ for all $(\alpha,\beta)\le (\gamma,\delta)$, and
the composition $E^\beta_\alpha \xrightarrow[]{\ell} E^\delta_\gamma\xrightarrow[]{\ell} E^\zeta_\epsilon$ is given by $\ell\colon E^\beta_\alpha \to E^\zeta_\epsilon$ by the transitivity in $\sO(\sP)$. The natural transformation $k$ yields the differential $k\colon E^\gamma_\beta\to E^\beta_\alpha$ such that the diagrams\footnote{In the special cases $(\alpha,\beta)\le (\alpha,\gamma)$ and $(\alpha,\gamma)\le (\beta,\gamma)$ the morphisms $\ell$ are denoted by $i$ and $j$ respectively.
} 
\begin{equation}
\label{unpackeddiag12}
\begin{tikzcd}[column sep=small]
E^\beta_\alpha \arrow[rr, "i"] &                   & E^\gamma_\alpha \arrow[ld, "j"] \\
                  & E^\gamma_\beta \arrow[lu, "k"] &                  
\end{tikzcd}
\qquad\qquad
\begin{tikzcd}
E^\gamma_\beta \arrow[r, "k"] \arrow[d, "\ell"'] & E^\beta_\alpha \arrow[d, "\ell"] \\
E^\zeta_\epsilon \arrow[r, "k"]                 & E^\epsilon_\delta               
\end{tikzcd}
\end{equation}
are exact and commutative  for all $(\alpha,\beta,\gamma)\le (\delta,\epsilon,\zeta)$.
By construction $\ell\colon E^\beta_\alpha \xrightarrow[\id]{\ell} E^\beta_\alpha$ is the identity homomorphism and the exactness of \eqref{unpackeddiag12}[left] shows that $E^\alpha_\alpha =0$ for all $\alpha \in \sO(\sP)$, cf.\ \cite{HelleRognes}.
    Morphisms between Cartan-Eilenberg  systems  $\bfE$ and $\bfE'$ 
    can be described in terms of the $E$-terms, i.e. a morphism\index{Cartan-Eilenberg system!morphism between} is a natural transformation $h\colon \sE \Rightarrow \sE'$ which, in terms of $E$-terms, implies that
there exist homomorphisms  $h^{\beta}_{\alpha}\colon E^\beta_\alpha \to E'^\beta_\alpha$ which commute with the morphisms in $\bfE$ and $\bfE'$ respectively:
\begin{equation}
    \label{homofCE}
    \begin{tikzcd}[column sep=large, row sep=large]
{} \arrow[r] & E^\beta_\alpha \arrow[r, "i"] \arrow[d, "h^\beta_\alpha"'] & E^\gamma_\alpha \arrow[r, "j"] \arrow[d, "h^\gamma_\alpha"'] & E^\gamma_\beta \arrow[r, "k"] \arrow[d, "h^\gamma_\beta"'] & E^\beta_\alpha \arrow[r] \arrow[d, "h^\beta_\alpha"'] & {} \\
{} \arrow[r] & E^{'\beta}_\alpha \arrow[r, "i'"]                 & E^{'\gamma}_\alpha \arrow[r, "j'"]                 & E^{'\gamma}_\beta \arrow[r, "k'"]                 & E^{'\beta}_\alpha \arrow[r]                 & {}
\end{tikzcd}
\end{equation}
for every ordered triple$(\alpha,\beta,\gamma)$.

Since $\sO(\sP)$ is a finite lattice it suffices to define
a Cartan-Eilenberg system with exact triangle\index{Exact triangle} and commutative squares\index{Commutative square} for ordered pairs called an \emph{exact couple system}, cf.\ \cite{Matschke}.\index{Exact couple system} To be more specific we consider the diagrams:
\begin{equation}
\label{unpackeddiag13}
\begin{tikzcd}[column sep=small]
E^\alpha_\varnothing \arrow[rr, "i"] &                   & E^\beta_\varnothing \arrow[ld, "j"] \\
                  & E^\beta_\alpha \arrow[lu, "k"] &                  
\end{tikzcd}
\qquad\qquad
\begin{tikzcd}
E^\beta_\alpha \arrow[r, "k"] \arrow[d, "\ell"'] & E^\alpha_\varnothing \arrow[d, "i"] \\
E^\delta_\gamma \arrow[r, "k"]                 & E^\gamma_\varnothing        
\end{tikzcd}
\end{equation}
which are exact and commutative  for all $(\alpha,\beta)\le (\gamma,\delta)$.
For an ordered triple  $(\alpha,\beta,\gamma)$ 
the composition $E^\gamma_\beta \xrightarrow[]{k} E^{\beta}_\varnothing \xrightarrow[]{j} E^\beta_\alpha$
 defines the connecting homomorphism (differential)
$k_{\alpha\beta\gamma}\colon E^\gamma_\beta \to E^\beta_\alpha$.
Since \eqref{unpackeddiag13} is exact for $(\beta,\gamma)\in \IIi$ we have that $k j=0$ and thus $k_{\alpha\beta\gamma} k_{\beta\gamma\delta}  = j (k j) k =0$.
Any ordered triple $(\alpha,\beta,\gamma)$ yields the following  {octahedral} diagram:\footnote{If there is no ambiguity about the domain and codomain the sub-indices are omitted from the maps $k$, $i$,  $j$,  $\ell$ and $\tilde k$.} 
\begin{equation}
\label{exact4}
\begin{tikzcd}[column sep=small]
                                                 &                                    & E^\alpha_\varnothing \arrow[lldd, "i"', bend right] \arrow[rrdd, "i", bend left] &                                     &                                    \\
                                                 & E^\beta_\alpha \arrow[ru, "k"] \arrow[rr, dashed, "l"'] &                                                               & E^\gamma_\alpha \arrow[lu, "k"'] \arrow[ld, dashed, "l"'] &                                    \\
E^\beta_\varnothing \arrow[ru, "j"] \arrow[rrrr, "i"', bend right] &                                    & E^\gamma_\beta \arrow[ll, "k"] \arrow[lu, dashed, "\tilde k"']                            &                                     & E^\gamma_\varnothing \arrow[ll, "j"] \arrow[lu, "j"']
\end{tikzcd}
\end{equation}
where the inner exact triangle (dashed) is induced by the three outer exact triangles, cf.\ \cite[Lem.\ 4.8]{Matschke}. 

\begin{theorem}[cf.\ \cite{Matschke}, Lem.\ 4.8]
    \label{equiv}
    An exact couple system over $\sO(\sP)$  extends to a Cartan-Eilenberg system over $\sO(\sP)$.
\end{theorem}

\subsection{The excisive property}
\label{excisiveprop}
For most algebraic topological  applications of Cartan-Eilenberg systems the excisive property of homology plays an important role which leads to the following definition.

\begin{definition}
\label{eecs}
A Cartan-Eilenberg system\index{Cartan-Eilenberg system!excisive}\index{Excisive Cartan-Eilenberg system}  $\bfE$ over $\sO(\sP)$
is called \emph{excisive} if
\begin{equation}
    \label{abs-exc}
E^\beta_{\alpha\cap \beta} \cong E^{\alpha\cup\beta}_\alpha,\quad \forall \alpha,\beta\in \sO(\sP).\footnote{The homomorphism $\ell\colon E^\beta_{\alpha\cap \beta} \to E^{\alpha\cup\beta}_\alpha$ is an isomorphism.}
\end{equation}
\end{definition}

Two ordered pairs $(\alpha,\beta),(\alpha',\beta')\in \IIi$ are \emph{equivalent},\index{Ordered pair!equivalent} if $\beta\smin\alpha=\beta'\smin \alpha'$.
The excisive property for a Cartan-Eilenberg system implies that the $E$-terms only depend on equivalent pairs up to isomorphism.
For $\beta\smin \alpha=\{p\}$ we abuse notation and write $E_p:= E^\beta_\alpha$ for all $p\in \sP$.
\begin{lemma}
    \label{equivpairs}
    Let $(\alpha,\beta)$ and $(\alpha',\beta')$  be equivalent pairs in $\IIi$. Then,
    $E^\beta_\alpha \cong E^{\beta'}_{\alpha'}$.
\end{lemma}
\begin{proof}
    Define $\tilde\alpha=\alpha\vee\alpha'$ and $\tilde\beta=\beta\vee\beta'$. Then,
$\alpha,\alpha'\le\tilde\alpha$, $\beta,\beta'\le \tilde\beta$ and 
$\beta\smin\alpha=\beta'\smin\alpha'=\tilde\beta\smin\tilde\alpha$. Consider $\alpha\le\tilde\alpha$ and $\beta\le\tilde\beta$. Then, $(\beta\cap\tilde\beta)\smin (\alpha\cup\tilde\alpha)=\beta\smin \tilde\alpha = 
\tilde\beta\smin\tilde\alpha$. Consequently,  $(\beta\cup\tilde\alpha)\smin\tilde\alpha = \tilde\beta\smin\tilde\alpha$ and therfore $\tilde\beta = \beta\cup\tilde\alpha$. Similarly, $\beta\smin(\beta\cap\tilde\alpha) = \beta\smin\alpha$ which implies that $\alpha = \beta\wedge \tilde\alpha$.
By \eqref{abs-exc} we conclude that
$
E^\beta_\alpha = E^\beta_{\tilde\alpha\cap \beta} \cong  E^{\tilde\alpha\cup\beta}_{\tilde\alpha}
= E^{\tilde\beta}_{\tilde\alpha},
$
By the same token one proves that 
$E^{\beta'}_{\alpha'} = E^{\beta'}_{\tilde\alpha\cap \beta'} \cong  E^{\tilde\alpha\cup\beta'}_{\tilde\alpha}
= E^{\tilde\beta}_{\tilde\alpha},
$ and thus $E^\beta_\alpha \cong E^{\beta'}_{\alpha'}$.
\end{proof}

Excisive Cartan-Eilenberg systems already appear in the seminal work by Franzosa on connection matrices for Morse representations, cf.\ \cite{fran}. In Franzosa's work such data structures of $R$-modules of $\K$-vector spaces are referred to a \emph{module braids}.\index{Module braid} As a matter of fact one can prove that these concepts are equivalent.

\begin{theorem}[cf.\ \cite{SpV2}]
\label{equivEEC}
A  module braid over the convex sets in $\sP$ is equivalent to
an excisive Cartan-Eilenberg system  over $\sO(\sP)$.
\end{theorem}

An important result for excisive Cartan-Eilenberg systems is a representation in terms of finitely generated differential modules.
For convenience we assume that the ring $R$ is a \emph{principal ideal domain}.\index{Principal ideal domain}
Recall that $\sP$-graded differential module\index{Differential module}\index{Differential module!$\sP$-graded} is denoted by $(C,\dff)$, with $C=\bigoplus_{p\in \sP} G_pC$, cf.\ App.\ \ref{gradedvs}-\ref{gradedcellchain}.
A $\sP$-graded differential module $(C,\dff)$ is free  if and only if the components $G_pC$ are free.\index{Differential module!free} 
Free graded differential modules are used to construct representations of Cartan-Eilenberg systems.
A $\sP$-graded differential module defines an $\sO(\sP)$-filtered differential module\index{Differential module!$\sO(\sP)$-filtered} via $\alpha \mapsto F_\alpha C:= \bigoplus_{p\in\alpha} G_pC$,  cf.\ App.\ \ref{gradedvs}-\ref{gradedcellchain}.
In general, $\sO(\sP)$-filtered differential modules induce
excisive Cartan-Eilenberg systems.  
Consider the short exact sequence:
\begin{equation}
    \label{firstexactseq12}
\begin{tikzcd}
0 \arrow[r] & F_\alpha C \arrow[r, "i","\subset"'] & F_\beta C \arrow[r, "j"] & \displaystyle{\frac{F_\beta C}{F_\alpha C}}
\arrow[r]  & 0
\end{tikzcd},\quad \alpha\le \beta.
\end{equation}
Since the differential $\dff$ preserves the filtering we may define
 the homologies 
 $E^\alpha_\varnothing := H(F_\alpha C,\dff)$, $E^\beta_\varnothing := H(F_\beta C,\dff)$ and
 $E^\beta_\alpha := H\bigl(F_\beta C/F_\alpha C,\dff\bigr)$. This yields the exact triangles in \eqref{unpackeddiag13} where 
 $k\colon H\bigl(F_\beta C/F_\alpha C,\dff\bigr) \to H(F_\alpha C,\dff)$ 
 is the connecting homomorphism constructed in the usual way. 
 All other axioms of exact couple systems are readily verified 
 which yields a Cartan-Eilenberg system denoted
 $\bfE(C,\dff)$. The excisive property follows from the fact that 
$F_{\alpha\cup\beta}C/F_\alpha C =(F_\alpha C+F_\beta C)/F_\alpha C
\cong F_\beta C/(F_\alpha C \cap F_\beta C) = F_\beta/F_{\alpha\cap \beta}C$. 
The (excisive) Cartan-Eilenberg system $\bfE(C,\dff)$ is the Cartan-Eilenberg system of the $\sO(\sP)$-filtered differential module $(C,\dff)$.\index{Cartan-Eilenberg system!of an $\sO(\sP)$-filtered differential module}
This implies that a Cartan-Eilenberg system of a $\sP$-graded differential module is automatically defined.\index{Cartan-Eilenberg system!of a $\sP$-graded differential module}

\begin{definition}
    \label{connmat1234}
    Let $\bfE$ be a Cartan-Eilenberg system over a finite distributive lattice $\sO(\sP)$.
 A free, $\sP$-graded differential group $(C,\dff)$, with $C=\bigoplus_{p\in \sP} G_pC$, is a \emph{$\sP$-graded representation}\index{Graded representation} for $\bfE$ if 
    \begin{equation}
        \label{conncond1}
    \bfE(C,\dff)\cong \bfE.
    \end{equation}
    A  $\sP$-graded representation is \emph{strict} if the free
    $\sP$-graded differential group $(C,\dff)$ is strict.\footnote{Recall that $(C,\dff)$ is strict if the differential restricted to $G_p C$, $p\in \sP$, is trivial, cf.\ Defn.\ \ref{Pchaincompl}.}\index{Graded representation!strict}\index{Differential!strict}\index{Strict differential}
\end{definition}
The main theorem of this section states that in most cases $\sP$-graded representations for Cartan-Eilenberg systems exist and are unique up to conjugacy.
The existence part was proved in \cite{fran} and applies to Cartan-Eilenberg systems due to Theorem \ref{equivEEC}. The existence result in \cite{fran} assumes that every $E$-term is the homology of a differential module. We say in this case that the Cartan-Eilenberg system is \emph{chain generated}.\index{Cartan-Eilenberg system!chain generated}
To be more precise, for all $(\alpha,\beta)\in \IIi$ there exist differential modules $\bigl(C^\alpha_\varnothing,d^\alpha_\varnothing \bigr)$,
$\bigl(C^\beta_\varnothing,d^\beta_\varnothing \bigr)$ and $\bigl(C^\beta_\alpha,d^\beta_\alpha \bigr)$, and short exact sequences
\[
\begin{tikzcd}
0 \arrow[r] & C^\alpha_\varnothing \arrow[r, "i"] & C^\beta_\varnothing \arrow[r, "j"] & C^\beta_\alpha \arrow[r] & 0,
\end{tikzcd}
\]
such that $E_\varnothing^\alpha=H\bigl(C^\alpha_\varnothing,d^\alpha_\varnothing \bigr)$,
$E_\varnothing^\beta = H\bigl(C^\beta_\varnothing,d^\beta_\varnothing \bigr)$ and $E_\alpha^\beta = H\bigl(C^\beta_\alpha,d^\beta_\alpha \bigr)$. By the standard construction of the connecting homomorphisms this
 yields a Cartan-Eilenberg system $\bfE$. This concept is more general than a Cartan-Eilenberg system generated by an $\sO(\sP)$-filtered differential module, or a $\sP$-graded differential module, cf.\ \cite{SpV2}. 

\begin{theorem}[\cite{fran}, Thm.\ 4.8]
    \label{prrepExis}
    Let $\bfE$ be a chain generated, excisive Cartan-Eilenberg system over $\sO(\sP)$. Then, there exists a free, $\sP$-graded differential group $(C,\dff)$ --- a $\sP$-graded representation --- such that $\bfE \cong \bfE(C,\dff)$.
\end{theorem}

An $R$-module is \emph{finitely generated}\index{Finitely generated module}\index{Module!finitely generated} if it has a finite generating set. 
The running assumption in this chapter is that $R$ is a principal ideal domain.
This implies that a module $C\cong R^n\oplus \Tor(C)$, where $\Tor(C)$ is the \emph{maximal torsion submodule}\index{Torsion submodule}\index{Maximal torsion submodule} of $C$\footnote{Recall that $c\in \Tor(C)$ if the exists an $r\in R$ such that $rc=0$.} and $\Tor(C) \cong \bigoplus_i R/(d_i)$, where $d_i$ are the non-zero invariant factors of $C$. The integer $n$ is called the \emph{rank} of $C$.\index{Module!rank}\index{Rank}
A Cartan-Eilenberg system $\bfE$ is finitely generated if all modules $E_\alpha^\beta$, 
$(\alpha,\beta)\in \IIi$.\index{Cartan-Eilenberg system!finitely generated}
Suppose $\bfE$ is a finitely generated, excisive Cartan-Eilenberg system over $\sO(\sP)$. For $E^\beta_\alpha$ let $s^\beta_\alpha$ be the rank of $E^\beta_\alpha$  and $r^\beta_\alpha$ is the number of non-zero invariant factors of the maximal torsion submodule $\Tor(E^\beta_\alpha)$.
For $\beta\smin\alpha=\{p\}$ we denote $s^\beta_\alpha$   and $r^\beta_\alpha$ by $s_p$ and $r_p$ respectively.
\begin{definition}
    \label{connmat4567}
    Let $\bfE$ be a finitely generated, excisive Cartan-Eilenberg system over a finite distributive lattice $\sO(\sP)$.
 A free, $\sP$-graded differential module $(C,\dff)$, with $C=\bigoplus_{p\in \sP} G_pC$, is a \emph{principal representation} for $\bfE$ if 
 \begin{enumerate}
     \item [(i)] $\bfE(C,\dff)\cong \bfE$;
     \item [(ii)] $\rank G_p C = s_p+2r_p$
      for all $p\in \sP$.
 \end{enumerate}
The differential $\dff$ is called a \emph{spectral matrix} for $\bfE$.
\end{definition}

Since $R$-modules over a principal ideal domain allow length 1 free resolutions the existence of a principal representation is guaranteed by Theorem \ref{prrepExis}.
The main result in \cite{SpV2} states that such a representation is unique up to isomorphism of Cartan-Eilenberg systems which implies that differentials (spectral matrices) are unique up to conjugacy. 

\begin{remark}
    \label{homequal}
    For strict $\sP$-graded differential modules the homology  satisfies $H(G_pC,\dff)= G_pC\cong E^\beta_\alpha$, with $\beta\smin\alpha=\{p\}$ for all $p\in \sP$.
\end{remark}

\begin{remark}
    In the context of dynamical systems a spectral matrix is referred to as a \emph{connection matrix}.\index{Connection matrix} We will refer to this nomenclature when we apply Cartan-Eilenberg systems for bi-topological spaces involving the \bflta.
\end{remark}

\begin{remark}
    If a Cartan-Eilenberg system is generated by an $\sO(\sP)$-filtered differential $\K$-vector space, then \cite{robbin:salamon2} provides a simplified proof of Theorem \ref{prrepExis}. 
\end{remark}

\subsection{Cartan-Eilenberg systems of a filtered topological space}
\label{CanCETop}
As before let $(\sP,\le)$ be a finite poset 
and $(X,\scrT)$ be a topological space.
Consider a $\sP$-graded decomposition of $X$ given by $X= \bigcup_{p\in \sP} G_p X$,
cf.\ App.\ \ref{gradfilt}.  
Dual to a grading is a lattice filtering $\grade^{-1}\colon \sO(\sP) \to \sSet(X)$ given by the lattice homomorphism $\alpha \mapsto F_\alpha X :=\grade^{-1}\alpha$.\footnote{If 
 $(\sP,\le)$ is $\scrT$-consistent then $\tile$ is a continuous map, and
   $\grade^{-1}\colon \sO(\sP) \to \scrC(X,\scrT)$ is filtering of closed set of $X$. 
   The mapping $\grade^{-1}\colon \sU(\sP) \to \scrO(X,\scrT)$ yields a filtering of open sets in $X$. This can also be obtained by considering a $\scrT$-co-consistent grading, i.e. continuous with respect to opposite poset $\sP^*$.} 
From this point on one can invoke a (generalized) (co)homology theory by assigning $E^\beta_\alpha := H(F_\beta X,F_\alpha X)$ for every $(\alpha,\beta)\in \IIi$. From
the Eilenberg-Steenrod axioms we have  exact triangles (and the connecting homomorphisms) and commutative squares:
\[
\begin{tikzcd}[column sep=tiny]
H(F_\alpha X) \arrow[rr, "i"] &                   & H(F_\beta X) \arrow[ld, "j"] \\
                  & H(F_\beta X,F_\alpha X) \arrow[lu, "k"] &                  
\end{tikzcd}
\qquad 
\begin{tikzcd}
H(F_\beta X,F_\alpha X) \arrow[r, "k"] \arrow[d, "\ell"'] & H(F_\alpha X) \arrow[d, "i"] \\
H(F_\delta X,F_\gamma X) \arrow[r, "k"]                 & H(F_\gamma X)      
\end{tikzcd}
\]
for $(\alpha,\beta)\le (\gamma,\delta)$.
For example the homology functor can be taken to be singular homology. Note that we suppress the $\Z$-grading as the Cartan-Eilenberg theory works in the setting of differential modules.
Singular homology yields a chain generated Cartan-Eilenberg system $\bfE^\sing(X)$.\index{$\bfE^\sing(X)$}
Denote by $C(X)$ the $R$-module of singular chains over a ring $R$  and $\dff\colon C(X) \to C(X)$ is the singular boundary operator, or differential making $\bigl(C(X),\dff\bigr)$ a differential module.\footnote{If we invoke $C(X) = \bigoplus_q C_q(X)$ as 
the $\Z$-graded module of singular chains then $(C,\dff)$ is a chain complex.} The filtering of $X$ yields a filtering of $C(X)$, i.e.
 $\alpha \mapsto F_\alpha C(X)$
   with $F_\alpha C(X) := C\bigl(F_\alpha X\bigr)$, where $C\bigl(F_\alpha X\bigr)$ are the singular chains in $C(X)$ restricted to $F_\alpha X$.
   For singular chains it holds that $F_{\alpha\cap \beta} C(X) = F_\alpha C(X) \cap F_\beta C(X)$.
   The same relation with  respect to union does not hold in general. The differential satisfies $\dff F_\alpha C(X) \subset F_\alpha C(X)$ for all $\alpha\in \sO(\sP)$,  making $\alpha \mapsto F_\alpha C(X)$ a  meet semi-lattice filtered module (chain complex).
   The fact that the latter is not an $\sO(\sP)$-filtered module prevents us from regarding $\bfE^\sing(X)$ as generated by an $\sO(\sP)$-filtered differential module.
   However, 
for the filtering we obtain the following short exact sequences:
\[
\begin{tikzcd}
0 \arrow[r] & C(F_\alpha X) \arrow[r, "i"] & C(F_\beta X) \arrow[r, "j"] & \frac{C(F_\beta X)}{C(F_\alpha X)} \arrow[r] & 0,
\end{tikzcd}
\]
$(\alpha,\beta)\in \IIi$, which represent the modules $C_\varnothing^\alpha$, $C_\varnothing^\beta$ and $C_\alpha^\beta$.
The associated homologies $H(F_\alpha X) = H\bigl(C(F_\alpha X),\dff\bigr)$, $H(F_\beta X)  = H\bigl(C(F_\beta X),\dff\bigr)$ and 
\[
H(F_\beta X,F_\alpha X) = H\bigl(C(F_\beta X)/C(F_\alpha X),\dff\bigr),
\]
yield the above exact triangle for a pair. 
We conclude that $\bfE^\sing(X)$ is chain generated as explained in Section \ref{excisiveprop}.
The fact that $(C(X),\dff)$ is not lattice filtered implies that the associated Cartan-Eilenberg system is \emph{not} excisive in general.
Depending on the homology theory, or on properties of the grading of $X$, we can relate the homologies 
$H(F_\beta X,F_\alpha X)$ and $H\bigl(F_\beta X/F_\alpha X\bigr)$ which are not necessarily isomorphic.
%
%
If a grading $X=\bigcup_{p\in \sP} G_pX$ is \emph{natural} then the associated filtering $\alpha\mapsto F_\alpha X$ consists of mutually good pairs, cf.\ Defn.\ \ref{regudisc}.
For example if we consider  singular homology then $H(F_\beta X,F_\alpha X)\cong H\bigl(F_\beta X/F_\alpha X\bigr)$ for all $\alpha\subset\beta$, cf.\ \cite[Prop.\ 2.22]{Hatcher}, \cite[Thm.\ 3.2.9]{Weintraub}.
In this case the relative singular homology satisfies the excisive property, i.e. $H(F_{\alpha\cup\beta} X,F_\alpha X)\cong H\bigl(F_{\alpha\cup\beta} X/F_\alpha X\bigr) \cong H\bigl(F_\beta X/F_{\alpha\cap \beta} X\bigr)\cong H(F_\beta X,F_{\alpha\cap \beta} X)$, for all $\alpha,\beta\in \sO(\sP)$.
%
If $X$ is a compact Hausdorff space and the poset $\sP$ is $\scrT$-consistent (not necessarily natural), i.e. the filtering consists of closed sets $F_\alpha X\in \scrC(X,\scrT)$.
By the closedness of $F_\alpha X$ we have the homeomorphisms $F_\beta X/F_\alpha X \smin [F_\alpha X] \cong F_\beta X\smin F_\alpha X$, for all $\alpha\subset\beta$.
Let $\bar H$ represent \emph{Alexander-Spanier cohomology}.\index{Alexander-Spanier cohomology} From the excisive property of Alexander-Spanier cohomology we have that\footnote{Here $\bar H_c$ denote compactly supported Alexander-Spanier cohomology.} 
\[
\begin{aligned}
\bar H\bigl( F_\beta X,F_\alpha X\bigr) &\cong \bar H_c\bigl(F_\beta X \smin F_\alpha X \bigr)
\cong \bar H_c\bigl(F_\beta X / F_\alpha X\smin [F_\alpha X]\bigr)\\
&\cong \bar H\bigl(F_\beta X / F_\alpha X,[F_\alpha X] \bigr) =: \bar H\bigl(F_\beta X / F_\alpha X \bigr),
\end{aligned}
\]
cf.\ \cite[Ch.\ 6, Sect.\ 6, Lem.\ 11]{Spanier}, \cite[Ch.\ V, Sect.\ 2.A]{Dieudonne}.
The same result can be obtained by using compactly supported \emph{Alexander-Spanier homology},\index{Alexander-Spanier homology!compactly supported} cf.\ \cite[Cor.\ 9.4]{Massey}. 

\begin{remark}
\label{BMequiv}
Suppose $X$ is a (locally) compact Hausdorff space homeomorphic to a finite CW-complex and
the poset $\sP$ is $\scrT$-consistent such that $F_\alpha X$ is homeomorphic to a closed subcomplex. Then, $F_\beta X/F_\alpha X$ is the one-point compactification of
$G_{\beta\smin\alpha} X := F_\beta X\smin F_\alpha X$. For the singular homology we have the isomorphism $H\bigl(F_\beta X/F_\alpha X\bigr) \cong H^\BM\bigl(G_{\beta\smin\alpha} X \bigr)$, where $H^\BM$ indicates the \emph{Borel-Moore homology}\index{Borel-Moore homology} of $G_{\beta\smin\alpha} X$, cf.\ \cite{BM}. 
As a matter of fact the Borel-Moore chain groups $C^\BM\bigl(F_{\beta}X\smin F_\alpha X)$
yield a short exact sequences
 \begin{equation}
     \label{exact2bbbb}
     0 \to C^\BM\bigl(F_\beta X\smin F_\alpha X \bigr) \xrightarrow[]{i} C^\BM\bigl(F_\gamma X\smin F_\alpha X\bigr) \xrightarrow[]{j}
     C^\BM\bigl(F_\gamma X \smin F_\beta X\bigr) \to 0,
 \end{equation}
 as opposed to the weakly exact sequences in \eqref{exact2bb}.
\end{remark}

\begin{remark}
    \label{franquot}
    In \cite{fran,fran2} considers the sequence of pointed quotient spaces $F_\beta X/F_\alpha X$ 
which induces the weakly exact sequence\footnote{Weakly exact sequences yield exact triangles homology.
Recall that a seqeunce $A\xrightarrow[]{i} B \xrightarrow[]{j} C$ is \emph{weakly exact}\index{Weakly exact seqeunce} is $i$ is injective, $j\circ i=0$ and the quatient map $B/\image i \to C$ induced by $j$ yields an isomorphism $H(B/\image j) \cong H(C)$, cf.\ \cite{kurland}  and \cite{fran} for more detail.}\index{Weakly exact sequence}
 \begin{equation}
     \label{exact2bb}
\begin{tikzcd}
C(F_\alpha X) \arrow[r, "i"] & C(F_\beta X) \arrow[r, "j"] & C\bigl(F_\beta X/F_\alpha X \bigr) 
\end{tikzcd}
 \end{equation}
 of singular chains on the quotient spaces for good pairs. The approach in \cite{fran2} allows slightly weaker conditions on the good pairs, cf.\ \cite{kurland}.
\end{remark}

\section{Tessellar homology
}\index{Tessellar homology}
\label{tesshom} 
The objective of the homological algebra in this section is to obtain algebraic topological invariants of $X$ via finite algebraic information; \emph{discretization of algebraic topology}.   
To do so we employ the abstract formalism of Cartan-Eilenberg systems  as explained in Section \ref{CEsystems}.
Let $\disc\colon (X,\scrT)\twoheadrightarrow (\cX,\leq)$ be a natural  discretization map,\index{Discretization map!natural} i.e. $\leq$ is a $\scrT$-consistent pre-order and consist of mutually good pairs. Let  $\cX/_\sim$ be the poset of partial equivalence classes of $(\cX,\leq)$. Then, the map $X\twoheadrightarrow \cX/_\sim$ given by the composition
\begin{equation}\label{eq:discgrading}
\begin{tikzcd}[column sep=large]
X \arrow[r, two heads, "\disc"] 
& \cX \arrow[r] \arrow[r] \arrow[r, two heads, "\pi"] & \cX/_\sim,
\end{tikzcd}
\end{equation}
is natural and yields a $\cX/_\sim$-grading $X= \bigcup_{[\xi]} G_{[\xi]}X$, cf.\ Rem.\ \ref{partialequiv} and App.\ \ref{gradfilt}. 
The associated filtering 
%
$\cU \mapsto F_\cU X$, $\cU \in \sO(\cX,\le)$, defined by $\disc^{-1}$, consists of good pairs and yields an excisive Cartan-Eilenberg system $\bfE^\disc$ as outlined in Section \ref{CanCETop}. For simplicity we assume that $\bfE^\disc$ is \emph{finitely generated}\index{Cartan-Eilenberg system!finitely generated}\index{Finitely generated} for the remainder of this chapter.
We now explain the construction of an associated homology theory.

\subsection{The tessellar differential module}\label{tesschaincomp}
For $\xi\in \cX$ define the $E$-terms for $\bfE^\disc$ via relative (singular) homology
 $E_{[\xi]} := H\bigl(F_{\downarrow\xi} X,F_{\downarrow\xi^\pred} X \bigr)$, which is chain generated and finitely generated with \emph{coefficients in a principal ideal domain} $R$. The  \emph{tessellar modules},\index{Tessellar module}\index{Tessellar chain groups} or \emph{tessellar chain groups} are given
 by the external direct sum
\begin{equation}
    \label{tess1}
    C^\disc(X) := \bigoplus_{[\xi]\in \ccX/_\sim} 
    G_{[\xi]}C^\disc(X),
\end{equation}
where $G_{[\xi]}C^\disc(X) := H\bigl(F_{\downarrow\xi} X,F_{\downarrow\xi^\pred} X \bigr)$\footnote{By the isomorphism $\sO(\ccX/_\sim) \cong \sO(\ccX,\le)$ we have that $\big\downarrow \xi = \big \downarrow [\xi]$.} 
if the latter is a free $R$-module, 
or
else choose a free differential module $\bigl(G_{[\xi]}C^\disc(X),\dff\bigr)$ such that 
\begin{enumerate}
    \item [(i)] $H\bigl( G_{[\xi]}C^\disc,\dff\bigr)\cong H\bigl(F_{\downarrow\xi} X,F_{\downarrow\xi^\pred} X \bigr)$;
    \item [(ii)] $G_{[\xi]} C^\disc(X) \cong R^{s_\xi+2r_\xi}$, cf.\ Defn.\ \ref{connmat4567}(ii).
\end{enumerate}
By Theorem \ref{prrepExis} there exists
an $\sO(\cX)$-filtered differential 
\[
\cm^\disc\colon  C^\disc(X) \to  C^\disc(X)
\]
such that 
$\bfE\bigl( C^\disc,\cm^\disc\bigr)\cong \bfE^\disc$ and the spectral matrix $\cm^\disc$ is unique up to isomorphism, cf.\ Sect.\ \ref{excisiveprop}.
If the homologies $H\bigl(F_{\downarrow\xi} X,F_{\downarrow\xi^\pred} X \bigr)$ are free (finitely generated), then 
  $\bigl(C^\disc,\cm^\disc\bigr)$ is a strict $\cX/_\sim$-graded differential module with $C^\disc(X)\\ = \bigoplus_{[\xi]\in \ccX/_\sim} 
    H\bigl(F_{\downarrow\xi} X,F_{\downarrow\xi^\pred} X \bigr)$.  
  We refer to $\bigl(C^\disc,\cm^\disc\bigr)$ as the 
  $\cX/_\sim$-graded \emph{tessellar differential module}, or
  $\cX/_\sim$-graded \emph{tessellar chain complex}\index{Tessellar differential module}\index{Tessellar chain complex} based on singular homology. The associated homology is called the \emph{tessellar homology} of $\disc\colon X\twoheadrightarrow \cX$ and is denoted by $H^\disc(X)\cong H(X)$. 
%
  Following Appendix \ref{gradedvs}
  the restricted module $G_{\cU\smin\cU'}C^\disc(X)$ is well-defined for
 every convex set $\cU\smin \cU'$, with $\cU,\cU'\in\sO(\cX,\le)$ and 
\begin{equation}
    \label{tess2}
    G_{\cU\smin\cU'}C^\disc(X) := \bigoplus_{[\xi]\subset \cU\smin\cU'}
    G_{[\xi]}C^\disc = \frac{F_\cU C^\disc}{F_{\cU'}C^\disc}.
\end{equation}
The differential  is the restriction of $\cm^\tess$ to $G_{\cU\smin\cU'}C^\tess(X)$.
The associated homology $H^\disc(G_{\cU\smin\cU'}X) := H\bigl(G_{\cU\smin\cU'}C^\disc(X),\cm^\disc \bigr)$ is the tessellar homology of the locally closed set\footnote{A locally closed subset of $X$ is given as the intersection an open and a closed subset in $X$.} $G_{\cU\smin\cU'}X$. The latter also defines the relative tessellar homology $H^\disc\bigl( F_\cU X,F_{\cU'}X\bigr)$.\index{Tessellar homology!relative}
\begin{theorem}
\label{tess3}
Let $\disc\colon X\twoheadrightarrow \cX$ be a natural 
discretization map. Then, the tessellar homology satisfies $H^\disc(X) \cong H(X)$.
In particular, for every convex set $\cU\smin\cU'\subset \cX$, $\cU,\cU'\in\sO(\cX,\le)$, we have that $H^\disc(G_{\cU\smin\cU'}X) \cong H(F_{\cU} X,F_{\cU'} X)$.
\end{theorem}

\begin{proof}
By definition the modules $G_{[\xi]}C^\disc$ 
are free. By  \cite[Thm.\ 4.8]{fran} we have a differential $\cm^\disc$ which retrieves  the homologies.
\end{proof}

The discretization map $\disc\colon  X\twoheadrightarrow \cX$ discretizes $X$, while the construction of $\bigl(C^\disc(X),\cm^\disc\bigr)$ 
discretizes the algebraic topology of $X$.  
More details on tessellar homology are discussed in \cite{SpV2}.

\subsection{The tessellar  complex}\index{Complex!tessellar}\index{Tessellar complex}
\label{tcc}
It it sometimes useful 
to describe the tessellar homology in combinatorial terms. To do some we define a slight generalization of a (Lefschetz) complex,\index{Lefschetz  complex}\index{Complex!Lefschetz} cf.\ \cite{hms}.

\begin{definition}
    \label{defnTC}
    A \emph{complex}\index{Complex} is a triple $(\ccT,\le,\kappa)$ where $(\ccT,\le)$ is a finite pre-order and  $\kappa\colon\ccT\times \ccT\to R$ is a function satisfying
    \begin{enumerate}
        \item [(i)] (upper-triangular) $\kappa(\vartheta,\vartheta') \neq 0$ implies $\vartheta\le \vartheta'$, $\vartheta\neq\vartheta'$; 
        \item [(ii)] (boundary) $\sum_{\vartheta'} \kappa(\vartheta,\vartheta')\kappa(\vartheta',\vartheta'') =0$.
    \end{enumerate}
    The function $\kappa$ is called the \emph{incidence function}\index{Incidence function} and its values in the ring $R$  are called the \emph{incidence numbers}.\index{Incidence number}
\end{definition}

For a cell complex we can define (free) canonical differential $R$-module.
Define the free $R$-module over $\ccT$:
\[
C(\ccT~) := \bigoplus_{\vartheta\in \cccT} R \langle \vartheta\rangle\footnote{We use the notation $\langle\vartheta\rangle$ to express $\vartheta$ as basis for $C(\ccfT~)$.} 
\]
with differential
\[
\dff_{\cccT~} \langle\vartheta \rangle:= \sum_{\vartheta'\in \cccT} \kappa(\vartheta,\vartheta')\langle\vartheta'\rangle, 
\]
which makes $\bigl(C(\ccT~),\dff_{\cccT~} \bigr)$ a $\ccT\,/_\sim$-graded differential $R$-module. For every convex set $\ccC\subset \ccT~$ the restriction of $C(\ccT~)$ to $C(\ccC)$ and differential accordingly defines a \emph{subcomplex}\index{Cell complex!subcomplex of a} $\bigl( C(\ccC),\dff_{\cccC}\bigr)$.
The associated Cartan-Eilenberg system is excisive and is denoted by $\bfE\bigl( C(\ccT~),\dff_{\cccT~}\bigr)$ with $E$-terms given by $E^\beta_\alpha:= H\bigl( C(\ccC),\dff_{\cccC}\bigr)$, with $\beta\smin\alpha =\ccC$.
%
Consider the diagram
\begin{equation}
    \label{TCdiag}
\begin{tikzcd}
X \arrow[r, "\disc", two heads] & \cX \arrow[r, "\pi", two heads] & \cX/_\sim                           \\
                            &                             & \ccT \arrow[u, "\varpi"', two heads]
\end{tikzcd}
\end{equation}
which implies that the equivalence classes in $\ccT~$ coincide with the equivalence classes in $\cX$, i.e. $\ccT\,/_\sim \cong \cX/_\sim$.
\begin{definition}
    \label{TCprin}
    A  triple $(\ccT,\le,\kappa)$ is \emph{tessellar  complex}\index{Complex!tessellar} for a natural discretization map $\disc\colon X\twoheadrightarrow \cX$ if
if
\[
\bigl( C(\ccT~),\dff_{\cccT~}\bigr) \cong \bigl(C^\disc,\dff^\disc\bigr),
\]
as $\sO(\cX)$-filtered chain isomorphic $\cX/_\sim$-graded differential modules. 
\end{definition}
A  tessellar  complex yields a differential matrix with $R$-coefficients via the incidence function $\kappa$. This justifies the terminology spectral matrix.
For a convex set $\ccC\subset \cX$ we use the notation
\begin{equation}
    \label{notfordischom}
    H^\disc(G_{\cU\smin\cU'}X) \cong  H\bigl( C(\ccC~),\dff_{\cccC~}\bigr) \cong H(F_\cU X,F_{\cU'}X),
\end{equation}
with $\ccC=\cU\smin \cU'$ is independent of the pair $\cU'\subset \cU$ and where
$\bigl(C(\ccC),\dff_{\cccC}\bigr)$ is restriction of $\bigl( C(\ccT~),\dff_{\cccT~}\bigr)$ to $\ccC$.
In Section \ref{cellhom12} we discuss the special case of a CW-decomposition. 

\begin{remark}
Given a natural discretization map $\disc\colon X\twoheadrightarrow \cX$ one
     can always construct a map $\disc\colon X\to \cX$ (not necessarily surjective) such that there exists  a  tessellar  complex $(\cX,\le,\kappa)$. Such a discretrization map is a \emph{representable} natural discretization.\index{Natural discretization map!representable}\index{Discretization map!representable}
\end{remark}


\subsection{Linear discretization, split grading and bi-graded Betti numbers}
\label{dimensiongrading}
In certain cases a natural discretization map $\disc\colon X\twoheadrightarrow \cX$ allows another scalar discretization map via an order-preserving map $\ind\colon\cX\to \Z$ (not necessarily surjective). By construction $\ind$ factors through $\cX/_\sim$:
\[
\begin{tikzcd}
X \arrow[r, "\disc", two heads] \arrow[rrr, "\skel", bend right, shift right=1] 
& \cX \arrow[r, two heads] \arrow[r] \arrow[r, "\pi"] \arrow[rr, "\ind", bend left, shift left=1] & \cX/_\sim \arrow[r] & \Z
\end{tikzcd}
\]
The map $\ind$ maps to a linear order and the composition $\skel$ is called a \emph{linear discretization}.\index{Linear discretization map}\index{Discretization map!linear}
The discretization $\ind$ 
is a coarsening of $\disc$ and  is therefore  a natural discretization map, which 
makes $\bigl(C^\disc,\cm^\disc)$ a $\Z$-graded differential module.
Since $\ind$ is order preserving the tessellar homology of $G_p X :=\skel^{-1} p \subset X$ is well-defined. For the latter
we consider a standard spectral sequence. Define, using \eqref{tess2},
\[
G_p C^\disc(X) = \bigoplus_{[\xi]\subset \ind^{-1} p} G_{[\xi]} C^\disc(X),
\]
which gives the $\Z$-grading $C^\disc(X) = \bigoplus_{p\in \Z} G_p C^\tess(X)$.
As explained in Section \ref{excisiveprop} we obtain the short exact sequences
\[
\begin{tikzcd}
0 \arrow[r] & F_{\downarrow (p-1)} C^\disc \arrow[r, "i_{p-1}"] & F_{\downarrow p}C^\disc \arrow[r, "j_p"] & G_{p}C^\disc \arrow[r] & 0
\end{tikzcd}
\]
From the tessellar boundary operator $\cm^\disc$ we compute the homology
which provide the zeroth and first pages $\scrE^0=\bigoplus_p \scrE^0_p$ and 
$\scrE^1=\bigoplus_p \scrE^1_p$ with
\[
\scrE^0_p = G_pC^\disc(X)\quad \text{and}\quad 
\scrE^1_p =
H\bigl(G_pC^\disc \bigr) \cong H^\disc\bigl(G_p X\bigr).
\]
This yields the exact triangles, using $H^\disc(F_{\downarrow p} X) = H\bigl(F_{\downarrow p}C^\disc\bigr)$,

\begin{equation}
     \label{dimensiongrading12}
\begin{tikzcd}[column sep=tiny]
H^\disc(F_{\downarrow (p-1)}X) \arrow[rr, "i_{p-1}"] &                   & H^\disc(F_{\downarrow p}X) \arrow[rr, "i_p"] \arrow[ld, "j_p"] &                                                       & H^\disc(F_{\downarrow (p+1)}X) \arrow[ld, "j_{p+1}"] \\
                  & \scrE^1_{p} \arrow[lu, "k_p"] &                                   & \scrE^1_{p+1} \arrow[lu, "k_{p+1}"] \arrow[ll, "\dff^1_{p+1}"', dotted, bend left] &                  
\end{tikzcd}
 \end{equation}
where
 $\dff^1_p = j_{p-1}k_p$  are the connecting homomorphisms computed from $\cm^\disc$.
Recursively define $\scrE^{r+1} = H(\scrE^r,\dff^r)$ where $\dff^r_p\colon \scrE^r_p\to \scrE^r_{p-r}$
 with $\dff^r = ji^{1-r}k$.
Since $\ind$ defines a finite filtering on $C^\disc(X)$ 
the spectral sequence converges and
$\scrE^\infty_p \cong G_pH(C^\disc) = G_pH^\disc(X)$. If we set $\beta^\disc_p= \rank G_pH^\disc(X)$, called the \emph{Betti numbers},\index{Betti number} then
\[
\sum_{p\in \Z} \beta^\disc_p = \rank H^\disc(X).
\]
In particular, when $R=\K$, then $\ind$ yields a grading on the tessellar homology. For ring coefficients this is more complicated and it is not true in general that $\Gr H^\disc(X) =\bigoplus_{p\in \sP} G_pH^\disc(X)$ is isomorphic to $H^\disc(X)$, cf.\ App.\ \ref{gradedvs}. However, if the differential satisfies $\dff^\disc G_pC^\disc\subset G_{p-1}C^\disc$, then 
the tessellar differential module $\bigl(C^\disc,\cm^\disc)$, with 
$\dff^\disc_p := \dff^\disc(p,p-1)$,
 \footnote{In Appendix \ref{gradedvs} the entries of $\dff$ are explained.}
is a chain complex $\bigl(G_pC^\disc,\cm_p^\disc)$. 
We write $\dff^\disc = \bigoplus_{p\in \sP} \dff_p^\disc$. In this case the linear discretization $\ind$ is said to be \emph{split grading}\index{Split grading}\index{Discretization map!split grading} for the tessellar differential module and $\ind$ induces a $\Z$-grading on the associated tessellar homology. 
The $\Z$-grading given by $\ind$ can be useful is some cases e.g. cellular homology and the treatment of tessellar homology for Morse pre-orders, cf.\ Sect.'s \ref{cellhom12} -- \ref{doublegr}.

The advantage of using Betti numbers is that one can treat the $\Z$-grading given by $\ind$ as a grading on the Betti numbers.
We start with using the $\Z$-grading on singular homology: $H^\disc(X) = \bigoplus_{q\in \Z} H_q^\disc(X)$. Define the \emph{bi-graded Betti numbers}\index{Betti number!bi-graded}\index{Bi-graded Betti number} as: $\beta^\disc_{p,q} =\rank G_pH^\disc_q(X)$.
The  double \emph{tessellar Poincar\'e polynomial}\index{Poincar\'e polynomial!tessellar} of $X$ is defined as 
\begin{equation}
    \label{poin1}
    P^\disc_{\lambda,\mu}(X) = \sum_{p,q\in \Z}\beta^\disc_{p,q}\lambda^p\mu^q.
\end{equation}
%
%
The latter satisfies a variation on the standard \emph{Morse relations}.\index{Morse relations} Define $\beta^\disc_{p,q}(\xi) = \rank G_pH_q^\disc(G_{[\xi]}X)$ and associated
Poincar\'e polynomial
\[
P_{\lambda,\mu}^\disc\bigl(G_{[\xi]}X \bigr) = \sum_{p,q\in \Z} \beta^\disc_{p,q}(\xi)\lambda^p\mu^q.
\]
The singular homology grading yields the splitting $\dff_p^r=
\bigoplus_{q\in \Z} \dff_{p,q}^r$.
\begin{theorem}[Bi-graded Morse relations]\index{Morse relations!bi-graded}\index{Bi-graded Morse relations}
\label{morserel1}
Let $X\xtwoheadrightarrow{\disc}\cX\xrightarrow{\ind} \Z$ be a natural linear discretization. Then,
\begin{equation}
    \label{morsegenrel}
    \sum_{[\xi]\subset \ccX}
    P_{\lambda,\mu}^\disc\bigl(G_{[\xi]} X \bigr) = P_{\lambda,\mu}^\disc(X) + \sum_{r=1}^\infty (1+\lambda^r \mu) Q^r_{\lambda,\mu},
\end{equation}
where $Q^r_{\lambda,\mu} = \sum_{p,q\in \Z} (\rk \image \dff^r_{p+r,q+1}) \lambda^p \mu^q\ge 0$. The sum over $r$ is finite.
\end{theorem}
\begin{proof}
In terms of the spectral sequences we have that $\scrE^0=\bigoplus_{p,q\in \Z}\scrE^0_{p,q}$ and $\scrE^1=\bigoplus_{p,q\in \Z}\scrE^1_{p,q}$ are isomorphic.
For the spectral sequence we have the short exact sequences
\[
0\longrightarrow \ker \dff^r_{p,q} \longrightarrow \scrE^r_{p,q} \longrightarrow \image \dff^r_{p,q} \longrightarrow 0,\quad\text{and}
\]
\[
0\longrightarrow \image \dff^r_{p+r,q+1} \longrightarrow \ker \dff^r_{p,q} \longrightarrow \scrE^{r+1}_{p,q} \longrightarrow 0.
\]
the implies the relation
\[
\rk \scrE^r_{p,q} = \rk \scrE^{r+1}_{p,q} + \rk \image \dff^r_{p,q} + \rk \image \dff^r_{p+r,q+1}.
\]
Define the double Poincar\'e polynomials $P_{\lambda,\mu}(\scrE^r) := \sum_{p,q\in \Z} (\rk \scrE^r_{p,q}) \lambda^p \mu^q$ and $Q^r_{\lambda,\mu} = \sum_{p,q\in \Z} (\rk \image \dff^r_{p+r,q+1}) \lambda^p \mu^q$. Then, the Poincar\'e polynomials satisfy
 $P_{\lambda,\mu}(\scrE^r) = P_{\lambda,\mu}(\scrE^{r+1}) + (1+\lambda^r \mu) Q^r_{\lambda,\mu}$.
Iterating the above identities for $P_{\lambda,\mu}(\scrE^r)$ and using the fact that the spectral sequence converges yields
Equation \eqref{morsegenrel}.
\end{proof}

Using  bi-graded  tessellar Betti numbers will prove to be very useful in setting up a more refined  theory of spectral matrices. In Sect.\ \ref{lapgrading} we exploit this idea  in the setting of parabolic flows.
This approach is reminiscent of the detailed connection matrix in \cite{Bart1}.
If we again ignore the natural grading of singular homology the Morse relations will be
\begin{equation}
    \label{morsegenrel12}
    \sum_{[\xi]\subset \ccX}
    P_{\lambda}^\disc\bigl(G_{[\xi]}X \bigr) = P_{\lambda}^\disc(X) + \sum_{r=1}^\infty (1+\lambda^r) Q^r_{\lambda},
\end{equation}
which is obtained by setting $\mu=1$. Note that the property for $\ind$ to be split grading is that $Q^r_\lambda=0$ for $r\ge 2$. The maximal value for $r$ in \eqref{morsegenrel12} can be utilized to further coarsen $\ind$ in order to obtain a linear discretization that is split grading.

\begin{remark}
\label{finergrad}
A similar procedure to bi-graded tessellar Betti numbers can be followed for the $\cX/_\sim$-grading by using spectral systems, cf.\ \cite{Matschke}.
\end{remark}

\begin{remark}
    We do not refer to  $C^\disc(X) = \bigoplus_{p\in \Z} G_p C^\disc(X)$ as the skeletal differential module since it is a coarsening of the tessellar differential module. The differential is obtained by coarsening the information. This issue comes up again in the next section.
\end{remark}

\section{Cellular homology}
\label{cellhom12}
Let $X$ be a finite CW-complex, i.e. a compact Hausdorff space that admits a CW-decomposition  map 
$\cell\colon X\twoheadrightarrow \cX$, where $(\cX,\le)$ is the poset of CW-cells with the face partial order.  
The \emph{cellular differential module}, or \emph{cellular chain complex},\index{Differential module!cellular}\index{Cellular differential module}\index{Chain complex!cellular}\index{Cellular chain complex} denoted $C^\cell$, is constructed according to the theory in Section~\ref{tesshom}. 
This coincides with the classical construction, as we outline below. 

From the definition of CW-decomposition we have the composition
\[
\begin{tikzcd}[column sep=large]
X \arrow[r, two heads, "\cell"] \arrow[rr, "\skel", bend right] & \cX \arrow[r] \arrow[r] \arrow[r, "\dim"] & \N
\end{tikzcd}
\]
which is 
a coarsening of the discretization $\cell$. Since $\dim$ is order-preserving and since $\cell$ is a natural discretization, the composite discretization $\skel$ is natural and linear and thus continuous. Therefore $\skel$ defines a $\scrT$-consistent (not surjective) linear discretization of $X$. In the traditional set-up the cellular homology of $X$ is defined in terms of the natural discretization map $\skel$.\index{Discretization map!natural}
As before define the filtering $\big\downarrow p\mapsto \skel^{-1}\big\downarrow p =: F_{\downarrow p} X$
with skeletal chain complex $C^\skel(X) := \bigoplus_{p\in \N} H\bigl(F_{\downarrow p} X, F_{\downarrow (p-1)} X\bigr)$.\footnote{By Remark \ref{BMequiv} we also have $H\bigl(F_{\downarrow p} X, F_{\downarrow (p-1)} X\bigr) \cong H^\BM(G_p X)$.} 
The sets $G_p X =F_{\downarrow p} X\smin F_{\downarrow (p-1)} X$ are a disjoint union of $p$-cells $|\xi|$ in $X$ and the homology is given by $H\bigl(F_{\downarrow p} X, F_{\downarrow (p-1)} X\bigr) \cong \bigoplus_{\xi \in G_p\cccX} R \langle \xi \rangle$, where $G_p\cX = \dim^{-1} p$ and $R$ is a principal ideal domain. 
Since the homologies $H\bigl(F_{\downarrow p} X, F_{\downarrow (p-1)} X\bigr)$ are 
free Theorem \ref{prrepExis} yields
 a differential (strict spectral matrix) $\cm^\skel\colon C^\skel(X)\to C^\skel(X)$ such that $H^\skel(X) \cong H(X)$. 
 We will explain this construction now in a more detailed way.

\begin{remark}
The above construction is the traditional way of constructing cellular homology for a finite CW-complex $X$. The $\N$-grading is special in the sense that if $H$ is the singular homology functor
and $\dim$ plays the role in $\ind$ in Sect.\ \ref{dimensiongrading},
then $G_p H_q(X)\neq 0$  
if and only if $p=q$. Since all components are homeomorphic the skeletal chain complex is given by $C^\skel_q(X) =\bigoplus_{\dim(\xi) = q} R \langle \xi \rangle$ with boundary operator $\cm^\skel_q$ with $\dff^\skel= \bigoplus \dff^\skel_q$. The linear discretization $\skel$ is split grading.\index{Discretization map!split grading}
Even though $\skel$ is natural this condition is not needed since the order is linear.
\end{remark}

A more detailed way to treat the cellular complex is to use the face partial order $(\cX,\le)$. 
As in the more general tessellar case we define the cellular
chain complex\footnote{Even though we do not utilize the $\Z$-grading of singular homology we refer to $C^\cell$ as chain complex as opposed to differential module since the construction is based on the singular chain complex.} by
\[
 C^{\cell}(X) = \bigoplus_{\xi \in \cccX} G_\xi C^\cell(X),\quad
G_\xi C^\cell(X) = H\bigl(F_{\downarrow \xi }X, F_{\downarrow \xi^{\pred}}X \bigr)\cong H^\BM(G_\xi X),\footnote{Here $\big\downarrow\xi^\pred:= \bigl(\big\downarrow\xi\bigr)^\pred$ denotes the immediate predecessor of $\big\downarrow\xi$.}
\]
where $G_\xi X=|\xi|$ and its Borel-Moore homology\index{Borel-Moore homology} of $G_\xi X$ is given by $H^\BM(G_\xi X)\cong R$.
The module $C^\cell$ is a special case of the tessellar module for the discretization map $\cell$. The  filtering 
$\cU \mapsto F_\cU X$, $\cU \in \sO(\cX,\le)$, defined by $\cell^{-1}$, consists of good pairs and yields a chain generated, excisive Cartan-Eilenberg system $\bfE^\cell$ as outlined in Sections \ref{CanCETop} and \ref{tesshom}. Since $G_\xi C^\cell(X)\cong R$  the system $\bfE^\cell$ is finitely generated.
We follow the procedure of Section \ref{tesshom}.
Theorem \ref{prrepExis} yields the existence of (strict)\index{Differential!strict}
\[
\cm^\cell\colon C^\cell(X)\to C^\cell(X),
\]
 such that $\bfE\bigl(C^\cell,\dff^\cell\bigr)\cong \bfE^\cell$. In particular,   $H^\cell(X) \cong H(X)$ and 
\begin{equation}
    \label{BMhom1}
  H^\cell(G_{\cU\smin \cU'}X) 
    \cong H(F_{\cU} X,F_{\cU'} X) \cong H^{\BM}(G_{\cU\smin \cU'}X),\footnote{For latter isomorphism on Borel-Moore homology, cf.\ \cite{goresky,Bredon}.}
\end{equation}
for all $\cU,\cU'\in \sO(\cX,\le)$,  $\cU'\subset \cU$.
The pair $\bigl(C^\cell,\dff^\cell\bigr)$ is the \emph{cellular chain complex}.\index{Chain complex!cellular}\index{Cellular chain complex}
For the differential $\cm^{\cell}$  we express the strict upper-triangular structure by
$\dffbf^\cell(\xi,\xi')\colon G_{\xi'} C^\cell\to G_\xi C^\cell$ and 
  $\dffbf^{\cell}(\xi,\xi')\neq 0$ implies that $\xi < \xi'$.
  The following lemma follows from choice of singular homology in the definition of $\bfE^\cell$.
  \begin{lemma}
      \label{cover12}
      If $\dffbf^\cell(\xi,\xi')\neq 0$, then $\xi'$ covers $\xi$.\footnote{In a finite partial order $\xi$ \emph{covers} $\xi$ if $\xi < \xi'$ and $[\xi,\xi']=\{\xi,\xi'\}$. The pair $\{\xi,\xi'\}$ is a\index{Covering pair}\index{Covering pair!covers} \emph{covering pair}.}
  \end{lemma}
  \begin{proof}
  For the singular chain complex the connecting homomorphism $k_q$ in
  \[
  \begin{tikzcd}[column sep=small]
{} \arrow[r, dashed] & H_q(F_{\downarrow\xi^\pred} X) \arrow[r, "i_q"]    & H_q(F_{\downarrow\xi} X) \arrow[r, "j_q"] 
                  & H_q(F_{\downarrow\xi} X,F_{\downarrow\xi^\pred} X) \arrow[r, "k_q"] &          H_{q-1}(F_{\downarrow\xi^\pred} X)  \arrow[r, dashed] & {}  
\end{tikzcd}
  \]
  is degree $-1$ which implies that $\dff^\cell$ is also degree $-1$ with respect to the $\Z$-grading of singular homology, i.e.
  $\dffbf^\cell(\xi,\xi') = \bigoplus \dffbf^\cell_q(\xi,\xi')$ and $\dffbf^\cell_q(\xi,\xi')\colon G_{\xi'} C^\cell_q\to G_\xi C^\cell_{q-1}$, cf.\ \cite{fran}.
  For the cellular chain groups it holds that
  \[
G_\xi C_q^\cell(X)
 =
H^{}_q \bigl(F_{\downarrow \xi} X, F_{\downarrow \xi^\pred} X \bigr) \cong
\begin{cases}
 R    & \text{ if } q=\dim \xi; \\
  0    & \text{ if } q\neq\dim \xi,
\end{cases}
\]
Therefore,  $\dffbf^{\cell}_q(\xi,\xi')=0$ unless $\dim \xi'=q$ and $\dim \xi=q-1$.
  \end{proof}

In contrast to the general construction of Section \ref{tesshom}, 
all of the nonzero entries $\dffbf^\cell(\xi,\xi')$ are determined by 
the octahedral diagrams, i.e. the rolled out middle triangle in \eqref{exact4} (homology braid). 
Indeed, for a covering pair $\xi < \xi'$, with $\dim \xi'=q$, we choose a triple of down-sets (closed sets) 
\[
\big\downarrow \xi'^\pred \smin \xi\subset \big\downarrow \xi'^\pred \subset \big\downarrow\xi', 
\]
which yields $\cm^\cell_q(\xi,\xi')\colon G_{\xi'}C_q^{\cell} \to  G_\xi C_{q-1}^{\cell}$ given by the 
composition\footnote{In Franzosa's connection matrix theory these entries are called \emph{flow defined}.}
\begin{equation}\label{cellbound}
\begin{tikzcd}[column sep=small]
H^{}_q\bigl(F_{\downarrow \xi'} X, F_{\downarrow \xi'^\pred }X\bigr) \arrow[r, "k_q"] & {H^{}_{q-1}\bigl(F_{\downarrow \xi'^\pred}X\bigr)} \arrow[r] & H^{}_{q-1}\bigl(F_{\downarrow \xi'^\pred}X,F_{\downarrow \xi'^\pred\smin {\xi} }X\bigr).
\end{tikzcd}
\end{equation}
where $G_{\xi'}C_q^{\cell} = H^{}_q\bigl(F_{\downarrow \xi'} X, F_{\downarrow \xi'^\pred }X\bigr)$ and $ G_\xi C_{q-1}^{\cell} = H^{}_{q-1}\bigl(F_{\downarrow \xi'^\pred}X,F_{\downarrow \xi'^\pred\smin {\xi} }X\bigr)$.
By the excisive property 
this construction is independent of the triple $\cU\subset\cU'\subset\cU''$ with $\xi=\cU'\smin \cU$ and $\xi'=\cU''\smin \cU'$.

 Per Section \ref{dimensiongrading}  
we  consider the order-preserving map $\dim \colon \cX\to \N$ which plays the role of $\ind$ in order to obtain an $\N$-grading of $H^\cell$. For the composition $\skel$ the sets $G_p X:= \skel^{-1} p$ are convex. 
By Lemma \ref{cover12} the differential $\dff^\cell$ acts on $G_p C^\cell(X)$ as
\[
\dff^\cell\colon G_p C^\cell_p(X) \to G_p C^\cell_{p-1}(X).
\]
If we write the restriction to $G_pC^\cell(X)$ as $\dff^\cell_p$ then
 $\bigl( G_p C^\cell,\dff^\cell_p \bigr)$ is a chain complex and $\dim$ is split grading
 for cellular homology.\index{Split grading}
 As a consequence $\dim$ yields the natural $\N$-grading of  cellular homology.

\begin{remark}
As in Section \ref{tcc} we can define the standard Lefschetz  complex and incidence numbers from the cellular homology.
\end{remark}





\begin{remark}
For the spectral sequence in Theorem \ref{morserel1} for $\dim$ we have that $\dff^r_{p,q}=0$ for $r\ge 2$ and therefore $Q^r_{\lambda,\mu}=0$ for $r\ge 2$. Moreover, all homologies $\scrE^r_{p,q} =0$ for $p\neq q$ and $r\le 2$. For the Morse relations this implies 
\[
P_{\lambda \mu}(C^\cell) = P_{\lambda,\mu}(C^\cell) = P_{\lambda,\mu}^\cell(X) + 
(1+\lambda \mu) Q_{\lambda,\mu}^1
= P_{\lambda\mu}^\cell(X) + 
(1+\lambda \mu) Q_{\lambda \mu}^1.
\]
Moreover, $P_{\lambda \mu}(C^\cell)= \sum_{\xi\in \cccX} P_{\lambda\mu}^\cell(\xi)$, with $P_{\lambda\mu}^\cell(\xi) = (\lambda\mu)^q$ and $q=\dim \xi$. This yields
\[
\sum_{\xi\in \cccX} (\lambda\mu)^{\dim \xi}
= P_{\lambda\mu}^\cell(X) + 
(1+\lambda \mu) Q_{\lambda \mu}^1, 
\]
which retrieve the standard Morse relations. The latter also follows if we use the fact that $\bigl(C^\cell,\dff^\cell \bigr)$ is a chain complex.
\end{remark}



\section{Composite gradings and the homology for Morse pre-orders}
\label{doublegr}
The most important objective in chapter is to build an  homology theory for the discretization map $\tile\colon X \to \sSC$.\index{$\tile$}
In the first sections of this chapter we utilized Cartan-Eilenberg systems to discretize the algebraic topological information for arbitrary topological spaces.
In this section we outline how discretization can be employed in a bi-topological setting.
The objective is to use the factorization so that we can  discretize two topologies: the space topology and the block-flow topology. To do so one can factor the two topologies in $(X,\scrT,\scrTbf)$ in  different ways.
For example $(X,\scrT) \rightarrow (\cX,\le) \rightarrow (\cX,\le^\dagger) \rightarrow (\sSC,\le)$, or $(X,\scrTbf)\rightarrow (\cX,\lebf ) \dashrightarrow (\cX,\le^\dagger)  \rightarrow (\sSC,\le)$, cf.\ Diagram \eqref{ordpreforgrd}.
One can also invoke to the topology $\scrT^\dagger$ in this setting.
In a more general setting the bi-topological discretization of algebraic invariants can be organized via the following diagram continuous maps
\[
\begin{tikzcd}
(X,\scrT,\scrT') \arrow[r, "\disc"', bend left, shift left] \arrow[r, "\disc", dashed, bend right, shift right] & (\cX,\le^\dagger) \arrow[r, "\ppart"] & (\sP,\le),
\end{tikzcd}
\]
where $(X,\scrT,\scrT')$ is a bi-topological space, $(\cX,\le^\dagger)$ an antagonistic pre-order\index{Antagonistic pre-order}\index{Pre-order!antagonistic} and $(\sP,\le)$ a finite poset. The factorization via $(\cX,\le)$ and $(\cX,\le')$ is given by Diagram \eqref{ordpreforgrd}. Recall,
\begin{equation*}
\begin{tikzcd}[column sep=huge, row sep=large]
& (\cX,\le) \arrow[d, "\id"'] \arrow[rd, "\ppart"]                 &   \\
(X,\scrT,\scrT') \arrow[r, "\disc"', bend left=29, shift right] \arrow[r, "\disc", dashed, bend right=29, shift left] \arrow[ru, "\disc", shift left=2] \arrow[rd, "\disc"',  shift right=2] & (\cX,\le^\dagger) \arrow[r, "\ppart"]                                  & (\sP,\le) \\
& (\cX,\le') \arrow[u, "\id", dashed] \arrow[ru, "\ppart"', dashed] &  
\end{tikzcd}
\end{equation*}
We illustrate the discretization by considering one of the two topologies and the associated factorization:
\begin{equation}
\label{factorGen}
    \begin{tikzcd}
(X,\scrT) \arrow[r, "\disc", two heads] \arrow[rrr, "\tile", bend right, shift right=1] 
& (\cX,\le) \arrow[r, two heads] \arrow[r] \arrow[r, "\pi"] \arrow[rr, "\ppart", bend left, shift left=1] & (\cX/_\sim,\le) \arrow[r, "\varpi"] & (\sP,\le)
\end{tikzcd}
\end{equation}
As before the discretization map $\tile$ yields a chain-generated Cartan-Eilenberg system $\bfE^\tile$, cf.\ Sect.\ \ref{CanCETop} and Sect.\ \ref{tesschaincomp}. Under the assumption that $\disc$ is natural and yields finitely generated tessellar homology the same applies to $\bfE^\tile$.
Define, using the singular chain complex for $(X,\scrT)$,
\begin{equation}
    \label{tess456}
    C^\tile(X) := \bigoplus_{p\in \sP} 
    G_{p}C^\tile(X),
\end{equation}
where $G_{p}C^\tile(X) := H\bigl(F_{\downarrow p} X,F_{\downarrow p^\pred} X \bigr)$
if the latter is a free $R$-module, 
or
else choose a free differential module $\bigl(G_{p}C^\tile(X),\dff\bigr)$ such that 
\begin{enumerate}
    \item [(i)] $H\bigl( G_{p}C^\tile,\dff\bigr)\cong H\bigl(F_{\downarrow p} X,F_{\downarrow p^\pred} X \bigr)$;
    \item [(ii)] $G_{p} C^\tile(X) \cong R^{s_p+2r_p}$, cf.\ Defn.\ \ref{connmat4567}(ii).
\end{enumerate}
By Theorem \ref{prrepExis} there exists
an $\sO(\sP)$-filtered differential 
\[
\cm^\tile\colon  C^\tile(X) \to  C^\tile(X)
\]
such that 
$\bfE\bigl( C^\tile,\cm^\tile\bigr)\cong \bfE^\tile$, cf.\ Sect.\ \ref{excisiveprop}. In particular, $H^\tile(X)\cong H(X,\scrT)$.
The tessellar homology $H^\tile$ can also be defined for $(X,\scrT')$ and for $(X,\scrT^\dagger)$. These are different invariants for the  discretization $\tile\colon (X,\scrT,\scrT') \to (\sP,\le)$.
Since $\disc\colon (X,\scrT)\to (\cX,\le)$ is a natural discretization map we can utile the tessellar homology of $\disc$.
Consider the Cartain-Eilenberg system ${\bfE^\tile}$ where the $E$-terms are defined using the tessellar homology $H^\disc$ which isomorphic to the singular homology, i.e. $H^\disc(G_pX) \cong H\bigl(F_{\downarrow p} X,F_{\downarrow p^\pred} X \bigr)$.
The tessellar homology for $\disc$ is defined via a $\cX/_\sim$-graded tessellar module $C^\disc(X) =\bigoplus_{[\xi]\in \cccX/_\sim} G_{[\xi]} C^\disc(X)$.
Convex sets in $\sP$ are convex sets in $\cX$ and we define ${G_pC^\tile(X)}$ as $H^\disc(G_p X)$ is the latter is free. Otherwise choose ${G_pC^\tile(X)}$ as explained above.
In the case of field coefficient in $\K$ the homology $H^\disc$ is free and  there exists a strict
 $\sP$-graded differential module  $\bigl({C^\disc(X)},{\cm^\disc}\bigr)$, cf.\ Theorem \ref{prrepExis} and cf.\ \cite{robbin:salamon2}.

\begin{remark}
    \label{algconn12}
    In the case we use field coefficients $R=\K$
the algorithm \textsc{ConnectionMatrix}~\cite[Algorithm 6.8]{hms}, which is based on algebraic-discrete Morse theory, takes as input $\bigl( C^\disc(X),\cm^\disc\bigr)$ and outputs a strict $\sP$-graded differential module 
$\bigl({C^\tile(X)},{\cm^\tile}\bigr)$
which is $\sO(\sP)$-filtered chain equivalent to the differential module $\bigl( C^\disc(X),\cm^\disc\bigr)$ via  $\sO(\sP)$-filtered chain maps $h\colon {C^\tile(X)} \to C^\disc(X)$ and $h'\colon C^\disc(X) \to {C^\tile(X)}$.  The $\sO(\sP)$-filtered chain equivalence $h$ induces isomorphisms $\bigl\{h_{\beta\smin\alpha}\colon H^\tile(G_{\beta\smin \alpha}X)\to H^\disc (G_{\beta\smin \alpha}X)\mid \alpha,\beta\in \sO(\sP)\bigr\}$ which form an isomorphism of Cartan-Eilenberg systems:\footnote{cf.\ \cite[Eqn.\ (1.2)]{fran}.}
\[
\begin{tikzcd}[column sep=small, row sep=large]
 H^\disc(G_{\beta\smin \alpha}X) \arrow[r]                & H^\disc(G_{\gamma\smin \alpha}X) \arrow[r]                & H^\disc(G_{\gamma\smin \beta}X) \arrow[r]                & H^\disc(G_{\beta\smin \alpha}X)                 \\
 H^\tile(G_{\beta\smin \alpha}X) \arrow[r] \arrow[u, "h_{\beta\smin\alpha}"] & H^\tile(G_{\gamma\smin \alpha}X) \arrow[r] \arrow[u, "h_{\gamma\smin\alpha}"] & H^\tile(G_{\gamma\smin \beta}X) \arrow[r] \arrow[u, "h_{\gamma\smin\beta}"] & H^\tile(G_{\beta\smin \alpha}X) \arrow[u, "h_{\beta\smin\alpha}"] 
\end{tikzcd}
\]
for all $\alpha,\beta,\gamma \in \sO(\sP)$.  
\end{remark}

\begin{remark}\label{rem:quasiiso}
There is a special case when $\disc$ is a quasi-isomorphism, i.e.
\[H(\disc)\colon H(X)\to H(\cX)\]
induces an isomorphism $H(X,\scrT)\cong H(\cX,\leq)$, where the latter is taken to be the singular homology of the finite topological space $(\cX,\leq)$. This is a situation which commonly arises in practice, e.g. $\disc$ is a CW decomposition map and $(\cX,\leq)$ is a simplicial or cubical complex.\footnote{Or more generally, if each $f_\xi$ is a homeomorphism; such a CW complex is called regular.}  
In this case we interpret \eqref{factorGen} as
\[
\begin{tikzcd}
(X,\scrT) \arrow[rr, "\disc", "\cong"'] \arrow[rd, "\tile"'] &   & (\cX,\le) \arrow[ld, "\ppart"] \\
                                   & (\sP,\le) &                  
\end{tikzcd}
\]
Let  $C^\ppart(\cX)$ be the tessellar $\sP$-graded  differential module for $\ppart$, for which $\bfE^\ppart$ is isomorphic to $\bfE^\tile$ given by the $\sP$-graded differential module $C^\tile(X)$ (this follows from an elementary five lemma argument), q.v.\ Sect. \ref{postlude:discretization} for further discussion.
\end{remark}

After this general interlude of bi-topological discretzation we return to
the case $(X,\scrT,\scrTbf)$.
Assume that $(X,\scrT)$ is a finite regular CW-complex and let
 $\disc=\cell$, $\ppart = \dyn=\varpi$ and $\sP=\sSC =\cX/_{\sim^\dagger}$, cf.\ \eqref{specdiscdiag12}. 
If we use the discrete space $(\cX,\le^\dagger)$ then the convex sets in $\sSC$ are given by
$\cU\smin\cU'$, $\cU,\cU'\in \sO(\cX,\le^\dagger)$.
The homology $H^\ppart({\cU\smin\cU'})$ is well-defined and isomorphic to $H^\tile(G_{\cU\smin\cU'}X)$ per Remark \ref{rem:quasiiso}.
 The homology $H^\tile(G_{\cU\smin\cU'}X)$ is given by the relative homology $H(F_\cU X,F_{\cU'} X)$ and can be computed from a cellular chain complex.
Summarizing, we have:
\begin{theorem}
\label{tilehomtoBM}
Given the composition  $X\xrightarrow[]{\cell} \cX \xrightarrow[]{\dyn}\sSC$. Then,
\begin{equation}
\label{finhom12} 
H^\dyn(\cU\smin\cU') \cong H^\tile(G_{\cU\smin\cU'}X) \cong H^\cell(G_{\cU\smin\cU'}X)\cong H^\BM(G_{\cU\smin\cU'}X),
\end{equation}
where $H^\BM(G_{\cU\smin\cU'}X)$ is the Borel-Moore homology of the Morse tile $G_{\cU\smin\cU'}X$.
\end{theorem}
The homologies $H^\tile(G_{\cU\smin \cU'}X)$ are invariants of a Morse pre-order $(\cX,\le^\dagger)$.
We can visualize this structure by augmenting the di-graph for $(\sSC,\le)$ by the homology $H^\ppart(\cS) \cong H^\BM(G_\cS X)$ at the nodes of the graph, cf.\ Fig.\ \ref{fig:braid:cmg12}. The associated Cartan-Eilenberg system yields the homologies $H^\tile(G_{\cU\smin\cU'}X)$.
Such invariants can also be defined using the topologies $\scrTbf$ and $\scrT^\dagger$, cf.\ Sect.\ \ref{MDs}.
In the forthcoming sections we will implement these ideas for a large class of flows, so-called parabolic flows, and produce algebraic-combinatorial representations of its global dynamics.
\begin{equation}
    \label{specdiscdiag12}
    \begin{tikzcd}[column sep=huge, row sep=large]
& (\cX,\le) \arrow[d, "\id"'] \arrow[rd, "\dyn"]                 &   \\
(X,\scrT,\scrTbf) \arrow[r, "\cell"', bend left=29, shift right] \arrow[r, "\cell", dashed, bend right=29, shift left] \arrow[ru, "\cell", shift left=2] \arrow[rd, "\cell"',  shift right=2] & (\cX,\le^\dagger) \arrow[r, "\dyn"]                                  & (\sSC,\le) \\
& (\cX,\lebf) \arrow[u, "\id", dashed] \arrow[ru, "\dyn"', dashed] &  
\end{tikzcd}
\end{equation}

\begin{remark}
\label{tileviadisc}
The map $\tile \colon X\to \sSC$ is a discretization map and  the homology $H^\tile$ is the tessellar homology of the discretization map $\tile$. The two stage approach in this section allows us to compute $H^\tile$ from an $\sSC$-graded cellular chain complex.
For any discretization $\ind\colon \sSC\to \Z$ we can 
 consider $H^\tile$ as a bi-graded homology theory (in the case of field coefficients) as described in Section \ref{dimensiongrading}. The bi-graded Betti numbers/homology depend on our choice of the discretization $\ind$. A possible choice for $\ind$ is a linear extension of $\sSC$. In Section \ref{tesspara} we consider tessellar homology for parabolic 
 flows with 
 a linear discretization. 
 \end{remark}

Consider the diagram
\[
\begin{tikzcd}
(X,\scrT) \arrow[r, "\cell"] \arrow[rr, "\tile", bend right, shift right] \arrow[rrr, "\skel", bend right, shift right=2] & (\cX,\le) \arrow[r, "\dyn"] 
& (\sSC,\le) \arrow[r, "\ind"] & \Z
\end{tikzcd}
\]
where we treat $\tile$ as discretization map 
and $\ind\colon\sSC\to \Z$ is a linear discretization map.
Then for for any convex set $\cU\smin \cU'$ the Morse relations in Thoerem \ref{morserel1} are given by
\begin{equation}
    \label{morseforconvex}
 \sum_{\cS\in \cU\smin\cU'} P_{\lambda,\mu}^\tile(G_\cS X) = P^\tile_{\lambda,\mu}(G_{\cU\smin\cU'}X) + \sum_{r=1}^\infty (1+\lambda^r \mu) Q^r_{\lambda,\mu}.
\end{equation}
The poset $(\sSC,\le)$
yields a partial order on the pairs $\bigl(\cS,P^\tile_{\lambda,\mu}(G_\cS X)\bigr)$, $\cS\in \sSC$, via
\begin{equation}
    \label{popairs1}
    \bigl(\cS,P^\tile_{\lambda,\mu}(G_\cS X)\bigr) \le
    \bigl(\cS',P^\tile_{\lambda,\mu}(G_\cS' X)\bigr) \quad \Longleftrightarrow\quad \cS\le\cS'.
\end{equation}
The poset of pair $\bigl(\cS,P^\tile_{\lambda,\mu}(G_\cS X)\bigr)$ is
denoted by $(\tessph,\le^\dagger)$ and 
is called the \emph{tessellar phase diagram}\index{Tessellar phase diagram} of the Morse pre-order $(\cX,\le^\dagger)$.
Considering only non-trivial Poincar\'e polynomials yields the 
\emph{pure tessellar phase diagram}\index{Tessellar phase diagram!pure}\index{Pure tessellar phase diagram} $(\otessph,\le^\ddagger)$, with order-embedding
\begin{equation}
    \label{ordembfortessph}
(\otessph,\le^\ddagger) \hookrightarrow (\tessph,\le^\dagger),
\end{equation}
cf.\ Fig.\ \ref{morsetess2ex}.
The treatment of semi-flows up to this point is a tale of two topologies. In that setting the $p$-index in $\beta^\tile_{p,q}$ can be thought of as a manifestation of the topology induced by the semi-flow and the $q$-index as a manifestation of the space topology.

\section{Morse tessellations, Morse decompositions and the Conley index}
\label{MDs}
For a Morse pre-order $(\cX,\le^\dagger)$ we have a Morse tessellation as given in \eqref{Morsetessll11}. As a matter of fact the filtering $\cU \mapsto F_\cU X$, $\cU\in \sO(\cX,\le^\dagger)$ defines a finite, distributive lattice $\sN = \bigl\{F_\cU X\mid \cU \in \sO(\cX,\le^\dagger)\bigr\}\subset \sABlock(\varphi)$ of attracting blocks.\index{Attracting block} This yields the Morse tessellation $(\sT,\le)$ given by
%
\begin{equation}
    \label{morsetess1212}
    \sT(\sN):= \bigl\{T=U\smin U^\pred  \mid U\in \sJ(\sN) \bigr\}, \quad T\le T' ~~\Longleftrightarrow ~~U\subset U',
\end{equation}
where $T=G_\cS X$, $U=F_{\downarrow\cS} X$ and $U^\pred=F_{\downarrow\cS\smin \cS} X$, $\cS\in \sSC$, cf.\ \cite{lsa3}.
%
Such a tessellation does model the `direction' of dynamics but not the invariant dynamics.
However, the compactness of the phase space $(X,\scrT)$ does imply the existence of key invariant sets: attractors. Recall that a set $A\subset X$ is an \emph{attractor}\index{Attractor} if there exists an attracting block $U$ such that $A=\omega(U)$. The attractors of $\varphi$  form  a bounded, distributive lattice $\sAtt(\varphi)$ with binary operations $A\vee A' = A\cup A'$ and $A\wedge A' = \omega(A\cap A')$. The map
$U\mapsto \omega(U)$ is a surjective lattice homomorphism, cf.\ \cite{lsa}.
By compactness:
\[
U\in \sABlock(\varphi),~~ U\neq \varnothing,~~\implies ~~A=\omega(U) \neq \varnothing.
\]
In terms of the sublattice $\sN$ we obtain a sublattice $\omega\colon \sN \twoheadrightarrow \sA\subset \sAtt(\varphi)$, where $\sA := \bigl\{ A\in \sAtt(\varphi)\mid A=\omega(U),~U\in \sN\bigr\}$.
In general this map need not be a isomorphism which yields an inportant conclusion: knowing $\sN$ provides no insight into structure of $\sA$ from information given by $(\cX,\le^\dagger)$. 
However, topology can partly answer this question. From the map $\omega\colon \sN\twoheadrightarrow\sA$ we have the following congruence relation:
$U\sim U'$ if and only if $\omega(U) = \omega(U')$.
Since we cannot utilize this relation solely on the information given by $(\cX,\le^\dagger)$ we use a topological principle for flows also known as\index{Wazewski's principle} \emph{Wazewski's principle},\footnote{To establish Wasewski's principle the continuity of $\varphi$ is used a crucial way.} cf.\ \cite[Sect.\ II.2]{conley:cbms}. This can be restated as follows:
\begin{equation}
    \label{waz1}
    U\sim U' \quad\implies\quad H(U,U\cap U')\cong H(U\cup U',U')\cong 0,
\end{equation}
where $H$ is the singular homology functor. 
The key representation of dynamics is via a  \emph{tessellated Morse decomposition}\index{Tessellated Morse decomposition}\index{Morse decomposition!tessellated} which we define as the dual of the homomorphism $\omega\colon \sN\twoheadrightarrow\sA$. By Birkhoff duality we obtain an injective order preserving map $\sJ(\omega)\colon \sJ(\sA) \hookrightarrow \sJ(\sN)$. Invoking the Conley form we obtain an injection
$\pi\colon \sM(\sA) \hookrightarrow \sT(\sN)$,
where the poset $(\sM(\sA),\le)$, with $\sM(\sA) := \bigl\{ M= A\rmin A^\pred\mid A\in \sJ(\sA)\bigr\}$, is called a \emph{Morse representation}.\index{Morse representation} The notation
$A\rmin A^\pred:= A\cap (A^\pred)^*\neq \varnothing$ is the Conley form on $\sAtt(\varphi)$, cf.\ \cite{lsa3}. As before $M\le M'$ if and only if $A\subset A'$. \emph{The Morse sets in a Morse representation are never the empty set}!
Via Birkhoff duality we can given an explicit formula for the embedding $\pi$. However, from \cite{lsa3} there exists a left inverse to $\pi$: the unique image $M\mapsto T$ satisfies $\Inv(T) = M$,
where $\Inv(T)$ is the maximal invariant sets in $T$.
The latter provides a easy way to construct the embedding $\pi$. As before we do not have control over the poset set $\sM(\sA)$.
Consider a Morse set $M$ and $\pi(M) = T$. Then,
for any pair $U,U'\in \sN$ with $U\smin U'=\pi(M)$ we have $\Inv(U\smin U')=\omega(U)\rmin \omega(U') = M\neq \varnothing$ and $U\not\sim U'$.
For the latter we can now invoke Wazewski's principle in \eqref{waz1}:
\begin{equation}
    \label{waz2}
    H(U, U\cap U')\cong H(U\cup U',U')\neq 0 \quad \implies \quad M= \Inv(U\smin U')\neq  \varnothing.
\end{equation}
By construction $H^\tile(G_{\cU\smin\cU'}X) \cong H(F_\cU X,F_{\cU'} X)$ only depends on $\cU\smin \cU'$ which justifies the definition
\begin{equation}
    \label{CI}
    HC(T) := H^\tile(G_{\cU\smin\cU'}X),\quad T=F_\cU X\smin F_{\cU'} X,
\end{equation}
and is called the \emph{Conley index}\index{Conley index} of a Morse tile $T$, cf.\ \cite{conley:cbms}. 
In the case of a CW-decomposition via $\cell$ the Conley index $T$ is given as the Borel-Moore homology of the tile $T$.
The Conley index is an algebraic invariant for congruent  pairs $(U,U')$.
By construction $HC(T)\neq 0$ implies that $T$ is the image of a Morse set $M$ under $\pi$, i.e.
$M=\Inv(T) \neq \varnothing$.
The Conley index in \eqref{CI} is not only well-defined  for tiles in $\sT(\sN)$ but for any Morse tile $T=U\smin U'$ obtained
from attracting blocks $U,U'\in \sN$.\footnote{The above arguments apply to sublattices of $\sN$ such as $\varnothing \subset U\cap U' \subset U\subset X$, and $\varnothing \subset  U' \subset U\cup U'\subset X$. }


The algebraic topological approach in this chapter is to use invariants based on the topology $(X,\scrT)$. Wazewski's principle allows an interpretation of the invariants that yield information about the  invariant dynamics of the flow which defines the second topology $(X,\scrTbf)$. In the application of the theory to dynamical systems the topology $(X,\scrT)$ is assumed to be given while the second topology $(X,\scrTbf)$ is not known a priori. However, for the latter we have information about discretizations $(\cX,\lebf)$.
This track of combining information of two topologies and invoking Wazewski's principle yields a powerful algebraic topological tool for studying invariant sets of flows.
In Section \ref{doublegr} we also indicated that the tessellar homology can be defined for three topologies. The case outlined above is worked out.
Invariants based on $(X,\scrT,\scrTbf)$ and $(\cX,\le^\dagger)$ take into account the bi-topological nature of the problem. 
In future work we will examine these more detailed algebraic topological invariants.
For the application in the forthcoming chapter and applications to Conley theory the appoach in Section \ref{doublegr} suffices.

%% file: parabolic.tex
\chapter{Parabolic recurrence relations and flows}
\label{parabolic}
In the final part of this text we study a class of flows for which we demonstrate the discretization of topology and dynamics as described in the preceding chapters. 
The class of systems we consider are called \emph{discrete parabolic flows}. 
Discrete parabolic flows and associated {parabolic recurrence relations} occur in various applications of dynamical systems and represent  important classes of conservative dynamics as well as dissipative dynamics,  cf.\ \cite{bgvw,day,miro}.   
Examples of parabolic recurrence relations are  discretizations of uniform parabolic PDE's,   monotone twist maps and fourth order conservative differential equations, etc, cf.\ \cite{gv,im}. 
The nature of discrete parabolic flows makes them very suitable for displaying the theory developed in
this paper. 
The application to parabolic systems entails the definition of explicit CA-discretizations and MA-discretizations.
The algebraic methods reveal a new invariant for braids and parabolic flows, q.v.\ Sect.\ \ref{stabofclasses}.



\section{Discretized braid diagrams}
\label{subsec:discbr}
Braids\index{Braid} can be treated in various ways. One way is to regard braids as a path in a two dimensional configuration space. The more hands-on way to think of braids as a collection of `strands' between to copies of the Eucledian plane. A generic projection onto the strip $\R\times[0,1]$  contains all information by tagging intersections as postive, or negative. This representation is called a \emph{braid diagram}.\index{Braid diagram} In this text we consider a special class
braid diagrams, piecewise linear and with positive intersections.
From \cite{im} we recall the notion of closed  \emph{discretized braid daigram}.\index{Braid diagram!discretized}\index{Discretized braid daigram}
%



%
\begin{definition}
\label{PL}
The space of \emph{closed discretized period $d$  braid diagrams}\index{Braid diagram!closed discretized}\index{Closed discretized braid diagram} on $m$ strands,
denoted $\Conf^d_m$,\index{$\Conf^d_m$} is the space of  
unordered collections of \emph{strands}
$x=\{x^\alpha\}_{\alpha=0}^{m-1}$,
defined as follows:
%
\begin{enumerate}
\item[(i)]  (Strands):
	each strand  
	$x^\alpha=(x^\alpha_0,x^\alpha_1,\ldots,x^\alpha_d)\in\R^{d+1}$
	consists   of $d+1$ {\em anchor points} $x_j^\alpha$;
\item[(ii)] (Periodicity): $x^\alpha_d = x^{\theta(\alpha)}_0$
	for all $\alpha=1,\ldots,m$, for some permutation $\theta\in S_m$ (symmetric group);
\item[(iii)] (Non-degeneracy):
	for any pair of distinct strands $x^\alpha$ and $x^{\alpha'}$
	such that $x^\alpha_i=x^{\alpha'}_i$ for some $i$,
	the following \emph{ transversality} condition holds
	 \[
	 \bigl(x^\alpha_{i-1}-x^{\alpha'}_{i-1}\bigr)
	\bigl(x^\alpha_{i+1}-x^{\alpha'}_{i+1}\bigr) < 0.
	\]
\end{enumerate}
%
The elements $x \in \Conf_m^d$ are referred to as \emph{discretized braids}, or \emph{discretized braid diagrams}. 
The spaces $\Conf_m^d$ can be topologized as metric spaces, cf.\ \cite{im},
and 
the connected components of $\Conf_m^d$, called \emph{discrete braid classes},\index{Braid class}\index{Braid class!discretized}\index{Discretized braid class} are denoted by $[x]$.\index{$[x]$}
\end{definition}
\begin{remark}
    \label{1dbraids}
    Note that the space $\Conf^d_1$ consists of tuples $x=(x_0,\cdots,x_d)$, with $x_0=x_d$ and no additional conditions on $x$. Therefore, $\Conf^d_1\cong \R^d$ as metric spaces. 
\end{remark}

Let us start with an important invariant of discretized braids.
Given a discretized braid $x\in \Conf_m^d$.
Two of its strands $x^\alpha$ and $x^{\alpha'}$ \emph{intersect}\index{Intersection}\index{Intersection!transverse} if 
\begin{enumerate}
    \item [(i)] $\bigl(x_{i}^\alpha -x_{i}^{\alpha'}\bigr)\bigl(x_{i+1}^\alpha -x_{i+1}^{\alpha'}\bigr) <0$ for some $i$, or
    \item [(ii)] if for some $i$, $x_i^\alpha = x_i^{\alpha'}$ and
    $\bigl(x^\alpha_{i-1}-x^{\alpha'}_{i-1}\bigr)
	\bigl(x^\alpha_{i+1}-x^{\alpha'}_{i+1}\bigr) < 0$.
\end{enumerate}
Define   $\iota(x^\alpha,x^{\alpha'}) := \#\bigl\{\text{number of intersection between}$ $x^\alpha$ and $x^{\alpha'}\bigr\}$ as the local intersection number. Define the \emph{crossing number}\index{$\iota(x^\alpha,x^{\alpha'})$}\index{$\cross(x)$} by\index{Crossing number}\index{Local intersection number}
\[
\cross(x) := \frac{1}{2}\sum_{\alpha\not=\alpha'}  \iota(x^\alpha,x^{\alpha'}) \in \N.
\]
The crossing number is an invariant for a braid classes $[x]$, i.e. $\cross$ is constant on components $[x]\subset \Conf_m^d$.
\begin{remark}
\label{genin}
Generically 
a discrete braid has the property $x_i^\alpha \neq x_i^{\alpha'}$ for all $i$ and for all $\alpha \neq \alpha'$. 
The local intersection number can be defined for generic braids by counting sign changes, i.e. indices $i$ for which (i) is satisfied.
\end{remark}

Unordered sets $x=\{x^\alpha\}_{\alpha=1}^m$
 for which Definition \ref{PL}(iii) is not satisfied are called \emph{singular braids}\index{Singular braid}\index{Braid!singular} and are denoted by $\Sing_m^d$.\index{$\Sing_m^d$} 
 Pairs of strands for which Definition \ref{PL}(iii) is not satisfied are called \emph{non-transverse}\index{Intersection!non-transverse} and the crossing number is not defined in this case.
We can however consider a variation on the crossing number that is defined for both discretized braids and singular braids.
Let $x\in \Sing_m^d$ be a singular braid. Then,  
following \cite{FuscoOliva}, we define the set
\[
\scrS_\epsilon(x) := \Bigl\{\tilde x\in \Conf_m^d \mid  |\tilde x_i^\alpha-x_i^{\alpha}|<\epsilon,~~\forall i \hbox{~and~} \forall \alpha \Bigr\}\neq \varnothing.
\]
The crossing numbers
\[
\cross_-(x) := \min_{\scrS_\epsilon(x)}\cross
\quad \text{and}\quad \cross_-(x) := \max_{\scrS_\epsilon(x)}\cross
\]
are independent of $\epsilon$ provided $\epsilon>0$ is sufficiently small and therefore are well-defined.
For $x\in \Conf_m^d$ the numbers $\cross_-(x)$ and $\cross_+(x)$ are also defined  in which case $\cross_-(x)=\cross_+(x)=\cross(x)$, cf.\ \cite{FuscoOliva}.\index{$\cross_+(x)$}\index{$\cross_-(x)$}

For a discrete braid $x$  its \emph{braid components}\index{Braid component} are given by the cycles of the permutation $\tau$. The the orders of the cycles which we refer to as the \emph{cycle orders}\index{Cyclic order} is another invariant for a braid class $[x]$. For example for $\theta\in S_5$ given by $\theta = (01)(234)$ the cycle orders are $2$ and $3$. We can define a special space of braids by coloring components.
In this text  we are interested in particular in braids with dual  coloring.
The space of \emph{2-colored discretized braids}\index{Colored discretized braid}\index{Braid!2-colored discretized} $\Conf_{n,m}^d$ consists of ordered pairs $\bigl(x,y\bigr)$, where $x=\{x^\alpha\}_{\alpha=0}^{n-1}\in \Conf_n^d$ and $y=\{y^\beta\}_{\beta=0}^{m-1}\in \Conf_m^d$, and $(x,y)$ satisfies Definition \ref{PL}(i)-(iii).
In other words $(x,y)\in \Conf_{n+m}^d$. We denote a 2-colored discretized braid by $x\rel y$.
The canonical projection $\pi\colon \Conf_{n,m}^d \to \Conf_m^d$ given by $x\rel y \mapsto y$ yields 
the fibers
 $\Conf_n^d\rel y = \pi^{-1}(y)$. A connected component $[x]\rel y$ is called a \emph{(discrete) relative  braid class component}\index{Relative  braid class component} with skeleton $y$. 
A connected component $[x\rel y]$ in $\Conf^d_{n+m}$ is called a 
\emph{(discrete) relative braid class}.\index{Relative braid class} The fibers
$\pi^{-1}(y)\cap [x\rel y]$ consist of relative braid class components $[x]\rel y$ and are referred to as the \emph{(discrete) relative braid class fibers}\index{Relative braid class fiber} of $[x\rel y]$. In most situations a braid class fiber consists of a single braid class component.
 
\begin{definition}
\label{properbr}
A relative braid class $[x\rel y]$ is called \emph{non-degenerate}, or \emph{proper}
if the cycle orders in $x$ differ form the cycle orders in $y$.\index{Braid!non-degenerate}\index{Braid!proper}\index{Non-degenerate braid}\index{Proper braid} 
In particular, for $x\rel y\in \Conf^d_{1,m}$ is proper if the cycle orders in $y$ are all larger than $1$.   In latter case the \emph{skeleton} $y$ is also called \emph{proper}.\index{Proper braid}\index{Proper skeleton}\index{Skeleton!proper}
\end{definition}
If $[x\rel y]$ is proper, then its fibers $\pi^{-1}(y)\cap [x\rel y]$ are and therefore also its components $[x]\rel y$ are proper as well.
For a skeleton $y\in \Conf_m^d$ we consider
the singular braids in $x\rel y\in \Sing_{n,m}^d$.
Denote the fiber of singular relative braids by $\Sing_n^d \rel y
=\pi^{-1}(y)$.

\section{Parabolic flows}
\label{paraflows123}
Discretized braids introduced in the previous section are intimately related to a class of recurrence relations.
\begin{definition}
A {\em parabolic recurrence relation}\index{Parabolic recurrence relation} (of period $d>0$) is a system of equations of the form
\begin{equation}\label{eqn:parabolic}
R_i(x_{i-1},x_i,x_{i+1}) = 0,\quad  i\in \Z,
\end{equation}
where each $R_i\colon \R^3\to \R$ is a smooth function such that 
\begin{enumerate}
    \item [(i)] $R_{i+d}=R_i$ for all $i$;
    \item [(ii)] $\partial_1 R_i>0$ and $\partial_3 R_i> 0$.\footnote{One can weaken the monotonicity with one of the inequalities to be $\ge$ for every $i$.
    For convenience in this paper we assume strict inequalities for both partial derivatives unless indicated explicitly.} 
\end{enumerate}
We denote the recurrence relation by $\mathbf{R} = (R_i)$.
\end{definition}
The following proposition establishes periodic solutions of parabolic recurrence relations as discretized braid diagrams. At a latter stage we also define associated flows which yields an even stronger symbiosis between parabolic recurrence relations and discretized braids.
\begin{proposition}[cf.\ \cite{im}]
Let $x=\{x^\alpha\}_{\alpha=0}^{m-1}$ be a set of strands satisfying Definition \ref{PL}(i)-(ii) with the property that $R_i(x_{i-1}^\alpha,x_i^\alpha,x_{i+1}^\alpha)=0$, for all $i$ and all $\alpha$. Then, $x\in \Conf_m^d$, i.e. Definition \ref{PL}(iii) is also satisfied. Such a discretized braid is called a stationary braid with respect to \eqref{eqn:parabolic}.\index{Braid!stationary}\index{Stationary braid}
In particular, the crossing number of $x$ is well-defined.
\end{proposition}

Associated to a parabolic recurrence relation we consider the following system of differential equations:
\begin{equation}
    \label{paraboliceq1}
    \dot x_i = R_i(x_{i-1},x_i,x_{i+1}),\quad  i\in \Z.
\end{equation}
The solution operator as well as the system of  differential equations will be referred to as a \emph{discrete parabolic flow}.\index{Discrete parabolic flow}\index{Parabolic flow}
The objective is to find $k$-periodic solutions of parabolic recurrence relations, i.e.  sequences $(x_i)$, with $x_{i+kd} = x_i$ for some $k\in \N$, which satisfies Equation \eqref{eqn:parabolic}. In order to build a suitable theory for periodic solutions we use the concept of discretized braids, cf.\ \cite{im}
Multiple  periodic sequences with possible different periods
may be regarded as a {discretized braid diagram}.

The parabolic equation 
 in \eqref{paraboliceq1} defines a \emph{local}\index{Flow!local} flow $\varphi$ on
 the space of 1-periodic sequences $\Conf_1^d\cong \R^d$ with the standard metric topology. We refer to $\varphi$ as a \emph{parabolic flow}.
 We say that a braid $y\in \Conf_m^d$ is a \emph{skeletal braid}\index{Skeleal braid}\index{Braid!skeletal} for $\varphi$ if Equation \eqref{eqn:parabolic} is satisfied for all $y^\alpha \in y$. Shorthand notation $\mathbf{R}(y)=0$ and $y$ is also referred as a \emph{skeleton}\index{Skeleton} for $\varphi$. Recall that a skeleton for which the cycle orders are all larger than 1 is called a \emph{proper skeleton}.\index{Proper skeleton}\index{Skeleton!proper} Relative braids $x\rel y \in \Conf_1^d\rel y$ for which $y$ is proper are proper as relative braids.
 In this case there is an important relation between parabolic dynamics and the crossing number invariants.\index{Proper relative braid}\index{Relative braid!proper}

\begin{proposition}[cf.\ \cite{im}, \cite{FuscoOliva}]
\label{lapLyap}
Let $y\in \Conf_m^d$ be a proper skeleton (stationary braid) for $\varphi$, i.e. all cycle orders are strictly larger than 1, and let 
$x\rel y\in \Sing_1^d\rel y$. Then, for $\epsilon>0$ sufficiently small, we have that 
\begin{enumerate}
    \item[(i)] $\varphi(t,x)\rel y \in \Conf_1^d\rel y$ for all $0<|t|\le \epsilon$;
    \item[(ii)]  for all
    $-\epsilon \le t_- < 0 < t_+ \le \epsilon$ it holds that
    \[
    \begin{aligned}
    \cross_+(x\rel y) &=\cross\bigl(\varphi(t_-,x)\rel y\bigr) \\
    &> \cross\bigl(\varphi(t_+,y)\rel y\bigr) =\cross_-(x\rel y).
    \end{aligned}
    \]
\end{enumerate}
\end{proposition}
\begin{proof}
Denote by $k\#y$ the $k$-fold covering of $y$, i.e. we take $k$ concatenated copies of $y$. 
The braid $y$ gives a permutation $\theta$ on the on the symbols $\{1,\cdots,m\}$, cf.\ Defn. \ref{PL}(ii). Choose $k$ to be order of the permutation $\theta$. Then,
 $k\# y$ consists of 1-periodic sequences for $i\in \{0,\cdots, kd\}$.
Moreover, 
\begin{equation}
    \label{concat13}
\cross\bigl(k\# x \rel k\# y\bigr) = k \cross\bigl(x\rel y \bigr).
\end{equation}
Since $R_{i+d} = R_i$ we have that the flow $\varphi^k$ generated by $\cR$ on $\Conf_1^{kd}$ is given by the $k$-fold covering if we choose $x\in \Conf_1^d$, i.e. 
$
\varphi^k(t,k\# x) = k\# \varphi(t,x).
$
If $x\in \Sing_1^d\rel y$, then $k\# x\rel k\# y\in \Sing_1^{kd}\rel k\# y$.
Let $y^\alpha$ be a strand such $x$ and $y^\alpha$ are non-transverse. Since all cycle orders of $y$ are larger than 1 we have that all relative  braids $x\rel y$ are proper and thus
$x-y^\alpha\neq 0$ for all $\alpha$ (strands don not coincide).
By the main result in \cite{FuscoOliva} this implies that
$\iota\bigl(\varphi^k(t_-,x),k\# y^\alpha \bigr)
> \iota\bigl(\varphi^k(t_+,x),k\# y^\alpha \bigr)$
for all
    $-\epsilon \le t_- < 0 < t_+ \le \epsilon$.
If we combine this with the crossing number for $k$-fold coverings we obtain
\[
\sum_{\alpha}\iota\bigl(\varphi^k(t_-,x),k\# y^\alpha \bigr)
> \sum_{\alpha}\iota\bigl(\varphi^k(t_+,x),k\# y^\alpha \bigr),
\]    
which implies that $\cross\bigl(\varphi^k(t_-,x)\rel k\# y\bigr) > \cross\bigl(\varphi^k(t_+,x)\rel k\# y\bigr) $
and thus 
\[
\cross\bigl(k\#\varphi(t_-,x)\rel k\# y\bigr) > \cross\bigl(k\#\varphi(t_+,x)\rel k\# y\bigr).
\]
 Property \eqref{concat13} then gives
\[
\cross\bigl(\varphi(t_-,x)\rel y\bigr) > \cross\bigl(\varphi(t_+,x)\rel  y\bigr).
\]
From \cite[Thm.\ 1(ii)]{FuscoOliva} it also follows that 
$\cross\bigl(\varphi(t_-,x)\rel y\bigr) = \cross_+(x\rel y)$ and
$\cross\bigl(\varphi(t_+,x)\rel  y\bigr) = \cross_-(x\rel y)$, which completes the proof
\end{proof}

As a consequence of the above proposition we conclude that $\cross_+(\cdot\rel y)$ is a   discrete Lyapunov function for $\varphi$ and the value of the Lyapunov function strictly drops at singular relative braids.\index{Lyapunov function}

\begin{remark}
 If we do not require the cycle orders for $y$ to be strictly larger than 1 then if $x$ may coalesce with the 1-periodic strands in $y$ in which case associated the singular braid is stationary and Proposition \ref{lapLyap}(ii) does not hold in that case. For application of these technique for improper braid classes recall \cite{im,miro}.
\end{remark}


Besides crossing numbers and cycle orders 
another invariant for 
relative braid classes can be defined, cf.\  \cite{im}, and which is of  algebraic topological nature.
In this case we assume that the components $[x]\rel y$  are bounded as sets in $\R^d$.
In Section \ref{somebrth} we provided an extensive account of the algebraic invariant which is also referred to as  the \emph{Braid Conley index}.\index{Braid Conley index}\index{Conley index}
The theory in \cite{im} implies that the invariants are homotopy invariants, i.e.\
a homotopy $h_s(y)$ in $\Conf_m^d$ yields 
isomorphic invariants.
The latter is useful for choosing convenient representatives $y$  for studying the  relative braid class fibers in $\Conf_1^d\rel y$.

\section{Closure algebra discretizations for parabolic flows}\label{sec:parabolic:model}
In this section we assume that a skeleton $y\in \Conf_m^d$ always consist of two extremal strands: define $y^-=y^0$ and $y^+=y^{m-1}$ such that the remaining strands satisfy
$
y^-_i < y_i^\alpha < y_i^+,
$
for all $\alpha=1,\cdots, m-2$. The latter collection of strands is denoted by $\mathring y$. 
The skeleton $y$ now induces a bounded cubical complex as will be explained in Section \ref{metrcidisc}.
\begin{definition}
\label{spaceandflow}
Let $y\in \Conf_m^d$ be as described above and assume that $\mathring y$ is proper, cf. Defn.\ \ref{properbr}.
We consider points $x$ in 
\begin{equation}
    \label{thespace}
X= \bigl\{x\in \Conf_1^d \mid y_i^-\le x_i\le y_i^+\bigr\}\subset \R^d,
\end{equation}
which is a compact metric space with metric topology induced by $\R^d$, cf.\ Rem.\ \ref{1dbraids}.
Let $\varphi$ be a parabolic (local) flow on $X$ generated by Equation \eqref{paraboliceq1} where the parabolic recurrence relation $\mathbf{R}$ satisfies $\mathbf{R}(y) = 0$.
\end{definition}
Regard $\bar y=\{y^-,y^+\}$ as sub-skeleton.
For $x$ equal to either $y^-$ or $y^+$ the local flow $\varphi$ is stationary. If $x_i=y^-_i$, or $x_i=y^+$ for at least one $i$, i.e. $x\in \partial X$ as subset of $\R^d$, then $x\rel \bar y\in \Sing^d_1\rel \bar y$.
Proposition \ref{lapLyap} then implies that $\cross\bigl( \varphi(t,x)\rel \bar y\bigr)
=\cross(x\rel \bar y) =0$, which yields $\varphi(t,x)\in X$ for all $t\ge 0$.
If we invoke the remaining strands in $y$ the semi-flow $\varphi$ will display monotone behavior with respect to relative braid classes in the spirit of a Morse tessellation as is described in the forthcoming sections. In practical terms $\varphi(t,x)\rel y$ is generically contained in a relative braid and $\varphi(t,x)\rel y$ can evolve from one braid class to the next and not return. 
Summarizing, the parabolic flows satisfy the following properties:
\begin{enumerate}
    \item [(i)] $\varphi\colon \R^+\times X\to X$ is smooth semi-flow on $X$;
    \item[(ii)] the braid $y\in \Conf_m^d$ is a skeleton for $\varphi$.
\end{enumerate}
We now study parabolic flows on $X$ for a fixed skeleton $y\in \Conf_m^d$. From this point on we assume that $y$ is a skeleton as described above with $\mathring y$ proper.

\begin{remark}
\label{allproper}
The boundary strands $y^-$ and $y^+$ model boundary conditions on the parabolic flow pushing in. Using alternating strands we can also model boundary condition pushing out, cf. \cite{im}.
Variations on push-in or push-out boundary conditions can be obtained by different combinations of multi-strand braids. In this paper we restrict to push-in boundary conditions modeled via the strands $y^-$ and $y^+$.
%
\end{remark}


\subsection{Discretization of the metric topology on $X$}
\label{metrcidisc}
We define a cubical complex $\cX$ for $X$ as follows.
From the definition of $X$ there exists a natural grading on $X$ via co-dimension of tangencies. Define
\[
\begin{aligned}
G_d X &:= \{x\mid x_i \neq y_i^{\alpha_i}, \forall i,\alpha\};\\
G_{d-k} X&:= \{ x\mid x_{i_1}=y_{i_1}^{\alpha_{i_1}},\cdots, x_{i_k}=y_{i_k}^{\alpha_{i_k}},~ i_1\neq\cdots\neq i_k\},
\end{aligned}
\]
where $k=1,\cdots,d$. Note that the indices $\alpha_{i_j}$ are not necessarily distinct. If
$\alpha=\alpha_{i_1}=\cdots = \alpha_{i_d}$, then $\alpha \in \{0,m-1\}$ by the definition of $X$ and the assumption that $\mathring y$ is proper.
The top cells $\xi \in \cX^\topc =: G_d\cX$ label
the connected components of the set  $G_d X$, which we refer to as \emph{generic braids}.\index{Braid!generic} 
The $(d-k)$-dimensional cells $\xi\in G_{d-k}\cX$ label
the connected components of the set $G_{d-k} X$,
cf.\ Fig.\ \ref{fig:parabolic:lap}(a). All cells realize as open rectangular cuboids $|\xi|$ in $X$ and are thus homeomorphic to open $k$-balls in $\R^k$, $k>0$.
The set $\cX$ of cells is forms a CW-decomposition of $X$, cf.\ Sect.\ \ref{CW}. The map 
\begin{equation}
    \label{paracell}
\cell\colon X \to \cX, \text{~~defined by ~~} \cell(x) = \xi,\text{~~ for ~~} x\in |\xi|,
\end{equation}
is a CW-decomposition map and thus Boolean.
The co-dimension provides the dimension grading of $\cX$: 
\[
\dim\colon X \to \cX, \text{~~with~~}
\dim \xi = q \text{~~ if and only if ~~} |\xi|\in G_qX.
\]
The face partial order on $\cX$ defines the discrete closure operator $\cl~\xi = \big\downarrow \xi$. The triple $(\cX,\cl,|\cdot|)$ is a CA-discretization for $X$ and the
  finite algebra $(\sSet(\cX),\cl)$ is the associated closure algebra discretezation of $(\sSet(X),\cl)$.

Note that generically a relative braid $x\rel y\in \Conf_1^d\rel y$
is a point in a top cell $|\xi|$, $\xi\in \cX^{\topc}$.
This justifies the notation $\cross(\xi)$ as the crossing number of a top cell.
The same applies to the crossing numbers $\cross_-$ and $\cross_+$, i.e.
$\cross_-(\xi) :=\min\bigl\{\cross(\eta)\mid \eta\in \st \xi\cap \cX^{\topc}\bigr\}$ and $\cross_+(\xi) := \max\bigl\{\cross(\eta)\mid \eta\in \st \xi\cap \cX^{\topc}\bigr\}$.
Define  the following combined crossing number function on
$\cX$:
\[
\Lap\colon \cX \to \N\times \N,\quad \xi \mapsto \Lap(\xi) := \bigl(\cross_-(\xi),\cross_+(\xi)\bigr)
\]
\begin{figure}[h!]
\begin{minipage}{.3\textwidth}
\begin{center}
\begin{tikzpicture}[dot/.style={draw,circle,fill,inner sep=.75pt},line width=.7pt,scale=.8]
\foreach \x in {0,1,2}
     \draw[thin,gray,->] (\x, -.25) -- (\x, 4);
\foreach \x in {0,1,2}
    \foreach \y in {0,.5,1,1.5,2,2.5,3,3.5}
        \node (\x\y) at (\x, \y)[dot] {};
\foreach \y in {0,3.5}
    \foreach \x  in {0,1}
        \draw (\x,\y) to (\x+1,\y);
\foreach \x in {0} {
     \draw (\x, .5) to (\x+1, .5);
     \draw (\x, 1) to (\x+1, 2);
     \draw (\x, 1.5) to (\x+1, 1);
     \draw (\x, 2) to (\x+1, 3);
     \draw (\x, 2.5) to (\x+1, 2.5);
     \draw (\x, 3) to (\x+1, 1.5);
     }
\foreach \x in {1} {
     \draw (\x, .5) to (\x+1, 1.5);
     \draw (\x, 1) to (\x+1, 1);
     \draw (\x, 1.5) to (\x+1, 2.5);
     \draw (\x, 2) to (\x+1, .5);
     \draw (\x, 2.5) to (\x+1, 2);
     \draw (\x, 3) to (\x+1, 3);
     }
\end{tikzpicture}
\end{center}
\end{minipage}
\begin{minipage}{.3\textwidth}
\begin{center}
\begin{tikzpicture}[dot/.style={draw,circle,fill,inner sep=.75pt},line width=.7pt,scale=.75]
\foreach \x in {0,1,2}
    \draw[thin,gray,->] (\x, -.25) -- (\x, 4);
    
\foreach \x in {0,1,2}
    \foreach \y in {0,.5,1,1.5,2,2.5,3,3.5}
        \node (\x\y) at (\x, \y)[dot] {};
\foreach \y in {0,3.5}
    \foreach \x  in {0,1}
        \draw (\x,\y) to (\x+1,\y);
\foreach \x in {0} {
     \draw (\x, .5) to (\x+1, .5);
     \draw (\x, 1) to (\x+1, 2);
     \draw (\x, 1.5) to (\x+1, 1);
     \draw (\x, 2) to (\x+1, 3);
     \draw (\x, 2.5) to (\x+1, 2.5);
     \draw (\x, 3) to (\x+1, 1.5);
     }
\foreach \x in {1} {
     \draw (\x, .5) to (\x+1, 1.5);
     \draw (\x, 1) to (\x+1, 1);
     \draw (\x, 1.5) to (\x+1, 2.5);
     \draw (\x, 2) to (\x+1, .5);
     \draw (\x, 2.5) to (\x+1, 2);
     \draw (\x, 3) to (\x+1, 3);
     }

\node (x1) at (-.35, 3.85) {$x_0$}; 
\node (x2) at (.65, 3.85) {$x_1$}; 
\node (x) at (-.25, 1.75) {$x$};

\node[fill=red,draw=red] (f1) at (0,1.75)[dot] {};
\node[fill=red,draw=red] (f2) at (1,2.25) [dot] {};
\node[fill=red,draw=red] (f3) at (2,1.75)[dot] {};
\draw[red] (f1) -- (f2) -- (f3);
\end{tikzpicture}
\end{center}
\end{minipage}
\begin{minipage}{.3\textwidth}
\begin{center}
\begin{tikzpicture}[dot/.style={draw,circle,fill,inner sep=1pt},line width=.7pt,scale=1]
\draw[->,thin] (-.5,-.5) -- (.5,-.5);
\draw[->,thin] (-.5, -.5) -- (-.5,.5);
\node (u) at (.75,-.5) {$x_0$}; 
\node (u) at (-.5,.75) {$x_1$}; 
\draw[densely dotted] (0,0) -- (2,0) -- (2,2) -- (0,2) -- (0,0);
\node (f2) at (1,1)[dot,fill=red,draw=red] {};
\node (u) at (1.2,.8) {$x\rel y$};
\node (u) at (1,2.25) {$[x]\rel y$};
\end{tikzpicture}

\end{center}
\end{minipage}
\caption{Skeleton $y$ [left] and $y$ with free strand $x$ (in red) [middle]. A relative braid class component $[x]\rel y$ which is a top-cell in $\cX$ [right].}
\label{fig:braidclass}
\end{figure}

\begin{remark}
\label{canrep1}
It is often convenient to use the following normal form\index{Skeleton!normal form} for
the skeleton $y$: for any fixed $i$ the cross-section $(y_i^0,y_i^1,\ldots,y_i^{m-1})$ is a permutation of $\{0,1,\ldots,m-1\}$.  That is, the anchor points $(y_i^\alpha)$ are integers and take unique values between $0$ and $m-1$.  This implies that the pairs $(i,y_i^\alpha)$ lay on the integer lattice within the box $[1,d+1]\times [0,m-1]$.  The cubical complex $\cX$ is comprised of
cells $\xi$ which we define as follows.
Consider sets $I^1_i = (i,i+1)\subset \R$ for for $i\in \{0,\cdots,m-2\}$
and sets $I^2_i = \{i\}$, for $i\in \{0,\cdots,m-1\}$.
A cell $\xi$ is given as
\begin{equation}
    \label{cellrepr12}
|\xi| 
= I^{j_1}_{i_1}\times I^{j_2}_{i_2}\times \ldots
\times I^{j_d}_{i_d},
\end{equation}
where $j_k\in \{1,2\}$, $i_k \in \{0,\cdots,m-2\}$ for $I^1_{i_k}$ and
$i_k \in \{0,\cdots,m-1\}$ for $I^2_{i_k}$.
Figure \ref{fig:parabolic:lap}(a) below shows a skeleton $y$ in normal form and Figure \ref{fig:parabolic:lap}(b) depicts the cubical complex consisting of unit squares. 
We can now use the cubical complex to describe a normal form of the braid component $[x]\rel y$ with $y$ in normal form.
The dimension of a cell $\xi$ is given by $\dim \xi = \sum_{j_k\not = 2} j_k$, where the integers are determined by the representation of $|\xi|$ in \eqref{cellrepr12}.
\end{remark}

  Figure  \ref{fig:braidclass}[left] above is an example of a skeleton $y\in \Conf_8^2$. 
  Figure  \ref{fig:sec112}[left] displays the CW-decomposition in terms of cubes.
  Figure  \ref{fig:sec112}[left] also gives the crossing numbers $\cross(\xi)$ for the top cells $\xi\in \cX^\topc$.

\subsection{Discretization of  flow topologies}
\label{fltodisc}
In the previous section we described a natural CW-decomposition to discretize the metric topology on $X$.
We now define a second discrete topology on $\cX$ via a pre-order such that
 the map $\Lap\colon \cX\to \N\times \N$ is order-preserving and which serves the purpose of discretizing the \bflt $\scrTbf$.
 The order on $\N\times \N$ is the product order, i.e. $(a,b)\le (a',b')$ if and only if $a\le a'$ and $b\le b'$.
We build a pre-order for discretizing $\scrTbf$ in  two steps:
\begin{itemize}
    \item[(i)] Define the (symmetric) relation $\cE\subset \cX\times \cX$ as the \emph{partial adjacency relation}
    \[
        (\xi,\xi'), (\xi',\xi)\in \cE \iff  \xi'\in \cX^{\topc}\text{~~and~~}  \xi\in \cl~\xi'.\footnote{In term of the face partial order this reads $\xi\le \xi'$.}
\] 
    \item[(ii)] Define the \emph{discrete flow relation} $\cR\subset \cE$ as follows:  
\[
(\xi,\xi')\in \cR  \iff   \Lap(\xi)\leq \Lap(\xi')
\]
\end{itemize}
By construction $\Lap$ is order-preserving.
Observe that $(\xi,\xi)\in \cR$ if and only if $\xi \in \cX^{\topc}$.
%
Since every relative braid and relative singular braid is associated to a unique cell $\xi$ we devide up the cells in $\cX$ into two groups, the \emph{regular cells} $\xi\in \cXr$ and the \emph{singular cells} $\xi\in \cXs$ which may be characterized as follows:
\[
\begin{aligned}
\xi\in \cXr,\quad&\text{if and only if}\quad \cross(\xi)=\cross_-(\xi) = \cross_+(\xi);\\
\xi\in \cXs,\quad&\text{if and only if}\quad \cross_-(\xi) < \cross_+(\xi).
\end{aligned}
\]
The transitive closure of $\cR$ yields the discrete derivative operator 
$\derr^+:= (\cR^{\scaleto{\bm{+}}{4pt}})^{-1}$
and the transitive, reflexive closure $\cR^{\bm +=}$ defines the pre-order $\le^+$ and yields the discrete closure operator 
$\ccl^+:= (\cR^{\scaleto{\bm{+=}}{4pt}})^{-1}$,
and  $\ccl^+ = \id \cup \derr^+$. 

\begin{lemma}
\label{closappr1}
The discrete closure operator $\ccl^+$ defined above satisfies the continuity condition $\cl^+|\cU|\subset |\ccl^+\cU|$ for all $\cU\subset \cX$.
\end{lemma}
\begin{proof}
By Remark \ref{discotherfltop} 
it suffices to show that $\varphi(t,x)\in |\ccl^+\xi|$ for all $t\ge 0$, for all $x\in |\xi|$ and for all $\xi\in \cX$.
Let $\xi\in \cX$, then $\varphi(t,x)\in |\st\xi|$ for all $0\le t\le \tau_x$ for some $\tau_x$ sufficiently small. If $\xi\in \cXr$, then $\st\xi\subset \ccl^+\xi$\footnote{Here $\st$ is associated with the discrete topology $(\ccX,\le)$ given by the CW-decomposition.} and if $\xi\in \cXs$, then $\st\xi\not\subset \ccl^+\xi$.
In the former case $\varphi(t,x)\in |\st\xi|\subset |\ccl^+\xi|$ for all $0\le t\le \tau_x$ for some $\tau_x$ sufficiently small. 
%
%
In the latter case we argue as follows. 
From Proposition \ref{lapLyap} we have that 
$\varphi(-t,x)\in |\xi^-|$,\footnote{A cell $\xi^-$ does not exist when $\xi$ corresponds to a boundary cell for $X$.} $\varphi(t,x)\in |\xi^+|$, for all $0<t\le \tau_x$, with $\xi^-,\xi^+\in \st\xi\cap \cX^{\topc}$.
Moreover, $\cross(\xi^-) = \cross_+(\xi)$ and $\cross(\xi^+) = \cross_-(\xi)$.
This implies that 
\begin{equation}
    \label{Laps}
\Lap(\xi^-) > \Lap(\xi) = \bigl( \cross_+(\xi),\cross_-(\xi)\bigr) > \Lap(\xi^+).
\end{equation}
This proves that $(\xi^+,\xi) \in \cR$ and thus $\varphi(t,x)\in  |\ccl^+\xi|$ for all $0\le t\le \tau_x$ for some $\tau_x$ sufficiently small. 

The above arguments show that for every $x\in X$ the flow satisfies $\varphi(t,x)\in |\eta|$ for all $0<t\le \tau_x$ for some $\eta\in \cXr$.
Let $t'>\tau_x$ be the first time that $\varphi(t',x)\in |\zeta|$ for some $\zeta\in \cXs\cap \ccl~\eta$. Then, by \eqref{Laps},
$\Lap(\eta)>\Lap(\zeta)$ and thus $(\zeta,\eta)\in \cR$ and consequently $\zeta\in \ccl^+ \xi$.
We can repeat the above argument  to conclude that 
the criterion in Remark \ref{discotherfltop} 
is satisfied for all $t\ge 0$.
%
%
\end{proof}
\begin{figure}[hbt]
\begin{minipage}{.3\textwidth}
\centering
\begin{tikzpicture}[dot/.style={draw,circle,fill,inner sep=.75pt}, scale=.45]
\draw[step=1cm,gray,very thin] (0,0) grid (7,7);

{\tiny 

\foreach \x in {.5} {
    \foreach \y [evaluate = \y as \z using int(2*(\y-.5))]  in {.5, 1.5, 2.5, 3.5, 4.5, 5.5, 6.5} {
        \node (\x\y) at (\x, \y) {$\z$}; }
    \foreach \y [evaluate = \y as \z using int((12-2*(\y-.5)))]  in {.5, 1.5, 2.5, 3.5, 4.5, 5.5, 6.5} {
        \node (\x\y) at (\x+6, \y) {$\z$}; }
}
\foreach \y in {.5} {
     \foreach \x [evaluate = \x as \z using int(2*(\x-.5))]  in {1.5, 2.5, 3.5, 4.5, 5.5} {
        \node (\x\y) at (\x, \y) {$\z$};  }
    \foreach \x [evaluate = \x as \z using int(12-2*(\x-.5))]  in {1.5, 2.5, 3.5, 4.5, 5.5} {
        \node (\x\y) at (\x, \y+6) {$\z$}; }
}
\node (11) at (1.5,1.5) {$2$}; 
\node (21) at (2.5,1.5) {$4$}; 
\node (31) at (3.5,1.5) {$4$}; 
\node (41) at (4.5,1.5) {$6$}; 
\node (51) at (5.5,1.5) {$8$}; 

\node (12) at (1.5,2.5) {$4$}; 
\node (22) at (2.5,2.5) {$4$}; 
\node (32) at (3.5,2.5) {$2$}; 
\node (42) at (4.5,2.5) {$4$}; 
\node (52) at (5.5,2.5) {$6$}; 

\node (13) at (1.5,3.5) {$6$}; 
\node (23) at (2.5,3.5) {$6$}; 
\node (33) at (3.5,3.5) {$4$}; 
\node (43) at (4.5,3.5) {$6$}; 
\node (53) at (5.5,3.5) {$6$}; 

\node (14) at (1.5,4.5) {$6$}; 
\node (24) at (2.5,4.5) {$4$}; 
\node (34) at (3.5,4.5) {$2$}; 
\node (44) at (4.5,4.5) {$4$}; 
\node (54) at (5.5,4.5) {$4$}; 

\node (15) at (1.5,5.5) {$8$}; 
\node (25) at (2.5,5.5) {$6$}; 
\node (35) at (3.5,5.5) {$4$}; 
\node (45) at (4.5,5.5) {$4$}; 
\node (55) at (5.5,5.5) {$2$}; 

\draw[thick, -stealth] (0,0) -- (7.1,0);
\draw[thick,-stealth] (0,0) -- (0,7.1);
\draw[thick] (0,0) rectangle (2,2);
\draw[thick] (3,2) rectangle (5,4);

}
\end{tikzpicture}
\end{minipage}
\begin{minipage}{.35\textwidth}
\centering
\begin{tikzpicture}[dot/.style={circle,fill,inner sep=1.5pt},2cell/.style={rectangle,fill,inner sep=2.5pt}, edge/.style = {circle, draw,inner sep=1.5pt},vert/.style = {circle, fill, inner sep=1.5pt},line width=.75pt,scale=1.3]
\def\h{.3}
\def\v{.2}
\node (0) at (0, 0) [2cell] {};
\node at (-\h,-\v) {\scriptsize	$(0,0)$};

\node (1) at (1, 0) [edge] {};
\node at (1,-\v) {\scriptsize $(0,2)$};

\node (2) at (2, 0) [2cell] {};
\node at (2+\h,-\v) {\scriptsize $(2,2)$};

\node (3) at (0, 1) [edge] {};
\node at (-\h,1) {\scriptsize $(0,2)$};

\node (4) at (1, 1) [vert] {};
\node at (1+1.25*\h,1) {\scriptsize $(0,2)$};

\node (5) at (2, 1) [edge] {};
\node at (2+\h,1) {\scriptsize $(2,2)$};

\node (6) at (0, 2) [2cell] {};
\node at (-\h,2+\v) {\scriptsize $(2,2)$};

\node (7) at (1, 2) [edge] {};
\node at (1,2+\v) {\scriptsize $(2,2)$};

\node (8) at (2, 2) [2cell] {};
\node at (2+\h,2+\v) {\scriptsize $(2,2)$};

\draw[-latex] (1) -- (0);
\draw[-latex] (3) -- (0);
\draw [-latex] (4) -- (0);

\draw[-latex] (2) -- (1);
\draw[-latex] (2) -- (4);
\draw[latex-latex] (2) -- (5);

\draw[-latex] (6) -- (3);
\draw[-latex] (6) -- (4);
\draw[latex-latex] (6) -- (7);

\draw[latex-latex] (8) -- (7);
\draw[-latex] (8) -- (4);
\draw[latex-latex] (8) -- (5);

\end{tikzpicture}

\end{minipage}
\begin{minipage}{.33\textwidth}
\centering
\begin{tikzpicture}[dot/.style={circle,fill,inner sep=1.5pt},,2cell/.style={rectangle,fill,inner sep=2.5pt}, edge/.style = {circle, draw,inner sep=1.5pt},vert/.style = {circle, fill, inner sep=1.5pt},line width=.75pt,scale=1.3]
\def\h{.3}
\def\v{.2}
\node (0) at (0, 0) [2cell] {};
\node at (-\h,-\v) {\scriptsize $(2,2)$};

\node (1) at (1, 0) [edge] {};
\node at (1,-\v) {\scriptsize $(2,4)$};

\node (2) at (2, 0) [2cell] {};
\node at (2+\h,-\v) {\scriptsize $(4,4)$};

\node (3) at (0, 1) [edge] {};
\node at (-\h,1) {\scriptsize $(2,4)$};

\node (4) at (1, 1) [vert] {};
\node at (1+1.25*\h,1) {\scriptsize $(2,6)$};

\node (5) at (2, 1) [edge] {};
\node at (2+\h,1) {\scriptsize $(4,6)$};

\node (6) at (0, 2) [2cell] {};
\node at (-\h,2+\v) {\scriptsize $(4,4)$};

\node (7) at (1, 2) [edge] {};
\node at (1,2+\v) {\scriptsize $(4,6)$};

\node (8) at (2, 2) [2cell] {};
\node at (2+\h,2+\v) {\scriptsize $(6,6)$};

\draw[-latex] (1) -- (0);
\draw[-latex] (3) -- (0);
\draw [-latex] (4) -- (0);

\draw[-latex] (2) -- (1);
\draw[latex-] (2) -- (5);

\draw[-latex] (6) -- (3);
\draw[latex-] (6) -- (7);

\draw[-latex] (8) -- (7);
\draw[-latex] (8) -- (4);
\draw[-latex] (8) -- (5);

\end{tikzpicture}

\end{minipage}
\vspace{2ex}
\caption{Cubical complex $\cX$ and values of $\cross(\xi)$, $\xi\in \cX^\topc$ [left], for the braid diagram displayed in Fig.\ \ref{fig:parabolic:lap}. The outlined regions are magnified in the middle and right figures and indicate the relation $\cR$. On the various cells $\xi\in \cX$ it gives the Lyapunov function $\Lap$.
}
\label{fig:sec112}
\end{figure}

\begin{remark}
\label{smalltimeonly}
 In view of Theorem \ref{localconstrGam}, since $X$ is compact, it suffices to prove that $\varphi(t,x)\in |\ccl^+\xi|$ for all $0\le t\le \tau_*$, some $\tau_*>0$.
\end{remark}

\begin{remark}
\label{Tmintop}
Equation \eqref{Laps} in the above proof also implies that $(\xi,\xi^-)\in \cR$ which is equivalent to $(\xi^-,\xi) \in \cR^{-1}$. The latter implies $\varphi(-t,|\xi|)\subset |\ccl^-\xi|$ for all $t\ge 0$ and for all $\xi\in\cX$, where $\ccl^-$ is the closure operator obtained from the opposite relation $\cR^{-1}\subset \cX\times \cX$. 
\end{remark}

Lemma \ref{closappr1} shows that $(\sSet(\cX),\ccl^+)$ is a CA-discretization for the Alexandrov topology $(X,\scrT^+)$ defined by the parabolic flow.
If we also invoke the observation that
 $(\xi,\xi)\in \cR$ if and only if $\xi \in \cX^{\topc}$ we can prove an even stronger statement.

\begin{lemma}
\label{closappr2}
The discrete  derivative operator $\derr^+$ satisfies $\der^+|\cU|\subset |\derr^+\cU|$ for all $\cU\subset \cX$.
\end{lemma}
\begin{proof}
The proof follows along the same lines as Lemma \ref{closappr1}.
As pointed out in Remark \ref{discotherfltop} 
it suffices to show that 
 $\varphi(t,x)\in |\derr^+\xi|$ for all $t > 0$, for all $x\in |\xi|$ and for all $\xi\in \cX$.
 If $\xi\in \cXr$, then either $\xi\in \cX^{\topc}$, or $\Lap(\eta) = \Lap(\xi)$ for all all $\eta\in \st\xi$.
 If $\xi\in \cX^{\topc}$, then $\varphi(t,x)$ stays in $|\xi|$ for all $0\le t\le \tau_x$ for some $\tau_x$ sufficiently small  and thus in $|\derr^+\xi|$ by the observation that
 $(\xi,\xi)\in \cR$ if and only if $\xi \in \cX^{\topc}$. The remaining case follows as before and therefore
$\varphi(t,x)\in |\st\xi|\subset |\derr^+\xi|$ for all $0\le t\le \tau_x$ for some $\tau_x$ sufficiently small. In particular for $0<t\le \tau_x$.
The case $\xi\in \cXs$ follows as in the proof of Lemma \ref{closappr1}.
\end{proof}

\begin{remark}
{
Note that Lemma \ref{closappr2} implies Lemma \ref{closappr1}. This indicates that the decomposition $\ccl^+= \id \cup \derr^+$ on $\sSet(X)$ allows a discretization $\ccl^+ = \id\cup \derr^+$ on $\sSet(\cX)$
with an explicit derivative operator $\derr^+$. 
The same conclusions hold for $\ccl^- = \id\cup \derr^-$ defined via the opposite relation $\cR^{-1}\subset \cX\times \cX$.
}
\end{remark}

The CA-discretizations for $(X,\scrT^-)$ and $(X,\scrT^+)$ are constructed according to the crossing number function $\Lap$  and are especially designed to display the behavior of $\varphi$ with respect to singular braids, cf.\ Prop.\ \ref{lapLyap}. One can define more refined discretizations that yield more discrete forward invariant sets.
The next lemma explains why we use this particular construction of
CA-discretization for $(X,\scrT^-)$ and $(X,\scrT^+)$.

\begin{figure}[hbt]
\begin{minipage}{.4\textwidth}
\centering
\begin{tikzpicture}[dot/.style={draw,circle,fill,inner sep=.75pt}, scale=.45]




\draw[step=1cm,gray,very thin] (0,0) grid (7,7);

{\tiny 

\foreach \x in {.5} {
    \foreach \y [evaluate = \y as \z using int(2*(\y-.5))]  in {.5, 1.5, 2.5, 3.5, 4.5, 5.5, 6.5} {
        \node (\x\y) at (\x, \y) {$\z$}; }
    \foreach \y [evaluate = \y as \z using int((12-2*(\y-.5)))]  in {.5, 1.5, 2.5, 3.5, 4.5, 5.5, 6.5} {
        \node (\x\y) at (\x+6, \y) {$\z$}; }
}
\foreach \y in {.5} {
     \foreach \x [evaluate = \x as \z using int(2*(\x-.5))]  in {1.5, 2.5, 3.5, 4.5, 5.5} {
        \node (\x\y) at (\x, \y) {$\z$};  }
    \foreach \x [evaluate = \x as \z using int(12-2*(\x-.5))]  in {1.5, 2.5, 3.5, 4.5, 5.5} {
        \node (\x\y) at (\x, \y+6) {$\z$}; }
}
\node (11) at (1.5,1.5) {$2$}; 
\node (21) at (2.5,1.5) {$4$}; 
\node (31) at (3.5,1.5) {$4$}; 
\node (41) at (4.5,1.5) {$6$}; 
\node (51) at (5.5,1.5) {$8$}; 

\node (12) at (1.5,2.5) {$4$}; 
\node (22) at (2.5,2.5) {$4$}; 
\node (32) at (3.5,2.5) {$2$}; 
\node (42) at (4.5,2.5) {$4$}; 
\node (52) at (5.5,2.5) {$6$}; 

\node (13) at (1.5,3.5) {$6$}; 
\node (23) at (2.5,3.5) {$6$}; 
\node (33) at (3.5,3.5) {$4$}; 
\node (43) at (4.5,3.5) {$6$}; 
\node (53) at (5.5,3.5) {$6$}; 

\node (14) at (1.5,4.5) {$6$}; 
\node (24) at (2.5,4.5) {$4$}; 
\node (34) at (3.5,4.5) {$2$}; 
\node (44) at (4.5,4.5) {$4$}; 
\node (54) at (5.5,4.5) {$4$}; 

\node (15) at (1.5,5.5) {$8$}; 
\node (25) at (2.5,5.5) {$6$}; 
\node (35) at (3.5,5.5) {$4$}; 
\node (45) at (4.5,5.5) {$4$}; 
\node (55) at (5.5,5.5) {$2$}; 

\draw[thick, -stealth] (0,0) -- (7.1,0);
\draw[thick,-stealth] (0,0) -- (0,7.1);
}

\def\hh{.8}
\def\bh{.2}

\def\hv{.8}
\def\bv{.2}

\foreach \x in {.5, 1.5} 
    \foreach \y  in {6.5,5.5,4.5,3.5,2.5,1.5} 
            \draw[-latex,line width=.2pt] (\x,\y-\bv) to (\x,\y-\hv);
\foreach \x in {2.5} 
    \foreach \y  in {6.5,5.5,3.5,2.5,1.5} 
            \draw[-latex,line width=.2pt] (\x,\y-\bv) to (\x,\y-\hv);
\draw[-latex,line width=.2pt] (2.5,3.5+\bv) to (2.5,3.5+\hv);

\foreach \x in {5.5, 6.5} 
    \foreach \y  in {5.5,4.5,3.5,2.5,1.5,.5} 
        \draw[-latex,line width=.2pt] (\x,\y+\bv) to (\x,\y+\hv);
\foreach \x in {4.5} 
    \foreach \y  in {5.5,4.5,3.5,1.5,.5} 
        \draw[-latex,line width=.2pt] (\x,\y+\bv) to (\x,\y+\hv);
\draw[-latex,line width=.2pt] (4.5,3.5-\bv) to (4.5,3.5-\hv);

\foreach \x in {3.5} 
    \foreach \y  in {0.5,1.5, 3.5} 
        \draw[-latex,line width=.2pt] (\x,\y+\bv) to (\x,\y+\hv);
\foreach \x in {3.5} 
    \foreach \y  in {6.5,5.5,3.5} 
            \draw[-latex,line width=.2pt] (\x,\y-\bv) to (\x,\y-\hv);

\draw[-latex,line width=.2pt] (1.5,.5+\bv) to (1.5,.5+\hv);
\draw[-latex,line width=.2pt] (2.5,.5+\bv) to (2.5,.5+\hv);
\draw[-latex,line width=.2pt] (2.5,1.5+\bv) to (2.5,1.5+\hv);

\draw[-latex,line width=.2pt] (4.5,6.5-\bv) to (4.5,6.5-\hv);
\draw[-latex,line width=.2pt] (4.5,5.5-\bv) to (4.5,5.5-\hv);

\draw[-latex,line width=.2pt] (5.5,3.5-\bv) to (5.5,3.5-\hv);
\draw[-latex,line width=.2pt] (5.5,6.5-\bv) to (5.5,6.5-\hv);

\foreach \y in {.5, 1.5} 
    \foreach \x  in {5.5,4.5,3.5,2.5,1.5} 
            \draw[-latex,line width=.2pt] (\x-\bh,\y) to (\x-\hh,\y);
\draw[-latex,line width=.2pt] (6.5-\bh,.5) to (6.5-.7,.5);
\draw[-latex,line width=.2pt] (6.5-\bh,1.5) to (6.5-.8,1.5);

\foreach \y in {5.5, 6.5} 
    \foreach \x  in {5.5,4.5,3.5,2.5} 
        \draw[-latex,line width=.2pt] (\x+\bh,\y) to (\x+\hh,\y);
\draw[-latex,line width=.2pt] (0.5+.3,6.5) to (0.5+.8,6.5);
\draw[-latex,line width=.2pt] (0.5+.3,5.5) to (0.5+.8,5.5);
\draw[-latex,line width=.2pt] (1.5+.3,6.5) to (1.5+.8,6.5);
\draw[-latex,line width=.2pt] (1.5+.3,5.5) to (1.5+.8,5.5);
\draw[-latex,line width=.2pt] (4.5-\bh,5.5) to (4.5-\hh,5.5);
\draw[-latex,line width=.2pt] (6.5-\bh,5.5) to (6.5-\hh,5.5);
\draw[-latex,line width=.2pt] (.5+\bh,1.5) to (.5+\hh,1.5);
\draw[-latex,line width=.2pt] (2.5+\bh,1.5) to (2.5+\hh,1.5);
\foreach \y in {2.5} 
    \foreach \x  in {6.5,5.5,4.5,2.5,1.5} 
            \draw[-latex,line width=.2pt] (\x-\bh,\y) to (\x-\hh,\y);
\foreach \y in {2.5} 
    \foreach \x  in {2.5,1.5,0.5} 
            \draw[-latex,line width=.2pt] (\x+\bh,\y) to (\x+\hh,\y);
\foreach \y in {3.5} 
    \foreach \x  in {6.5,5.5,4.5,2.5,1.5} 
            \draw[-latex,line width=.2pt] (\x-\bh,\y) to (\x-\hh,\y);
\foreach \y in {3.5} 
    \foreach \x  in {5.5,4.5,2.5,1.5,0.5} 
            \draw[-latex,line width=.2pt] (\x+\bh,\y) to (\x+\hh,\y);
            
\foreach \y in {4.5} 
    \foreach \x  in {6.5,5.5,4.5} 
            \draw[-latex,line width=.2pt] (\x-\bh,\y) to (\x-\hh,\y);
\foreach \y in {4.5} 
    \foreach \x  in {5.5,4.5,2.5,1.5,0.5} 
            \draw[-latex,line width=.2pt] (\x+\bh,\y) to (\x+\hh,\y);

\draw[thick] (1,0) -- (2,0) -- (2,1) -- (2,2) -- (0,2) -- (0,1) -- (1,1) -- cycle;
\draw[thick] (6,7) -- (5,7) -- (5,5) -- (7,5) -- (7,6) -- (6,6) -- cycle;
\draw[thick] (3,0) -- (3,1) -- (4,1) -- (4, 2) -- (3,2) -- (3,3) -- (0,3);
\draw[thick] (4,7) -- (4,6) -- (3,6) -- (3,5) -- (4,5) -- (4,4) -- (7,4) -- (7,5) -- (5,5) -- (5,7) -- cycle;
\draw[thick] (0,3) -- (3,3) -- (3,4) -- (2,4) -- (2,5) -- (1,5) -- (1,4) -- (0,4);
\draw[thick] (7,4) -- (4,4) -- (4,3) -- (5,3) -- (5,2) -- (6,2) -- (6,3) -- (7,3) -- cycle;



\end{tikzpicture}
\end{minipage}
\begin{minipage}{.5\textwidth}
\centering
\begin{tikzpicture}[node/.style = {ellipse, draw, inner sep = 1.5}, scale=.225]
\def\a{2}
\def\b{4}
\def\c{6}
\def\d{8}
\def\e{10}
\def\f{12}

\def\xa{0} 
\def\xb{2} 
\def\xc{4} 
\def\xd{6} 
\def\xe{8} 
\def\xf{10} 
\def\xg{12} 
\def\xh{14} 
\def\w{28}

\node[node] (0) at (\xa, 0) {\tiny $0$};
\node[node] (1) at (\w-\xa, 0)
{\tiny $0$};

\node[node] (2) at (\xb, \a)
{\tiny $2$ };
\node[node] (3) at (\xd, \a)
{\tiny $2$};
\node[node] (4) at (\w-\xd, \a)
{\tiny $2$};
\node[node] (5) at (\w-\xb, \a)
{\tiny $2$};
\node[node] (6) at (\xc, \b)
{\tiny $4$};
\node[node] (7) at (\xe, \b)
{\tiny $4$};
\node[node] (8) at (\xh, \b)
{\tiny $4$};
\node[node] (9) at (\w-\xe, \b)
{\tiny $4$};
\node[node] (10) at (\w-\xc, \b)
{\tiny $4$};
layer four
\node[node] (11) at (\xb, \c)
{\tiny $6$};
\node[node] (12) at (\xd, \c)
{\tiny $6$};
\node[node] (13) at (\xf, \c)
{\tiny $6$};
\node[node] (14) at (\w-\xf, \c)
{\tiny $6$};
\node[node] (15) at (\w-\xd, \c)
{\tiny $6$};
\node[node] (16) at (\w-\xb, \c)
{\tiny $6$};
\node[node] (17) at (\xc, \d)
{\tiny $8$};
\node[node] (18) at (\xe, \d)
{\tiny $8$};
\node[node] (19) at (\xg, \d)
{\tiny $8$};
\node[node] (20) at (\w-\xg, \d)
{\tiny $8$};
\node[node] (21) at (\w-\xe, \d)
{\tiny $8$};
\node[node] (22) at (\w-\xc, \d)
{\tiny $8$};
\node[node] (23) at (\xd, \e)
{\tiny $10$};
\node[node] (24) at (\xf, \e)
{\tiny $10$};
\node[node] (25) at (\w-\xf, \e)
{\tiny $10$};
\node[node] (26) at (\w-\xd, \e)
{\tiny $10$};
\node[node] (27) at (\xe, \f)
{\tiny $12$};
\node[node] (28) at (\w-\xe, \f)
{\tiny $12$};

layer one
\draw[->,>=stealth,thick] (2) to (0);
\draw[->,>=stealth,thick] (5) to (1);
\draw[->,>=stealth,thick] (6) to (2);
\draw[->,>=stealth,thick] (6) to (3);
\draw[->,>=stealth,thick] (7) to (3);
\draw[->,>=stealth,thick] (8) to (3);
\draw[->,>=stealth,thick] (8) to (4);
\draw[->,>=stealth,thick] (9) to (4);
\draw[->,>=stealth,thick] (10) to (4);
\draw[->,>=stealth,thick] (10) to (5);
\draw[->,>=stealth,thick] (11) to (6);
\draw[->,>=stealth,thick] (12) to (6);
\draw[->,>=stealth,thick] (14) to (6);

\draw[->,>=stealth,thick] (12) to (6);
\draw[->,>=stealth,thick] (12) to (7);
\draw[->,>=stealth,thick] (13) to (7);

\draw[->,>=stealth,thick] (13) to (8);
\draw[->,>=stealth,thick] (14) to (8);

\draw[->,>=stealth,thick] (14) to (9);
\draw[->,>=stealth,thick] (15) to (9);

\draw[->,>=stealth,thick] (15) to (10);
\draw[->,>=stealth,thick] (16) to (10);
\draw[->,>=stealth,thick] (13) to (10);

\draw[->,>=stealth,thick] (13) to (10);
\draw[->,>=stealth,thick] (17) to (11);

\draw[->,>=stealth,thick] (17) to (12);

\draw[->,>=stealth,thick] (18) to (12);
\draw[->,>=stealth,thick] (18) to (13);

\draw[->,>=stealth,thick] (19) to (13);

\draw[->,>=stealth,thick] (20) to (14);
\draw[->,>=stealth,thick] (21) to (14);
\draw[->,>=stealth,thick] (21) to (15);
\draw[->,>=stealth,thick] (22) to (15);
\draw[->,>=stealth,thick] (22) to (16);

\draw[->,>=stealth,thick] (23) to (17);
\draw[->,>=stealth,thick] (23) to (18);
\draw[->,>=stealth,thick] (24) to (18);
\draw[->,>=stealth,thick] (24) to (19);

\draw[->,>=stealth,thick] (25) to (20);
\draw[->,>=stealth,thick] (25) to (21);
\draw[->,>=stealth,thick] (26) to (21);
\draw[->,>=stealth,thick] (26) to (22);

\draw[->,>=stealth,thick] (27) to (23);
\draw[->,>=stealth,thick] (27) to (24);
\draw[->,>=stealth,thick] (28) to (25);
\draw[->,>=stealth,thick] (28) to (26);
\end{tikzpicture}
\end{minipage}
\vspace{2ex}
\caption{
Cubical complex $\cX$, $\cross$ restricted to $\cX^\topc$ and the closure operator $\ccl^\topc$ induced by the relation $\cR$ visualized as a directed graph [left].  Elements of $\sSC$ with more than one top cell are outlined. Poset $\sSC$, where vertices are labeled by $\cross$ [right].}
\label{fig:parabolic:lap}
\end{figure}
\begin{lemma}
\label{closappr3}
Let $\cU\subset \cX$ be a closed, forward invariant set for $\cR$, i.e. $\ccl~\cU=\cU$ and $\ccl^+\cU=\cU$. Then,
\begin{equation}
\label{interiorprop12}
\derr^+ \cU \subset \cInt \cU.
\end{equation}
In particular, $\cU$ is a regular closed set in $\cX$.
\end{lemma}

\begin{proof}
For $\xi\in \cU$ we distinguish between  $\xi \in\cX^\topc$ and $\xi\not \in \cX^\topc$.
We start with the latter.
For $\xi\not\in \cX^\topc$ consider two cases: (i) $\xi\in \cInt\cU$. Then,
 $\st\xi\subset \cInt\cU$, and
by the definition of $\cR$ we have that 
$\cR^{-1}[\xi] := \bigl\{\eta\mid (\eta,\xi)\in \cR \bigr\}
\subset \st\xi  \subset \Int \cU$.\footnote{Note that this does not require closedness for neither of the discrete topologies.} (ii) $\xi\in \cU\smin\cInt\cU= \ccl~\cU\smin\cInt\cU = \bd\cU$ (using the fact that $\ccl~\cU=\cU$). Then,
$\st\xi\not\subset\cU$.
Under the condition that $\ccl^+\cU=\cU$ 
regular cells $\xi\in \cXr$ satisfy the property that $\st\xi\subset\cU$. This implies that boundary cells are singular cells  $\xi\in \cXs$.
As before $\cR^{-1}[\xi]
\subset \st\xi \subset \cX^\topc$ and thus $\cR^{-1}[\xi]$ is open. Since $\ccl^+\cU\subset\cU$ also $\cR^{-1}[\xi]\subset \cU$ which yields
$\cR^{-1}[\xi]
\subset \cX^\topc\cap\cU\subset \cInt\cU$.

Consider the case $\xi\in \cU^\topc=\cU\cap\cX^\topc$. Then, $\cR^{-1}[\xi]\subset \cU\cap \ccl~\xi\subset \cU$ since
$\ccl~\cU=\cU$ and $\ccl^+\cU=\cU$.
Let $\eta\in \cR^{-1}[\xi]\smin \cX^\topc$ be a cell which is not interior to $\cU$, i.e. 
$\eta\in \bd\cU$.  
By the same argument as before $\eta\in \cXs$ and by Equation \eqref{Laps} we have that 
$\Lap(\eta)>\Lap(\xi)$ which yields $(\xi,\eta)\in \cR$ and $(\eta,\xi)\not\in \cR$.
The latter contradicts the existence of boundary cells $\eta \in\cR^{-1}[\xi]$.
Consequently,  $\cR^{-1}[\xi]\subset \cInt\cU$, which holds for every $\xi\in \cU$.

Iterating this procedure gives $(\cR^{-1})^k[\xi]\subset \cInt\cU$, $k\ge 1$ and thus $\derr^+\xi\subset \cInt\cU$, cf.\ Thm.\ \ref{localconstrGam}. This proves that $\derr^+\cU\subset \cInt\cU$.
The realization $|\cU|$ is a closed, forward invariant set and satisfies $\varphi(t,|\cU|) \subset \Int|\cU|$ for all $t>0$. Therefore $|\cU|$ is thus a closed attracting block. Such sets a regular closed by Theorem \ref{clattbl}. Since the CA-discretization $(\cX,\le,|\cdot|)$ is Boolean the same holds for the sets $\cU$ in $(\cX,\le)$.
\end{proof}

%

From the previous consideration we propose a pre-order $\le^\dagger$ and we show that $\le^\dagger$ yields the right continuity properties with respect to $\cell$ as defined in \eqref{paracell}.
Define the pre-order 
\begin{equation}
    \label{parapreorder}
    \le^\dagger \,:=\, \le \vee \le^+,\footnote{The meet of two pre-orders is defined as the transitive closure of the union of the two relations, i.e. $\le\vee\le^+ := (\le\cup\le^+)^{\bm +}$.}
\end{equation}
i.e.
$\cU$ is closed in $(\cX,\le^\dagger)$ if and only if $\ccl~ \cU=\cU$ and $\ccl^+\cU=\cU$.

\begin{theorem}
\label{chartwotop12}
The pre-order $\le^\dagger$ defines a Morse pre-order
for $(X,\scrT,\scrTbf)$, i.e.
the maps
\[
\cell\colon (X,\scrT) \to (\cX,\le^\dagger),\quad\text{and}\quad
\cell\colon (X,\scrTbf) \to (\cX,\ge^\dagger),
\]
are continuous.
\end{theorem}

\begin{proof}
To proof the above statement we need to show that for $\cU\subset \cX$
closed in $(\cX,\le^\dagger)$ the pre-image $\cell^{-1}\cU = 
|\cU|$ is closed in $(X,\scrT)$ and open in $(X,\scrTbf)$.
Let $\cU$ be closed in $(\cX,\le^\dagger)$. Then,
$\ccl~\cU=\cU$ and $\ccl^+\cU=\cU$ which implies that $|\cU|$ is closed in $(X,\scrT)$.
By Lemma \ref{closappr3}
\[
\varphi(t,|\cU|) \subset \bigcup_{t>0}\varphi(t,|\cU|) = \derr^+|\cU|
\subset |\derr^+\cU| \subset |\cInt\cU| = \Int|\cU|,~~\forall t>0,
\]
which, by Lemma \ref{charattbl12345}, proves that $|\cU|$ is open in 
$(X,\scrTbf)$.
\end{proof}

The above theorem shows that sets that are closed in both discrete topologies are closed attracting blocks and therefore
$\le \vee\le^+ = \le^\dagger$ defines a Morse pre-order for parabolic flows with skeleton $y$ and which is an antagonistic coarsening of discretizations for both 
$\scrT$ and $\scrTbf$, cf.\ Sect.\ \ref{bi-top}, i.e. take $\le^\dagger$ and $\ge^\dagger$ respectively.

\begin{remark}
\label{aboutopptop}
If $\cU\subset \cX$ is a closed, backward invariant set for $\cR$, i.e. $\ccl~\cU=\cU$ and $\ccl^-\cU=\cU$. Then, $\derr^- \cU \subset \cInt \cU$.
The arguments in the proof of Lemma \ref{closappr3} remain unchanged if we replace $\cR$ with $\cR^{-1}$. For the latter we   use the first part of the inequality in \eqref{Laps}. This proves the same statement for closed, backward invariant sets.
\end{remark}

\begin{remark}
\label{directbflt}
From the previous consideration we can define a discretization of the \bflt
$\scrTbf$. Using the relation $\cR$ as defined above we have the derivative operators $\derr^-,\derr^+\colon\sSet(\cX) \to \sSet(\cX)$.
Define the discrete operator $\udermm:= \derr^-\ccl$.
Since $\derr^-$ gives a DA-discretization for the Alexandrov topology $\scrT^-$ we have by Remark \ref{discotherfltop} that
\[
\varphi(-t,\ccl~|\xi|) = \varphi(-t,|\ccl~\xi|) \subset |\derr^-\ccl~\xi| = |\udermm\xi|,\quad \forall t>0.
\]
The discrete operator $\udermm$ satisfies the hypotheses of Lemma \ref{udermlem} and therefore $\uclbff\colon
\sSet(\cX) \to \sSet(\cX)$, given by Lemma \ref{udermlem} 
defines a CA-discretization for $(X,\scrTbf)$. 
The antagonistic coarsening of $(\cX,\ccl,\uclbff,|\cdot|)$ is the pre-order $\le^\dagger$ in \eqref{parapreorder}.
\end{remark}

The following lemma gives a characterization of the transitive, reflexive closure of $\cR$.

\begin{lemma}
\label{charofR}
 $(\xi,\xi')\in \cR^{\bm{ +=}}$ if and only if there exist cells $\xi_0,\cdots \xi_{\ell}$, with $\xi_0=\xi$, $\xi_\ell=\xi'$ and
 $(\xi_i,\xi_{i+1})\in \cX^\topc\times G_{d-1}\cX$, or
 $(\xi_i,\xi_{i+1})\in G_{d-1}\cX\times \cX^\topc$,
such that
\begin{equation}
\label{chainform}    
(\xi_0,\xi_1),
(\xi_1,\xi_2),
\cdots, 
(\xi_{\ell-2},\xi_{\ell-1}),
(\xi_{\ell-1},\xi_{\ell})\in \cR.
\end{equation}
\end{lemma}

\begin{proof}
We prove the lemma in one direction since \eqref{chainform} trivially implies
$(\xi,\xi')\in \cR^{\bm{+=}}$.
We may assume without loss of generality that $\xi,\xi'\in \cX^\topc$, $\xi\neq \xi'$. Indeed,
if $\xi\not\in \cX^\topc$ then $\xi\in \cXr$, or $\xi \in \cXs$.
For the former we can
choose $\tilde\xi\in \st\xi\cap \cX^\topc $ such
that $\Lap(\tilde\xi) = \Lap(\xi)$, and for the latter case $\cross_-(\xi)<\cross_+(\xi)$. Therefore we can choose 
$\tilde\xi\in \st\xi\cap \cX^\topc $ such
that $\bigl(\cross_-(\xi),\cross_-(\xi)\bigr)=\Lap(\tilde\xi) < \Lap(\xi) = \bigl(\cross_-(\xi),\cross_+(\xi)\bigr)$.
We can thus choose $\tilde\xi\in \st\xi\cap \cX^\topc $ such
that $\Lap(\tilde\xi)\le \Lap(\xi)$.
Similarly, if $\xi'\not\in \cX^\topc$ we can choose $\tilde\xi'\in \st\xi'\cap \cX^\topc $ such
that $\Lap(\xi)\le \Lap(\tilde\xi')$.

By definition $(\xi,\xi')\in \cR$ if (a) $\xi'\in \cX^\topc$ and $\xi\in \ccl~\xi'$, or (b) $\xi\in \cX^\topc$ and $\xi'\in \ccl~\xi$, or (c) $\xi=\xi'$
which we may exclude in the proof.
Consequently, $(\xi,\xi')\in \cR^{\bm{+=}}$ implies the existence of 
\[
(\xi,\eta_0),(\eta_0,\sigma_1),\cdots,(\sigma_{k-1},\eta_{k-1}),(\eta_{k-1},\xi')\in\cR,
\]
with $\eta_j\not\in \cX^\topc$ and $\sigma_j\in \cX^\topc$.
The lemma is proved if we prove \eqref{chainform} for the case
 $(\sigma,\eta),(\eta,\sigma')\in \cR$, $\sigma,\sigma'\in \cX^\topc$ and
 $\eta\not\in \cX^\topc$. Assume without loss of generality that $\eta\in G_k\cX$, $k<d-1$.
 By the definition of $\cR$ it holds that  $\Lap(\sigma)\le \Lap(\eta) \le \Lap(\sigma')$, which is equivalent to
 \[
 \bigl(\cross(\sigma),(\cross(\sigma)\bigr) \le 
 \bigl(\cross_-(\eta),(\cross_+(\eta)\bigr) \le \bigl( \cross(\sigma'),(\cross(\sigma')\bigr).
 \]
Since $\cross_-(\eta)$ is minimal over $\st\eta\cap \cX^\topc$
and $\cross_+(\eta)$ is maximal over $\st\eta\cap \cX^\topc$ it follows that 
$\cross(\sigma)=\cross_-(\eta) \le \cross_+(\eta)=\cross(\sigma')$. Moreover, every $\sigma''\in \st\eta\cap \cX^\topc$
satisfies $\cross(\sigma)\le\cross(\sigma'')\le\cross(\sigma')$. Choose $\sigma''\in \st\eta\cap \cX^\topc$ such that $\ccl~\sigma\cap\ccl~\sigma''\cap G_{d-1}\cX\neq\varnothing$ and let $\eta'\in G_{d-1}\cX$ be the unique cell in $\ccl~\sigma\cap\ccl~\sigma''\cap G_{d-1}\cX$. Then,
$\Lap(\sigma) \le \Lap(\eta')\le \Lap(\sigma'')\le\Lap(\sigma')$
and  $(\sigma,\eta'),(\eta',\sigma'')\in \cR$. 
Observe that 
$\ccl~\sigma'$ and $\ccl~\sigma''$ intersect in a cell $\tilde\eta \in \st\eta$
with $\Lap(\sigma'') \le \Lap(\tilde\eta)\le \Lap(\sigma')$
and $\st \tilde\eta \subset \st\eta$. Now repeat the above steps by using the cells in $\st\eta\cap\cX^\topc$ at most once. This process terminates after finitely many steps, proving the lemma.
\end{proof}

\begin{theorem}
\label{brclasses123}
The partial equivalence classes of $\cR^{\bm +=}$ correspond to the discrete relative braid class components in $\Conf_1^d \rel y$ for a given skeleton $y$.
\end{theorem}

\begin{proof}
We distinguish regular and singular cells. 
Every regular cell $\xi$ determines a discrete relative braid. This assignment is not one-to-one in general. Let the braid class component $[x]\rel y$ be the connected component of $x\rel y$ for some $x\in |\xi|$.
Any point $x'\rel y\in [x]\rel y$ corresponds to a cell $\xi'\in \cXr$,
with $x\in |\xi'|$. Let $\gamma\colon[0,1]\to [x]\rel y$ be a path joining
$x$ and $x'$. Then, $\gamma(s) \in |\xi_s| \subset [x]\rel y$, $\xi_s\in \cXr$ for all $s\in [0,1]$.
If $\xi\in \cXr$, then
$\st\xi\subset \cXr$ and $\Lap(\eta)=\Lap(\xi)$ for all $\eta\in \st\xi$. Therefore, 
the set $\bigcup_{s\in [0,1]}|\st\xi_s|$ is an open covering of $\gamma$. Since $\Lap(\st\xi_s)$ is constant for all $s$, the compactness of path $\gamma$ implies that $\Lap$ is constant on $\bigcup_{s\in [0,1]}\st\xi_s$ and in particular $\Lap(\xi)=\Lap(\xi')$.
The path also
yields a chain $\xi_i$ satisfying \eqref{chainform} which proves, using Lemma \ref{charofR}, that $(\xi,\xi')\in \cR^{\bm{+=}}$.
Since this holds for any two relative braids in $[x]\rel y$ we conclude that the set  of all cells $\cU$ with $|\cU|=[x]\rel y$ is contained in a partial equivalence class of $\cR^{\bm{+=}}$.
Conversely, equivalent cells belong to the same braid class component which proves that the braid class components are realized by the partial equivalence classes of $\cR^{\bm{+=}}$.
%
\end{proof}


From the theory in Section \ref{dyngrad12} we have that $\le^\dagger$ restricted to $\cX^{\topc}$ defines a \discresol  $(\cX^{\topc},\le^\topc)$. 
The \discresol  $(\cX^{\topc},\le^\topc)$ can be characterized as follows.
We define a relation $\cF$ on $\cX^{\topc}$ in  two steps.
\begin{enumerate}
    \item[(i)] Let  $\cE^\topc\subset \cX^{\topc}\times \cX^{\topc}$ be the  (symmetric) \emph{adjacency relation} given by
    \[
(\xi,\xi')\in \cE^\topc \iff G_{d-1}\cX \cap \ccl~\xi\cap \ccl~\xi'\neq \varnothing\quad \text{and}\quad \xi\neq \xi',
\] 
    i.e. $(\xi,\xi')\in \cE^\topc$ if and only if the cells $\lbr\xi\rbr$ and $\lbr\xi'\rbr$ intersect along a $(d-1)$-dimensional face.\footnote{Recall that $G_{d-1}\cX$ is the skeleton of co-dimension one cells in $\cX$.} 
    \item[(ii)] Let $\cF\subset \cE^\topc$ be defined  as follows:  
\[
(\xi,\xi')\in \cF  \iff   \cross(\xi)\leq \cross(\xi'),
\]
cf.\ Fig.\ \ref{fig:parabolic:lap}[left] and [right].
\end{enumerate}
\begin{theorem}
\label{statetrans12}
The transitive, reflexive closure $\cF^{\bm{+=}}$ is the restriction $\le^\topc$ of $\le^\dagger$ to $\cX^{\topc}$.
In other words, $\cF^{\bm{+=}}$ is a \discresol for $\varphi$.
\end{theorem}

\begin{proof}
Let $(\xi,\xi')\in \cF$, then $\xi$ and $\xi'$ are adjacent top cells and $\cross(\xi)\le \cross(\xi')$. By definition there exists a unique cell $\eta\in G_{d-1}\cX\cap \ccl~\xi\cap \ccl~\xi'$,
and $\st\eta = \{\eta,\xi,\xi'\}$. The Lyapunov functions for these cells are given by
$\Lap(\xi) = \bigl( \cross(\xi),\cross(\xi)\bigr)$, $\Lap(\eta) =  \bigl( \cross(\xi),\cross(\xi')\bigr)$ and $\Lap(\xi') = \bigl( \cross(\xi'),\cross(\xi')\bigr)$. Consequently,
$\Lap(\xi)\le \Lap(\eta)\le\Lap(\xi')$ and thus $(\xi,\eta), (\eta,\xi')\in \cR$.
By definition $(\xi,\xi') \in \cF^{\bm{+=}}$ is equivalent to pairs 
\[
(\xi_0,\xi_1),\cdots,(\xi_{\ell-1},\xi_\ell) \in\cF,
\]
with $\xi_0=\xi$ and $\xi_k=\xi'$. Therefore, 
$(\xi,\xi') \in \cF^{\bm{+=}}$ implies $(\xi,\xi') \in \cR^{\bm{+=}}$.
We obtain the inclusion, $\cF^{\bm{+=}} \subset \cR^{\bm{+=}} \cap (\cX^\topc\times\cX^\topc)$.

Conversely, let $(\xi,\xi') \in \cR^{\bm{+=}} \cap (\cX^\topc\times\cX^\topc)$.
By Lemma \ref{charofR} there exist $\xi_0,\cdots,\xi_{2\ell}$, with $\xi_0=\xi$ and $\xi_{2\ell}=\xi'$, such that \eqref{chainform} holds. This yields
\[
(\xi_{2i},\xi_{2i+2})\in \cF, \quad i=0,\cdots,\ell-1,
\]
which implies that $(\xi,\xi')\in \cF^{\bm{+=}}$. This provides the opposite inclusion
$\cR^{\bm{+=}} \cap (\cX^\topc\times\cX^\topc)\subset \cF^{\bm{+=}}$ and thus $\cR^{\bm{+=}} \cap (\cX^\topc\times\cX^\topc)= \cF^{\bm{+=}}$.

Since $\le$ restricted to $\cX^\topc$ is an anti-chain  the restriction
of $\le^\dagger$ to $\cX^\topc$ is equal to the restriction
of $\le^+$ to $\cX^\topc$, i.e. 
$\cR^{\bm{+=}}\cap (\cX^\topc\times\cX^\topc)$.
Therefore, the restriction
of $\le^\dagger$ to $\cX^\topc$ is equal to 
$\cF^{\bm{+=}}\cap (\cX^\topc\times\cX^\topc)$, which completes the proof.
\end{proof}

By Theorem \ref{statetrans12} the transitive, reflexive closure $\cF^{\bm{+=}}$ is a closure operator on $\cX^\topc$, i.e. 
$\ccl^\topc = (\cF^{\scaleto{\bm{+=}}{4pt}})^{-1}$, 
and is thus a \discresol for $\varphi$.
By Theorem \ref{brclasses123} the braid class components correspond to the partial equivalence classes of $\cR^{\bm +=}$ and therefore with the partial equivalence classes of
$\cF^{\bm +=}$. By the order-preserving map $\dyn\colon (\cX,\le) \twoheadrightarrow (\sSC,\le)$ given in \eqref{dyndefn12} we obtain the partial equivalence classes of $\le^\dagger$. The interior yields the braid class components: the open sets $\Int |\dyn^{-1}\cS|$, $\cS\in \sSC$, describe all discrete braid class components. The advantage of $|\dyn^{-1}\cS|$ is that its Borel-Moore homology gives its Conley index via $H^\tile(\cS)$. Another way to retrieve the braid class components from $\le^\topc$ is to use Section \ref{regclMT12}. In this way we obtain the closures of the braid class components. This way one cannot immediately determine its Conley index

In the next section we carry out a specific analysis for a number of examples of parabolic systems.

\section{Recipe for global decompositions}
\label{gldecomppf}


In this section we apply the methodology of this text to parabolic flows in combination with the theory of discretized braids.  This application will use all of the ingredients described in the previous chapters.  We also place an emphasis on the computational aspects and highlight how these can be carried out in practice.  The goal is to obtain a Morse pre-order for a discrete parabolic flow $\varphi$, which encodes the directionality of $\varphi$, from which we can determine a (graded) tessellar differential module, a graded representation (connection matrix), and a tessellar phase diagram, whose structures reveal information about the invariance and connecting orbits of $\varphi$.  

Section \ref{algimpl} outlines the general recipe for computing  a Morse pre-order and a graded representation  applied  to parabolic flows.  In order to make use of the algorithm \textsc{ConnectionMatrix} of \cite{hms}, we also assume that we work with homology over fields.   Section \ref{lapgrading} is specific to parabolic flows, and  introduces parabolic Betti numbers/homology\index{Betti number!parabolic}\index{Parabolic homology} using the lap number grading.

\subsection{Computing Morse pre-orders and tessellar chain complexes}\label{algimpl}

We divide the computations into three appropriate steps: topologization, discretization and algebraization.   These steps use the tools of graph theory and computational algebraic topology.\footnote{These computations can be set up for a given skeleton using the open-source software package \textsc{PyChomp}~\cite{cmcode}.  Of note is that the software is very efficient, and can calculate \discresols and connection matrices for examples of parabolic flows with $|\ccX|\approx 2.5\times 10^{10}$, and $|\sSC|\approx 6.2\times 10^4$, cf.\ \cite{hms2}. More details on the software and algorithms, in addition to timing information for computational experiments, can be found in~\cite{hms,hms2}.}




\subsubsection{Topologization}
(a) The space of $d$-periodic sequences is a cube in $\R^d$ and is given the standard metric topology. 
The block-flow topology given by a parabolic flow is derived from the backward image operator as explained in Section \ref{sec:parabolic:model}. The idea in Section \ref{sec:disc-dyn} is to construct a pre-order that discretizes both the metric topology as well as the \bflt\!\!. This is carried out such that CW-discretization map $\cell\colon X \twoheadrightarrow\cX$ has the right continuity properties.

\subsubsection{Discretization}
We breakdown discretization into steps (b)--(f); steps (b)--(d) are represented in Fig.'s \ref{fig:braidclass}--\ref{fig:sec112} and step (e) in Fig.~\ref{fig:parabolic:lap}.

 (b) For parabolic flows we use specific discretizations that are compatible with the braid classes for a given skeleton $y$, i.e., the top-cells $\cX^\topc$ correspond to generic braids given by  $G_dX$.  For a given  skeletal braid $y\in \Conf_m^d$ that is stationary for $\varphi$ and for which $\mathring y$ is proper, the phase space is given by Equation \eqref{thespace}.
Following the representation in Remark \ref{canrep1} we represent $y$ in normal form\footnote{None of the results here are affected by choosing a normal form because all skeleta are homotopic, cf.\ \cite{im}.} which yields
a  cubical CW-decomposition $\cX$ with the appropriate number of cubes, q.v.\ Fig.~\ref{fig:parabolic:lap}.

(c) The top cells $\xi\in \cX^\topc$ of the cubical CW-decomposition described in (a) correspond to the subsets $|\xi|\subset G_dX$ of generic braids $x\rel y \in  \Conf_1^d\rel y$, cf.\ Sect.\ \ref{metrcidisc}.
For the top cells we determine the symmetric adjacency relation $\cE^\topc\subset \cX^\topc\times\cX^\topc$ as described in Section \ref{fltodisc}.

(d) For every generic braid diagram  $x\rel y$ the crossing number $\cross(x\rel y) = \cross(\xi)\in \N$ is well-defined and can be given as the crossing number $\cross(\xi)$ of the unique top cell representating $x\rel y$. From the description in Section \ref{fltodisc} we obtain the generating relation $\cF\subset \cE^\topc$ for the \discresol $\le^\topc$ via $(\xi,\xi')\in \cF$ if and only if $\cross(\xi)\le \cross(\xi')$.

(e) For the relation $\cF$ we compute the poset of strongly connected components $(\sSC,\le)$. This can be done in time $O(|\ccX^\topc|+|\cF|)$ using Tarjan's algorithm\index{Tarjan's algoritm} \cite{tarjan1972depth}.\footnote{Note that without a \discresol it would take time $O(|\cX|+|\!\!\leq^\dagger\!\!|)$ to compute the poset $\sSC$.}
The elements $\cS\in \sSC$ correspond to discrete braid class components via $\Int | \cl~\cS|$, where $\cl$ is closure in $(\cX,\le)$ and $\Int$ is interior in $(X,\scrT)$. 

(f) We use the formula for the map $\dyn\colon \cX \twoheadrightarrow \sSC$ given by Theorem \ref{thethmdyn} to reconstruct the pre-order $(\cX,\le^\dagger)$. The partial equivalence classes of $\le^\dagger$ are given by $\dyn^{-1}\cS$. Note that $\cl~\dyn^{-1}\cS = \cl~\cS$.
The difference between $\dyn^{-1}\cS$ and $\cl~\cS$ is that the former is  convex in $(\cX,\le^\dagger)$ and thus locally closed  in $(\cX,\le)$. This implies that $|\dyn^{-1}\cS|$ is locally compact, and that the Borel-Moore homology is well-defined and can be computed via the cellular homology.
The pre-order $(\cX,\le^\dagger)$  defines a Morse tessellation.
The Morse tiles are given by the formula $|\dyn^{-1}\cS|  = G_\cS X$, cf.\ Eqn.\ \eqref{finhom12}.
Having the pre-order $(\cX,\le^\dagger)$ now establishes the discretization map
\[
\begin{tikzcd}
X \arrow[r, two heads, "\cell"] \arrow[rr, two heads, "\tile", bend right] & \cX \arrow[r] \arrow[r] \arrow[r, two heads, "\dyn"] & \sSC
\end{tikzcd}
\]
which also induces a non-trivial grading of $X$.

\subsubsection{Algebraization}
The steps (a)--(f) yield the cubical CW-decomposition $\cX$, the poset $(\sSC,\le)$ and the map $\dyn\colon \cX\twoheadrightarrow \sSC$.

(g) The CW-decomposition map $\cell\colon (X,\scrT) \twoheadrightarrow(\cX,\le)$ together with the map $\dyn\colon (\cX,\le) \twoheadrightarrow (\sSC,\le)$ form an $\sSC$-graded cell complex, cf.\  \cite{hms}, which is the input for the algorithm \textsc{ConnectionMatrix} of~\cite{hms}.
\begin{figure}[h!]
\centering
\begin{minipage}{.625\textwidth}
\centering
\begin{tikzcd}[column sep = 0.7em]
\vphantom{v}\\
0 \arrow{r} &
\bigoplus_{i\in\{6,8,10\}} \Z_2 \langle \cS_i\rangle\arrow{r}{\cm_1} & [2em]
\bigoplus_{i\in\{0,1,3,4\}}\Z_2 \langle \cS_i\rangle \arrow{r} & 0,
\end{tikzcd}
\end{minipage}
\begin{minipage}{.3\textwidth}
\begin{align*}
\cm_1 = 
\bordermatrix{  & \cS_6 & \cS_8 & \cS_{10}   \cr
              \cS_0 & 1 & 0 & 0   \cr
              \cS_1 & 0 & 0 & 1 \cr
              \cS_3 & 1 & 1 & 0   \cr
              \cS_4 & 0 & 1 & 1  }
\end{align*}
\end{minipage}
\vspace{2ex}
    \caption{Graded tessellar differential module $C^\tile(X)$ for the example in Figure \ref{fig:braidclass}[left]. The differential $\cm^\tile$ is computed using $\Z_2$ coefficients [right].}
    \label{fig:braid:cm}
\end{figure}

As output, we obtain the graded tessellar differential module $C^\tile(X)$, as described in Section \ref{doublegr}, cf.\ Fig.\ \ref{fig:braid:cm}.  In particular, we  obtain the Borel-Moore homologies $H^\BM(G_\cS X)$, and 
all Borel-Moore homologies are finitely generated.  In general the time complexity of this step is $O(|\cX|^3)$, however in practice it is linear~\cite{hms2}.

(h) Since the homology is computed over a field, i.e. $\K=\Z_2$, it is completely described by its Betti numbers/Poincar\'e polynomials. 
We visualize the ensemble of Morse pre-order and tessellar differential module by augmenting the Hasse diagram for $\sSC$ by providing the Borel-Moore Poincar\'e polynomials of the Morse tiles $G_\cS X$.  This visualization of the {tessellar phase diagram} $(\tessph,\le^\dagger)$ in given in Figure \ref{fig:braid:phase}, cf.\ Sect.\ \ref{doublegr}.  
\begin{figure}[h!]
\centering
\begin{minipage}{1.0\textwidth}
\centering
\begin{tikzpicture}[node/.style = {ellipse, draw, inner sep = 1.5}, scale=.325]
\def\a{2}
\def\b{4}
\def\c{6}
\def\d{8}
\def\e{10}
\def\f{12}

\def\xa{0} 
\def\xb{2} 
\def\xc{4} 
\def\xd{6} 
\def\xe{8} 
\def\xf{10} 
\def\xg{12} 
\def\xh{14} 
\def\w{28}

\node[node] (0) at (\xa, 0) {\scriptsize $\cS_0 : \mu^0$};
\node[node] (1) at (\w-\xa, 0)
{\scriptsize $\cS_1 : \mu^0$};

\node[node] (2) at (\xb, \a)
{\scriptsize $\cS_2$ };
\node[node] (3) at (\xd, \a)
{\scriptsize $\cS_3 : \mu^0$};
\node[node] (4) at (\w-\xd, \a)
{\scriptsize $\cS_4 : \mu^0$};
\node[node] (5) at (\w-\xb, \a)
{\scriptsize $\cS_5$};
\node[node] (6) at (\xa, \b)
{\scriptsize $\cS_6 : \mu^1$};
\node[node] (7) at (\xe, \b)
{\scriptsize $\cS_7$};
\node[node] (8) at (\xh, \b)
{\scriptsize $\cS_8 : \mu^1$};
\node[node] (9) at (\w-\xe, \b)
{\scriptsize $\cS_9$};
\node[node] (10) at (\w-\xa, \b)
{\scriptsize $\cS_{10} : \mu^1$};
layer four
\node[node] (11) at (\xb, \c)
{\scriptsize $\cS_{11}$};
\node[node] (12) at (\xd, \c)
{\scriptsize $\cS_{12}$};
\node[node] (13) at (\xf, \c)
{\scriptsize $\cS_{13}$};
\node[node] (14) at (\w-\xf, \c)
{\scriptsize $\cS_{14}$};
\node[node] (15) at (\w-\xd, \c)
{\scriptsize $\cS_{15}$};
\node[node] (16) at (\w-\xb, \c)
{\scriptsize $\cS_{16}$};
\node[node] (17) at (\xc, \d)
{\scriptsize $\cS_{17}$};
\node[node] (18) at (\xe, \d)
{\scriptsize $\cS_{18}$};
\node[node] (19) at (\xg, \d)
{\scriptsize $\cS_{19}$};
\node[node] (20) at (\w-\xg, \d)
{\scriptsize $\cS_{20}$};
\node[node] (21) at (\w-\xe, \d)
{\scriptsize $\cS_{21}$};
\node[node] (22) at (\w-\xc, \d)
{\scriptsize $\cS_{22}$};
\node[node] (23) at (\xd, \e)
{\scriptsize $\cS_{23}$};
\node[node] (24) at (\xf, \e)
{\scriptsize $\cS_{24}$};
\node[node] (25) at (\w-\xf, \e)
{\scriptsize $\cS_{25}$};
\node[node] (26) at (\w-\xd, \e)
{\scriptsize $\cS_{26}$};
\node[node] (27) at (\xe, \f)
{\scriptsize $\cS_{27}$};
\node[node] (28) at (\w-\xe, \f)
{\scriptsize $\cS_{28}$};

layer one
\draw[->,>=stealth,thick] (2) to (0);
\draw[->,>=stealth,thick] (5) to (1);
\draw[->,>=stealth,thick] (6) to (2);
\draw[->,>=stealth,thick] (6) to (3);
\draw[->,>=stealth,thick] (7) to (3);
\draw[->,>=stealth,thick] (8) to (3);
\draw[->,>=stealth,thick] (8) to (4);
\draw[->,>=stealth,thick] (9) to (4);
\draw[->,>=stealth,thick] (10) to (4);
\draw[->,>=stealth,thick] (10) to (5);
\draw[->,>=stealth,thick] (11) to (6);
\draw[->,>=stealth,thick] (12) to (6);
\draw[->,>=stealth,thick] (14) to (6);

\draw[->,>=stealth,thick] (12) to (6);
\draw[->,>=stealth,thick] (12) to (7);
\draw[->,>=stealth,thick] (13) to (7);

\draw[->,>=stealth,thick] (13) to (8);
\draw[->,>=stealth,thick] (14) to (8);

\draw[->,>=stealth,thick] (14) to (9);
\draw[->,>=stealth,thick] (15) to (9);

\draw[->,>=stealth,thick] (15) to (10);
\draw[->,>=stealth,thick] (16) to (10);
\draw[->,>=stealth,thick] (13) to (10);

\draw[->,>=stealth,thick] (13) to (10);
\draw[->,>=stealth,thick] (17) to (11);

\draw[->,>=stealth,thick] (17) to (12);

\draw[->,>=stealth,thick] (18) to (12);
\draw[->,>=stealth,thick] (18) to (13);

\draw[->,>=stealth,thick] (19) to (13);

\draw[->,>=stealth,thick] (20) to (14);
\draw[->,>=stealth,thick] (21) to (14);
\draw[->,>=stealth,thick] (21) to (15);
\draw[->,>=stealth,thick] (22) to (15);
\draw[->,>=stealth,thick] (22) to (16);

\draw[->,>=stealth,thick] (23) to (17);
\draw[->,>=stealth,thick] (23) to (18);
\draw[->,>=stealth,thick] (24) to (18);
\draw[->,>=stealth,thick] (24) to (19);

\draw[->,>=stealth,thick] (25) to (20);
\draw[->,>=stealth,thick] (25) to (21);
\draw[->,>=stealth,thick] (26) to (21);
\draw[->,>=stealth,thick] (26) to (22);

\draw[->,>=stealth,thick] (27) to (23);
\draw[->,>=stealth,thick] (27) to (24);
\draw[->,>=stealth,thick] (28) to (25);
\draw[->,>=stealth,thick] (28) to (26);

\end{tikzpicture}
\end{minipage}
\vspace{2ex}
\caption{Tessellar phase diagram for $(\tessph,\le^\dagger)$ for the example in Figure \ref{fig:braidclass}. The elements in $(\sSC,\le)$ are labeled with the natural numbers. Elements $\cS\in \sSC$ with trivial tessellar homology are indicated only by label.}\label{fig:braid:phase}
\end{figure}

This structure is an algebraic and combinatorial description of the global dynamics of $\varphi$, encoding both the directionality and the invariance.   We can also  display the tessellar differential (connection matrix data), although in practice we regard this as a separate, queryable data structure which lives over the tessellar phase diagram.
The pure tessellar phase diagram is obtained by the subposet of vertices with non-trivial homology, cf.\ Fig.\ \ref{fig:braid:cmg}. 
\begin{figure}[h!]
\centering
\begin{minipage}{.75\textwidth}
\centering
\begin{tikzpicture}[scale=.85,node distance=.2cm and .25cm]
\def\a{0}
\def\b{3}
\def\c{6}
\def\d{9}
\def\e{1.5}
\def\f{4.5}
\def\g{7.5}
\node[draw, ellipse] (0) at (\a, 0) {\scriptsize $\cS_0 : \mu^0$};
\node[draw, ellipse] (1) at (\b, 0) {\scriptsize $\cS_3 : \mu^0$};
\node[draw, ellipse] (2) at (\c, 0) {\scriptsize $\cS_4 : \mu^0$};
\node[draw, ellipse] (3) at (\d, 0) {\scriptsize $\cS_1 : \mu^0$};

\node[draw, ellipse] (4) at (\e, 1.5) {\scriptsize $\cS_6 : \mu^1$};
\node[draw, ellipse] (5) at (\f, 1.5) {\scriptsize $\cS_8 : \mu^1$};
\node[draw, ellipse] (6) at (\g, 1.5) {\scriptsize $\cS_{10} : \mu^1$};

 
\draw[->,>=stealth,thick] (4) to (0);
\draw[->,>=stealth,thick] (4) to (1);
\draw[->,>=stealth,thick] (5) to (1);
\draw[->,>=stealth,thick] (5) to (2);
\draw[->,>=stealth,thick] (6) to (2);
\draw[->,>=stealth,thick] (6) to (3);

\end{tikzpicture}
\end{minipage}
\vspace{2ex}
\caption{The pure tessellar phase diagram $(\otessph,\le^\ddagger)$ for the example in Figure \ref{fig:braidclass}.  }\label{fig:braid:cmg}
\end{figure}


\subsection{Lap number grading and parabolic homology}
\label{lapgrading}
From this point on we use field coefficients $R=\K$.
In the above calculations  we compute the tessellar homology associated with the discretization $\tile\colon X\to \sSC$. The elements of $\sSC$ represent discretized braid class components and therefore the crossing number $\cross(\cS)\in 2\N$ is well-defined for all $\cS\in \sSC$, i.e. we have an order-preserving map $\cS \mapsto \frac{1}{2}\cross(\cS)$.
The latter may be regarded as a discretization map and will be denoted by $\lap\colon \sSC \to \N$. This yields the following factorization
\begin{equation}
\label{factorforLap}
   \dgARROWLENGTH=3.5em
    \begin{diagram}
        \node{} \node{\sSC}\arrow{s,r}{\lap}\\
        \node{X}\arrow{e,l}{\pb}\arrow{ne,l}{\tile} \node{\N}
    \end{diagram}
\end{equation}
where $\pb\colon X\to \N$ is the composition of $\tile$ and $\lap$.\index{$\pb$}
By the same token as before the discretization $\lap$\index{Lap number}\index{Lap number grading} yields an
$\N$-graded differential module/chain complex $\bigl(C^\tile,\cm^\tile \bigr)$ and  an $\N$-grading of the tessellar homology $H^\tile$. Following the procedure in Section \ref{dimensiongrading} 
we have:
\begin{definition}
\label{parahom1}
The scalar discretization $\lap\colon \sSC \to \N$ yields a bi-graded homology theory for $H^\tile$ 
 which is be denoted by 
\begin{equation}
    \label{parahom2}
\tH_{p,q}(X):=G_pH^\tile_q(X),\quad p,q\in \N,
\end{equation}
and will be referred to as\index{Parabolic homology} the \emph{parabolic homology}.\footnote{The definition of the parabolic homology and the bi-grading uses the fact that we use field coefficients. In general we only obtain bi-graded Betti numbers.}
\end{definition}

The parabolic homology is defined for all  sets $G_{\cU\smin\cU'}X$, with $\cU\smin\cU'$ convex in $(\sSC,\le)$. In particular,
%
if we restrict to the convex sets $\{\cS\}$ in $(\sSC,\le)$ we obtain 
\begin{equation}
    \label{singleBM}
\tH_{p,q}(G_\cS  X) = \begin{cases}
     H^\BM_q(G_\cS X) & \text{ for } p=\lap\, G_\cS X; \\
      0 & \text{otherwise}.
\end{cases}
\end{equation}
For convex sets $\cU\smin\cU'\subset \cX$ we have: \[
\vec{P}_{\lambda,\mu}(G_{\cU\smin\cU'}X) = \sum_{p,q\in \N} \bigl({\mathrm{rank~}} \tH_{p,q}(G_{\cU\smin\cU'}X)\bigr) \lambda^p \mu^q
\]
and 
$
P_{\lambda,\mu}\bigl(C^\tess(G_{\cU\smin\cU'}X)\bigr) = \sum_{\cS\in \cU\smin\cU'}\sum_{p,q\in \N} \bigl(\rk \tH_{p,q}(G_\cS X)\bigr) \lambda^p \mu^q.
$
The Morse relation from Theorem \ref{morserel1} yield
\begin{equation}
    \label{MRagain1}
    \sum_{\cS\in \sSC}
    \vec{P}_{\lambda,\mu}\bigl(G_\cS X \bigr) = \vec{P}_{\lambda,\mu}(X) + \sum_{r=1}^\infty (1+\lambda^r \mu) Q^r_{\lambda,\mu},
\end{equation}
with $Q^r_{\lambda,\mu}\ge 0$.

The matrices below describe the tessellar differential with $q$ and $p$ grading respectively for the skeleton $y$ in Figure \ref{fig:braidclass}. 
\[\small
\begin{blockarray}{ccccc|ccc}
 & ~~\cS_0\!\!\!\!\!&\cS_1\!\!\!\!\!&\cS_3\!\!\!\!\!&\cS_4&\cS_6\!\!\!\!\!&\cS_8\!\!\!\!\!\!&\cS_{10}\\ 
 \begin{block}{c(cccc|ccc)}
   \cS_0 &  &  & & &1&0&0        \\  
   \cS_1 &  &  & & &0&0&1 \\  
   \cS_3 &  &   & & &1&1&0       \\
   \cS_4 &  &   & & &0&1&1       \\\cline{1-8}
   \cS_6 &  &   & & & &  &        \\  
   \cS_8 &  &  & & & &  &    \\  
   \cS_{10} &  & & & &  &        \\  
 \end{block}
\end{blockarray}\quad\quad\quad
\begin{blockarray}{ccc|cc|ccc}
 & ~~\cS_0\!\!\!\!\!&\cS_1\!\!\!&\cS_3\!\!\!\!\!&\cS_4\!\!\!&\cS_6\!\!\!\!\!&\cS_8\!\!\!\!\!\!&\cS_{10}\\ 
 \begin{block}{c(cc|cc|ccc)}
   \cS_0 &  &  &0 &0 &1&0&0        \\  
   \cS_1 &  &  &0 &0 &0&0&1 \\  \cline{1-8}
   \cS_3 &  &   & & &1&1&0       \\
   \cS_4 &  &   & & &0&1&1       \\\cline{1-8}
   \cS_6 &  &   & & & &  &        \\  
   \cS_8 &  &  & & & &  &    \\  
   \cS_{10} &  & & & &  &        \\  
 \end{block}
\end{blockarray}
\]
Using  $\K=\Z_2$
the left matrix is $\cm^\tile$ as chain complex boundary for $C^\tess(X) = \bigoplus_{q\in \{0,1\}} C_q^\tile(X)$, with 
\[
C_0^\tile(X) = \bigoplus_{i\in \{0,1,3,4\}}
\Z_2\langle \cS_i\rangle, \quad\text{~~ and ~~}\quad C_1^\tess(X) = \bigoplus_{i\in \{6,8,10\}}
\Z_2\langle \cS_i\rangle.
\]
The right matrix is $\cm^\tile$ as differential for
the $\Z$-graded vector space $C^\tile(X) = \bigoplus_{p\in \{0,1,2\}} G_pC^\tile(X)$,
with 
\[
G_0C^\tile(X) = \!\!\!\bigoplus_{i\in \{0,1\}}
\Z_2\langle \cS_i\rangle,~~~ G_1C^\tile(X) = \!\!\!\bigoplus_{i\in \{3,4\}}
\Z_2\langle \cS_i\rangle,~~~
G_2C^\tile(X) = \!\!\!\!\!\bigoplus_{i\in \{6,8,10\}}\!\!
\Z_2\langle \cS_i\rangle.
\]
Figure \ref{fig:braid:cm} depicts $C^\tile(X)$  as  chain complex  and Figure \ref{fig:braid:cm12}
 below depicts $C^\tile(X)$ the differential vector space.
\begin{figure}[h!]
\centering
\begin{minipage}{.825\textwidth}
\centering
\begin{tikzcd}[column sep = 1.1em]
 0 \arrow{r} &
\bigoplus_{i\in\{6,8,10\}} \Z_2 \langle \cS_i\rangle\arrow{r}{\tiny\begin{pmatrix}1~~~1~~~0\\
0~~~1~~~1\end{pmatrix}}
\ar[to path={ -- ([yshift=-3.5ex]\tikztostart.south) -| (\tikztotarget)},
        rounded corners]{rr}
& [2.5em]
\bigoplus_{i\in\{3,4\}}\Z_2 \langle \cS_i\rangle \arrow{r}{\tiny\begin{pmatrix}0~~~0\\
0~~~0\end{pmatrix}} 
\arrow[phantom]{d}[pos=0.35]{\tiny{\begin{pmatrix} 1~~~0~~~0\\ 0~~~0~~~1\end{pmatrix} }}
& [2em]
\bigoplus_{i\in\{0,1\}}\Z_2 \langle \cS_i\rangle \arrow{r} & 0\\
 &  & \phantom{.}  
\end{tikzcd}
\end{minipage}
    \caption{The tessellar differential module $C^\tile(X)$ with the terms of the tessellar differential acting on the different groups $G_pC^\tile(X)$, $p=0,1,2$.}
    \label{fig:braid:cm12}
    \vskip-.4cm
\end{figure}
If we determine the parabolic homology of the sets $\cS\in \sSC$ in Figure \ref{fig:braidclass} we obtain a refinement of the reduced tessellar phase given in Figure \ref{doublepoinpol}.
To compute the bi-graded homology we use the spectral recurrence procedure.
To illustrate the lap number grading we start with computing $\tH_{p,q}(X)$ for all $p,q$. 
Recall that $G_p X$ is a convex set and is the union  of braid classes with lap number $p$ and $F_{\downarrow p} X$ is the 
union   of braid classes with lap number less or equal to $p$. 
We compute the homologies from the chain complex $(C^\tile,\cm^\tile)$. We have the following short exact sequences:
\[
\begin{aligned}
0 \xrightarrow[]{~~~~} F_{\downarrow 0} C_0^\tile \xrightarrow[]{i_{0,0}} F_{\downarrow 1} C_0^\tile \xrightarrow[]{j_{1,0}}
G_1 C_0^\tile \xrightarrow[]{~~~~} 0;\\
0 \xrightarrow[]{~~~~} F_{\downarrow 1} C_0^\tile \xrightarrow[\cong]{i_{1,0}} F_{\downarrow 2} C_0^\tile \xrightarrow[]{j_{2,0}}
G_2 C_0^\tile \xrightarrow[]{~~~~} 0;\\
0 \xrightarrow[]{~~~~} F_{\downarrow 1} C_1^\tile \xrightarrow[]{i_{1,1}} F_{\downarrow 2} C_1^\tile \xrightarrow[]{j_{2,1}}
G_2 C_1^\tile \xrightarrow[]{~~~~} 0,
\end{aligned}
\]
and the isomorphisms 
\[
F_{\downarrow 0} C_0^\tile \xrightarrow[\cong]{j_{0,0}} G_0C_0^\tile,\quad
F_{\downarrow 1} C_1^\tile \xrightarrow[\cong]{j_{1,1}} G_1C_1^\tile \cong 0,\quad G_2C_0^\tile \cong 0.
\]
The vector spaces are generated by the non-trivial classes given in Figure \ref{fig:braid:cm}.
To build the spectral sequence we have have page 1: $E^1_{0,0} = H^\tile_0(G_0 X)\cong \Z_2^2$, $E^1_{1,0} = H^\tile_0(G_1 X)\cong \Z_2^2$ and $E^1_{2,1} = H^\tile_1(G_2 X)\cong \Z_2^3$, cf.\ Sect.\ \ref{dimensiongrading}. Furthermore,
$E^1_{0,1}$, $E^1_{1,0}$, $E^1_{1,1}$, $E^1_{2,2}$ and $E^1_{p,q}$, for $p,q>2$, are trivial.
The relevant spectral sequences are:
\[
0\xrightarrow[]{~~~~} E^1_{2,1}\xrightarrow[]{\dff^1_{2,1}} E^1_{1,0}\xrightarrow[]{~~~~} 0,
\quad 
0\xrightarrow[]{~~~~} E^1_{2,1}\xrightarrow[]{\dff^1_{1,1}} E^1_{1,0}\xrightarrow[]{~~~~} 0,
\]
where $E^1_{1,1}=0$ and $\dff^1_{1,1}=0$.
The differential $\dff^1_{2,1}$ can be determined from the above data.
From the the third short exact sequence above we have
\[
\cdots\xrightarrow[]{} H_1^\tile(F_{\downarrow 1} X)
\xrightarrow[]{i_{1,1}} H_1^\tile(F_{\downarrow 2} X)\xrightarrow[]{j_{2,1}} H_1^\tile(G_2 X)
\xrightarrow[]{k_{2,1}} H_0^\tile(F_{\downarrow 1} X)\xrightarrow[]{}\cdots,
\]
where $k_{2,1}$ is the connecting homomorphism.
The vector space $E^1_{2,1}$ is given by  $\cm_1(\cU)\colon G_2C_1^\tile \to G_1C_0^\tile$, where $\cm_1(\cU)$, with $\cU = \lap^{-1} 2$, is the zero matrix. Let $\gamma = (a,b,c)\in \bigoplus_{i\in \{6,8,10\}} \Z_2\langle\cS_i\rangle$ be a cycle. Then, the inverse image under $j_{2,1}$ is the same element $\gamma\in F_{\downarrow 2} C_1^\tile$. Apply $\cm^\tile$, i.e. $\cm_1 \gamma=
(a,c,a+b,b+c)\in F_{\downarrow 1} C_0^\tile$, which is also the homology class in $H_0^\tile(F_{\downarrow 1} X)$.
This calculation yields: $k_{2,1} = \cm_1$. 
From the the first short exact sequence above we have
\[
0\xrightarrow[]{} H_0^\tile(F_{\downarrow 0} X)
\xrightarrow[]{i_{0,0}} H_0^\tile(F_{\downarrow 1} X)\xrightarrow[]{j_{1,0}} H_0^\tile(G_1 X)
\xrightarrow[]{} 0,
\]
where the map $j_{1,0}$ is given by 
$(a,b,c,d)\mapsto (0,0,c,d)$. The differential $\dff^1_{2,1}$ is given by $\dff^1_{2,1} = j_{1,0}\cdot k_{2,1}$ and
\[
\dff^1_{2,1} = \begin{pmatrix}
0~~~0~~~1~~~0\\
0~~~0~~~0~~~1
\end{pmatrix} 
\begin{pmatrix}
1~~~0~~~0\\
0~~~0~~~1\\
1~~~1~~~0\\
0~~~1~~~1
\end{pmatrix} =
\begin{pmatrix}
1~~~1~~~0\\
0~~~1~~~1
\end{pmatrix} \colon \Z_2^3 \longrightarrow \Z_2^2.
\]
The next page of the spectral sequence yields the vector spaces
\[
E^2_{2,1} = \ker \dff^1_{2,1}\cong\Z_2,~~~E^2_{1,0} =\frac{\Z_2^2}{\image \dff^1_{2,1}}\cong 0,~~~
E^2_{1,1}=0,~~~E^2_{0,0}\cong \frac{\Z_2^2}{\image \dff^1_{1,1} } \cong \Z_2^2.
\]
To complete the lap number  grading of tessellar homology we compute the third page of the spectral sequence.
The remaining spectral sequence is 
\[
0\xrightarrow[]{~~~~} E^2_{2,1}\xrightarrow[]{\dff^2_{2,1}} E^2_{0,0}\xrightarrow[]{~~~~} 0,
\]
where the differential $\dff^2_{2,1}$ is computed as follows.
From the theory of spectral sequences we have that
$\dff^2_{2,1} = j_{0,0}\cdot i^{-1}_{0,0}\cdot k_{2,1}$. 
Let $\gamma = (a,a,a)\in E^2_{2,1}$, then $k_{2,1} \gamma = (a,a,0,0)\in E^1_{1,0}=H_0^\tile(F_1 X)$. Under $i^{-1}_{0,0}$ we obtain $i^{-1}_{0,0}(a,a,0,0) = (a,a)\in  H_0^\tile(F_{\downarrow 0}\cX)$. Since $F_{\downarrow 0} C_0^\tile \xrightarrow[=]{j_{0,0}} G_0C_0^\tile$ the map $j_{0,0}$ is the identity which implies that $\dff^2_{2,1}$ is given by $(a,a,a) \mapsto (a,a) \in E^2_{0,0}$. This map is
the restriction to $E^2_{2,1}$ of the map 
\[
\widetilde \dff^2_{2,1} = \begin{pmatrix}
1~~~0~~~0\\
0~~~0~~~1
\end{pmatrix}\colon E^1_{2,1}\longrightarrow E^1_{0,0}.
\]
Consequently, $\ker \dff^2_{2,1}=0$ and thus $E^3_{2,1} =0$. Moreover, $\image \dff^2_{2,1} \cong \Z_2$ and thus  $E^3_{0,0} = \frac{\Z_2^2}{\image \dff^2_{2,1}}\cong \Z_2$.
The spectral sequence stabilizes at $r\ge 3$ and we 
define $\tH_{p,q}(X) := E^r_{p,q}$, $r\ge 3$. In particular,
\[
\tH_{p,q}(X) = \begin{cases}
  \Z_2   & \text{ for }  (p,q)= (0,0);\\
   0   & \text{ for }  (p,q)\neq (0,0),
\end{cases}
\]
and thus $\vec{P}_{\lambda,\mu}(X) = 1$. The same procedure can be carried out for other convex sets in $\sSC$.
%
\begin{figure}[h!]
\vspace{2ex}
\centering
\begin{minipage}{.8\textwidth}
\centering
\begin{tikzpicture}[scale=.85,node distance=.2cm and .25cm]
\def\a{0}
\def\b{3}
\def\c{6}
\def\d{9}
\def\e{1.5}
\def\f{4.5}
\def\g{7.5}
\node[draw, ellipse] (0) at (\a, 0) {\scriptsize $\cS_0 : \lambda^0\mu^0$};
\node[draw, ellipse] (1) at (\b, 1.5) {\scriptsize $\cS_3 : \lambda^1 \mu^0$};
\node[draw, ellipse] (2) at (\c, 1.5) {\scriptsize $\cS_4 : \lambda^1\mu^0$};
\node[draw, ellipse] (3) at (\d, 0) {\scriptsize $\cS_1 : \lambda^0\mu^0$};

\node[draw, ellipse] (4) at (\e, 3) {\scriptsize $\cS_6 : \lambda^2\mu^1$};
\node[draw, ellipse] (5) at (\f, 3) {\scriptsize $\cS_8 : \lambda^2\mu^1$};
\node[draw, ellipse] (6) at (\g, 3) {\scriptsize $\cS_{10} : \lambda^2\mu^1$};

 
\draw[->,>=stealth,thick] (4) to (0);
\draw[->,>=stealth,thick] (4) to (1);
\draw[->,>=stealth,thick] (5) to (1);
\draw[->,>=stealth,thick] (5) to (2);
\draw[->,>=stealth,thick] (6) to (2);
\draw[->,>=stealth,thick] (6) to (3);

\end{tikzpicture}
\end{minipage}
\vspace{2ex}
\caption{The pure tessellar phase diagram $(\otessph,\le^\ddagger)$ with parabolic Poincar\'e polynomials. The vertices are positions with height corresponding to their lap numbers.  }\label{doublepoinpol}
\vskip-.4cm
\end{figure}
If we apply the Morse relations in \eqref{MRagain1}
for the skeleton $y$ in Figure \ref{fig:braidclass} we obtain
\[
P_{\lambda,\mu}(C^\tile) = 2+2\lambda+3\lambda^2 \mu =
1+ (1+\lambda \mu) \cdot 2\lambda + (1+\lambda^2 \mu)\cdot 1,
\]
which implies that $Q^1_{\lambda,\mu} = 2\lambda$ and $Q^2_{\lambda,\mu} = 1$. This provides information about the differentials 
$\dff^1_{2,1}$  and $\dff^2_{2,1}$ in the associated spectral sequence.
Note that the sum of the ranks of $\dff^1_{2,1}$  and $\dff^2_{2,1}$ equals the rank of $\cm^\tile$. 


\begin{remark}
If we apply the Cartan-Eilenberg theory, and Theorem \ref{prrepExis} in particular,  to the discretization $\pb\colon X\to \Z$ given by the composition $X\xrightarrow[]{\tile}\sSC \xrightarrow[]{\lap}\Z$ we obtain tessellar differential of the form:
\[
\cm^\pb = \bordermatrix{  & E^1_{0,0} & E^1_{1,0} & E^1_{2,1}   \cr
              E^1_{0,0} & 0 & 0  & \widetilde\dff^2_{2,1}  \cr
              E^1_{1,0} & 0 & 0  & \dff^1_{2,1}\cr
              E^1_{2,1} & 0 & 0 & 0}   
\]
This is exactly the tessellar boundary operator $\cm^\tile$ as given in Fig.\ \ref{fig:braid:cm}. 
The entries of $\cm^\pb$ can also be marked as $\cm^r_{p,q}$.
In this case $\cm^1_{1,0}=0$, $\cm^1_{2,1}=\dff^1_{2,1}$ and
$\cm^2_{2,1} = \widetilde\dff^2_{2,1}$.
\end{remark}



\section{Differential modules and tessellar phase diagrams for positive braids}
\label{tesspara}
For a parabolic flow $\varphi$ with a proper skeleton $y$, i.e. $\mathring y$ is proper, we obtained a canonical Morse pre-order $\le^\dagger$ via a CW-decomposition induced by $y$ and the parabolic recurrence relation $\RR$.
The poset $\sSC$ induced by $y$ describes the Morse tessellation of all discrete braid class components
$[x]\rel y$. The Cartan-Eilenberg theory then provides an algebraic topological data structure that contains topological information about the braid class components (Morse tiles) and algebraic information on how the classes are stitched together.
We give a substantial extension of the results in \cite{im} by
showing that these data structures are invariants of positive conjugacy classes of braids.

\subsection{Some braid theory}
 \label{somebrth}
Following \cite{im,day,CzechV} we recall some basic ideas from braid theory.
Let $x\in \Conf_n^d$ be a discrete braid diagram.
Generically strands in $x$ intersect with $x^\alpha_i\neq x^{\alpha'}_i$, cf. Rmk.\ \ref{genin}. To such generic braid diagrams $x\in \Conf_n^d$ one assigns
 a unique positive word $\beta = \beta(x)$ given by:
\begin{equation}
\label{eqn:word1}
x \mapsto \beta(x) = \sigma_{\alpha_1} \cdots \sigma_{\alpha_\ell},
\end{equation}
 where $\alpha_k$ and $\alpha_k +1$ are the positions of intersection that intersect, cf.\ Rmk.\ \ref{canrep1}.
The  (algebraic) \emph{Artin braid group}\index{Braid group}\index{Artin braid group} $\BB_{n}$ is a free group spanned by the $m-1$ generators $\sigma_{\alpha}$, modulo
following relations:
\begin{align}\label{eqn:braidrel}
 \begin{cases}
  \sigma_{\alpha} \sigma_{\alpha'} = \sigma_{\alpha'} \sigma_{\alpha}, & \ |\alpha-\alpha'| \geq 2,\ \alpha,\alpha \in \{0, \dots ,n-2\} \\
  \sigma_{\alpha} \sigma_{\alpha+1} \sigma_{\alpha} = \sigma_{\alpha+1} \sigma_{\alpha} \sigma_{\alpha+1}, & \ 0\le \alpha \le n-3.
 \end{cases}
\end{align} 
%
Presentations of words consisting only of the $\sigma_i$'s (not the inverses) and the relations in \eqref{eqn:braidrel} form a monoid
which is called the \emph{positive braid monoid}\index{Positive braid monoid} $\BB_n^+$.
%
Two positive words\index{Braid word}\index{Braid word!positive} $\beta$ and $\beta'$ are \emph{positively equal}\index{Braid word!positively equal}
if they represent the same element in $\BB_n^+$ by using  the $\sigma$-relations in the braid group. Notation $\beta \sdoteq\beta'$.
Two positive words $\beta,\beta'$ are \emph{positively conjugate}\index{Braid word!positively conjugate} if there exists a sequence of words $\beta_0,\cdots,\beta_\ell\in \BB_n^+$, with $\beta_0=\beta$ and $\beta_\ell =\beta'$, such
that for all $k$, either  $\beta_k\sdoteq \beta_{k+1}$, or $\beta_k\equiv \beta_{k+1}$, where the latter is defined by
\[
\sigma_{\alpha_1}\sigma_{\alpha_2}\cdots\sigma_{\alpha_p} \equiv \sigma_{\alpha_2}\cdots\sigma_{\alpha_p}\sigma_{\alpha_1},
\]
cf.\ \cite[Sect.\ 2.2]{day}.
Notation
$\beta \possim \beta'$. Positive conjugacy is an equivalence relation on $\BB^+_n$ and 
the positive conjugacy class of $\beta\in \BB^+_n$ is denoted by $\langle\beta\rangle$. 
The set of all positive conjugacy classes in $\BB^+_n$  is  denoted by
$\langle \BB_n^+\rangle$.

\begin{definition}
\label{defn:topeq}
Two discretized braids $x, x'\in \Conf_n^d$ are \emph{topologically equivalent}\index{Braid!positively equivalent} if 
$\beta(x)$ and $\beta(x')$ are positively conjugate.
Notation: $x \possim x'$.
\end{definition}

Recall that $x\sim x'$ if and only if $x,x'\in [x]$. Clearly,
$x\sim x'$ implies $x \possim x'$
which defines a coarser equivalence relation on $\Conf_n^d$.
Denote the equivalence classes with respect to $\possim$  by $[x]_{\possim}$.
The converse   is not true in general, cf.\ \cite[Fig.\ 8]{im}.
Following \cite[Def.\ 17]{im}, a discretized braid class $[x]$ is \emph{free} if
$[x] = [x]_{\possim}$.
\begin{remark}
For non-generic $x\in \Conf_n^d$ we choose $\beta(x)$ to be any representative in the positive conjugacy class $\langle\beta(x')\rangle$, for any $x'\sim x$. 
\end{remark}

\begin{proposition}[\cite{im}, Prop.\ 27]
\label{prop:free}
If $d>\cross(x)$, then $[x]$ is a free braid class.
\end{proposition}

For the space of $2$-colored discretized braid diagrams $\Conf_{1,m}^d$
there exists a natural embedding $\Conf_{1,m}^d \hookrightarrow \Conf_{1+m}^d$ by regarding $x\rel y$ as braid in $\Conf_{1+m}^d$.
Via the embedding we   define the notion of topological equivalence of two 2-colored discretized braids:\index{Braid!2-colored}
$x\rel y \possim x'\rel y'$ if they are topologically equivalent as braids in $\Conf_{1+m}^d$.
The associated equivalence classes are denoted by $[x\rel y]_{\possim}$, which are
not necessarily connected sets in $\Conf_{1,m}^d$. A 2-colored discretized braid class $[x\rel y]$ is free if $[x\rel y] = [x\rel y]_{\possim}$.
If $d>\cross(x\rel y)$, then  $[x\rel y]$ is free by Proposition \ref{prop:free}. 

In Theorem \ref{brclasses123} we showed that each partial equivalence class of 
$(\cX,\le^\dagger)$
corresponds with braid class components $[x]\rel y$ of the fiber $\pi^{-1}(y)$.
The Borel-Moore homology of $[x]\rel y$
is independent of the choice of parabolic recurrence relations $\RR$ for which $\RR(y)=0$, cf.\ \cite[Thm.\ 15(a)-(b)]{im}, and therefore the parabolic homology is independent of $\RR$, i.e.
\begin{equation}
    \label{BMI1}
    \tH_{p,q}\bigl([x]\rel y;\varphi\bigr) = \tH_{p,q}\bigl([x]\rel y;\varphi'\bigr).
\end{equation}
A similar statement holds for the braid classes $[x\rel y]$. 
Let  $x\rel y \sim x'\rel y'$ and let $\varphi$ and $\varphi'$ be parabolic flows with skeleton $y$ and $y'$ respectively. Then, 
\begin{equation}
    \label{BMI2}
    \bigoplus_{[x]\rel y\subset\atop \pi^{-1}(y)\cap[x\rel y]}\!\!\!\!\!\!\!\!\! \tH_{p,q}\bigl([x]\rel y;\varphi\bigr) ~~\cong\!\!\!\!\!\!\!\!\!
    \bigoplus_{[x']\rel y'\subset\atop \pi^{-1}(y')\cap[x'\rel y']}\!\!\!\!\!\!\!\!\! \tH_{p,q}\bigl([x']\rel y';\varphi'\bigr)
\end{equation}
 cf.\ \cite[Thm.\ 15(c)]{im}.
This makes the latter  an invariant of the discrete 2-colored braid class $[x\rel y]$
and justifies the notation:
\begin{equation}
    \label{BMI3}
\tH_{p,q}\bigl([x\rel y]\bigr)~~ := \!\!\!\!\!\!\!\!\!\bigoplus_{[x]\rel y\subset\atop \pi^{-1}(y)\cap[x\rel y]} \!\!\!\!\!\!\!\!\!\tH_{p,q}\bigl([x]\rel y;\varphi\bigr).
\end{equation}
%
The homology $\tH_{p,q}\bigl([x\rel y]\bigr)$  is not necessarily independent of to the number of discretization points $d$.
In order to have  independence also with respect to $d$, another invariant for discrete braid classes was introduced in \cite{im}.
Consider the equivalence class $[x\rel y]_{\possim}$ of  discrete 2-colored braid diagrams induced by the relation $x\rel y \possim x'\rel y'$ on $\Conf_{1,m}^d$.
As before the projection $\pi\colon \Conf_{1,m}^d \to \Conf_m^d$ given by $x\rel y\mapsto y$ yields the fibers $\pi^{-1}(y)$ and components
$[x]\rel y \subset \pi^{-1}(y)\cap [x\rel y]_{\possim}$
and defines the homology:
\begin{equation}
    \label{BMI4}
    \tH_{p,q}\bigl([x\rel y]_{\possim}\bigr)~~ := \!\!\!\!\!\!\!\!\!\bigoplus_{[x]\rel y\subset\atop \pi^{-1}(y)\cap [x\rel y]_{\possim}} \!\!\!\!\!\!\!\!\!\tH_{p,q}\bigl([x]\rel y;\varphi\bigr).
\end{equation}
Note that if $d>\cross(x\rel y)$ then the homology in \eqref{BMI4} corresponds to the homology in \eqref{BMI3}.
As before $\tH_{p,q}\bigl([x\rel y]_{\possim}\bigr)$ is an invariant and does not depend on the choice of $\varphi$ and $y$. In the next section we explain the invariance
by investigating the dependence on the number of
 discretization points $d$ making it a topological invariant for relative braid classes.

The braid class components $[x]\rel y$ comprise the elements of $\sSC$ and the equivalence relation $\possim$ yields a span which we express in terms of tessellar phase diagrams
\[
\begin{tikzcd}
(\cX,\le^\dagger) \arrow[r, two heads] \arrow[rd, two heads] & (\tessph,\le^\dagger) \arrow[d, two heads] & \otessph,\le^\ddagger) \arrow[d, two heads] \arrow[l, hook'] \\
                                             & ({\tessph(y)},\le)                      & (\otessph(y),\le) \arrow[l, hook']                     
\end{tikzcd}
\]
where $(\tessph(y),\le)$ is the coarsening of $(\tessph,\le^\dagger)$ defined by unionizing
equivalent pairs $[x]\rel y,[x']\rel y\subset \pi^{-1}(p)\cap [x\rel y]_{\possim}$ whenever
$[x]\rel y \possim [x']\rel y$ 
and taking transitive, reflexive closure.
The equivalence classes are parallel in $\tessph$ and thus the coarsening yields a well-defined poset $(\tessph(y),\le)$.
The poset $(\otessph(y),\le)$ is the restriction of pairs in $\tessph(y)$ for which the Poincar\'e polynomials that are the non-zero. 
The elements of $\tessph(y)$ can be identified with $[x\rel y]_{\possim}$. 
By defining the groups in \eqref{BMI3} we consider convex sets in $\sSC$. Therefore, the differential $\cm^\tile$ yields a induced differential $\cm^\para$ on the groups $\tH_{p,q}\bigl([x\rel y]_{\possim}\bigr)$.
This leads to the following definition:
\begin{definition}
\label{parabolicmod}
Let $y\in \Conf^d_m$ be a proper discrete braid.
Define the 
differential module\footnote{We use the term differential module even though we use field coefficients which makes $\scrA$ a differential vector space.}
\begin{equation}
    \label{span123}
\rA^\para(X):= \bigoplus_{[x\rel y]_{\possim}} \bigoplus_{p,q} \tH_{p,q}\bigl([x\rel y]_{\possim}\bigr),\quad \cm^\para\colon C^\para(X)\to C^\para(X).
\end{equation}
where
$\cm^\para$ is the induced   differential  for $(\cX,\le^\dagger)$.
The non-trivial homologies $\tH_{p,q}\bigl([x\rel y]_{\possim}\bigr)\neq 0$
yield an $\otessph(y)$-grading of $C^\para(X)$.
The $\otessph(y)$-graded differential module $\scrA = \scrA(y):= \bigl( \rA^\para,\cm^\para\bigr)$ is called the  \emph{parabolic differential module} for $y$.\index{Parabolic differential module}\index{Differential module!parabolic}
\end{definition}
Convex sets in $\otessph(y)$ yield convex sets in $\tessph(y)$, i.e. $\sCo(\otessph(y)) \hookrightarrow \sCo(\tessph(y))$\footnote{Recall that for a poset $(\sP,\le)$ the lattice of convex sets in $\sP$ is denoted by $\sCo(\sP)$, cf.\ \cite{birkhbennet}.} via convex hull: if $I\in \sCo(\otessph(y))$, then the convex hull $\langle I\rangle$ is contained in $\sCo(\tessph(y))$. This implies that $C^\para(X)$ is well-defined over every convex set in $\sCo(\otessph(y))$.
By construction the homology is given by $H(\scrA) \cong \K$.
The grading by lap number is a first $\Z$-grading. 
The dimension grading of Borel-Moore homology comes as a second $\Z$-grading for $\scrA$. This yields the bi-graded homology   $\tH_{p,q}(\scrA)$.

\begin{remark}
For most of the examples  in this paper 
the partial order on $\scrA$ is given by $\otessph$, i.e. $\otessph \cong \otessph(y)$. 
In general this holds for $d$ large enough.
\end{remark}

\begin{remark}
In general there  does not exist an order-retraction $\tessph \twoheadrightarrow\otessph$, cf.\ \cite{kkv}. If such an order-retraction exists one obtains the discretization $(\cX,\le^\dagger) \twoheadrightarrow(\tessph,\le^\dagger) \twoheadrightarrow(\otessph,\le^\ddagger)$ for which all tiles have non-trivial tessellar homology. It is not clear if this behaves well under stabilization. 
The same question applies to $\tessph(y)$ and $\otessph(y)$.
In the case of parabolic homology the remedy is two use the bi-grading and the order given by $\otessph(y)$.
\end{remark}

Figure \ref{fig:braidclass1234} below shows a representation of $\scrA(y)$ for the discrete skeletal braid $y$ displayed in Figure \ref{fig:braid:cmg12}.
The advantage of this representation is that the $\otessph(y)$-grading is expressed in the diagram. Contracting to lap numbers yields the diagram in Figure \ref{fig:my_label2b}.
\vspace{-3ex}
\begin{figure}[h!]
\centering
\tikzcdset{every label/.append style = {font = \tiny}}
\tikzcdset{close/.style={outer sep=-5pt}}
\begin{tikzcd}[column sep=0.7em, row sep=1.5em]
            &                         & \Z_2\langle \cS_8\rangle  \arrow[rd, "(1)"] &                         & \Z_2\langle \cS_5\rangle \arrow[rd, "(1)"] \arrow[rrd, bend left, "\begin{pmatrix}0\\ 1 \end{pmatrix}" ] &             &                         & \Z_2\langle \cS_1\rangle \arrow[rd] &   \\
0 \arrow[r] & \Z_2\langle \cS_9\rangle \arrow[ru,"(1)"] \arrow[rd, "(1)"'] &              & \Z_2\langle \cS_6\rangle \arrow[ru, "(0)"] \arrow[rd,"(0)"'] &                                     & \Z_2\langle \cS_3 \rangle  \arrow[r,"{\tiny \begin{pmatrix}0\\ 0 \end{pmatrix}}"] &[3ex] \Z_2^2\langle \cS_2\rangle \arrow[ru,"(1~~~0)"] \arrow[rd, "(1~~~0)"'] &              & 0 \\
            &                         & \Z_2\langle \cS_7\rangle  \arrow[ru,"(1)"'] &                         & \Z_2\langle \cS_4\rangle \arrow[ru,"(1)"']                        &             &                         & \Z_2\langle \cS_0\rangle \arrow[ru] &  
\end{tikzcd}
\vspace{2ex}
 \caption{
 The diagram displays the lap number grading as well as the 
$\otessph$-grading of $\scrA(\beta)$. The $\Z_2$-groups are the parabolic homologies $\tH_{p,q}(\cS_j)$ 
}\label{fig:braidclass1234}
\end{figure}

\subsection{Stabilization}
\label{stabofclasses}

In order to formulate the main results of this section we recall some definitions and methods from \cite{im}.
For a skeleton $y\in \Conf_m^d$ we define the extension operator $\E\colon \Conf_m^d \to \Conf_m^{d+1}$ as:
\[
(\E y)^\alpha_i  := \begin{cases}
    y^\alpha_i  & \text{ for } i=0,\cdots, d, \\
     y^\alpha_d & \text{ for } i=d+1.
\end{cases}
\]
For a given braid class $[x\rel y]$ extension does not change any of the properties, i.e.
$[\E x\rel \E y]$ is both bounded and non-degenerate. The same remains true under repeated application of the operator $\E$.
The main result in \cite[Thm.\ 20]{im} states that the Borel-Moore homology of $[x\rel y]_{\possim}$ remains unchanged under application of $\E$. 
For the parabolic homology this implies:
\begin{theorem}
 \label{stabparhombrcl}
Let $[x\rel y]_{\possim}$ be a relative braid class then
\begin{equation}
    \label{BMI5}
    \tH_{p,q}\bigl([\E x\rel \E y]_{\possim}\bigr)\cong \tH_{p,q}\bigl([x\rel y]_{\possim}\bigr)
\end{equation}
\end{theorem}

\begin{proof}
The lap number does not change under the extension operator $\E$, i.e. $\lap\big([\E x\rel \E y]_{\possim}\bigr) = \lap\bigl([x\rel y]_{\possim} \bigr)$, where the latter is defined as the lap number of is braid class fibers. Since the Borel-Moore homologies are isomorphic by  \cite[Thm.\ 20]{im} Equation \eqref{singleBM} implies that the parabolic homologies are isomorphic.
\end{proof}
By Proposition \ref{prop:free} $[x\rel y]$ 
a relative braid class is free if $d$ is sufficiently large and thus
$[\E^\ell x\rel \E^\ell y]$ is free for $\ell\ge 0$ sufficiently large.
Let $x\rel y \possim x'\rel y'$, then $\E^\ell x\rel \E^\ell y \possim \E^\ell x'\rel \E^\ell y'$
and both are contained in the same connected component $[\E^\ell x\rel \E^\ell y]_{\possim}$.
By Equation \eqref{BMI2} this implies that 
$\tH_{p,q}\bigl([\E^\ell x\rel \E^\ell y]_{\possim}\bigr)\cong 
\tH_{p,q}\bigl([\E^\ell x'\rel \E^\ell y']_{\possim}\bigr)$. Combining these isomorphisms gives
\[
\begin{aligned}
\tH_{p,q}\bigl([x\rel y]_{\possim}\bigr) &\cong \tH_{p,q}\bigl([\E^\ell x\rel \E^\ell y]_{\possim}\bigr)\cong 
\tH_{p,q}\bigl([\E^\ell x'\rel \E^\ell y']_{\possim}\bigr)\\
&\cong \tH_{p,q}\bigl([x'\rel y']_{\possim}\bigr),
\end{aligned}
\]
which establishes $\tH_{p,q}\bigl([x\rel y]_{\possim}\bigr)$ as a topological invariant for relative braid classes.
The proof of \cite[Thm.\ 20]{im}  is based on a singular perturbation argument for parabolic recurrence relations. We use the same technique now to show that the graded differential module
$\scrA(y)$ is also invariant under extension by $\E$:
\begin{theorem}
 \label{stableh1}
 Let $y\in \Conf^d_m$ be a proper discrete braid. Then,
 \begin{equation}
     \label{stableh2}
     \scrA(\E y)  \cong \scrA(y).
 \end{equation}
 In particular, the grading is given by $\otessph(\E y) \cong \otessph(y)$.
\end{theorem}
\begin{proof}
We will outline the main steps in the proof and indicate the additional results that can be obtained from this method.
In order to accommodate the extension operator $\E$ acting on discrete braids we consider the parabolic recurrence relation $\RR^\epsilon$ defined as $R^\epsilon_i=R_i$ for $i=0,\cdots,d-1$ and $R^\epsilon_d = \epsilon^{-1}(x_{d+1}-x_d)$, and
$R_0 = R_0(x_0,x_1)$. 
For $\epsilon>0$ the parabolic flow is denoted by $\varphi^\epsilon$ on the augmented phase space $\widetilde X$, with $\widetilde x\in \widetilde X$ given by $\widetilde x = (x_0,\cdots, x_d) = (x,x_d)$. 
In the singular limit $\varphi^0 = \E\varphi$ is the induced parabolic flow on $\E X$.
Consider the coordinates
$(x,z)$ with $z=x_{d+1}-x_d = x_0-x_d$. With the reparametrization of time by $\tau=t/\epsilon$ we obtain the differential equations:
\begin{equation}
    \label{exteqn1}
    \begin{aligned}
    x' &= \epsilon X(x,z);\\
    z' &= -z + \epsilon Z(x),
    \end{aligned}
\end{equation}
where $X(x,z)$ and $Z(x)$ are given by: $X_i(x)= R_i(x_{i-1},x_i,x_{i+1})$, with  $i=0,\cdots,d-2$,
$X_{d-1}(x,z) = R_{d-1}(x_{d-2}, x_{d-1},x_0-z)$  and $Z(x) = R_0(x_0,x_1)$. The associated flow on $\widetilde X$ is denoted by $\varphi^\epsilon$.

As for $\varphi$ consider $\varphi^\epsilon$ with skeleton $\E y$ defined on the compact phase space $\widetilde X$. The construction in Section \ref{sec:parabolic:model} of the Morse pre-order for $\varphi^\epsilon$ follows by the same token as before and we denote the Morse pre-order by $(\widetilde X,\le^\dagger_\epsilon)$. The poset of partial equivalence classes is denoted by $(\sSC_\epsilon,\le_\epsilon)$. 
For $\epsilon >0$ we obtain a tessellar phase diagram $(\tessph_\epsilon,\le^\dagger_\epsilon)$ and a pure tessellar phase diagram $(\otessph_\epsilon,\le^\ddagger_\epsilon)$. Since the diagrams only depend on topological data they are independent of $\epsilon>0$.
The results in \cite{im} imply that for every fiber $\pi^{-1}(y)\cap [x\rel y]_{\possim}$ in
$\Conf_{1,m}^d$ and $\pi^{-1}(\E y)\cap [\E x\rel \E y]_{\possim}$ in
$\Conf_{1,m}^{d+1}$ the Borel-Moore homologies are the same, cf.\ Thm.\ \ref{stabparhombrcl}.
The tessellar phase diagrams associated to the non-trivial homologies are denoted $\tessph(\E y)$ and $\otessph(\E y)$.
It also follows that for all classes in $[\widetilde x]\rel \E y \subset \pi^{-1}(\E y)$ for which
 $\pi^{-1}(\E y)\cap \E X = \varnothing$ the Borel-Moore homology is trivial, cf.\ \cite[Rmk.\ 23]{im}. These are exactly the classes that emerge when we extend via the operator $\E$. 
 Indeed, if the maximal invariant set $\Inv([\widetilde x]\rel \E y)\neq \varnothing$ then the convergence properties of \eqref{exteqn1}, cf.\ \cite[Lem.\ 22]{im}, imply that there is a non-trivial invariant set for $\epsilon=0$, which contradicts the fact that $\pi^{-1}(\E y)\cap \E X = \varnothing$.
The pure tessellar phase diagram $\otessph(\E y)$ obtained from $\tessph(\E y)$  remains unchanged since the non-trivial since a braid class $[x\rel y]_{\possim}$ has non-trivial homology if and only if $[\E x\rel \E y]_{\possim}$ has non-trivial homology. The partial order also remains unchanged.
Indeed, by the above construction $(\tessph,\le^\dagger) \hookrightarrow (\tessph_\epsilon,\le^\dagger_\epsilon)$ and thus two elements in $\otessph(y)$ are ordered if and only if they are ordered in $\otessph(\E y)$, which proves that the pure tessellar phase diagram is invariant under the action of $\E$.
By the same token the invariance of the homologies for convex sets in $\otessph(y)$ are non-trivial if and only the homologies of the associated convex sets in $\otessph(\E y)$ are non-trivial. By the nature of the induced connection matrix $\cm^\para$ the homology braids are isomorphic and therefore the differential on $C^\para$ for $\E y$ can be chosen to be the same all $\epsilon\ge 0$.
These facts combined prove the  theorem.
\end{proof}

The final result is to prove that the parabolic differential module is in fact an invariant of positive conjugacy classes of braid diagrams.

\begin{theorem}
 \label{stableh6}
 The parabolic differential module $\scrA(\beta)$ of a (topological)  positive braid  $\beta$ is a
 positive conjugacy class invariant.
\end{theorem}

\begin{proof}
Let $x\rel y\in \Conf^d_{1,m}$ and $x'\rel y'\in \Conf^{d'}_{1,m}$ 
be relative braids such that $\beta(x\rel y) \possim \beta(x'\rel y')$. By Proposition
\ref{prop:free} we have that $\E^\ell x\rel \E^\ell y \sim \E^{\ell'}x \rel \E^{\ell'} y'$ for $\ell+d=\ell'+d'$ and  $\ell,\ell'$ sufficiently large. This implies
\[
\begin{aligned}
\tH_{p,q}\bigl([x\rel y]_{\possim}\bigr) &\cong \tH_{p,q}\bigl([\E^\ell x\rel \E^\ell y]_{\possim}\bigr)\cong 
\tH_{p,q}\bigl([\E^{\ell'} x'\rel \E^{\ell'} y']_{\possim}\bigr)\\
&\cong \tH_{p,q}\bigl([x'\rel y']_{\possim}\bigr),
\end{aligned}
\]
and therefore $\scrA(y) \cong \scrA(y')$. This justifies the notation $\scrA(\beta) := \scrA(y)$ with $\beta = \beta(y)$.
\end{proof}

Since $\scrA(\beta)$ is an invariant for a positive braid $\beta$ the associated reduced tessellar phase is denoted by $\bigl(\otessph(\beta),\le\bigr)$. From the tessellar phase diagram we can immediately derive the Poincare polynomial $\vec P_{\lambda,\mu}(\scrA)$ by summing up the term $\vec P_{\lambda,\mu}$ in the tessellar phase diagram.

\subsection{An example}
For the skeletons $y$ we use in this paper the strands $y^-$ and $y^+$ do not intersect with the remaining strands in $\mathring y$, nor do they intersect with $x$. For that reason the convention is to label the basic words in $\BB_m$ by $\sigma_-$ (intersections between $y^-$ and $y^1$), $\sigma_\alpha$, $\alpha=1,\cdots,m-3$ (intersections between $y^\alpha$ and $y^{\alpha+1}$), and
$\sigma_+$ (intersections between $y^{m-2}$ and $y^{m-1}$).
This implies that the words $\beta(y)$ in our setting consists only of the letters $\sigma_1,\cdots,\sigma_{m-3}$ and maybe regarded as a word $\beta(\mathring y) \in \BB^+_{m-2}$. 

In the examples below the words $\beta$ are understood to be words $\beta(\mathring y) \in \BB^+_{m-2}$.
We compute the parabolic module $\scrA(\beta)$ and the pure tessellar phase diagrams for various positively conjugate representation of the braid $\beta = \sigma_1^2(\sigma_2\sigma_1)^2 \in \BB^+_3$.
\begin{lemma}
$\beta \possim \sigma_1^5\sigma_2$.
\end{lemma}
\begin{proof}
From the braid group relations and positive conjugacy we have:

\noindent $
\beta = \sigma_1^2(\sigma_2\sigma_1)^2 = \sigma_1^2\sigma_2\sigma_1\sigma_2\sigma_1 \sdoteq \sigma_1^3\sigma_2\sigma_1^2
\possim \sigma_1^5\sigma_2.
$
\end{proof}

\begin{figure}[htb]
\centering
\vspace{-2ex}
\begin{minipage}{.25\textwidth}
\begin{tikzpicture}[dot/.style={draw,circle,fill,inner sep=.75pt},line width=.7pt,scale=.65]
\def\a{0}
\def\b{1}
\def\c{2}
\def\d{3}
\def\e{4}
\foreach \x in {0,1,2,3}
    \foreach \y in {\a,\b,\c,\d,\e}
        \node (\x\y) at (\x, \y)[dot] {};
\foreach \y in {\a,\e}
    \foreach \x  in {0,1,2}
        \draw (\x,\y) to (\x+1,\y);
\foreach \x in {0,2} {
     \draw (\x, \b) to (\x+1, \c);
     \draw (\x,\c) to (\x+1,\d);
     \draw (\x,\d) to (\x+1,\b);
     }
\draw (1,\d) to (2,\c);
\draw (1,\c) to (2,\b);
\draw (1,\b) to (2,\d);
\end{tikzpicture}
\end{minipage}
\begin{minipage}{.3\textwidth}
\begin{tikzpicture}[dot/.style={draw,circle,fill,inner sep=.75pt},line width=.7pt,scale=.65]
\def\a{0}
\def\b{1}
\def\c{2}
\def\d{3}
\def\e{4}
\foreach \x in {0,1,2,3,4}
    \foreach \y in {\a,\b,\c,\d,\e}
        \node (\x\y) at (\x, \y)[dot] {};
\foreach \y in {\a,\e}
    \foreach \x  in {0,1,2,3}
        \draw (\x,\y) to (\x+1,\y);
\draw (0,\b) to (1,\c);
\draw (1,\c) to (2,\b);
\draw (2,\b) to (3,\c);
\draw (3,\c) to (4,\d);
\draw (0,\c) to (1,\b);
\draw (1,\b) to (2,\c);
\draw (2,\c) to (3,\d);
\draw (3,\d) to (4,\b);
\draw (0,\d) to (1,\d);
\draw (1,\d) to (2,\d);
\draw (2,\d) to (3,\b);
\draw (3,\b) to (4,\c);
\end{tikzpicture}
\end{minipage}
\begin{minipage}{.25\textwidth}
\begin{tikzpicture}[dot/.style={draw,circle,fill,inner sep=.75pt},line width=.7pt,scale=.65]
\def\a{0}
\def\b{1}
\def\c{2}
\def\d{3}
\def\e{4}
\foreach \x in {0,1,2,3,4,5}
    \foreach \y in {\a,\b,\c,\d,\e}
        \node (\x\y) at (\x, \y)[dot] {};
\foreach \y in {\a,\e}
    \foreach \x  in {0,1,2,3,4}
        \draw (\x,\y) to (\x+1,\y);
\draw (0,\b) to (1,\c);
\draw (1,\c) to (2,\b);
\draw (2,\b) to (3,\c);
\draw (3,\c) to (4,\b);
\draw (4,\b) to (5,\d);
\draw (0,\c) to (1,\b);
\draw (1,\b) to (2,\c);
\draw (2,\c) to (3,\b);
\draw (3,\b) to (4,\c);
\draw (4,\c) to (5,\b);
\draw (0,\d) to (1,\d);
\draw (1,\d) to (2,\d);
\draw (2,\d) to (3,\d);
\draw (3,\d) to (4,\d);
\draw (4,\d) to (5,\c);
\end{tikzpicture}
\end{minipage}
\vspace{2ex}
\caption{Three graphical representatives of positively conjugate braids whose  words $\beta(\mathring y)$ are given by $\sigma_2\sigma_1^2\sigma_2^2\sigma_1$, $\sigma_1^2(\sigma_2\sigma_1)^2$ and $\sigma_1^5\sigma_2$ respectively.}
\label{fig:my_label}
\end{figure}

\begin{lemma}
$\beta \possim \sigma_2\sigma_1^2\sigma_2^2\sigma_1 $.
\end{lemma}
\begin{proof}
As before: $\beta  = \sigma_1^2\sigma_2\sigma_1\sigma_2\sigma_1
\sdoteq \sigma_1^2\sigma_2^2\sigma_1\sigma_2 \possim \sigma_2\sigma_1^2\sigma_2^2\sigma_1$.
\qed
\end{proof}

 Figure \ref{fig:my_label}  depicts the three presentations of the braids conjugate to $\beta$ (including $\beta$). The presentation $\sigma_2\sigma_1^2\sigma_2^2\sigma_1$ has the minimal number of discretization points.
The braid $\sigma_2\sigma_1^2\sigma_2^2\sigma_1$ can be represented in $\Conf^3_5$, the braid $\sigma_1^2(\sigma_2\sigma_1)^2$ in $\Conf^4_5$ and $\sigma_1^5\sigma_2$
in $\Conf^5_5$.
Figure \ref{fig:pseudo1}  below shows the tessellar phase diagrams 
of $\sigma_2\sigma_1^2\sigma_2^2\sigma_1$ 
and $\sigma_1^2(\sigma_2\sigma_1)^2$ for $d=3$ and $d=4$ respectively. As expected the tessellated phase diagrams are not isomorphic since $y\in \Conf^4_5$ allows more relative braid classes. However, if we reduce the tessellated phase diagrams to only those with non-trivial Borel-Moore homology we obtain isomorphic posets $\otessph\bigl(\sigma_2\sigma_1^2\sigma_2^2\sigma_1\bigr) \cong \otessph\bigl(\sigma_1^2(\sigma_2\sigma_1)^2\bigr)$. 
Moreover 
 the associated parabolic modules $\scrA$ for $\sigma_2\sigma_1^2\sigma_2^2\sigma_1$ 
and $\sigma_1^2(\sigma_2\sigma_1)^2$ are isomorphic, cf.\ Thm.\ \ref{stableh6}.

%

\begin{figure}[H]
\centering
\begin{tikzcd}[column sep = 0.7em]
0 \arrow{r} &
\Z_2\langle \cS_{15}\rangle
\arrow{r}{{\begin{pmatrix}
1
\end{pmatrix}}}&[3em] 
\Z_2\langle \cS_{13}\rangle
\arrow{r}{\begin{pmatrix}
0
\end{pmatrix}} & [3em]
\Z_2\langle \cS_{8}\rangle
\arrow{r}{\begin{pmatrix}
1\\
1
\end{pmatrix}} & [3em]
\Z_2\langle \cS_4\rangle \oplus
\Z_2\langle \cS_0\rangle
\arrow{r} & 
0
\end{tikzcd}
\vspace{2ex}
    \caption{The parabolic differential module $\scrA(\beta)$, computed over $\Z_2$ coefficients, represented as chain complex.}
    \label{fig:my_label777}   
\end{figure}
If we apply the Morse relations in \eqref{morsegenrel} we obtain
\[
\vec P_{\lambda,\mu}(\scrA) = 2+\lambda^2 \mu +\lambda^3 \mu^2 + \lambda^4 \mu^3 =
1+ (1+\lambda \mu) \cdot \lambda^3 \mu^2 + (1+\lambda^2 \mu)\cdot 1,
\]
which implies that $Q^1_{\lambda,\mu} = \lambda^3 s^2$ and $Q^2_{\lambda,\mu} = 1$. This provides information about the differentials 
$\dff^1_{4,3}$ and $\dff^2_{2,1}$.
A straightforward but meticulous verification of the Morse relations in \eqref{morsegenrel} show that the choices for $Q^1$ and $Q^2$ are unique. The ranks of  $\dff^1_{4,3}$ and $\dff^2_{2,1}$ correspond to the ranks of the connection matrix $\cm^\tile$.
\begin{figure}[H]
\centering
\begin{minipage}{.44\textwidth}
\centering
\begin{tikzpicture}[node/.style = {ellipse, draw, inner sep = 1.5}, scale=.225]
\def\a{3.15} %
\def\b{2*\a} %
\def\c{3*\a} %
\def\d{4*\a} %

\def\xa{0} 
\def\xb{3} 
\def\xc{6} 
\def\xd{9} 
\def\w{18}

\node[node] (0) at (\xb, 0) {\scriptsize $\cS_0: \lambda^0 \mu^0$};
\node[node] (1) at (\w-\xb, 0)
{\scriptsize $\cS_1: \lambda^0 \mu^0$};

\node[node] (2) at (\xa, \a)
{\scriptsize  };
\node[node] (3) at (\xb, \a)
{\scriptsize };
\node[node] (4) at (\xc, \a)
{\scriptsize };
\node[node] (5) at (\w-\xc, \a)
{\scriptsize  };
\node[node] (6) at (\w-\xb, \a)
{\scriptsize};
\node[node] (7) at (\w-\xa, \a)
{\scriptsize};
\node[node] (8) at (\xa, \b)
{\scriptsize };
\node[node] (9) at (\xb, \b)
{\scriptsize };
\node[node] (10) at (\xd, \b)
{\scriptsize $\cS_{10} : \lambda^2 \mu^1$};

\node[node] (11) at (\w-\xb, \b)
{\scriptsize };
\node[node] (12) at (\w-\xa, \b)
{\scriptsize };
\node[node] (13) at (\xd, \c)
{\scriptsize $\cS_{13} : \lambda^3 \mu^2$};

{\scriptsize };
\node[node] (15) at (\xd, \d)
{\scriptsize $\cS_{14} : \lambda^4 \mu^3$};

\draw[->,>=stealth,thick] (2) to (0);
\draw[->,>=stealth,thick] (3) to (0);
\draw[->,>=stealth,thick] (4) to (0);
\draw[->,>=stealth,thick] (5) to (1);
\draw[->,>=stealth,thick] (6) to (1);
\draw[->,>=stealth,thick] (7) to (1);
\draw[->,>=stealth,thick] (8) to (2);
\draw[->,>=stealth,thick] (9) to (2);
\draw[->,>=stealth,thick] (9) to (3);
\draw[->,>=stealth,thick] (10) to (2);
\draw[->,>=stealth,thick] (10) to (3);
\draw[->,>=stealth,thick] (10) to (4);
\draw[->,>=stealth,thick] (10) to (5);
\draw[->,>=stealth,thick] (10) to (6);
\draw[->,>=stealth,thick] (10) to (7);
\draw[->,>=stealth,thick] (11) to (6);
\draw[->,>=stealth,thick] (11) to (7);
\draw[->,>=stealth,thick] (12) to (7);
\draw[->,>=stealth,thick] (13) to (8);
\draw[->,>=stealth,thick] (13) to (9);
\draw[->,>=stealth,thick] (13) to (10);
\draw[->,>=stealth,thick] (13) to (11);
\draw[->,>=stealth,thick] (13) to (12);

\draw[->,>=stealth,thick] (15) to (13);
\end{tikzpicture}
\end{minipage}
%
%
\begin{minipage}{.55\textwidth}
\centering
\begin{tikzpicture}[node/.style = {ellipse, draw, inner sep = 1.5}, scale=.225]
\def\a{3.15}
\def\b{2*\a} 
\def\c{3*\a} 
\def\d{4*\a} 
\def\e{4.66*\a} 
\def\f{5.33*\a} 
\def\g{6*\a} 

\def\xa{0} 
\def\xb{4} 
\def\xc{8} 
\def\xd{12} 
\def\xe{14} 
\def\w{28}

\node[node] (0) at (\xc, 0) {\scriptsize $\cS_0: \lambda^0 \mu^0$};
\node[node] (1) at (\w-\xc, 0)
{\scriptsize $\cS_1: \lambda^0 \mu^0$};

\node[node] (2) at (\xc, \a)
{\scriptsize  };
\node[node] (3) at (\xd, \a)
{\scriptsize };
\node[node] (4) at (\w-\xd, \a)
{\scriptsize };
\node[node] (5) at (\w-\xc, \a)
{\scriptsize  };
\node[node] (6) at (\w-\xb, \a)
{\scriptsize};
\node[node] (7) at (\w-\xa, \a)
{\scriptsize};
\node[node] (8) at (\xa, \b)
{\scriptsize };
\node[node] (9) at (\xb, \b)
{\scriptsize };
\node[node] (10) at (\xc, \b)
{\scriptsize };
\node[node] (11) at (\xe, \b)
{\scriptsize $\cS_{11} : \lambda^2 \mu^1$};

\node[node] (12) at (\w-\xc, \b)
{\scriptsize };
\node[node] (13) at (\w-\xb, \b)
{\scriptsize };
\node[node] (14) at (\w-\xa, \b)
{\scriptsize };
\node[node] (15) at (\xb, \c)
{\scriptsize };
\node[node] (16) at (\xe, \c)
{\scriptsize $\cS_{15} : \lambda^3 \mu^2$};

\node[node] (17) at (\xb, \d)
{\scriptsize };
\node[node] (18) at (\xc, \d)
{\scriptsize };
\node[node] (19) at (\xe, \d)
{\scriptsize $\cS_{18} : \lambda^4 \mu^3$};
\node[node] (20) at (\w-\xc, \d)
{\scriptsize };
\node[node] (21) at (\w-\xb, \d)
{\scriptsize };
\node[node] (22) at (\xb, \e)
{\scriptsize };
\node[node] (23) at (\xc, \e)
{\scriptsize };
\node[node] (24) at (\xd, \e)
{\scriptsize };
\node[node] (25) at (\w-\xd, \e)
{\scriptsize };
\node[node] (26) at (\w-\xc, \e)
{\scriptsize };
\node[node] (27) at (\w-\xb, \e)
{\scriptsize };
\node[node] (28) at (\xd, \f)
{\scriptsize };
\node[node] (29) at (\w-\xd, \f)
{\scriptsize };

\draw[->,>=stealth,thick] (2) to (0);
\draw[->,>=stealth,thick] (3) to (0);
\draw[->,>=stealth,thick] (4) to (1);
\draw[->,>=stealth,thick] (5) to (1);
\draw[->,>=stealth,thick] (6) to (1);
\draw[->,>=stealth,thick] (7) to (1);
\draw[->,>=stealth,thick] (8) to (2);
\draw[->,>=stealth,thick] (9) to (2);
\draw[->,>=stealth,thick] (10) to (2);
\draw[->,>=stealth,thick] (11) to (2);
\draw[->,>=stealth,thick] (11) to (3);
\draw[->,>=stealth,thick] (11) to (4);
\draw[->,>=stealth,thick] (11) to (5);
\draw[->,>=stealth,thick] (11) to (6);
\draw[->,>=stealth,thick] (11) to (7);
\draw[->,>=stealth,thick] (12) to (4);
\draw[->,>=stealth,thick] (12) to (5);
\draw[->,>=stealth,thick] (13) to (5);
\draw[->,>=stealth,thick] (13) to (6);
\draw[->,>=stealth,thick] (14) to (6);
\draw[->,>=stealth,thick] (14) to (7);
\draw[->,>=stealth,thick] (15) to (8);
\draw[->,>=stealth,thick] (15) to (9);
\draw[->,>=stealth,thick] (16) to (8);
\draw[->,>=stealth,thick] (16) to (9);
\draw[->,>=stealth,thick] (16) to (10);
\draw[->,>=stealth,thick] (16) to (11);
\draw[->,>=stealth,thick] (16) to (12);
\draw[->,>=stealth,thick] (16) to (13);
\draw[->,>=stealth,thick] (16) to (14);
\draw[->,>=stealth,thick] (17) to (15);
\draw[->,>=stealth,thick] (17) to (16);
\draw[->,>=stealth,thick] (18) to (16);
\draw[->,>=stealth,thick] (19) to (16);
\draw[->,>=stealth,thick] (20) to (16);
\draw[->,>=stealth,thick] (21) to (16);
\draw[->,>=stealth,thick] (22) to (17);
\draw[->,>=stealth,thick] (22) to (19);
\draw[->,>=stealth,thick] (23) to (17);
\draw[->,>=stealth,thick] (23) to (18);
\draw[->,>=stealth,thick] (23) to (19);
\draw[->,>=stealth,thick] (24) to (19);
\draw[->,>=stealth,thick] (25) to (19);
\draw[->,>=stealth,thick] (26) to (19);
\draw[->,>=stealth,thick] (26) to (20);
\draw[->,>=stealth,thick] (27) to (19);
\draw[->,>=stealth,thick] (27) to (20);
\draw[->,>=stealth,thick] (27) to (21);
\draw[->,>=stealth,thick] (28) to (22);
\draw[->,>=stealth,thick] (28) to (23);
\draw[->,>=stealth,thick] (28) to (24);
\draw[->,>=stealth,thick] (29) to (25);
\draw[->,>=stealth,thick] (29) to (26);
\draw[->,>=stealth,thick] (29) to (27);
%
\end{tikzpicture}
\end{minipage}
\vspace{2ex}
\caption{Tessellar phase diagrams for $\beta = \sigma_2\sigma_1^2\sigma_2^2\sigma_1$ [left]
and $\beta= \sigma_1^2(\sigma_2\sigma_1)^2$ [right]. The posets $\sSC$ for both examples are different but the posets $\otessph\bigl(\sigma_2\sigma_1^2\sigma_2^2\sigma_1\bigr)$ and $\otessph\bigl(\sigma_1^2(\sigma_2\sigma_1)^2\bigr)$ are isomorphic.}\label{fig:pseudo1}
\end{figure}

%% file: postlude.tex
\chapter{Postlude}
\label{postlude}
In this section we will address open questions and directions for further research based on the ideas in this text.

\section{Topologization}
\label{topologization}
In this section we comment on
 the general scheme of modeling dynamics as topology. 

\subsection{Variations on flow topologies.}\label{exttopology} If we consider the $\tau$-\emph{forward-image} operator $\Gammatau$ by considering forward images from $t\ge \tau>0$
we obtain a derivative operator which is variation on $\der^+$. The derivative $\Gammatau$ is not idempotent in general. The fixed points of $\Gammatau$ comprise the invariant sets of $\varphi$. The same can be done for $(-\tau)$-backward images which yields the operators $\Gamma^-_{-\tau}$. 
As for the \bflt we can also consider operators 
$\Gamma^{-\tau}_\sqbullet := \Gamma^-_{-\tau}\cl$
and the topologies $\scrT^{-\tau}_\sqbullet$,
which are refinements of $\scrTbf$.
A different class of flow topologies can be derived by
composing closure and forward image in reversed order, e.g. $\Gamma_{-\tau}^\sqbullet:= \cl~\Gamma^-_{-\tau}$ 
is a derivative operator on $\sSet(X)$,
and associated topologies  $\scrTbfop$ and $\scrT_{-\tau}^\sqbullet$. If $\varphi$ is continuous then $\Gamma^\sqbullet_\tau$ defines a derivative operator on $\sSet(X)$.
%
%
%
%
The $\tau$-forward image derivative can be used to define
the omega limit set operator $\deromega U:=\omega(U)$.
A flow topology of particular interest is based on omega limit sets.
 Since $\omega(U)$ is forward invariant and closed it holds that $\omega(\omega(U)) \subset \omega(U)$. In addition,
$\omega$ is a additive and $\omega(\varnothing) = \varnothing$,
which proves that $\omega$ defines the derivative operator $\deromega$ on $\sSet(X)$.
Therefore, $U\mapsto \cl_\omega U := U\cup\deromega U$ is a closure operator and thus defines the topology $\scrT_\omega$ on $X$. 
Sets $U\subset X$ that are open in $\scrT$ and closed on $\scrT_\omega$ are exactly open attracting neighborhoods.\index{Omega limit set}
It is sometimes  convenient to consider trapping regions instead of attracting blocks. This can be achieved 
the operator:
$
U\mapsto \cl^+_\omega U:= U\cup \mathrm{\Gamma^+} U \cup \deromega U.
$

\subsection{Dynamical systems in the large.}\label{dyninlarge}
The flow topologies $\scrT^+$ and $\scrT^-$ and the derived flow topologies such as 
$\scrTbf$ and $\scrTbfop$ inherit properties of the semi-flow $\varphi$ at hand. For example if $\varphi$ is trivial\footnote{i.e. $\varphi(t,x)=x$ for all $x\in X$ and for all $t\ge 0$.} then $\scrTbf = \scrTbfop =\scrT$, and if $X$ allows a dense flow  $\varphi$ then for instance $\overline\scrT_0$ is the trivial topology.
In general, the flow topologies are not Hausdorff, independent of $\scrT$.
A natural question to ask is which topologies are manifested as a derived flow topologies. In this setting we can think of a dynamical system in larger terms as a relation on $X$.  This raises  a deeper question whether we can equivalently study certain aspects of dynamical systems in the large as bi-topological spaces and bi-closure algebras. For the latter closure operators  serve as a generalization of a dynamical system.

\subsection{Beyond semi-flows.}\label{beyonfsemiflow}
Most of the considerations in this paper carry over to discrete time dynamical systems such as iterating a map in $X$, cf. Remark \ref{rem:phi_continuity}. Another aspect that has played a minor role in the construction in this paper is the continuity of the dynamics with respect to the phase variable $x$. Finding discretizations, Morse pre-orders and Morse tessellation does not require continuity. The one instance where continuity plays a role is tying invariant dynamics to Morse tessellations. The latter uses algebraic topology and Wazewski's principle to conclude that non-trivial Borel-Moore homology yields non-trivial invariant dynamics. In particular, Wazewski's principle crucially uses the continuity of $\varphi$.  
For discrete time systems one can use a different notion of Conley index, cf.\ \cite{robbin:salamon2,mrozek10,richeson1,richeson2}.



\section{Discretization}\label{postlude:discretization}
In terms of discretization, there are many intersetions with finite and combinatorial topology.

\subsection{Quasi-isomorphic discretization and finite topologies} In practice (indeed, for all of the examples in this paper) the map $\disc\colon (X,\scrT)\to (\cX,\leq)$ is a quasi-isomorphism, i.e. induces an isomorphism $H(X)\cong H(\cX,\leq)$, where the latter is taken to be the singular homology of the finite topological space.  In this case, as per Remark \ref{rem:quasiiso}, one can discard $\disc$ and compute the graded tessellar complex $C^\ppart(\cX)$ from 
$\ppart\colon \cX\to \sP$ 
directly. This turns on the singular homology of $(\cX,\leq)$ being easy to understand, which is often the case, e.g., $\disc$ is a CW-decomposition map.  More generally, the singular homology of a finite topological space can be understood through a theorem of M. McCord~\cite{mccord1966singular}, which says that there is a simplicial complex (the order complex) $K(\cX)$ and a weak-homotopy equivalence $k\colon |K(\cX)|\to \cX$. This can be used to induce a new map 
$\ppart_k\colon K(\cX)\xrightarrow{k} \cX\to \sP$ 
and form a graded chain complex $C^{\ppart_k}(K(\cX))$. As $k$ is a weak homotopy equivalence (thus a quasi-isomorphism) $C^{\ppart_k}(K(\cX))$ has an isomorphic homology braid to $C^\ppart(\cX)$.  This provides another route for computation, and the role that finite topologies play seems worth examining.




\subsection{Semi-conjugacies and finite models.} 
A (bi-topological) discretization map  $\disc\colon (X,\scrT,\scrTbf)
\twoheadrightarrow (\cX,\le,\le^0)$ is analogous to a semi-conjugacy in that one has a model system onto which the system of interest is mapped.  There are situations in which it could be of interest to construct a semi-conjugacy of the semi-flow $\varphi$. The work of~\cite{mcmodels} provides results in this direction. If $\sSC$ obeys a particular condition
\footnote{Namely, what \cite{bjorner1984posets} terms a CW-poset, i.e. the face partial order of a regular CW complex.}
and further conditions
on the Conley indices and connection matrix $\dff^\tile$ are met,\footnote{Termed H0-H3 in~\cite{mcmodels}, these conditions ensure that $C^\tile(X)$ `looks-like' the Morse complex of a gradient flow.  Some of the braid examples from this text, after appropriate modification \`a la \cite{kkv}, satisfy these assumptions.}  
then one can construct a surjective semi-conjugacy from $X$ to $|K(\sSC)|$, the order complex associated to $\sSC$. Given that our starting point is $\tile$, rather than the $C^\tile(X)$ itself, one has more data than assumed in \cite{mcmodels} (in which only the nature of Conley indices and $\dff^\tile$ is known), and it seems reasonable that these assumptions could be weakened.



\section{Algebraization}

There are further refinements and generalizations of the homological algebraic techniques we use to build algebraic models of dynamics.

\subsection{Grading tessellar homology.}\label{Gtess}
In Remark \ref{dimensiongrading} we explained via spectral sequences that one obtains a $\sP$-grading on tessellar homology in the case that $\sP$ is a linear order. Applying this idea in the setting of parabolic recurrence relations yields the bi-graded parabolic homology. Via the bi-grading we   obtain refined information about invariant sets for  parabolic flows and connections between invariant sets. We have not fully explored the case of arbitrary partial orders $\sP$ is this setting, i.e. grading tessellar homology by an arbitrary partial order. To do so we need to explore the theory of spectral systems, cf.\ \cite{Matschke}.
The grading related to dynamics is denoted by $p$ and the grading related to topology by $q$. In some case it is beneficial to disregard the $q$-grading and only consider the dynamical grading $p$.

\subsection{Exact couple systems.}\label{ECS}
In \cite{Matschke} Matschke introduces exact couple systems as a generalization of Massey's exact couples which allows generalizing spectral sequences to spectral systems. In Section \ref{CEsystems}
we discuss Franzosa's theory of connection matrices in terms of exact couple systems and Cartan-Eilenberg systems, cf.\ \cite{SpV2}. 


\section{Braid and knot invariants}
\label{brandkn}
The application of discretization techniques to braids opens the door for investigating invariants in more general contexts.

\subsection{Braid Floer homology.}\label{BFH} In Section \ref{tesspara} we discuss an invariant for positive conjugacy classes of braid  in terms of a poset graded differential module.
Our techniques in this paper use parabolic flows which restrict to positive braid diagrams and positive conjugacy.
The general problem of braids is addressed in \cite{bgvw} and defines invariants for relative braid classes $x\rel y$ via Floer homology. 
One of the main results in \cite{bgvw} is that the homology invariants are isomorphic to the Floer homology of positive braids via Garside's normal form for braids.
The latter makes that the Floer homology invariants for relative braids can be computed from the discrete invariants introduced in Section \ref{tesspara}.
The problem of formulating a differential module invariant in the general case for braids based on Floer homology is much harder. An extension of the results in \cite{bgvw} is to define a differential module for  braids and show that it is an invariant of conjugacy classes of braids. A natural next step is to investigate the link to the discrete case in order to compute the differential module invariants. 
Another question in this setting is to understand how Markov moves impact the braid invariants which is important for investigating its relation with  knot invariants.

\subsection{Some immediate applications of the parabolic differential module.}\label{appdm}
The theory of parabolic recurrence relations has been successfully applied to scalar parabolic differential equations of the type
\[
u_t = u_{xx}+g(x,u,u_x).
\]
cf.\ \cite{gv}. A collection of stationary solutions to the above equation is regarded as a continuous, positive braid $\beta$. As in the discrete case one may consider various relative braids $\alpha\rel\beta$. In particular, when $\alpha$ represents a single strand braid the analogy with the discrete theory is obtained by representing $\alpha\rel \beta$ as a piecewise linear braid $x\rel y$. 
The results in \cite{gv} make $\scrA(\beta)$ also an invariant for the above parabolic equation. The reduced tessellar phase diagram $\otessph(\beta)$ yields stationary solutions for every vertex in $\otessph(\beta)$. Moreover,
in Conley index theory the boundary operator $\cm^\para$ contains information about connecting orbits between stationary solutions, cf.\ \cite{gv,day}. The highlight of the discrete theory is insight into the infinite dimensional Morse theory for the above parabolic equation.

The reduced tessellar phase diagram given by $\scrA(\beta)$ as provides detailed information about periodic points of surface diffeomorphism and diffeomorphisms of the 2-disc in particular.
In \cite{CzechV} mapping classes of diffeomorphisms of the 2-disc are related to braids and positive braids in particular. To obtain a forcing theory for additional periodic points the Conley indices of braids $\alpha\rel\beta$ is computed. As in the previous example the reduced tessellar phase diagram given by $\scrA(\beta)$ forces new periodic points.

%% file: binrelprops.tex
\chapter{Binary relations and operators}\index{Relation!see Binary relation}
\label{binrel}

In this appendix we summarize some elementary facts on binary relations and operators relevant for this text. We use \cite{Schmidt} as are main reference for the theory of binary relations and \cite{Gor} for operators.

\section{Binary relations}
\label{binrel122}
Let $X,Y$ be  point sets. A \emph{binary relation}\index{Binary relation} is a subset $\phi\subset X\times Y$. If $X=Y$, then $\phi\subset X\times X$ is called an \emph{endorelation}, or homogeneous (binary) relation on $X$.\index{Endorelation}\index{Binary relation!homogeneous} The top relation $X\times Y$ is denoted by $\top$ and the bottom relation $\varnothing$ by $\bot$. The identity is the diagonal $\id = \{(x,x)\mid x\in X\}$.
The \emph{opposite relation}, or \emph{inverse relation}\index{Binary relation!opposite}\index{Binary relation!inverse} is denote by $\phi^{-1} = \{(y,x)\mid (x,y)\in \phi\}\subset Y\times X$. The \emph{complement} $\phi^c$ is defined as
$\phi^c:= \{(x,y)\mid (x,y)\not\in \phi\}$, which is the set-complement of $\phi$.\index{Binary relation!complement}
Some obvious properties are:
\begin{enumerate}
    \item [(i)] $(\phi^{-1})^{-1}=\phi$;
    \item [(ii)] $(\phi^{-1})^c = (\phi^c)^{-1}$;
    \item [(iii)] $\phi \subset \psi$ if and only if $\phi^{-1} \subset \psi^{-1}$ if and only if $\phi^c \supset \psi^c$;
    \item [(iv)] $(\phi\cup \psi)^{-1} = \phi^{-1}\cup \psi^{-1}$ and $(\phi\cap \psi)^{-1} = \phi^{-1}\cap \psi^{-1}$;
    \item [(v)] $(\phi\cup \psi)^{c} = \phi^{c}\cap \psi^{c}$ and $(\phi\cap \psi)^{c} = \phi^{c}\cup \psi^{c}$ --- De Morgan's laws.
\end{enumerate}
An important operation on relations is \emph{composition}:\index{Binary relation!composition}
given $\phi\subset X\times Y$ and $\psi\subset Y\times Z$, then
\[
\psi\cdot \phi := \bigl\{(x,z)\mid (x,y)\in \phi\text{~~and~~} (y,z)\in \psi\text{~~~for some~~}y\in Y\bigr\}\subset X\times Z.
\]
Some additional properties:
\begin{enumerate}
    \item [(vi)] $(\psi\cdot \phi)^{-1} = \phi^{-1}\cdot \psi^{-1}$;
    \item [(vii)] $\phi \subset \phi\cdot \phi^{-1}\cdot \phi$ and $\phi^{-1} \subset \phi^{-1}\cdot \phi\cdot \phi^{-1}$.
\end{enumerate}
The set of all binary relations on $X\times Y$ is the power set $\sSet(X\times Y)$, which is a complete and atomic Boolean algebra.
In particular, for any $S\subset \sSet(X\times Y)$:
\begin{equation}
    \label{infcup}
\bigcup_{\phi\in S} \phi :=\sup\bigl\{\phi\in S\mid S\subset \sSet(X\times Y) \bigr\},
\end{equation}
is a well-defined relation in $\sSet(X\times Y)$. The same applies the infima:
\begin{equation}
    \label{infcap}
\bigcap_{\phi\in S} \phi :=\inf\bigl\{\phi\in S\mid S\subset \sSet(X\times Y) \bigr\},
\end{equation}
is a well-defined relation in $\sSet(X\times Y)$.

A binary relation $\phi$ is \emph{left total}\index{Binary relation!left total} if for all $x\in X$ there exists an $y\in Y$ such that $(x,y)\in \phi$. A binary relation $\phi$ is \emph{right total}, or \emph{surjective}\index{Binary relation!right total}\index{Binary relation!surjective} if for all $y\in Y$ there exists an $x\in X$ such that $(x,y)\in \phi$, i.e. $\phi^{-1}$ is left total. Composition with the opposite relation yields the following properties:
\begin{enumerate}
    \item [(viii)] $\phi$ is left total if and only if $\id \subset \phi^{-1}\cdot \phi$;
    \item [(ix)] $\phi$ is right total if and only if $\id \subset \phi\cdot \phi^{-1}$.
\end{enumerate}
Binary relations come with many different properties. In this text partial orders and pre-orders are special cases of homogeneous binary relations on $X$ that play a pivotal role in the theory.
For convenience we now consider homogeneous relations on $X$.
Instead of using the latter terminology we will refer to homogeneous binary relations as binary relations on $X$, which are elements of the Boolean algebra $\sSet(X\times X)$.
 A selection of special properties of binary relations on $X$ are:
\begin{enumerate}
    \item [-] \emph{reflexive}\index{Binary relation!reflexive} if $\id \subset \phi$;
    \item [-] \emph{irreflexive}\index{Binary relation!irreflexive} if $\id\subset \phi^c$;
    \item [-] \emph{symmetric}\index{Binary relation!symmetric} if $\phi = \phi^{-1}$;
    \item [-] \emph{asymmetric}\index{Binary relation!asymmetric} if $\phi^{-1}\subset \phi^c$, i.e. $(x,y)\in \phi$ implies $(y,x)\not\in \phi$;
    \item [-] \emph{anti-symmetric}\index{Binary relation!anti-symmetric} if $\phi\cap \phi^{-1}\subset \id$, i.e. $(x,y)\in \phi$ and $(y,x)\in \phi$ implies $x=y$;
    \item [-] \emph{transitive}\index{Binary relation!transitive} if $\phi^2\subset \phi$, i.e. $(x,y),(y,z)\in \phi$ implies that $(x,z)\in \phi$;
    \item [-] \emph{dense}\index{Binary relation!dense} if $\phi \subset \phi^2$, i.e. for every $(x,y)\in \phi$ there exists a $z\in X$ such that $(x,z),(z,y)\in \phi$.
\end{enumerate}
In particular, every reflexive relation is dense and  both left and right total.
With the above properties one can indicate a number of common binary relations on $X$. A binary relation $\phi$  on $X$ is a(n):
\begin{enumerate}
    \item [-] \emph{pre-order}\index{Pre-order} if $\phi$ is reflexive and transitive, i.e. $\id\subset \phi$ and $\phi^2=\phi$;
    \item [-] \emph{partial order}\index{Partial order}
    if $\phi$ is reflexive, anti-symmetric and transitive, i.e. $\id\subset \phi$, $\phi\cap\phi^{-1}\subset \id$ and $\phi^2=\phi$;
    \item [-] \emph{strict partial order}\index{Strict partial order}\index{Partial order!strict} if $\phi$ is irreflexive and transitive, which is equivalent to asymmetric and transitive;
    \item [-] \emph{linear order}, or \emph{total order}\index{Partial order!linear}\index{Partial order!total} if $\phi$ is a partial such that $(x,y)\in \phi$, or $(y,x)\in \phi$ for all pairs $(x,y)\in X\times Y$ (the last conditions equivalent to $\phi\cup \phi^{-1}=\top$);
    \item [-] \emph{equivalence relation}\index{Equivalence relation} if $\phi$ is reflexive, transitive and symmetric.
\end{enumerate}
For a binary relation on $X$ we use different notations: $(x,y)\in \phi$, which is equivalent to $x\,\phi\, y$. In particular for partial orders and pre-orders we write $x\le y$, or $x\le_\phi y$. For equivalence relations we use $x\sim y$, or $x\sim_\phi y$.
Given a partial order we can write the associated strict order and vice verse, i.e. given a partial order $\phi$, then $\phi_\bullet:= \phi\cap \id^c$ is the associated strict partial order, and given a strict partial order $\psi$, then $\psi^\bullet:= \psi\cup\id$ is the associated partial order. This yields the correspondences:
\[
(\phi_\bullet)^\bullet=\phi,\quad\quad (\psi^\bullet)_\bullet=\psi.
\]
A \emph{Hasse relation}, or \emph{Hasse diagram}\index{Hasse diagram}\index{Hasse relation} for a partial order $\phi$ is defined as:
\[
\phi_H := \phi_\bullet \cap (\phi_\bullet^2)^c,
\]
and corresponds to the usual notion of Hasse diagram for partial orders on a finite set $X$. Transitive reflexive closure, which we now explain, of the Hasse relation  retrieves the partial order.

Given a binary relation $\phi$ on $X$, then its \emph{transitive closure}\index{Binary relation!transitive closure} is defined as
\[
\phi^{\bm{+}} := \inf\bigl\{\psi\subset X\times X \mid \phi\subset \psi,~~\psi^2\subset \psi\bigr\} = \bigcup_{k\ge 1} \phi^k. 
\]
The notion of \emph{transitive reduction}\index{Binary relation!transitive reduction} only makes sense for finite sets $X$ but is not well-defined for infinite sets in general.
The \emph{transitive reflexive closure}\index{Binary relation!transitive reflexive closure} is defined as
\[
\phi^{\bm{+=}} := \inf\bigl\{\psi\subset X\times X \mid \id\cup \phi\subset \psi,~~\psi^2\subset \psi\bigr\} = \bigcup_{k\ge 0} \phi^k. 
\]
The transitive reflexive closure $\phi^{\bm{+=}}$ of a binary relation $\phi$ is pre-order on $X$.
Via \emph{reflexive closure}\index{Binary relation!rflexive closure} $\phi^{\bm{=}}:=\id\cup\phi$ we have that 
$\phi^{\bm{+=}}= \id\cup \phi^{\bm{+}}$. \emph{Reflexive reduction}\index{Binary relation!reflexive reduction} is defined for all binary relations and is given as $\phi^{\bm{\neq}} := \phi\cap \id^c$.
Finally, the \emph{strongly connected components}\index{Strongly connected component}\index{Binary relation!strongly connected component} of a binary relation on $X$ are given by the equivalence relation
\begin{equation}
    \label{SCbin}
\phi_\sSC :=\phi^{\bm{+=}} \cap \bigl(\phi^{\bm{+=}}\bigr)^{-1},
\end{equation}
whose equivalence classes are the strongly connected components of $\phi$.
The latter is of particular importance for binary relations on finite sets (digraphs), cf.\ Rem.\ \ref{perspectives}. The pre-order $\phi^{\bm{+=}}$ yields a partial order on the strongly connected components, cf.\ App.\ \ref{posets} (ordered tessellations).

\section{Modal operators}\index{Operator}
\label{operators}
Related to binary relations $\phi\subset X\times Y$ is the notion of operator on the algebra of subsets of $X$ and $Y$.
Let $\phi\subset X\times Y$ be a binary relation. Then, for $U\subset X$ define
\[
\phi\, U := \bigl\{ y\in Y\mid (x,y)\in \phi~\text{~~for some~~} x\in U\bigr\} =\bigcup_{x\in U}\phi\,x \subset Y,
\]
where $\phi\, x = \bigl\{ y\in Y\mid (x,y)\in \phi\bigr\}$.
By definition $\phi\,\varnothing=\varnothing$ and $\phi(U\cup U')
=\phi\,U\cup \phi\,U'$, which show that $\phi$ regarded as operator is a \emph{modal operator}\index{Modal operator} from $\sSet(X)$ to $\sSet(Y)$.
In this text the operators are mostly from $\sSet(X)$ to $\sSet(X)$.
The following properties follow from the definition of operator:
\begin{enumerate}
    \item [(i)] $\phi\bigl( \bigcup_{i\in I}U_i\bigr) =
    \bigcup_{i\in I} \phi\, U_i$;
    \item [(ii)] $\phi\bigl( \bigcap_{i\in I}U_i\bigr) \subset
    \bigcap_{i\in I} \phi\, U_i$, both for arbitrary families $\{U_i\}_{i\in I}$ of subsets in $X$.
\end{enumerate}
The modal operator defined by a binary relation are completely additive.\index{Modal operator!completely additive}
The opposite relation $\phi^{-1}$ regarded as operator is related to $\phi$ as operator in a similar way as the inverse function to a function. First of all (i)-(ii) also holds for $\phi^{-1}$ as operator. Moreover,
\begin{enumerate}
 \item [(iii)] $U\cap \phi^{-1}Y \subset \phi^{-1}\bigl( \phi\,U\bigr)$, for all $U\subset X$;
    \item [(iv)] $V\cap \phi\,X \subset \phi\bigl( \phi^{-1}V\bigr)$, for all $V\subset Y$.
\end{enumerate}
\begin{remark}
    If $\phi$ is left total then \ref{operators}(iii) corresponds to \ref{binrel122}(viii), and \ref{operators}(iv) corresponds to \ref{binrel122}(ix). Note that if $V=\phi U$ in \ref{operators}(iv), then the identity \ref{binrel122}(vii) is satisfied. The same holds for \ref{operators}(iii).
\end{remark}

To prove (iii) and (iv) we argue as follows.
The set $\phi^{-1}Y = \bigl\{ x\in X\mid (y,x)\in \phi^{-1}\text{~~for some~~} y\in Y\bigr\} = \bigl\{ x\in X\mid (x,y)\in \phi\text{~~for some~~} y\in Y\bigr\}$ is the domain $X$ for which $\phi\,x\neq\varnothing$. Futhermore,
\[
\begin{aligned}
    \phi\,U &= \bigl\{ y\in Y\mid (x,y)\in \phi\text{~~for some~~} x\in U\bigr\}\\
    &= \bigl\{ y\in Y\mid (x,y)\in \phi\text{~~for some~~} x\in U\cap \phi^{-1}Y\bigr\}= \phi(U\cap \phi^{-1}Y),
\end{aligned}
\]
and thus $\phi$ may be regarded as a left total relation on $\phi^{-1}Y\times Y$. 
For any subset $U\subset X$, $U\cap \phi^{-1}Y$ is a subset of 
$\phi^{-1}Y$. By \ref{binrel122}(viii) this gives
\[
U\cap \phi^{-1} Y \subset \phi^{-1} \bigl(\phi(U\cap \phi^{-1} Y) \bigr)= \phi^{-1} (\phi\,U),
\]
which proves (iii). Statement (iv) is proved by changing the role of $\phi$ and $\phi^{-1}$.

A useful identity for $\phi^{-1}U$ is given by
\begin{equation}
    \label{oppform}
    \phi^{-1}V = \bigl\{x\in X\mid \phi\,x\cap V\neq\varnothing \bigr\},\quad V\subset Y.
\end{equation}
Indeed, 
\[
\begin{aligned}
    \phi^{-1}V &= \bigl\{ x\in X\mid (y,x)\in \phi^{-1}\text{~~for some~~} y\in V\bigr\}\\
    &= \bigl\{ x\in X\mid (x,y)\in \phi\text{~~for some~~} y\in V\bigr\}\\
    &= \bigl\{ x\in X\mid y\in\phi\,x\text{~~for some~~} y\in V\bigr\}\\
    &=\bigl\{x\in X\mid \phi\,x\cap V\neq\varnothing \bigr\}.
\end{aligned}
\]
As operator $\bar \phi:=\phi^{-1}$ is also referred to as \emph{conjugate operator}.\index{Operator!conjugate}
Of other crucial importance is the \emph{dual operator} related to $\phi$.\index{Dual operator}\index{Operator!dual}
Define,
\begin{equation}
    \label{dualop123}
    \phi^*U := \bigl(\phi\,U^c \bigr)^c,\quad U\subset X.
\end{equation}
\begin{remark}
    In general this definition works for modal operators on Boolean algebras, not just completely additive operators as discussed in this appendix. For example $\Int$ is the dual operator of $\cl$ on $\sSet(X)$. If $\cl$ is defined via $\phi$, i.e. Alexandrov topology, then the conjugate operator is ${\mathrm{star}}$, cf.\ Sect.\ \ref{closalgdisc12}.\index{Star}\index{Conjugate closure operator}\index{Closure operator!conjugate}
\end{remark}
By the same token we define the dual of $\phi^{-1}$:
\begin{equation}
    \label{dualop456}
    \phi^{-*}V := \bigl(\phi^{-1}V^c \bigr)^c,\quad V\subset Y.
\end{equation}
Useful identities in this setting are:
\begin{enumerate}
\item [(v)] $\phi^{*}U^c = \bigl(\phi\,U\bigr)^c$ and $\bigl(\phi^{*}U\bigr)^c = \phi\,U^c$ for all $U\subset X$;
    \item [(vi)] $\phi^{-*}V^c = \bigl(\phi^{-1}V\bigr)^c$ and $\bigl(\phi^{-*}V\bigr)^c = \phi^{-1}V^c$ for all $V\subset Y$. 
\end{enumerate}
As for $\phi^{-1}$ there is a convenient characterization of $\phi^{-*}$:
\begin{equation}
    \label{invform}
    \phi^{-*}V = \bigl\{x\in X\mid \phi\,x\subset V \bigr\},\quad V\subset Y.
\end{equation}
Indeed, by \eqref{oppform} and \eqref{dualop456}
\[
\phi^{-*}V = \bigl( \phi^{-1}V^c\bigr)^c=
\bigl\{ x\in X\mid \phi\,x\cap V^c\neq \varnothing\bigr\}^c
= \bigl\{x\in X\mid \phi\,x\subset V \bigr\},
\]
which proves \eqref{invform}.
Note that by the latter characterization we have $\phi^{-*}V \subset \phi^{-1}V$ for all $V\subset Y$. This gives:
\begin{enumerate}
    \item [(vii)] $\bigl( \phi^{-*}V\bigr)^c \supset \phi^{-*} V^c$;
    \item [(viii)] $\bigl( \phi^{-1}V\bigr)^c \subset \phi^{-1} V^c$, both for all $V\subset Y$.
\end{enumerate}
\begin{remark}
    \label{alsoforphi}
    The characterizations in \eqref{oppform} and \eqref{invform} also hold for $\phi$ and $\phi^*$ by simply replacing $\phi$ by $\phi^{-1}$. This yields the same identities in (vii)-(viii) for $\phi$ and $\phi^*$.
\end{remark}
The modal operators defined by binary relations are completely additive. As a consequence of the definition of dual the latter are \emph{completely multiplicative}.\index{Modal operator!completely multiplicative}
\begin{enumerate}
    \item [(ix)] $\phi^*\bigl( \bigcup_{i\in I}U_i\bigr) \supset
    \bigcup_{i\in I} \phi^* U_i$;
    \item [(x)] $\phi^*\bigl( \bigcap_{i\in I}U_i\bigr) =
    \bigcap_{i\in I} \phi^* U_i$, both for arbitrary families $\{U_i\}_{i\in I}$ of subsets in $X$.
\end{enumerate}
The same holds for $\phi^{-*}$. For example (x) is derived as follows:
\[
\phi^*\Bigl( \bigcap_{i\in I}U_i\Bigr)=
\Bigl( \phi\bigl(\bigcap_{i\in I}U_i\bigr)^c\Bigr)^c=
\Bigl( \phi\bigl(\bigcup_{i\in I}U_i^c\bigr)\Bigr)^c
= \Bigl( \bigcup_{i\in I}\phi\, U_i^c\Bigr)^c=
\bigcap_{i\in I}\bigl( \phi\,U_i^c\bigr)^c =  \bigcap_{i\in I} \phi^* U_i.
\]
Due to the definition of dual there are additional properties with respect to composition in contrast to (iii) and (iv):
\begin{enumerate}
    \item [(xi)] $U \subset \phi^{-*}\bigl( \phi\,U\bigr)$, for all $U\subset X$;
    \item [(xii)] $V \supset \phi\bigl( \phi^{-*}V\bigr)$, for all $V\subset Y$.
\end{enumerate}
To prove (xi) observe that
\[
\phi^{-*}\bigl( \phi\,U\bigr) = \bigl\{ x\in X\mid \phi\,x\subset \phi\,U\bigr\} \supset U.
\]
As for (xii) we have:
\[
\begin{aligned}
\phi\bigl( \phi^{-*}V\bigr) &= \bigl\{ y\in Y\mid (x,y)\in \phi~\text{~~for some~~} x\in \phi^{-*}V\bigr\}\\
&=\bigl\{ y\in Y\mid (x,y)\in \phi~\text{~~for some~~}x\text{~~such that~~} \phi\,x\subset V\bigr\}\subset V,
\end{aligned}
\]
i.e. $y\in \phi\,x\subset V$, which proves (xii).
\begin{remark}
In the special case that $\phi$ is given by a function $f\colon X\to Y$, i.e. $\phi = \{(x,y)\mid y=f(x)\}$, then the opposite relation $\phi^{-1}$ is given by: $\phi^{-1}=\{(y,x)\mid y=f(x)\}$. As operators we have $\phi^{-1}V$ and $\phi^{-*}V$ and,
since $\{f(x)\}$ is a singleton set,
\[
\phi^{-1}V = \big\{x\in X\mid \{f(x)\}\cap V\neq\varnothing \bigr\} = \big\{x\in X\mid \{f(x)\}\subset V \bigr\}
=\phi^{-*}V.
\]
We write
$=f^{-1}V = \phi^{-1}V=\phi^{-*}V$. For composition this implies by (iv) and (xii):
\[
f(X)\cap V = \phi\,X \cap V \subset \phi(\phi^{-1}V)=
\phi(\phi^{-*}V) \subset V,
\]
and therefore $f(X)\cap V\subset f(f^{-1}V) \subset V$, which yields identity exactly when $f$ is surjective. Similarly we have  $U\subset f^{-1}f(U)$.
\end{remark}
\begin{remark}
Properties of binary relations imply properties on operators. For instance an operator that satisfies $\phi^2U\subset \phi\,U$ for all $U\subset X$, corresponds to a transitive relation.
\end{remark}

\section{Duality}
\label{complBoolAlg}
In this section we recall the duality for binary relations on $X$ and operators on $\sSet(X)$. The main source of reference for this section is \cite{Goldblatt}.
The prime example of a complete and atomic Boolean algebra\index{Complete Boolean algebra}\index{Boolean algebra!complete and atomic} is the power set of a point set $X$ denoted by $\sSet(X)$. The latter is a Boolean algebra with respect to intersection, union and complement, and is closed under arbitrary intersections and unions,
and every element is the union of atoms $\{x\}$, $x\in X$.
Let $Y$ be another point set. A Boolean homomorphism\index{Boolean homomorphism} is a lattice homomorphism that preserves the units, i.e. $\varnothing$ and $X$.
A Boolean homomorphism $\Phi\colon \sSet(Y)\to \sSet(X)$ between the complete and atomic Boolean algebras $\sSet(Y)$ and $\sSet(X)$ is \emph{completely additive}\index{Boolean homomorphism!completely additive}\index{Completely additive Boolean homomorphism} if
\begin{equation}
    \label{mod123}
\Phi\Bigl(\bigcup_{i\in I} U_i\Bigr) = \bigcup_{i\in I} \Phi U_i,
\end{equation}
for  arbitrary unions $\bigcup_{i\in I} U_i$, $U_i\subset X$.
Since $\Phi$ is Boolean the above complete additivity is equivalent to 
\begin{equation}
    \label{mod456}
\Phi\Bigl(\bigcap_{i\in I} U_i\Bigr) = \bigcap_{i\in I} \Phi U_i,
\end{equation}
for  arbitrary intersections $\bigcap_{i\in I} U_i$, $U_i\subset X$.
Related to $\Phi$ we define a binary relation on $X\times Y$:
\begin{equation}
    \label{firstdefnbinrel}
    (x,y)\in \phi\quad\text{if and only if}\quad x\in \Phi\{y\}.
\end{equation}
\begin{lemma}
\label{whyphi}
$\Phi = \phi^{-1}$ as operators from $\sSet(Y)$ to $\sSet(X)$.
\end{lemma}
\begin{proof}
    By defintion
    \[
    (x,y)\in \phi\quad\text{if and only if}\quad 
    (y,x)\in \phi^{-1}\quad\text{if and only if}\quad x\in \phi^{-1}\{y\}.
    \]
    This implies that $\Phi\{y\} = \phi^{-1}\{y\}$ for all $y\in Y$. By the complete additivity of $\Phi$ we have:
    \[
    \phi^{-1}V = \phi^{-1}\Bigl( \bigcup_{y\in V}\{y\}\Bigr)
    =\bigcup_{y\in V}\phi^{-1}\{y\} = \bigcup_{y\in V}\Phi\{y\}
    =\Phi V,\quad V\subset Y.
    \]
    Moreover, $\phi^{-1}\varnothing=\Phi\varnothing=\varnothing$,
    which completes the proof.
\end{proof}
In the Boolean setting the properties on $\Phi$ yield a restriction on $\phi$ as is displayed in the following result.
\begin{proposition}
    \label{charcompladd}
    If $\Phi\colon \sSet(Y) \to \sSet(X)$ is completely additive Boolean homomorphism\index{Boolean homomorphism!completely additive}, then $\Phi=\phi^{-1}$, where
    \[
    \phi = \bigl\{(x,f(x)\mid f\colon X\to Y \bigr\},
    \]
    and $(x,y)\in \phi$ uniquely determines $f(x)=y$. We say that the binary relation $\phi$ is functional\index{Binary relation!functional} and we write $\Phi = f^{-1}$.
\end{proposition}
\begin{proof}
    For $y\neq y'$ we have $\Phi\{y\}\cap \Phi\{y'\} = \Phi\bigl( \{y\}\cap \{y'\}\bigr) =\Phi\varnothing= \varnothing$.
    Also $X=\Phi Y =\Phi\bigl(\bigcup \{y\} \bigr)= \bigcup \Phi\{y\}$, which uses the complete additivity. This implies that given $x\in X$, then there exists a unique $y\in Y$ such that $x\in \Phi\{y\}$ since the sets $\Phi\{y\}$ are disjoint for distinct $y\in Y$.
    The latter defines the functional relation $\phi$ and $f(x) =y$, i.e. the points in
    $\phi$ are given by the pairs $(x,f(x))$.
\end{proof}

\begin{remark}
    The definition of $\phi$ related to $\Phi$ in \eqref{firstdefnbinrel} is a choice that will be crucial in our treatment of duality. Some authors use \eqref{firstdefnbinrel} to define $\phi^{-1}$. Our choice in consistent with the conventions in the logics liturature, cf.\ \cite{Goldblatt}. A more compelling reason for the above choice is Proposition \ref{charcompladd}:
    \[
    \Phi\colon \sSet(Y) \to \sSet(X)\quad\iff\quad \Phi=\phi^{-1}, \text{~~with~} \phi=\bigl\{(x,f(x)\mid f\colon X\to Y \bigr\},
    \]
    which characterizes all completely additive Boolean homomorphisms.
\end{remark}

The notion of Boolean homomorphism can be weakened  to the notion of modal operator.
Let $\Phi\colon \sSet(Y) \to \sSet(X)$ be a \emph{completely additive modal operator},\index{Modal operator!completely additive} i.e. $\Phi$ satisfies $\Phi\varnothing=\varnothing$ and the additivity condition in \eqref{mod123}.
By the definition in \eqref{firstdefnbinrel} and Lemma \ref{whyphi} every completely additive operator $\Phi$ uniquely determines a binary relation
$\phi\subset X\times Y$. However, since $\Phi$ is not necessarily Boolean the relation $\phi$ is not necessarily functional.
Instead, we obtain any binary relation on $X\times Y$. We use the notation:
\begin{equation}
   \label{JT222} 
\Phi = \phi^{-1}\colon\sSet(Y)\to\sSet(X),\quad\text{and}\quad \phi = \Phi^{-1}\subset X\times Y.
\end{equation}
In particular,
\begin{equation}
   \label{JT1} 
(\phi^{-1})^{-1} = \phi,,\quad\text{and}\quad (\Phi^{-1})^{-1}=\Phi.
\end{equation}
The duality between operators and  relations explained in this appendix applies to complete and atomic Boolean algebras. In \cite{JonssonTarski} and \cite{Halmos} this duality is extended to arbitrary
Boolean algebras and a special classes of binary relations --- \emph{Boolean relations}.\index{Boolean relation} 
In the case of a single Boolean algebra with (modal) operator $(\sB,\bvtheta)$ (not necessarily completely additive) is embedded in a complete and atomic Boolean algebra with with a completely additive extension of $\bvtheta$.

%% file: order2.tex
\chapter{Order Theory}
In this appendix we outline some of the most prominent concepts of order theory that are used in this text. Our main references are \cite{davey:priestley}, \cite{roman}
and \cite{JonssonTarski}.
\section{Posets and pre-orders}
\label{posets}
In Appendix \ref{binrel} we defined binary relations and in particular partial orders and pre-orders. In the setting of finite sets we will focus here on finite partially ordered sets, or \emph{posets} which we will denote by $(\sP,\le)$, where $\sP$ is a finite set and $\le$ a partial order (or pre-order).
If there is no ambiguity on the partial order we typically denote a finite poset\index{Poset}\index{Poset!finite} by $\sP$, or $\sQ$.
A function $\nu \colon\sP\to \sQ$ between finite posets is \emph{order-preserving} if $p\leq q$ implies that $\nu(p)\leq \nu(q)$.\index{Order-preserving map}\index{Poset!order-preserving map}

The \emph{category of finite posets}, denoted $\bFPoset$,\index{$\bFPoset$} is the category whose objects are finite posets and whose morphisms are order-preserving maps.
The category of finite pre-orders is denoted by $\bFPreOrd$.\index{$\bFPreOrd$}

Let $\sP$ be a finite poset. An {\em up-set}\index{Up-set} of $\sP$ is a subset $I\subset \sP$ such that if $p\in \sU$ and $p\leq q$ then $q\in I$.  For $p\in \sP$ the {\em up-set at $p$} is $\big \uparrow p:=\{q\in \sP:p \leq q\}$
which is also called a {\em principal up-set}.\index{Principal up-set}  Following~\cite{lsa}, we denote the collection of up-sets by $\sU(\sP)$.  A {\em down-set}\index{down-set} of $\sP$ is a set $I\subset \sP$ such that if $q\in D$ and $p\leq q$ then $p\in I$.  The {\em down-set at $q$} is $\big \downarrow  q:=\{p\in \sP: p \leq q\}$ which is called a {\em principal down-set}.\index{Principal down-set}    Following~\cite{lsa}, we denote the collection of down-sets by $\sO(\sP)$.

For $p,q\in \sP$ the {\em interval}\index{Interval} from $p$ to $q$, denoted $[p,q], $ is the set $\{r\in \sP: p\leq r \leq q\}$. 
A subset $I\subset \sP$ is {\em convex}\index{Convex set} if whenever $p,q\in I$ then $[p,q]\subset I$. 
Every convex set is of the form $\alpha\smin\beta$ with $\alpha,\beta\in \sO(\sP)$.
We denote the collection of convex sets by $\sCo(\sP)$. 
 Every convex set of $\sP$ can be obtained as intersection of a down-set and an up-set.  Under a poset morphism the preimage of a convex set is a convex set, cf.\ \cite{roman}.

If $\sP$ is a pre-order, then the equivalence classes are ordered as $[p]\le [q]$ if and only if $p\le q$. The poset of equivalence class $\sP/_\sim$ is called an \emph{ordered tessellation}\index{Ordered tessellation} for $\sP$.
The latter is also referred to as an \emph{ordered partition}\index{Ordered partition} of $\sP$

\section{Lattices}\label{sec:prelims:lat}\index{Lattice}
Some texts introduce lattices as a particular type of poset.  Instead, we begin with the definition of lattice as an algebraic structure. For a discussion of the relationship of these two definitions the reader may consult~\cite[Chapter 2]{davey:priestley}, in particular~\cite[Theorem 2.9]{davey:priestley}.

\begin{definition}
A {\em lattice} is a set $\sL$ with the binary operations $\vee,\wedge \colon \sL\times \sL\to \sL$ satisfying the following four axioms:

\begin{enumerate}
\item[(i)]  $a\wedge a = a \vee a = a$ for all $a\in \sL$ (idempotence);
\item[(ii)] $a\wedge b = b\wedge a$ and $a\vee b = b \vee a$ for all $a,b\in \sL$  (commutativity);
\item[(iii)] $a\wedge (b\wedge c) = (a\wedge b)\wedge c$ and $a\vee(b\vee c) = (a\vee b)\vee c$ for all $a,b,c\in \sL$ (associativity);
\item[(iv)] $a\wedge (a\vee b) = a\vee (a\wedge b)=a$ for all $a,b\in \sL$ (absorption law).
\end{enumerate}
\noindent A lattice $\sL$ has an associated poset structure given by $a\leq b$ if $a=a\wedge b$ or, equivalently, if $b=a\vee b$.
A lattice $\sL$ is {\em distributive} if it satisfies the additional axiom:
\begin{enumerate}
\item[(v)] $a\wedge (b\vee c) = (a\wedge b)\vee (a\wedge c)$ and $a\vee (b\wedge c) = (a\vee b) \wedge (a\vee c)$ for all $a,b,c\in \sL$ (distributivity);\index{Distributive law}
\end{enumerate}

\noindent A lattice $\sL$ is {\em bounded}\index{Lattice!bounded} if there exist {\em neutral}\index{Neutral elements} elements $0$ and $1$ that satisfy the following property:
\begin{enumerate}
\item[(vi)] $0\wedge a = 0, 0\vee a = a, 1\wedge a = a, 1\vee a = 1$ for all $a\in \sL$.
\end{enumerate}
\end{definition}

A {\em complemented lattice},\index{Lattice!complemented} also called a \emph{Boolean algebra},\index{Boolean algebra} is a bounded lattice (with least element $0$ and greatest element $1$), in which every element $a$ has a complement, i.e. an element $b$ such that $a\vee b = 1$ and $a\wedge b= 0$.

A lattice homomorphism\index{Lattice homomorphism} $f \colon \sL\to \sM$ is a map such that if $a,b\in \sL$ then $f(a\wedge b) = f(a)\wedge f(b)$ and $f(a\vee b) = f(a)\vee f(b)$.  
If $\sL$ and $\sM$ are bounded lattices then we also require that $f(0)=0$ and $f(1)=1$.  In particular, we are interested in finite lattices.  Every finite lattice is bounded.  A subset $\sK\subset \sL$ is  a sublattice\index{Sublattice}\index{Lattice!sublattice of a} of $\sL$ if $a,b\in \sK$ implies that $a\vee b\in \sK$ and $a\wedge b\in \sK$. For sublattices of bounded lattices we impose the extra condition that $0,1\in \sK$.
%
The \emph{category of finite distributive lattices}, denoted $\bFDLat$,\index{$\bFDLat$} is the category whose objects are finite distributive lattices and whose morphisms are lattice homomorphisms.

An element $a\in \sL$ is {\em join-irreducible}\index{Join-irreducible element} if it has a unique \emph{immediate predecessor};\index{Immediate predecessor} given a join-irreducible $a$ we denote its unique predecessor by $a^\pred$.
The set of join-irreducible elements of $\sL$ is denoted by $\sJ(\sL)$. The join-irreducible elements form a poset $(\sJ(\sL),\leq)$, where the order $\leq$ is the restriction of the partial order of $\sL$.

\section{Birkhoff duality}\label{sec:birkhoff}
Given a finite distributive lattice $\sL$,  $\sJ(\sL)$ is a poset   with respect to set-inclusion.
Conversely, given a finite poset $(\sP,\leq)$ the set of downsets $\sO(\sP)$ is a bounded distributive lattice under $\wedge = \cap$ and $\vee = \cup$.  The following theorem often goes under the moniker `Birkhoff's Representation Theorem'\index{Birkhoff's Representation Theorem}
and the duality will be referred to as \emph{Birkhoff duality}.\index{Birkhoff duality}

\begin{theorem}[cf.\ \cite{roman}, Theorem 10.4 and \cite{kkv}]\label{thm:birkhoff}
The applications  $\sL \Rightarrow \sJ(\sL)$ and $\sP \Rightarrow \sO(\sP)$
 are contravariant functors from $\bFDLat$ to $\bFPoset$ and from $\bFPoset$ to $\bFDLat$ respectively.
A lattice homomorphism $h\colon \sK\to \sL$ is dual to an order-preserving map $\sJ(h)\colon \sJ(\sL) \to \sJ(\sK)$ and an order-preserving map $\nu\colon \sP \to \sQ$ is dual to a lattice homomorphism $\sO(\nu)\colon\sO(\sQ) \to \sO(\sP)$ given by the formulas
 This may be represent this via the following diagram:
\[
\begin{tikzcd}
\sK \ar[dd,"h",swap] && \sJ(\sK)\\
{} \arrow[rr,shorten >= 12pt,shorten <= 12pt,Rightarrow,"\sJ"]  && {}\\
\sL && \sJ(\sL)\ar[uu,swap,"\sJ(h)"]
\end{tikzcd}\quad\quad\quad
\begin{tikzcd}
\sP \ar[dd,"\nu",swap]&& \sO(\sP)\\
{} \arrow[rr,,shorten >= 12pt,shorten <= 12pt,Rightarrow,"\sO"]  && {}\\
\sQ && \sO(\sQ)\ar[uu,swap,"\sO(\nu)"]
\end{tikzcd}
\]
\begin{align*}\label{eqn:birkhoff:form}
&& \sJ(h)(a) &= \min h^{-1}\left(\big \uparrow a\right), & &\text{where } a\in \sJ(\sL), &\\
&& \sO(\nu)(\alpha)&= \nu^{-1}(\alpha), &  &\text{where } \alpha\in \sO(\sQ), &
\end{align*}
respectively.
Furthermore,
\[
\sL\cong \sO(\sJ(\sL))\quad\text{and}\quad \sP\cong \sJ(\sO(\sP)).
\]
\end{theorem}

The pair of contravariant functors $\sO$ and $\sJ$ called the {\em Birkhoff transforms}.   Given $\nu\colon \sP\to \sQ$ we say that $\sO(\nu)$ is the {\em Birkhoff dual}\index{Birkhoff dual} to $\nu$.  Similarly, for $h\colon \sK\to \sL$ we say that $\sJ(h)$ is the {\em Birkhoff dual} to $h$. A lattice homomorphism $h$ is injective if and only if $\sJ(h)$ is surjective, and $h$ is surjective if and only if $\sJ(h)$ is an order-embedding.

\begin{remark}
\label{extraformdual}
If $\sL = \sO(\sP)$ and $\sK=\sO(\sQ)$ then the homomorphism $\sJ(h)\colon \sP\to \sQ$ is given by the formula $\sJ(h)(p) = \min \bigl\{ q\in \sQ\mid p\in h\bigl(\big\downarrow q \bigr)\bigr\}$, cf.\ \cite[Thm.\ 5.19]{davey:priestley}.
\end{remark}

%% file: appendixCMT.tex

\chapter{Grading, filtering and differential modules}
\label{appendix:gfcm}
In this appendix we explain gradings and filterings in the context order theory applied to sets and modules
In this setting we discuss the duality between gradings and filterings in the spirits of Birkhoff's representation theorem.

\section{$\sP$-gradings and $\sO(\sP)$-filterings}
\label{gradfilt}
We start with 
the definitions of grading and filtering on a set $X$ which carry over to the setting of modules in Appendix \ref{gradedvs}.
\begin{definition}
\label{ordtessdefn}
An \emph{ordered tessellation}, or \emph{ordered partition}\index{Ordered tessellation}\index{Ordered partition} of a set $X$ is a 
poset $(\sT,\le)$ consisting of  non-empty subsets  $T\subset X$, 
such that
\begin{enumerate}
    \item [(i)] $T\cap T'=\varnothing$ for all distinct $T, T'\in \sT$;
    \item [(ii)] $\bigcup_{T\in \sT} T=X$.
\end{enumerate}
%
One can also consider more general binary relations on tessellations such as pre-orders.
\end{definition}
Given the set of down-sets $\sO(\sT)$ in $(\sT,\le)$ we  define the lattice:
\[
\sN(\sT) := \Bigl\{N\subset X\mid N=\bigcup_{T\in \alpha} T,~\alpha\in \sO(\sT)\Bigr\},
\]
which is a finite sublattice of $\sSet(X)$ with  binary operations  $\cap$ and $\cup$. 
\begin{definition}
\label{filteringnew}
A finite sublattice $\sN\subset \sSet(X)$ is called a \emph{filtering}\footnote{We use term \emph{filtering}\index{Filtering}\index{Filtering!of a set $X$} for an arbitrary sublattice of subsets in $\sSet(X)$ as opposed to \emph{filtration} which is commonly used for a linearly ordered sublattice.} on $X$.
\end{definition}
The sublattice $\sN(\sT)$ is a filtering on $X$, constructed from $\sT$.
If we start with a filtering $\sN\subset \sSet(X)$ we  construct an ordered tessellation from $\sN$ as follows.
For $N\in \sJ(\sN)$ define $T:= N\smin N^\pred$, where $N^\pred$ is the unique immediate predecessor of $N$ in the lattice $\sN$. From \cite{lsa3} it follows that for distinct $N,N'\in \sJ(\sN)$, then (i) $T,T'\neq \varnothing $, (ii) $T\cap T'=\varnothing$, and (iii) $\bigcup_{N\in \sJ(\sN)} N\smin N^\pred =X$.
We order the tiles $T$ as: $T\le T'$ if and only if $N\subseteq N'$. This makes $\bigl(\sT(\sN),\le\bigr) \cong (\sJ(\sN),\subset)$ an ordered tessellation of $X$.
By Birkhoff duality we conclude: 
\[\sN\bigl(\sT(\sN)\bigr) = \sN\quad \text{and}\quad
\sT\bigl(\sN(\sT)\bigr) = \sT.
\]
Filterings regarded as finite sublattices of subsets in $X$ provide an algebraic point of view for the construction ordered tessellations of $X$.
An \emph{$\sO(\sP)$-filtering}\index{$\sO(\sP)$-filtering} on  $X$ is a lattice homomorphism 
\[
\flt\colon \sO(\sP) \to \sSet(X),
\]
and the pair $(X,\flt)$ is called an \emph{$\sO(\sP)$-filtered set}.
The image $\sN\subset\sSet(X)$ of $\flt$ is a filtering on $X$.
 Common
 notation for an $\sO(\sP)$-filtering is:  $\alpha \mapsto  F_\alpha X$.
Given an $\sO(\sP)$-filtering $\flt\colon \sO(\sP) \to \sSet(X)$,
with image $\sN$,
then Birkhoff duality as described above 
yields the order-embedding 
$\iota\colon\sT(\sN)\hookrightarrow \sP$ of the induced ordered tessellation:
 \begin{equation}
     \label{grdeqn1}
    T \xmapsto{\iota} \min\Bigl\{p\in \sP\mid T \subset F_{\downarrow p}X\Bigr\}.
  \end{equation}
  The latter defines a discretization map\index{Discretization map} (not necessarily surjective, nor continuous) $\grade\colon X\to (\sP,\le)$ via: 
  \[
  \grade(x) := p, \quad x\in \iota^{-1}(p).
  \]
  The discretization map $\grade$ is called a \emph{$\sP$-grading}\index{$\sP$-grading} on $X$ and the subsets
  $G_pX=\grade^{-1} p$ yield a decomposition
 \begin{equation}
    \label{Pgradnot1}
    X = \bigcup_{p\in \sP} G_p X,
\end{equation}
which we refer to as a \emph{$\sP$-graded decomposition}\index{Graded decomposition}\index{Graded decomposition!of $X$} of $X$. The set $G_pX$ is also called the subsets of \emph{homogeneous elements of degree}\index{Homogeneous elements}\index{Homogeneous elements! of degree $p$} $p$. The non-empty sets $G_pX$ form a ordered tessellation of $X$.
In this construction the $\sP$-grading $\grade$ is induced by the filtering $\flt$.
In general any discretization map $\grade\colon X\to \sP$ yields a 
$\sP$-graded decomposition of $X$.
Given a $\sP$-grading $\grade\colon X\to \sP$, Birkhoff duality implies a  lattice homomorphism $\grade^{-1}\colon\sO(\sP) \twoheadrightarrow \sO(\sT)$. The formula
\begin{equation}
    \label{grdeqn2}
    \flt\colon \sO(\sP) \to \sN(\sT),\quad
    \alpha \mapsto F_\alpha X:= \bigcup \bigl\{ T\mid T\in  \grade^{-1}(\alpha)\bigr\},
\end{equation}
yields an $\sO(\sP)$-filtering on $X$, which establishes teh duality between $\sP$-gradings and $\sO(\sP)$-filterings of $X$.

\section{$\sP$-graded and $\sO(\sP)$-filtered modules}
\label{gradedvs}
In the spirit of gradings and filterings of a set $X$ we can do the same for $R$-modules.
Let $C$ be an $R$-module, or module for short, over a ring $R$.\index{$R$-module}\index{Module} The submodule\index{Submodule}
of $C$ are denoted by $\Sub\,C$ with binary operations $\cap$ and $+$ (span).
An \emph{$\sO(\sP)$-filtering}\index{$\sO(\sP)$-filtering!on a module $C$}\index{Filtering!on a module $C$} on $C$ is a lattice homomorphism 
\[
\flt\colon \sO(\sP) \to \Sub\,C,
\]
and the pair $(C,\flt)$ is called an \emph{$\sO(\sP)$-filtered}\index{$\sO(\sP)$-filtered module} module.
Common  notation for an $\sO(\sP)$-filtering is $\alpha\mapsto F_\alpha C$. 
For an $\sO(\sP)$-filtering $\sO(\sP)\twoheadrightarrow \Sub\,C$ define the external direct sum
\begin{equation}
    \label{quotdecmp}
\Gr C = \bigoplus_{\alpha\in \sJ(\sO(\sP))} \frac{F_\alpha C}{F_{\alpha^\pred}C} \cong \bigoplus_{p\in \sP} \frac{F_{\downarrow p} C}{F_{\downarrow p^\pred}C},
\end{equation}
where we use the fact that $F_\alpha C/F_\beta C\cong F_{\alpha'} C/F_{\beta'}C$ for all $\alpha\smin\beta = \alpha'\smin\beta'$.
The module $\Gr C$ is the \emph{associated graded module},\index{Associated graded module} cf.\ \cite{Hilton} and \cite{robbin:salamon2}.
In general $\Gr C$ is not isomorphic to $C$.
If the subquotients $G_pC$ are free, then $C$ is a free module and 
$\Gr C\cong C$. This always holds if $C$ is $\K$-vector space.
The decomposition in \eqref{quotdecmp} is called a \emph{$\sP$-graded decomposition} (of $\Gr C$).\index{$\sP$-graded decomposition!of a module}\index{Module!$\sP$-graded decomposition}
The subquotients $G_p C:= F_\alpha C/F_{\alpha^\pred} C$  are called \emph{factors}\index{Factor!of a grading} and
since $\flt$ is not necessarily injective some factors $F_\alpha C/F_{\alpha^\pred} C$ may be trivial, i.e. the zero module.
In general a decomposition 
\begin{equation}
    \label{quotdecmp345}
C =  \bigoplus_{p\in \sP} G_p C, \quad G_pC=\frac{F_{\downarrow p} C}{F_{\downarrow p^\pred}C},
\end{equation}
is called a $\sP$-graded decomposition of $C$. The element in $G_p C$ are called homogeneous elements of degree $p$. As before $G_pC$ may be trivial for some $p$.
Factors are also well-defined for any convex set $\beta\smin\alpha$ and are denoted by $G_{\beta\smin\alpha} C :=F_\alpha C/F_\beta C$.

A $\sP$-graded module $C =  \bigoplus_{p\in \sP} G_p C$ yields an $\sO(\sP)$-filtered module in a canonical way:
\begin{equation}
\label{firstdecomp1}
\flt\colon \sO(\sP)\to \Sub\,C,\quad \alpha \mapsto  F_\alpha V := \bigoplus_{p\in \alpha} G_p C,
\end{equation}
which is denoted by $(C,\flt)$.
If $C$ is a $\sP$-graded module then $\Gr C$ is  defined as before via the induces $\sO(\sP)$-filtered module and 
\[
\Gr C \cong C,
\]
which establishes one of the dualities for arbitrary modules.

\section{Differential modules}
\label{gradedcellchain}
The concept of $\sP$-grading can be applied to chain complexes and differential modules/vector spaces. 
A \emph{differential module}\index{Differential module} is a pair $(C,\dff)$ where $C$ is an $R$-module  and $\dff\colon C\to C$ is an endomorphism  satisfying $\dff^2=0$. If $C$ is a vector space, then we refer to $(C,\dff)$ as a \emph{differential vector space}.
We refer to the elements in $C$ as \emph{chains}.\index{Chain}
The chains in $C$ for which $\dff$ vanishes are called \emph{cycles}\index{Cycle} and are denoted by $Z(C,\dff)\subset C$. Chains in the range of $\dff$ are called \emph{boundaries}\index{Boundary} and are denoted by $B(C,\dff)\subset Z$. The \emph{homology}\index{Homology} of $(C,\dff)$ is define as $H(C,\dff) := Z(C)/B(C)$.
A homomorphism $h\colon C\to C'$ of $R$-modules is a \emph{D-homomorphism} if $\dff' h = h \dff$.

\begin{definition}
\label{Pchaincompl}
A {\em $\sP$-graded differential module}\footnote{A chain complex $(C,\dff)$ is an $\N$-graded, or $\Z$-graded differential module with the differential $\dff$ a degree $-1$ map, i.e. $\dff F_{\downarrow p} C \subset F_{\downarrow (p-1)} C$. In the literature $\sP$-graded differential modules are also simply called \emph{differential graded modules}.\index{Differential graded module}}\index{$\sP$-graded differential module}\index{Differential vector space} $(C,\dff)$
is given by an $\sP$-graded module $C= \bigoplus_{p\in \alpha} G_p C$ such that
\begin{equation}
    \label{filtdiffmod12}
    \dff F_\alpha C \subset F_\alpha C,\quad\forall \alpha\in \sO(\sP),
\end{equation}
which is equivalent to saying that $\dff$ is $\sO(\sP)$-filtered.
A $\sP$-graded differential module $C$ is \emph{strict} if $\dff|_{G_pC}=0$ for all $p\in \sP$.\index{Differential!strict}\index{Strict differential}   
More generally an \emph{$\sO(\sP)$-filtered differential module}\index{$\sO(\sP)$-filtered differential module} $(C,\dff)$ is given by a $\sO(\sP)$-filtered module $C$ equipped with an $\sO(\sP)$-filtered differential, i.e. $\dff$ satisfies 
\eqref{filtdiffmod12}.
\end{definition}

The differential $\dff$ on a $\sP$-graded module may be regarded as \emph{upper-triangular}\index{Upper-triangular} with entries $\dff(p,q)\colon G_q C \to G_p C$ due \eqref{filtdiffmod12}.
The latter implies
\[
\dff(p,q)\neq 0,\quad \implies\quad p\le q.
\]

As before a $\sP$-graded differential module also yields an $\sO(\sP)$-filtering on $(C,\dff)$ making the latter an $\sO(\sP)$-filtered differential module. The converse is not true in general. For example is we use field coefficients then an $\sO(\sP)$-filtered differential module is isomorphic to a $\sP$-graded differential module as described in the previous section.

 An {\em $\sO(\sP)$-filtered $D$-homomorphism}\index{$\sO(\sP)$-filtered $D$-homomorphism}
between $\sP$-graded differential modules 
 is a $D$-homomorphism $h\colon (C,\dff)\to (C',\dff')$, such that $h\colon C\to C'$ is an $\sO(\sP)$-filtered homomorphism, i.e. $h\bigl( F_\alpha C\bigr)\subset F_\alpha C$ for all $\alpha\in \sO(\sP)$.  

\begin{remark}
    \label{homhom}
    In our treatment of graded and filtered differential modules we assume that the differential is filtered as well as the homomorphisms. This allows more flexibility. For example for a (co)chain complex the differential is homogeneous of degree $\pm 1$. This implies that the differential is also filtered. The converse is not true.
\end{remark}